\documentclass[10pt]{article} 
\newcommand{\ifomaprintti}[1]{} %
\newcommand{\ifeiomaprintti}[1]{#1} %
\newcommand{\ifeiIEOT}[1]{#1} 
\newcommand{\ifIEOT}[1]{} 
\newcommand{\ifeiproofjo}[1]{#1} 
\usepackage{isolatin1,t1enc} %
\usepackage[finnish,english]{babel} %

\newcount\theTime
\newcount\theHour
\newcount\theMinute
\newcount\theMinuteTens
\newcount\theScratch
\theTime=\number\time
\theHour=\theTime
\divide\theHour by 60
\theScratch=\theHour
\multiply\theScratch by 60
\theMinute=\theTime
\advance\theMinute by -\theScratch
\theMinuteTens=\theMinute
\divide\theMinuteTens by 10
\theScratch=\theMinuteTens
\multiply\theScratch by 10
\advance\theMinute by -\theScratch
\def\timeHHMM{{\number\theHour:\number\theMinuteTens\number\theMinute}}

\ifomaprintti{
\textheight 24.0cm
\oddsidemargin -.2cm
\textwidth 12.5cm %
\topmargin -0.5cm
} %
\ifomaprintti{
\textheight 27.0cm %
\topmargin -2.1cm
\ifprintlabels{ %
\let\oldlabel=\label    %
\renewcommand{\label}[1]{\smash{\fboxrule.01pt\raise11pt\rlap{\fbox{\small $#1$}}}\oldlabel{#1}} %
\let\oldref=\ref %
\renewcommand{\ref}[1]{\smash{\fboxrule.01pt\raise11pt\rlap{{\tiny $\scriptsize #1$}}}\oldref{#1}}
\let\oldpageref=\pageref %
\renewcommand{\pageref}[1]{\smash{\fboxrule.01pt\raise11pt\rlap{{\tiny $\scriptsize #1$}}}\oldpageref{#1}}

}%

\pagestyle{myheadings}

\makeatletter
      \long\def\sectionmark#1{%
      }
\makeatother

} %

\renewcommand{\sectionmark}[1]{\markright{\thesection\ #1}}
\ifomaprintti{
\renewcommand{\sectionmark}[1]{\markright{\thesection\ \  #1\ 
 \ \ \tiednimi \ \ \   km@iki.fi\ \ \ \ \ 
 \today\ \timeHHMM}}
} %

\ifeiproofjo{ 
\newenvironment{proof}[1][Proof:]{#1}{} %
\renewenvironment{proof}[1][Proof:]{\par\noindent\begingroup{\bf #1\ }%
 }%
{\nopagebreak[4]\null\hspace{2em}\hfill\lower.3ex\hbox{$\square$}\linebreak
\vspace{-8.0ex}\null\hfill\linebreak\endgroup}
} %

\usepackage{mathrsfs} %
\usepackage{graphicx}
\usepackage{psfrag} %
\usepackage{oldgerm} %
\usepackage{amssymb} %

\usepackage{amsmath} %
\usepackage{isolatin1} %
\usepackage{rotating} %
\usepackage{epic} %
\usepackage{ifthen} %

\newtheorem{lemma}[subsection]{Lemma}

\newtheorem{theorem}[subsection]{Theorem}
\newtheorem{proposition}[subsection]{Proposition}
\newtheorem{cor}[subsection]{Corollary}
\newtheorem{defin}[subsection]{Definition}

\newtheorem{remark}[subsection]{Remark} 
\newtheorem{exampleTh}[subsection]{Example} %
\newtheorem{remarksTh}[subsection]{Remarks} %
\newtheorem{discussionTh}[subsection]{Discussion} %
\newtheorem{terminologyTh}[subsection]{Terminology} %
\newtheorem{algorithmTh}[subsection]{Algorithm} %
\newenvironment{example}[1][\dummyargumentti]{\ifx#1\dummyargumentti\begin{exampleTh}\rm\else\begin{exampleTh}[#1]\rm\fi}{\hfill$\triangleleft$\end{exampleTh}} %
\newcommand\dummyargumentti{dummyargumentti}

\newenvironment{remark-rm}[1][\dummyargumentti]{\ifx#1\dummyargumentti\begin{remark}\rm\else\begin{remark}[#1]\rm\fi}{\hfill$\triangleleft$\end{remark}} %

\newtheorem{applicationTh}[subsection]{Application} %

\newtheorem{standinghypothesis}[subsection]{Standing Hypothesis}{} %
{} %
{}  %
{} %
{} %
{} %
{} %

\newcommand{\Dom}{\mathop{\rm Dom}\nolimits}     %

\newcommand{\im}{\mathop{\rm Im}\nolimits}       %

\newcommand{\Ran}{\mathop{\rm Ran}\nolimits}     %
\newcommand{\re}{\mathop{\rm Re}\nolimits}       %
\newcommand{\wlim}{\mathop{\rm w\mbox{\rm -}lim}} %
\newcommand{\slim}{\mathop{\rm s\mbox{\rm -}lim}} %
\newcommand{\W}{\hbox{$\Vert\hspace{-.08em}\vert$}} %

\newcommand{\subc}{\linebreak[0]\hbox{ $\sub\lower.45ex\hbox{$\below{\,\hspace{.05em}c}$}$ } \linebreak[3] }
\newcommand{\subcd}{\linebreak[0]\hbox{ $\sub\lower.45ex\hbox{$\below{\hspace{-.00em}
	c\hspace{-.09em},\hspace{-.10em}d}$}\!$ } \linebreak[3] }
\newcommand{\subd}{\linebreak[0]\hbox{ $\sub\lower.45ex\hbox{$\below{\,\hspace{.05em}d}$}$ } \linebreak[3] }
\newcommand{\subsd}{\linebreak[0]\hbox{ $\sub\lower.45ex\hbox{$\below{\hspace{-.00em}sd}$}$ } \linebreak[3] }      

\newcommand{\subcsd}{\linebreak[0]\hbox{ $\sub\lower.45ex\hbox{$\below{\hspace{-.10em}
	c\hspace{-.10em},\hspace{-.14em}s\hspace{-.09em}d}$}\!\!$ } \linebreak[3] }

\newcommand{\led}{\,\raise.39ex\hbox{\,$\le$\,}\lower.30ex\hbox{$ _{_{\hspace{-0.75em} d }}\ $}\,\hspace{.02em}}
\newcommand{\lesd}{\,\raise.39ex\hbox{\,$\le$\,}\lower.30ex\hbox{$ _{_{\hspace{-0.94em} sd }}\ $}\,\hspace{.02em}}

\newcommand{\levect}{\,\raise.39ex\hbox{\,$\le$\,}\lower.36ex\hbox{$ _{_{\hspace{-1.15em}{\rm vect}}}$}\hspace{.02em}}
\newcommand{\eqvect}{=_{_{\hspace{-1.11em} {\rm \lower.35ex\hbox{\tiny vect\hspace{-.44em}}}}}}%
\newcommand{\lealg}{\,\raise.39ex\hbox{\,$\le$\,}\lower.36ex\hbox{$ _{_{\hspace{-.99em}{\rm alg}}}$}\hspace{.26em}}
\newcommand{\leTVS}{\,\raise.42ex\hbox{\,$\le$\,}\lower.39ex\hbox{$ _{_{\hspace{-1.15em}{\rm TVS}}}$}\hspace{.02em}}
\newcommand{\eqTVS}{=_{_{\hspace{-1.11em} {\rm \lower.35ex\hbox{\tiny TVS\hspace{-.44em}}}}}}%
\newcommand{\mkb}{_{_{\hspace{-1em} {\rm \lower.05ex\hbox{\tiny m.k.\hspace{-.6em}}}}}}%
\newcommand{\aeb}{_{_{\hspace{-1em} {\rm \lower.05ex\hbox{\tiny a.e.\hspace{-.6em}}}}}}%
\newcommand{\mkle}{\le _{_{\hspace{-1.0em} {\rm \lower.85ex\hbox{\tiny m.k.}}}}}%
\newcommand{\aele}{\le _{_{\hspace{-1.0em} {\rm \lower.85ex\hbox{\tiny a.e.}}}}}%
\newcommand{\mkeq}{=_{_{\hspace{-1em} {\rm \lower.05ex\hbox{\tiny m.k.\hspace{-.6em}}}}}}%
\newcommand{\aeeq}{=_{_{\hspace{-1em} {\rm \lower.05ex\hbox{\tiny m.k.\hspace{-.6em}}}}}}%
\newcommand{\loc}{{\rm loc}}

 {\nopagebreak[4]\null\hspace{2em}\hfill\lower.3ex\hbox{$\square$}\linebreak
  \vspace{-8.0ex}\null\hfill\linebreak\endgroup}

\newcommand{\talteen}[1]{} %
\newcommand{\dummy}[1]{} %
\newcommand{\extra}[1]{} %
\newcommand{\wrong}[1]{} %
\newcommand{\obsolete}[1]{} %
\newcommand{\toistoa}[1]{} %
\newcommand{\unused}[1]{} %
\newcommand{\p}[1]{{\ \!\hspace{-.20em}\langle\hspace{-.03em}%
#1\hspace{-.06em}\rangle\ \!\hspace{-.20em}}}
\newcommand{\raj}[1]{\lower.2ex\hbox{$|$}\lower.35ex\hbox{}_{#1}} %
\newcommand{\below}[1]{ _{_{\hspace{-1.0em} #1 }}\ } %
\newenvironment{itemlist}%
{%
\topsep\itemlisttopsep
\begin{description}%
\itemsep\itemlistitemsep
\advance\leftskip\itemlistleftskip
\itemindent\itemlistfirstlineskip %
}%
{\end{description}} 
\newcommand{\itemlisttopsep}{1.3ex plus 1pt minus 2pt}
\newcommand{\itemlistitemsep}{-.2ex} %
\newcommand{\itemlistleftskip}{-.2em}
\newcommand{\itemlistfirstlineskip}{-1.8em} %
{\vspace{-2.5ex}\begin{description} \itemsep-.7ex}%
{\end{description}\vspace{-5ex}\hfill\linebreak} %
{\vspace{-2.5ex}\begin{description} \itemsep-.7ex}%
{\end{description}\vspace{-7ex}\hfill\linebreak} %
{\hfill\linebreak %
\vspace{-5.4ex}\begin{description} \itemsep-.7ex} %
{\end{description}\vspace{-5ex}\hfill\linebreak} %
{\hfill\linebreak %
\vspace{-5.4ex}\begin{description} \itemsep-.7ex} %
{\end{description}\vspace{-7ex}\hfill\linebreak} %
\newcommand{\all}{\forall}
\newcommand{\ex}{\exists}
\newcommand{\notex}{\raise.23ex\hbox{$\not$}\,\hspace{-.06em}\ex}
\newcommand{\eps}{\epsilon}
\newcommand{\del}{\delta}
\newcommand{\khi}{{\raise.45ex\hbox{$\chi$}}} %

\newcommand{\Hoo}{\mathop{\H^\infty}} %

\let\ccal=\oldmathcal %

\newcommand{\cA}{{\ccal A}}
\newcommand{\cB}{{\ccal B}} %
\newcommand{\cC}{{\ccal C}} %
\newcommand{\cG}{{\ccal G}} %
\newcommand{\cJ}{{\ccal J}}
\newcommand{\cLL}{ {{\ccal L}^{\!\!\!\!\!\! \smile }} } %

\newcommand{\cP}{{\ccal P}} %
 
\newcommand{\cR}{{\ccal R}} %
\newcommand{\cU}{{\ccal U}}

\newcommand{\cY}{{\ccal Y}}

\newcommand{\I}{\,\big\vert\,} %
\newcommand{\sub}{\subset}
\newcommand{\pois}{\setminus}
\newcommand{\tyhja}{\emptyset}

\newcommand{\lland}{\ \&\ }

\newcommand{\THEN}{\Rightarrow}
\newcommand{\TTHEN}{\,\Longrightarrow\,} %
\newcommand{\TTTHEN}{ \ \Longrightarrow \ }
\newcommand{\IF}{\Leftarrow}

\newcommand{\IFF}{\Leftrightarrow}

\newcommand{\IIIFF}{\ \Longleftrightarrow\ }

\makeatletter
\@ifundefined{QED}{}{}
\makeatother
\newcommand{\C}{{\mathbb C}}
\newcommand{\N}{{\mathbb N}}

\newcommand{\R}{{\mathbb R}}
\newcommand{\Z}{{\mathbb Z}}
\renewcommand{\raj}[1]{\lower.6ex\hbox{$|$}\lower.75ex\hbox{}_{#1}} 
\newcommand{\strong}{{\rm strong}} %
\newcommand{\sint}{\mathop{\raise.24ex\hbox{\rm\smalls}\hspace{-.05em}\!\!\!\!\!\int}} %
\newcommand{\wint}{\mathop{\raise.12ex\hbox{\rm\smallw}\hspace{-.20em}\!\!\!\!\!\int}} %

\newcommand{\diag}{\mathop{\rm diag}}    %

\renewcommand{\cLL}{{{\ccal L}^{\!\!\!\!\!\!\!\!\!\hspace{.10ex}\raise-.03ex\hbox{$\scriptscriptstyle\smile$}}}} %
\newcommand{\Lap}{\cLL} %

\renewcommand{\khi}{{\raise.1ex\hbox{$\chi$}}} %
\newcommand{\tv}{\tilde{v}}
\newcommand{\tx}{\tilde{x}}

\newcommand{\teta}{\tilde{\eta}}

\newcommand{\hh}{\widehat{h}}
\newcommand{\hF}{\widehat{F}}
\newcommand{\hG}{\widehat{G}}

\newcommand{\cCo}{\cC_{0}} %
\renewcommand{\L}{{\rm L}} %
\renewcommand{\W}{{\rm W}} %
\newcommand{\Lc}{\L_{\rm c}} %
\newcommand{\La}{\Lambda} %
\newcommand{\efn}{{\rm e}} %
\newcommand{\Lstrong}{\L_\strong} %

\newcommand{\Loostrong}{\L^\infty_\strong}  %
\newcommand{\Lstrongloc}{\L_{\strong,\loc}} %

\renewcommand{\H}{{\rm H}} %
\newcommand{\Hoooo}{\H^\infty_\infty} %
\newcommand{\Hstrong}{\H_\strong} %
\newcommand{\Lloc}{\L_\loc} %

\newcommand{\noproof}{\nopagebreak[4]\nolinebreak[4]\null{}\phantom{}\hfill\lower.7ex\hbox{$\square$}} %
\newcommand{\noproofup}{\nopagebreak[4]\nolinebreak[4]\null{}\phantom{}\hfill\raise0.0ex\hbox{$\square$}} %
\newcommand{\itemlistnoproof}{ \ \nopagebreak[4]\nolinebreak[4]\null{}\phantom{}\vskip-7ex\hfill{$\square$}} %

\newcommand{\indterm}[1]{{\em #1}\index{#1}} %
\newcommand{\indtermem}[1]{{\em #1}\index{#1|emph}} %

\let\smalla=a
\let\smallb=b
\let\smallc=c
\let\smalld=d %
\let\smalle=e
\let\smallf=f 
\let\smallg=g 
\let\smallh=h 
\let\smalli=i 
\let\smallj=j 
\let\smallk=k
\let\smalll=l 
\let\smallm=m 
\let\smalln=n
\let\smallo=o 
\let\smallp=p 
\let\smallq=q
\let\smallr=r 
\let\smalls=s 
\let\smallt=t 
\let\smallu=u
\let\smallv=v
\let\smallw=w
\let\smallx=x
\let\smally=y 
\let\smallz=z
\newcommand{\twoone}[2]{\stackrel{\scriptstyle #1}{\scriptstyle #2}} %
\newcommand{\threeone}[3]{\stackrel{\twoone{#1}{#2}}{\scriptstyle #3}}

\newcommand{\twotwo}[4]{\twoone{#1}{#3}\,\twoone{#2}{#4}}

\newcommand{\threetwo}[6]{\threeone{#1}{#3}{#5}\,\threeone{#2}{#4}{#6}}
\newcommand{\threethree}[9]{\threeone{#1}{#4}{#7}\,\,\threeone{#2}{#5}{#8}\,\,\threeone{#3}{#6}{#9}}
\newcommand{\ctwoone}[2]{\lower.5ex\hbox{\ensuremath{\twoone{#1}{#2}}}}
\newcommand{\ctwotwo}[4]{\lower.5ex\hbox{\ensuremath{\twotwo{#1}{#2}{#3}{#4}}}}
\newcommand{\ctwothree}[6]{\lower.5ex\hbox{\ensuremath{\twotwo{#1}{#2}{#3}{#4}{#5}{#6}}}}
\newcommand{\cthreeone}[3]{\lower1ex\hbox{\ensuremath{\threeone{#1}{#2}{#3}}}}
\newcommand{\cthreetwo}[6]{\lower1ex\hbox{\ensuremath{\threetwo{#1}{#2}{#3}{#4}{#5}{#6}}}}
\newcommand{\cthreethree}[9]{\lower1ex\hbox{\ensuremath{\threethree{#1}{#2}{#3}{#4}{#5}{#6}{#7}{#8}{#9}}}}

\newcommand{\btwotwo}[4]{\left[\ctwotwo{#1}{#2}{#3}{#4}\right]}

\newcommand{\rmd}{{\rm d}} %

\newcommand{\MCSS}[1]{} %
\newcommand{\NotesFor}[1]{{\large\bf Notes for #1\par}} %
\newcommand{\oK}{{\qK}} %
\newcommand{\rN}{{\qN}} %
\newcommand{\rM}{{\qM}} %
\newcommand{\lN}{{\tqN}} %
\newcommand{\lM}{{\tqM}} %
\newcommand{\tALS}{{\widetilde{\ALS}}} %
\newcommand{\Ric}{\cP} %
\newcommand{\tRic}{\stilde{\cP}} %
\newcommand{\ALSExtGRP}{{\ALS_{\rm Joint}}} %
\newcommand{\ALStotal}{\ALSExtGRP} 

\newcommand{\qDo}{\qD^o}

\newcommand{\qBd}{\qB_{\rm d}}  %
\newcommand{\dU}{\cU}  %
\newcommand{\dY}{\cY}  %
\newcommand{\dUout}{\dU_{\rm out}} %

\newcommand{\dUexp}{\dU_{\rm exp}}       

\newcommand{\dUstr}{\dU_{\rm str}}       
\newcommand{\gUg}{\dU_*^*}        %

\newcommand{\qRR}{\qR}  %
\newcommand{\qQR}{\qQ}  %

\newcommand{\qRRClL}{\qRClL}  %
\newcommand{\w}{{\rm w}} %
\newcommand{\rms}{{\rm s}} %
\newcommand{\rmu}{{\rm u}} 
\newcommand{\rmc}{{\rm c}} %
\def\bsysm#1{\begin{sysmatrix}#1\end{sysmatrix}}
\def\bsysbm#1{\left[\sysm{#1}\right]}
\def\bsyspm#1{\left(\sysm{#1}\right)}

\def\msysm#1{\raise.3ex\hbox{$\begin{medsysmatrix}#1\end{medsysmatrix}$}}
\def\msysbm#1{[\msysm{#1}]} 

\makeatother 
\makeatletter

\def\etv{& \hskip-.3em\vrule\hskip-.3em &} 
\def\smalletv{&\vrule&} 
\def\crh{\vrule height0pt depth5pt width0pt \cr
 \noalign{\hrule}
 \vrule height12pt depth0pt width0pt} 
\def\smallcrh{\vrule height0pt depth2\ex@ width0pt
        \cr\noalign{\hrule}
        \vrule height6.5\ex@ depth0pt width0pt}
\newcommand\crharg[1]{\vrule height0pt depth2\ex@ width0pt
        \cr\noalign{\hrule}
        \vrule height#1\ex@ depth0pt width0pt} 
\newbox\smallstrutbox
\setbox\smallstrutbox=\hbox{\vrule height6pt depth1.5pt width\z@}
\def\smallstrut{\relax\ifmmode\copy\smallstrutbox\else\unhcopy\smallstrutbox\fi}

\def\tinyetv{&\vrule&} 
\let\oldcr=\cr
\newcommand\tinycrh{\vrule height0pt depth1\ex@ width0pt 
        \oldcr\noalign{\hrule}
        \vrule height5.5\ex@ depth0pt width0pt} 
\let\tinycr=\cr 
\def\semitinycrh{\vrule height0pt depth1\ex@ width0pt 
        \cr\noalign{\hrule}
        \vrule height6.5\ex@ depth0pt width0pt} 
\newcommand\supertinycrh{\vrule height0pt depth1\ex@ width0pt 
        \oldcr\noalign{\hrule}
        \vrule height5\ex@ depth0pt width0pt} 

\newbox\tinystrutbox
\setbox\tinystrutbox=\hbox{\vrule height888pt depth888pt width\z@} 
\def\tinystrut{\relax\ifmmode\copy\tinystrutbox\else\unhcopy\tinystrutbox\fi}

\newenvironment{sysmatrix}{
 \let\|=\etv
 \hskip \arraycolsep
 \begin{matrix}}
 {\end{matrix}
 \hskip \arraycolsep
}       

\def\sysm#1{\begin{sysmatrix}#1\end{sysmatrix}}

\newenvironment{smallsysmatrix}{\null\,\vcenter\bgroup
  \let\|=\smalletv
  \let\crh=\smallcrh
  \def\\{\smallstrut\math@cr}
  \restore@math@cr\default@tag
  \baselineskip\z@skip \lineskip\z@skip \lineskiplimit\lineskip
  \ialign\bgroup\hfil$\m@th\scriptstyle##$\hfil&&\thickspace\hfil
  $\m@th\scriptstyle##$\hfil\crcr
  \crcr\noalign{\vskip -.3\ex@}%
}{\crcr\noalign{\vskip -.2\ex@}%
  \crcr\egroup\egroup\,%
} 

\newbox\medstrutbox
\setbox\medstrutbox=\hbox{\vrule height33pt depth95pt width\z@} 
\def\medstrut{\relax\ifmmode\copy\medstrutbox\else\unhcopy\medstrutbox\fi}

\newenvironment{medsysmatrix}{\null\,\vcenter\bgroup 
  \let\|=\etv
  \def\\{\medstrut\math@cr}
  \restore@math@cr\default@tag
  \baselineskip\z@skip \lineskip\z@skip \lineskiplimit\lineskip
  \ialign\bgroup\hfil$\m@th\,##\,$\hfil&&\thickspace\hfil
  $\m@th\,##\,$\hfil\crcr
  \crcr\noalign{\vskip -.0\ex@}%
}{\crcr\noalign{\vskip -.0\ex@}%
  \crcr\egroup\egroup\,%
} 

\newenvironment{tinysysmatrix}{\null\,\vcenter\bgroup
  \let\|=\tinyetv
  \let\crh=\tinycrh 
  \let\cr=\tinycr 
  \def\\{\tinystrut\math@cr} 
  \restore@math@cr\default@tag
  \baselineskip\z@skip 
\lineskip\z@skip
 \lineskiplimit\lineskip
  \ialign\bgroup\hfil$\m@th\scriptscriptstyle##$\hfil&&\thickspace\hfil
  $\m@th\scriptscriptstyle##$\hfil\crcr
  \crcr%
\noalign{\vskip -.3\ex@}
}{\crcr\noalign{\vskip -.2\ex@}
  \crcr\egroup\egroup\,%
} 

\newenvironment{semitinysysmatrix}{\null\,\vcenter\bgroup
  \let\|=\tinyetv
  \let\crh=\semitinycrh
  \def\\{\tinystrut\math@cr} 
  \restore@math@cr\default@tag
  \baselineskip\z@skip \lineskip\z@skip \lineskiplimit\lineskip
  \ialign\bgroup\hfil$\m@th\scriptscriptstyle##$\hfil&&\thickspace\hfil
  $\m@th\scriptscriptstyle##$\hfil\crcr
  \crcr\noalign{\vskip -.3\ex@}
}{\crcr\noalign{\vskip -.2\ex@}
  \crcr\egroup\egroup\,%
} %

\newenvironment{tinymatrix}{\null\,\vcenter\bgroup
  \def\\{\tinystrut\math@cr} 
  \let\oldcr=\cr
  \def\cr{\vspace{.1ex}\oldcr} 
  \restore@math@cr\default@tag
  \baselineskip\z@skip \lineskip\z@skip \lineskiplimit\lineskip
  \ialign\bgroup\hfil$\m@th\scriptscriptstyle##$\hfil&&\thickspace\hfil
  $\m@th\scriptscriptstyle##$\hfil\crcr 
  \crcr\noalign{\vskip -.3\ex@}
}{%
{\let\oldcr=\cr}
\crcr\noalign{\vskip -.2\ex@}
  \crcr\egroup\egroup\,
} 

\makeatother

\def\ssysm#1{\begin{smallsysmatrix}#1\end{smallsysmatrix}}
\def\ssysbm#1{\left[\ssysm{#1}\right]}
\def\ssyspm#1{\left(\ssysm{#1}\right)} 

\def\tsysm#1{\begin{tinysysmatrix}#1\end{tinysysmatrix}}
\def\tsysbm#1{\left[\tsysm{#1}\right]}
\def\tsyspm#1{\left(\tsysm{#1}\right)}

\let\ClL=\circlearrowleft %
\newcommand{\opt}{{\rm opt}}

\newcommand{\ext}{{\rm ext}} 
\newcommand{\shat}[1]{{\widehat{#1}}} %
\newcommand{\stilde}[1]{{\widetilde{#1}}} %
\newcommand{\Refl}{\cR}     %

\newcommand{\suba}{\linebreak[0]\hbox{ $\sub\lower.45ex\hbox{$\below{\,\hspace{.20em}a}$}$ } \linebreak[3] }

\newcommand{\tmathop}[1]{{\ensuremath{\mathop{\rm #1}\nolimits}} }    %
\newcommand{\tmathopL}[1]{{\ensuremath{\mathop{{\rm #1}^{\L^1}}\nolimits}} }    %
\newcommand{\tmathopLCL}[1]{{\ensuremath{\mathop{{\rm #1}}\nolimits^{\L^1,\BL\cC}}} }    %
\newcommand{\tmathopLsub}[2]{{\ensuremath{\mathop{{\rm #1}^{\L^1}_{#2}}\nolimits}} }    %
\newcommand{\WPLS}{\tmathop{WPLS}}   %
\newcommand{\SOS}{\tmathop{SOS}}     %
\newcommand{\TIC}{\tmathop{TIC}}     %
\newcommand\TICexp{\TIC_{\rm exp}}    %
\newcommand{\WR}{\tmathop{WR}}     %
\newcommand\ULR{\tmathop{ULR}}    %

\newcommand{\WTIC}{\tmathopL{MTIC}} %
\newcommand{\WTICLC}{\tmathopLCL{MTIC}} %
\newcommand{\WTICsub}[1]{\tmathopLsub{MTIC}{#1}} %
 
\newcommand{\WTICinfty}{\WTICsub{\infty}}
\newcommand{\WTIComega}{\WTICsub{\omega}}

\newcommand{\MTIC}{\tmathop{\rm MTIC}} %

\newcommand{\CRHP}{\overline{\C^+}\cup\{\infty\}} %
\newcommand{\cRHP}{\overline{\C^+}} %
\newcommand{\uopt}{u_{\rm opt}} %

\newcommand{\xopt}{x_{\rm opt}} %
\newcommand{\yopt}{y_{\rm opt}}

\newcommand{\uc}{u_\ClL}

\newcommand{\vc}{v_\ClL} %
\newcommand{\yc}{y_\ClL} %
\newcommand{\tyc}{\tilde y_\ClL}
\newcommand{\xc}{x_\ClL} 

\newcommand{\tu}{\tilde{u}}
 
\newcommand{\ty}{\tilde{y}}
\newcommand{\hu}{{\shat{u}}} %
\newcommand{\huClL}{{\shat{u_\ClL}}} %
 
\newcommand{\huc}{{\shat{\uc}}} 
 
\newcommand{\hv}{{\shat{v}}} %
\newcommand{\hx}{{\shat{x}}} %
\newcommand{\hy}{{\shat{y}}} 
 
\newcommand{\hf}{{\shat{f}}} 
\newcommand{\hg}{{\shat{g}}} 

\newcommand{\piplus}{{\pi_{+}}}
\newcommand{\piminus}{{\pi_{-}}}

\newcommand{\piOt}{{\pi_{[0,t)}}} %
\newcommand{\pitO}{{\pi_{[-t,0)}}} %
\newcommand{\piTO}{{\pi_{[-T,0)}}} 
\newcommand{\piti}{{\pi_{[t,\infty)}}} %
\newcommand{\piit}{{\pi_{(-\infty,t)}}} %

\newcommand{\pitt}{{\pi_{[-t,t)}}} %
\newcommand{\piOT}{{\pi_{[0,T)}}} %
 
\newcommand{\piOI}{{\pi_{[0,1)}}} %
\newcommand{\Par}[1]{\left( #1 \right)}

\newcommand{\rplus}{{\R_+}}

\newcommand{\Sem}{\qA}

\newcommand{\In}{\qB} %

\newcommand{\Out}{\qC} %
\newcommand{\Feed}{\qK}

\newcommand{\IO}{\qD}
\newcommand{\FeedIO}{\qF}

\newcommand{\ALS}{\Sigma}

\newcommand{\BL}{\cB} %

\newcommand{\tA}{{\tilde{A}}} %
\newcommand{\tC}{{\tilde{C}}} %
\newcommand{\tB}{\tilde{B}}

\newcommand{\tJ}{{\tilde{J}}}
\newcommand{\tK}{{\tilde{K}}} 
\newcommand{\tL}{{\tilde{L}}} %

\newcommand{\tS}{{\tilde{S}}} %

\newcommand{\tU}{{\tilde{U}}}
\newcommand{\tY}{{\tilde{Y}}} 

\newcommand{\systemfont}{\mathscr} %
\renewcommand{\C}{{\fieldfont C}}
\renewcommand{\N}{{\fieldfont N}}

\renewcommand{\R}{{\fieldfont R}}
\renewcommand{\Z}{{\fieldfont Z}}

\newcommand{\qA}{{\systemfont A}}
\newcommand{\qB}{{\systemfont B}}
\newcommand{\qC}{{\systemfont C}}
\newcommand{\qD}{{\systemfont D}}
\newcommand{\qE}{{\systemfont E}}
\newcommand{\qF}{{\systemfont F}}
\newcommand{\qG}{{\systemfont G}}
\newcommand{\qH}{{\systemfont H}}

\newcommand{\qK}{{\systemfont K}}

\newcommand{\qM}{{\systemfont M}} 
\newcommand{\qN}{{\systemfont N}} 
\newcommand{\qO}{{\systemfont O}} 

\newcommand{\qQ}{{\systemfont Q}} 
\newcommand{\qR}{{\systemfont R}}
\newcommand{\qS}{{\systemfont S}}
\newcommand{\qT}{{\systemfont T}}

\newcommand{\qX}{{\systemfont X}}
\newcommand{\qY}{{\systemfont Y}}

\newcommand{\tqA}{{\tilde{\systemfont A}}} 
\newcommand{\tqB}{{\tilde{\systemfont B}}} 
\newcommand{\tqC}{{\tilde{\systemfont C}}} 
\newcommand{\tqD}{{\tilde{\systemfont D}}} 
 
\newcommand{\tqF}{{\tilde{\systemfont F}}}

\newcommand{\tqK}{{\tilde{\systemfont K}}} 
 
\newcommand{\tqM}{{\tilde{\systemfont M}}} 
\newcommand{\tqN}{{\tilde{\systemfont N}}}

\newcommand{\tqS}{{\tilde{\systemfont S}}}

\newcommand{\tqX}{{\tilde{\systemfont X}}} 
\newcommand{\tqY}{{\tilde{\systemfont Y}}}

\newcommand{\hqA}{{\hat{\systemfont A}}} %
\newcommand{\hqB}{{\hat{\systemfont B}}}

\newcommand{\hqC}{{\hat{\systemfont C}}} 
\newcommand{\hqD}{{\hat{\systemfont D}}} 
\newcommand{\hqE}{{\hat{\systemfont E}}} 
\newcommand{\hqF}{{\hat{\systemfont F}}}

\newcommand{\hqK}{{\hat{\systemfont K}}} 
 
\newcommand{\hqM}{{\hat{\systemfont M}}} 
\newcommand{\hqN}{{\hat{\systemfont N}}} 
\newcommand{\hqO}{{\hat{\systemfont O}}}

\newcommand{\hqS}{{\hat{\systemfont S}}} 
\newcommand{\hqT}{{\hat{\systemfont T}}}

\newcommand{\hqX}{{\hat{\systemfont X}}} 
 
\newcommand{\hqY}{{\hat{\systemfont Y}}}

\newcommand{\hqKClL}{{\shat{\systemfont K_\ClL}}} %
\newcommand{\hqAClL}{{\shat{\qAClL}}} %
\newcommand{\hqBClL}{{\shat{\qBClL}}} %
\newcommand{\hqCClL}{{\shat{\qCClL}}} %
\newcommand{\hqDClL}{{\shat{\qDClL}}}

\newcommand{\htqD}{{\hat{\tilde{\systemfont D}}}} %

\newcommand{\ALStau}{\ALS^\tau} %
\newcommand{\ALSexttau}{\ALSext^\tau} %
\newcommand{\ALSClLtau}{\ALS_\ClL^\tau} %

\newcommand{\ALSext}{\ALS_\ext}

\newcommand{\ALSd}{\ALS^\rmd}
\newcommand{\ALSExt}{\ALS_\ext}
\newcommand{\ALSClL}{\ALS_\ClL}

\newcommand{\bbbF}{{\mathbb F}} %
\newcommand{\bbbT}{{\mathbb T}} 

\newcommand{\Hgen}{H} %

\newcommand{\Xgen}{X} %

\newcommand{\qAopt}{\qA_\opt}  %
 
\newcommand{\qCopt}{\qC_\opt}  
 
\newcommand{\qKopt}{\qK_\opt}

\newcommand{\ALSopt}{\ALS_\opt} 
\newcommand{\hALSopt}{\shat{\ALSopt}}

\newcommand{\hqKopt}{\shat{\qK_\opt}}  
\newcommand{\Aopt}{A_\opt} %
\newcommand{\Copt}{C_\opt} 
 
\newcommand{\Kopt}{K_\opt}

\newcommand{\Koptw}{(K_\opt)_\w}

\newcommand{\qKoptt}{{\qK_\opt^t}}

\newcommand{\qAClL}{\qA_\ClL}
\newcommand{\qBClL}{\qB_\ClL}
\newcommand{\qCClL}{\qC_\ClL}
\newcommand{\qDClL}{\qD_\ClL}
\newcommand{\qKClL}{\qK_\ClL}
\newcommand{\qFClL}{\qF_\ClL}
 
\newcommand{\qGClL}{\qG_\ClL}

\newcommand{\qRClL}{\qR_\ClL} 
\newcommand{\tqAClL}{\tqA_\ClL}  %
 
\newcommand{\tqCClL}{\tqC_\ClL} 
\newcommand{\tqDClL}{\tqD_\ClL} 
\newcommand{\tqKClL}{\tqK_\ClL} 
\newcommand{\tqFClL}{\tqF_\ClL}

\newcommand{\qAt}{{\qA^t}} %
\newcommand{\qBt}{{\qB^t}} %
\newcommand{\qCt}{{\qC^t}} %
\newcommand{\qDt}{{\qD^t}} %
\newcommand{\qKt}{{\qK^t}}

\newcommand{\qFt}{{\qF^t}} 
\newcommand{\qNt}{{\qN^t}} 
\newcommand{\qMt}{{\qM^t}} 
\newcommand{\qXt}{{\qX^t}}

\newcommand{\qSt}{{\qS^t}}

\newcommand{\qAClLt}{{\qA_\ClL^t}}
\newcommand{\qBClLt}{{\qB_\ClL^t}}
\newcommand{\qCClLt}{{\qC_\ClL^t}}
\newcommand{\qDClLt}{{\qD_\ClL^t}}

\newcommand{\tqAClLt}{{{\tqA}_\ClL^t}}

\newcommand{\tqCClLt}{{{\tqC}_\ClL^t}}

\newcommand{\tqKClLt}{{{\tqK}_\ClL^t}}

\newcommand{\AClL}{A_\ClL}
\newcommand{\BClL}{B_\ClL}
\newcommand{\CClL}{C_\ClL}

\newcommand{\KClL}{K_\ClL}

\newcommand{\uZ}{Z^\rmu}

\newcommand{\othersub}{0} %
\newcommand{\oqB}{\qB_\othersub} %
\newcommand{\oqD}{\qD_\othersub}

\newcommand{\oqA}{\qA_\othersub}
\newcommand{\oqC}{\qC_\othersub}

\newcommand{\oqK}{\qK_\othersub}

\newcommand{\oqACK}{\tsysbm{\oqA\crh\oqC\cr \oqK}}
\newcommand{\oqAK}{\tsysbm{\oqA\cr \oqK}}

\newcommand{\oALS}{\ALS_\othersub}

\newcommand{\qKClLo}{\qK_{\ClL\othersub}}

\newcommand{\oA}{A_\othersub} %
\newcommand{\oC}{C_\othersub}
\renewcommand{\oK}{K_\othersub}
\newcommand{\oqAt}{{\qA_\othersub^t}}
\newcommand{\oqBt}{{\qB_\othersub^t}}
\newcommand{\oqCt}{{\qC_\othersub^t}}
\newcommand{\oqDt}{{\qD_\othersub^t}}

\newcommand{\oqKt}{{\qK_\othersub^t}}

\newcommand{\hoqC}{\shat{\qC_\othersub}}
\newcommand{\hoqD}{\shat{\qD_\othersub}}
\newcommand{\choqD}{{\check\qD_\othersub}}
\newcommand{\hoqK}{\shat{\qK_\othersub}}

\newcommand{\qCo}{\qC_\othersub}
\newcommand{\qKo}{\qK_\othersub}

\newcommand{\Ao}{A_\othersub} %
\newcommand{\Co}{C_\othersub}

\newcommand{\hqCo}{\shat{\qC_\othersub}}

\newcommand{\gUstar}{{\dU_*}}
\newcommand{\PTO}{\qS_{\rm PT}} %

\newcommand{\ALSp}{{\ALS_+}} %
\newcommand{\ALSpClL}{(\ALS_{+})_\ClL}

\newcommand{\qAp}{{\qA_+}}
\newcommand{\qBp}{{\qB_+}}
\newcommand{\qCp}{{\qC_+}}
\newcommand{\qDp}{{\qD_+}}
\newcommand{\qKp}{{\qK_+}}
\newcommand{\qFp}{{\qF_+}}
\newcommand{\qKFp}{\msysbm{\qKp\|\qFp}} %
\newcommand{\qXp}{{\qX_+}}

\newcommand{\qMp}{{\qM_+}}
\newcommand{\qNp}{{\qN_+}}

\newcommand{\hqDp}{\shat{\qD_+}}
\newcommand{\hqXp}{\shat{\qX_+}}

\newcommand{\dUoutp}{\dUout^{\ALSp}}        %
\newcommand{\dUexpp}{\dUexp^{\ALSp}}        

\renewcommand{\tALS}{\tilde{\ALS}}

\newcommand{\tALSClL}{\tilde{\ALS}_\ClL}

\newcommand{\lqD}{\underline{\qD}}
\newcommand{\Cc}{C_\rmc} %
\newcommand{\Dc}{D_\rmc}
\newcommand{\Kc}{K_\rmc}
\newcommand{\Fc}{F_\rmc}
\newcommand{\Xc}{\Xgen_\rmc} %

\newcommand{\Cs}{C_\rms}

\newcommand{\Bw}{B_\w} %
\newcommand{\Cw}{C_\w}
\newcommand{\Kw}{K_\w} %
\newcommand{\bigmatrix}[1]{\begin{matrix}#1\end{matrix}}

\newcommand{\sm}[1]{\begin{smallmatrix}#1\end{smallmatrix}}

\let\normalcr=\cr %

\makeatletter
\newcommand\anycrh[1]{\vrule height0pt depth1\ex@ width0pt %
        \oldcr\noalign{\hrule}
        \vrule height#1\ex@ depth0pt width0pt} %
\makeatother
\let\tmcr=\tinymatrixcr %
\newcommand{\tm}[1]{%
\bgroup{ %
\let\cr=\tmcr%
\begin{tinymatrix}%
#1\end{tinymatrix}}\egroup} %
\newcommand{\tbm}[1]{\left[\tm{#1}\right]} 
\let\sbm=\smallbmatrix
\let\bbm=\bigbmatrix

\newcommand{\SmallbSystem}{\ssysbm{\qA\|\qB\crh \qC\|\qD}}

\newcommand{\SmallbExtSystem}{\ssysbm{\Sem\|\In\crh \Out\|\IO\cr \Feed\| \FeedIO}}

\newcommand{\SmallbClLExtSystem}{\ssysbm{\qAClL\|\qBClL\crh \qCClL\|\qDClL\cr
               \qKClL\| \qFClL}}

\newcommand{\bExtSystem}{\bsysbm{\Sem\|\In\crh \Out\|\IO\cr \Feed\|\FeedIO}}
\newcommand{\ExttauSystem}{\bsysm{\Sem\|\In\tau\crh \Out\|\IO\phantom{\tau}\cr \Feed\|\FeedIO\phantom{\tau}}}
\newcommand{\qABCD}{\tsysbm{\qA\|\qB\crh \qC\|\qD}}
\newcommand{\qABtauCD}{\tsysbm{\qA\|\qB\tau\crh \qC\|\qD}}

\newcommand{\qAB}{\msysbm{\qA\|\qB}}

\newcommand{\qCD}{\msysbm{\qC\|\qD}} 

\newcommand{\qKF}{\msysbm{\qK\|\qF}} %
\newcommand{\qAK}{\tsysbm{\qA\crh\qK}} 

\newcommand{\qABKF}{\tsysbm{\qA\|\qB\crh \qK\|\qF}}

\newcommand{\qABQR}{\tsysbm{\qA\|\qB\crh \qQR\|\qRR}}  %
\newcommand{\qABCDClL}{\tsysbm{\qAClL\|\qBClL\crh \qCClL\|\qDClL}}

\newcommand{\tqKF}{\msysbm{\tqK\|\tqF}} %
\newcommand{\sZ}{Z^\rms}  %
\newcommand{\ABCD}{\tsysbm{A\|B\crh C\| D}} %
\newcommand{\ABC}{\tsysbm{A\|B\crh C\| }} 

\newcommand\CDreservedtrue{1}
\ifx \CDreserved\CDreservedtrue\else\fi
\newcommand{\hIRE}{$\shat{\rm IRE}$} %
\newcommand{\optIRE}{$\ALSopt$-IRE} %
\newcommand{\hoptIRE}{$\hALSopt$-IRE} %
\newcommand{\bfoptIRE}{$\pmb{\ALSopt}$-IRE} %
\newcommand{\SIRE}{$\qS^t$-IRE} %
\newcommand{\hSIRE}{$\hqS$-IRE} %
\newcommand{\boldedSIRE}{$\pmb{\qS^t}$-IRE} %
\newcommand{\boldedhSIRE}{$\pmb{\hqS}$-IRE} %

\newcommand{\BwARE}{$B^*_\w$-ARE}  %

\newcommand{\rconn}{\mathop{\rm rconn}\nolimits} %

\newcommand{\barz}{{\overline{z}}}

\newcommand{\inv}{{\rm inv}} %

\newcommand{\ALSinv}{\ALS^\inv}

\newcommand{\ALSoptp}{(\ALSopt)_+} 

\newcommand{\chqD}{\check\qD}
\newcommand{\chqF}{\check\qF}
\newcommand{\chqK}{\check\qK}

\newcommand{\chqX}{\check\qX}

\newcommand{\chqM}{\check\qM}

\newcommand{\chqFClL}{\check\qF_\ClL}

\newcommand{\hqDc}{\hqD_\ALS}

\newcommand{\hqXc}{\hqX_{\ALSExt}}

\renewcommand{\N}{{\mathbb N}} %
\renewcommand{\Z}{{\mathbb Z}}

\renewcommand{\R}{{\mathbb R}}
\renewcommand{\C}{{\mathbb C}}

\ifomaprintti{
\renewcommand{\Refl}{\cR} %
}%

\makeatletter
\def\atitem[#1]{%
  \if@noparitem
    \@donoparitem
  \else
    \if@inlabel
      \indent \par
    \fi
    \ifhmode
      \unskip\unskip \par
    \fi
    \if@newlist
      \if@nobreak
        \@nbitem
      \else
        \addpenalty\@beginparpenalty
        \addvspace\@topsep
        \addvspace{-\parskip}%
      \fi
    \else
      \addpenalty\@itempenalty
      \addvspace\itemsep
    \fi
    \global\@inlabeltrue
  \fi
  \everypar{%
    \@minipagefalse
    \global\@newlistfalse
    \if@inlabel
      \global\@inlabelfalse
      {\setbox\z@\lastbox
       \ifvoid\z@
         \kern-\itemindent
       \fi}%
      \box\@labels
      \penalty\z@
    \fi
    \if@nobreak
      \@nobreakfalse
      \clubpenalty \@M
    \else
      \clubpenalty \@clubpenalty
      \everypar{}%
    \fi}%
  \if@noitemarg
    \@noitemargfalse
    \if@nmbrlist
      \refstepcounter\@listctr
    \fi
  \fi
  \sbox\@tempboxa{{#1}}%
  \global\setbox\@labels\hbox{%
    \unhbox\@labels
    \hskip \itemindent
    \hskip -\labelwidth
    \hskip -\labelsep
    \ifdim \wd\@tempboxa >\labelwidth
      \box\@tempboxa
    \else
      \hbox to\labelwidth {\unhbox\@tempboxa}%
    \fi
    \hskip \labelsep}%
  \ignorespaces}
\let\olditem=\atitem
\makeatother

\newenvironment{itemlistbibitem}%
{%
\topsep\itemlisttopsep
\parskip-.0ex %
\begin{description}%
\itemsep\itemlistitemsep
\def\KMBibitemlistleftskip{0.7em} %
\def\itemlistfirstlineskip{-3.8em} %
\advance\leftskip\KMBibitemlistleftskip 
\itemindent\itemlistfirstlineskip}%
{\end{description}}

\newcommand{\KMBibitem}[2]{\olditem[\hbox to 3.7em{#1\hfil}]#2}

\renewcommand{\TTHEN}{\ \THEN\ } %
\renewcommand{\d}{\partial}

\renewcommand{\NotesFor}[1]{{\bf Notes for #1: }}

\newcommand{\pmbold}[1]{#1} %
\newcommand{\ItemName}[1]{{\rm (#1)}} %

\renewcommand{\raj}[1]{_{|#1}}
\begin{document} 
\author{\ifomaprintti{ \today\ \timeHHMM\ }Kalle M. Mikkola}

\ifeiIEOT{

\newbox\addressbox
\newbox\emailbox
\newbox\keywordsbox
\newbox\subjclassbox
\newcommand{\address}[1]{\setbox\addressbox=\vtop{\noindent#1}} 
\newcommand{\email}[1]{\setbox\emailbox=\hbox{#1}} 
\newcommand{\keywords}[1]{\setbox\keywordsbox=\vtop{\noindent{\bf Keywords:} #1.}}  %
\newcommand{\subjclass}[1]{\setbox\subjclassbox=\hbox{{{\bf AMS class:} #1.}}}

\newbox\abstractbox
\renewenvironment{abstract}{\global\setbox\abstractbox=\vtop\bgroup%
\noindent{\bf Abstract:} }%
{\egroup} %

}%

\address{%
        Helsinki University of Technology %
        Institute of Mathematics\\
        P.O. Box 1100; %
        FIN-02015 HUT, Finland\\
        GSM: +358-40-7545660,\ fax: +358-9-451 3016\\
        http://www.math.hut.fi/\~{}kmikkola/}
\email{%
        Kalle.Mikkola@iki.fi}

\newcommand{\qHG}{\ssysbm{\qH\crh\qG}} %
\newcommand{\qHGd}{\bsysbm{\qH^\rmd \|\qG^\rmd}} %
\newcommand{\qABCDKF}{\SmallbExtSystem} %
\newcommand{\qABCDp}{\bsysbm{\qAp\|\qBp\crh \qCp\|\qDp}}

\title{\ifomaprintti{ \today\ \timeHHMM\ }Riccati equations 
 \ifeiomaprintti{and optimal control
 of well-posed linear systems}\ifomaprintti{and ...}}

\begin{abstract}
We generalize the classical theory on algebraic Riccati equations
 and optimization to infinite-dimensional well-posed linear systems,
 thus completing the work of George Weiss, Olof Staffans and others.
We show that the optimal control is given by the 
 stabilizing solution of an integral Riccati equation. %
If the input operator is not maximally unbounded, 
 then this integral Riccati equation is equivalent to the algebraic 
 Riccati equation.

Using the integral Riccati equation,
 we show that for (nonsingular) minimization problems
 the optimal state-feedback loop is always well-posed.
In particular, the optimal state-feedback operator is admissible
 also for the original semigroup, not only for the closed-loop semigroup
 (as has been known in some cases);
 moreover, both settings are well-posed with respect to an external input.
This leads to the positive solution of several %
 central, previously open questions 
 on %
 exponential, output and dynamic (aka.\ ``internal'') stabilization
 and on coprime factorization of transfer functions.

Our theory covers all quadratic (possibly indefinite) cost functions,
 but the optimal state feedback need not be well-posed (admissible)
 unless the cost function is uniformly positive
 or the system is sufficiently regular.
\end{abstract}


\date{March 14, 2004}
\ifIEOT{%
{\parindent0em%
\noindent%
\maketitle}%
}\vspace{-4.5ex}

\ifeiIEOT{

\maketitle

\box\addressbox
\box\emailbox
\vspace{.5ex}
\box\abstractbox 
\vskip1.5ex%
\box\subjclassbox
\vspace{.5ex}
\box\keywordsbox

}%

\ifomaprintti{ %
\newdimen\kallelinespacing
\newdimen\kallebaselineskip
\kallelinespacing=\lineskip
\kallebaselineskip=\baselineskip
\lineskip=-8pt %
\baselineskip=-11pt 
\tableofcontents %
\lineskip=\kallelinespacing
\baselineskip=\kallebaselineskip
}%

\section{Introduction: systems with bounded generators}\label{sIntro} %
In this section we present a (mostly known) very special case of our results.
At the end of this section and in ``Conclusions''
 (Section~\ref{sConclusions})
 we explain, how we have generalized these results
 to more general systems, cost functions and 
 stability goals, in the other sections.\footnote{This is the March 14, 2004 draft (= the latest version before its split) as such except for this publication footnote (February 27, 2016). I was asked to publish it now in arXiv to allow referencing to results not published elsewhere. 
 In an earlier form it was circulated a few of months earlier. Later, parts of it were published, usually with several newer results: ``State-Feedback Stabilization of Well-Posed Linear Systems''  Integral Equations and Operator Theory 55 (2), pp. 249-271, 2006 (early/middle parts). ``Coprime factorization and dynamic stabilization of transfer functions'', SIAM Journal on Control and Optimization, 45 (6), pp. 1988-2010, 2007 (not systems, just transfer functions, unlike in the ones mentioned below). ``Weakly coprime factorization and state-feedback stabilization of discrete-time systems''
  Mathematics of Control, Signals, and Systems, 20 (4), pp. 321-350, 2008, ``Weakly coprime factorization and continuous-time systems''
 IMA Journal of Mathematical Control and Information, 25 (4): pp. 515-546, 2008. doi:10.1093/imamci/dnn011 
 Many of the results were in some form already in [M02].
 Most remaining main results, such as Theorem \ref{DCFiff} and output and measurement feedback
  stabilization results for WPLSs were published in 
  ``Coprime factorizations and stabilization of infinite-dimensional linear systems''
  Proceedings of CDC-ECC2005. Of those results I had two corresponding drafts fairly ready late 2007
   but then had to stop finishing them due to other responsibilities.
   I will probably publish also them in arXiv as such, if I do not find time to 
    update their references and shorten the presentation.} %
In the most simple case, a
 linear time-invariant
 control system is governed by the equations\label{pageWPLSmotivation}
\begin{equation}
  \label{exy=ABCDintro}
  \begin{aligned}
   \dot x(t)&=Ax(t)+Bu(t),\\
   y(t)&=Cx(t)+Du(t),\\
   x(0)&=x_0 %
  \end{aligned}
\end{equation}
 (for $t\ge0$), 
 where the
 {\em generators}\label{pagegenerators1}
 $\sbm{A&B\cr C&D}\in\BL(H\times U,H\times Y)$
 are matrices, or more generally, linear operators on
 Hilbert spaces ($U,H,Y$) of arbitrary dimensions.
There $u$ is the {\em input} (or {\em control}),
 $x$ the {\em state} and $y$ the {\em output} of the system.
Obviously, $x_0$ and $u$ determine $x$ and $y$ uniquely.
In this section, we shall allow $A$ to be unbounded
 as long as it generates a
 strongly continuous semigroup,
 which we denote by $\efn^{At}$;
 in later sections also $B$ and $C$ may be unbounded.

By $\BL(H,U)$ we denote the space of bounded linear operators $H\to U$,
 by $\R_+$ the set $[0,\infty)$
 and by $\L^2(\R_+;U)$\label{pageL2+} 
 the Banach space of (equivalence classes of Bochner) 
 measurable functions $u:\R_+\to U$ for which
 $\|u\|_2^2:=\int_0^\infty \|u(t)\|_U^2\,dt<\infty$.

We first take a look at the following (LQR) minimization problem.
Given any {\em initial state} $x_0\in H$,
 we want to minimize a cost function, such as
\begin{equation}
  \label{ecJx2u2}
  \cJ(x_0,u)=\int_0^\infty \left( \|x(t)\|_H^2+\|u(t)\|_U^2\right)\,dt.
\end{equation}
Observe that the output $y$ (and hence $C$ and $D$ too)
 is irrelevant to this problem.

A necessary condition for the existence of a minimum is the
 {\em state-FCC}\label{pagestate-FCC} (Finite Cost Condition):
\begin{equation} %
  \label{eFCCUexp}
 \text{For each } x_0\in H, \text{there exists some control } u\in\L^2(\R_+;U)
 \text{ such that } x\in \L^2(\R_+;H)
\end{equation}
 (i.e.\ $\inf_u \cJ(x_0,u)<\infty$ for all $x_0$, %
  so that we do not have to optimize over the empty set).
Thus, some {\em stable}\label{pageStableIntro} input ($u\in\L^2$) must make the state stable
 ($x\in\L^2$).
It is known that the state-FCC is also sufficient:
\begin{theorem}[$\int_0^\infty\|x\|^2+\|u\|^2$]\label{IntroTh-Uexp} %
The following are equivalent:
\begin{itemlist}
  \item[(i)] For each initial state $x_0\in H$,
 there exists a unique
 control that minimizes (\ref{ecJx2u2}).
  \item[(ii)] The algebraic Riccati equation (ARE)
\begin{equation}
  \label{eAREminBsimpleI}
  \Ric BB^*\Ric=A^*\Ric+\Ric A+I\ \ \ \text{on}\ \ \Dom(A)
\end{equation}
  has a solution $\Ric=\Ric^*\in\BL(H)$ that is {\em exponentially stabilizing},\label{pageExpStab0}
 i.e., such that the $C_0$-semigroup $\efn^{t(A+BK)}$ %
 is {\em exponentially stable}\label{pageExpStab},%
\footnote{$\|\efn^{t(A+BK)}\|\le M\efn^{-\eps t}$ for some $M,\eps>0$ and all $t>0$
 (cf.\  Lemma~\ref{lExpStable}).}
 where $K:=-B^*\Ric$.

  \item[(iii)] The state-FCC (\ref{eFCCUexp}) holds.
\end{itemlist}

Assume that (ii) has a solution. Then this solution is unique,
 and the (state-feedback) control $u(t)=Kx(t)$
 strictly minimizes the cost (\ref{ecJx2u2})
 for any initial state $x_0\in H$.
Moreover, the minimal cost equals $\p{x_0,\Ric x_0}_H$.\noproof
\end{theorem}

This is a special case of Corollary~\ref{cBARE}(a). %
In fact, a solution of (\ref{eAREminBsimpleI})
 is exponentially stabilizing iff it is nonnegative.

By the {\em transfer function}\label{pageTrFctIntro} of the system  (\ref{exy=ABCDintro})
 we mean the map $s\mapsto \hqD(s)\in\BL(U,Y)$, where
 \begin{equation}
   \label{eIntroTrFct}
   \hqD(s):=D+C(s-A)^{-1}B.
 \end{equation}
When $x_0=0$, we have $\hy (s)= \hqD(s)\hu(s)$
 for each $s$ on some right half-plane;
 here $\hu(s):=\int_0^\infty \efn^{-st}u(t)\,dt$
 is the {\em Laplace transform} of $u$.
This fact follows from the identity
 \begin{equation}
   \label{eIntroABhat}
   (s-A)\hx(s)=x_0+B\hu(s),
 \end{equation}
 which is a direct consequence of  (\ref{exy=ABCDintro})
 (and Lemma~\ref{lhf'}).

If we allow for an {\em external input}\label{pageextinputIntro} $\uc\in\L^2(\R_+;U)$
 to the state-feedback loop of Theorem~\ref{IntroTh-Uexp},
 i.e.\ $u(t)=Kx(t)+\uc(t)\ \all t\ge0$.
For $x_0=0$ this leads to
 $(s-A)\hx(s)=B(K\hx(s)+\huClL(s))$, i.e., to
 \begin{equation}
  \label{ehxhuhyClL}
   \hx(s)=(s-A-BK)^{-1}B\uc(s),\ \ \hu=\hqM\huc,\ \ \hy=\hqN\huc,
 \end{equation}
 on some right half-plane,
 where $\hqM(s):=I+K(s-A-BK)^{-1}B$,
 $\hqN(s)=D+(C+DK)(s-A-BK)^{-1}B$.

We call a state-feedback operator $K:\Dom(A)\to U$
 {\em admissible} 
 for the system (\ref{exy=ABCDintro})
 if the map $\uc\to u$ and its inverse are
  locally bounded in $\L^2$.
An equivalent requirement is that $\hqM$ and $\hqM^{-1}$
 are bounded on some right half-plane. %
A sufficient condition is that $K$ is {\em bounded} ($K\in\BL(H,U)$),
 but in a more general setting with an unbounded $B$ ($B\not\in\BL(U,H)$)
 one sometimes needs an unbounded $K$
 to make $\efn^{\cdot(A+BK)}$ stable.

Theorem~\ref{IntroTh-Uexp} implies the following:
\begin{cor}\label{cIntroOptim}%
The system satisfies the state-FCC (\ref{eFCCUexp})
 iff it is exponentially stabilizable.
\end{cor}

{\em Exponentially stabilizable} means that there exists
  an admissible $K$  s.t.\ the semigroup %
  generated by $A+BK$
  is exponentially stable.
Our generalization of Corollary~\ref{cIntroOptim}
 (Corollary~\ref{cOptExpstab1})
 solves positively the ``optimizability = exponential stabilizability'' problem
 studied in, e.g., [WR00].

Similar results also hold for the alternative (LQR) cost function
\begin{equation}
  \label{ecJy2u2}
  \cJ(x_0,u)=\int_0^\infty \left( \|y(t)\|_Y^2+\|u(t)\|_U^2\right)\,dt:
\end{equation}
\begin{theorem}[$\int_0^\infty\|y\|^2+\|u\|^2$]\label{IntroTh-Uout} %
Assume that $D=0$. Then the following are equivalent:
\begin{itemlist}
  \item[(i)] %
 For each initial state $x_0\in H$,
 there exists a unique
 control that minimizes (\ref{ecJy2u2}). 

  \item[(ii)] \ItemName{ARE} The algebraic Riccati equation
\begin{equation}
  \label{eAREminBsimpleCC}
  \Ric BB^*\Ric=A^*\Ric+\Ric A+C^*C
\end{equation}
  has a nonnegative solution $\Ric\in\BL(H)$.

 \item[(iii)] \ItemName{output-FCC}\label{pageoutput-FCC} For each $x_0\in H$,
 there is $u\in\L^2(\R_+;U)$ s.t.\ $y\in\L^2(\R_+;Y)$.
\end{itemlist}

Assume that (ii) has a solution. 
Then there is a smallest nonnegative solution $\Ric\in\BL(H)$
 of (\ref{eAREminBsimpleCC}),
 and the (state-feedback) control $u(t)=Kx(t)$
 strictly minimizes the cost (\ref{ecJy2u2})
 for any initial state $x_0\in H$, 
 where $K:=-B^*\Ric$.
Moreover, the minimal cost equals $\p{x_0,\Ric x_0}_H$.\noproof
\end{theorem} 

This is a special case of Corollary~\ref{cBARE}(b).

\begin{cor}\label{cQRCF-intro}%
The system satisfies the output-FCC \ref{IntroTh-Uout}(iii)
 iff it is output-stabilizable.
\end{cor}

{\em Output-stabilizable}\label{pageoutput-stabilizable} means that there exists
  an admissible (state-feedback operator)
 $K$ s.t.\ %
 $u,y\in\L^2$ for each initial state $x_0\in H$ under 
  $u(t)=Kx(t)$. %
Even more is true: $u,y\in\L^2$ for any $x_0\in H$
 and $\uc\in\L^2(\R_+;U)$, and the maps $\uc\to\sbm{y\cr y}$
 are coprime in a sense which we will describe below
 if we choose $K$ as in Theorem~\ref{IntroTh-Uout}.

The above claim ``$\uc\in\L^2 \TTHEN u,y\in\L^2$'' %
 implies that the transfer functions
 $\smash{\sbm{\hqN\cr \hqM}}:\huClL\to\smash{\sbm{\hy\cr \hu}}$ 
 (have holomorphic extensions that)
 are bounded on the right half-plane $\C^+:=\{s\in\C\I \re s>0\}$.
We also show that %
 the maps $\hqN$ and $\hqM$
 are {\em q.r.c.} (quasi--right coprime),
 which means that 
 $\smash{\sbm{\hqN\cr \hqM}}\hf\in\shat{\L^2}\TTHEN f\in{\L^2}\ \all f$
 (see Definition~\ref{drcf}(a));
 this implies that $\hqN$ and $\hqM$ do not have common zeros on $\C^+$
 and is as good as the ``standard right coprimeness''
 in typical applications
 (and equivalent to it
  at least if $\dim U<\infty$ and $\hqN,\hqM$ are continuous on $\CRHP$,
  by the proof of Lemma~\ref{lNMrcUfinite}). %

Using our generalization of Corollary \ref{cQRCF-intro},
 we show that any holomorphic map having a ``stable (right) factorization'' has
 a ``q.r.c.\ factorization'':
\begin{theorem}[Right-coprime factorization]\label{IntroQRC}
Given any holomorphic, bounded maps
 $\hqN:\C^+\to\BL(U,Y),\ \hqM:\C^+\to\BL(U)$
 such that $\hqM^{-1}$ exists and is bounded on some right half-plane,
 there are $\hqN_2,\hqM_2$ that satisfy the same conditions,
 $\hqN\hqM^{-1}=\hqN_2\hqM_2^{-1}$,
 and, in addition, $\hqN_2$ and $\hqM_2$ are q.r.c.
\end{theorem}

Thus, we can ``cancel any common zeros of $\hqN$ and $\hqM$ on $\C^+$''.
This and further equivalent conditions on the map $\hqN\hqM^{-1}$
 are given in Corollary~\ref{cQRCFiff}.

By applying Theorem~\ref{IntroTh-Uexp} to 
 the {\em dual system}\label{pagedualsystem}
 $\sbm{A^*&C^*\cr B^*&D^*}$ 
 in place of $\sbm{A&B\cr C&D}$,
 we see that the ``dual'' of the state-FCC (\ref{eFCCUexp}) holds
 iff there exists $\Hgen\in\BL(Y,H)$ s.t.\ $A+\Hgen C$ generates
 an exponentially stable semigroup.
This and (\ref{eFCCUexp}) lead to so called
 {\em doubly coprime factorization (d.c.f.)}
 of the transfer function $\hqD$ and to 
 {\em dynamic (output-feedback) stabilization} 
 of the system.
Conversely, dynamic stabilization leads to a d.c.f.,
 by Lemma~\ref{lDFCdcf} below;
   this is an infinite-dimensional version of 
   the result [S89] by Malcolm Smith.
See Corollary \ref{cJointlySD} and Theorem~\ref{JointlyUout}
  and the references below them for details.
Note that dynamic (I/O-)stabilizaion
 is the same as ``internal stabilization'' 
 in, e.g., [Q03], except that we require the I/O map
 and the controller to be well-posed, i.e.,
 both transfer functions must be bounded on some right half-plane.

The above results are well-known for bounded $B$
 (the same applies to most results mentioned in the remainder
  of this section if we ignore the IREs),
 except for the claims on coprimeness, which have been known for
 finite-dimensional $U,H,Y$ only.
In this article, we shall generalize the above results
 to WPLSs (see below)
 and to general quadratic cost functions in place of $\cJ$
 (including those that are indefinite with respect to $u$).
Also some other results are presented.
However, if $B$ and $C$ are extremely unbounded, %
 then one must use integral Riccati equations instead
 of the algebraic ones above
 and, in the case of indefinite $\cJ$, the optimal
 state-feedback need no longer be admissible.

In Section~\ref{sWPLS},
 we shall define
 {\em WPLSs} (well-posed linear systems, or the Salamon--Weiss class),
 which form a generalization of (\ref{exy=ABCDintro})
 allowing for rather unbounded $B$ and $C$
 (the ``feedthrough'' operator $D=\lim_{s\to+\infty}\hqD(s)$ need not exist;
  if it does, then the WPLS is called {\em regular}). %

In Section~\ref{sFeedback}, 
 we recall what {\em state feedback}
 (the above equation $u=Kx$)
 is in the WPLS context.

In Section~\ref{sOpt}, we shall define a general domain of optimization
 (``the set of admissible inputs $u$ for a given initial state $x_0$'') 
 to replace its special cases 
 (the set of $u$'s in (\ref{eFCCUexp}) or those
  in Theorem~\ref{IntroTh-Uout}(iii)).
Then we define a general cost function $\cJ$ and give sufficient conditions
 for the existence of an optimal control,
 i.e., a control that makes the derivative of the
 cost function vanish.
If the cost function is nonnegative, then such a control is
 cost-minimizing; in the general case it corresponds to a saddle point
 (``maximin'') control,
 which is used to solve, e.g., $\H^\infty$ control problems
 (``the best control for the worst disturbance''; see [M02]).

In Section~\ref{sStabOpt},
 we show that for uniformly positive quadratic cost functions
 (such as (\ref{ecJx2u2}) and (\ref{ecJy2u2})),
 under the (generalized) FCC,
 there is always a unique cost-minimizing state feedback.
The existence of a unique optimal control has been well known
 (see, e.g., [FLT88] or [Z96] for the cost function (\ref{ecJy2u2})),
 but it has not been known that it is given by (well-posed) state feedback.
The corollaries of this result,
 also given in Section~\ref{sStabOpt}, are perhaps the
 main results of this article --- 
 most of results \ref{IntroTh-Uexp}--\ref{IntroQRC}
 are special cases of some of them.

In Section~\ref{sARE} we shall generalize Theorems \ref{IntroTh-Uexp}
 and~\ref{IntroTh-Uout}:
 we shall show that for any regular WPLS
 and any quadratic (possibly indefinite) cost function
 (and any typical domain of optimization),
 there is a (regular) optimal state-feedback operator ($K$)
 iff the ARE has a stabilizing solution.
Here {\em stabilizing}
 means that the resulting controlled system
 is stable in the sense corresponding to the domain of optimization
 (cf.\ Theorem~\ref{IntroTh-Uexp}(iii)).
The necessity of the ARE was originally discovered independently
 by Olof Staffans [S97] and Martin Weiss and George Weiss [WW97],
 for stable regular WPLSs.
The author
 established the converse in [M97] and extended
 the equivalence to the unstable case in [M02].
This equivalence (Theorem~\ref{ARE})
 can be simplified in certain special cases, as we show in 
 Sections \ref{sARE} and~\ref{sMTIC} and in [M02].

In Section~\ref{sIRE}, we generalize Theorem~\ref{ARE}
 to general WPLSs.
Since the ARE cannot be defined for irregular systems, %
 we use the integral Riccati equation (IRE) instead:
 the IRE has a stabilizing solution iff there is some (well-posed) optimal
 state-feedback for the WPLS (Theorem~\ref{GenSpF}).
(The ARE can be used only when $\hqD(+\infty)$ and $\hqM(+\infty)$ exist.)

In Corollary~\ref{cPosIRE}
 we explain the results of Section~\ref{sStabOpt}
 in terms of IREs and AREs
 and show that the word ``stabilizing'' can be replaced by ``nonnegative''
 for the cost functions (\ref{ecJx2u2}) and (\ref{ecJy2u2}).

However, even if there exists a unique optimal control for
 each initial state, 
 the optimal control need not be given by any (well-posed) state feedback
 (except for uniformly positive cost functions,
   as shown in Section~\ref{sStabOpt}).
To treat this most general case,
 we show that a unique optimal control is always given
 by a ``generalized state feedback''
 (in the uniformly positive case this was already known [Z96]),
 and that a third equivalent condition is that a variant of the IRE
 has a stabilizing solution (Theorem~\ref{SIRE}).

Fortunately, if the original system is sufficiently regular
 (we give various alternative assumptions),
 then the ``generalized state feedback'' is nevertheless
 given by a well-posed, even regular state-feedback operator,
 thus making also the AREs and IREs equivalent to the three conditions 
 mentioned above; this is explained in Sections \ref{sMTIC} and~\ref{sARE}.
Section~\ref{sMTIC} focuses on systems for which
 $\efn^{\cdot A}B$ and $C\efn^{\cdot A}B$ are locally integrable.

Most of Sections \ref{soptIRE}--\ref{sStabOptproofs}
 consist of the proofs of the results mentioned above.
Only the simplest proofs have been included in the previous sections.

Theorem~\ref{IntroTh-Uout} was ``generalized'' to WPLSs
 having a bounded output operator ($C$) 
 by Franco Flandoli,
 Irena Lasiecka and Roberto Triggiani in [FLT88],
 using an ``ARE'' given on $\Dom(A+BK)$
 (although the well-posedness of $K$ was not known before this article).
We extend their result to regular WPLSs and to general cost functions
 and domains of optimization in Theorem~\ref{GenARE};
 see Theorem~\ref{OptIRE} (and~\ref{ALSopt}) for the irregular case. 
Also the other variants of the IRE are treated in
 Sections \ref{soptIRE} and~\ref{sIREmore}.

In Section~\ref{sJcoerc} we study the coercivity of the cost function,
 which is a sufficient
 (and in many cases also necessary) condition
 for the existence of a unique optimal control.

In Section~\ref{sConclusions} (``Conclusions''),
 we summarize the Riccati equation and optimization
 theory of this article.
The appendices contain some auxiliary results used in the proofs.

Thus, we generalize and extend most of the theory in [FLT88], %
 [Z96], [S97]--[S98b], [WW97], [M97] %
 and much of that in [M02].
Further notes are given at the end of each of the remaining sections.
Additional notes are given in [M02],
 which also provides numerous further results, details,
  explanations, examples, applications and references
  for much of the theory presented here,
 as well as the corresponding discrete-time results.
\NotesFor{Section~\ref{sIntro}} %
Corollaries \ref{cIntroOptim} and~\ref{cQRCF-intro}
 and Theorem~\ref{IntroQRC} will be extended to general WPLSs
 in Corollary~\ref{cOptExpstab1}, 
 Theorem~\ref{QRCF-Uout} and Corollary~\ref{cQRCFiff}, respectively.
Theorems \ref{IntroTh-Uexp} and~\ref{IntroTh-Uout} 
 will be extended to general WPLSs in 
  Corollary~\ref{cOptExpstab1} and Theorem~\ref{QRCF-Uout},
 respectively, combined with Corollary~\ref{cPosIRE}(a)\&(c);
 see (b) and the remarks below the corollary for cases where
 an ARE can be used instead of the IRE
 (by Theorem~\ref{ARE}, in those cases
   a state-feedback operator can be used instead of a state-feedback pair).
Further discussion on different extensions
 (and on what kind of extensions are not true) 
 is given in Section~\ref{sConclusions}.

\vspace{2ex} %
\noindent{\bf Notation:}%
\addcontentsline{toc}{section}{Notation}
\newbox\BoxA
\newcommand{\Symb}[2]{\vskip 1pt\noindent
\setbox\BoxA=\hbox{$#1$:\quad}%
\hangindent=2 cm%
\ifdim\wd\BoxA > 2.3cm \unhbox\BoxA #2.\else\hbox to
2cm{$#1$:\hfil}#2.\fi} %

\Symb{\ex,\all}{$\ex$ = ``there exists'', $\all$ = ``for all''}
\Symb{*}{unknown/omitted element (e.g., ``$X=\sbm{I&0\cr *&*}$'')}
\Symb{A^{-*}}{$A^*=$ the (Hilbert space) adjoint of $A$;
  $A^{-*}:=(A^{-1})^*=(A^*)^{-1}$}
\Symb{\p{\cdot,\cdot}}{Inner product 
 (usually in $\L^2$ over $\R$)}
\Symb{U,W,H,Y}{Hilbert spaces of arbitrary dimensions; cf.\ p.~\pageref{dWPLS0}}
\Symb{\BL(U,Y)}{Bounded linear maps $U\to Y$. $\BL(U):=\BL(U,U)$}
\Symb{A\gg0}{$A\ge \eps I$ for some $\eps>0$}
\Symb{\Dom(A)}{The domain of the semigroup generator $A$
 with the graph norm $(\|x\|_H^2+\|Ax\|_H^2)^{1/2}$. See Lemma~\ref{lABC}
 for details and for $\Dom(A^*)^*=H_{-1}\supset H$ and $\Dom(A)^*\supset H$}
\Symb{\cG}{The subset (often group) of
 {\em invertible}
 elements
 (e.g., $T\in\cG\BL(X,Y)$ if $ST=I_X$ and $TS=I_Y$ for some $S\in\BL(Y,X)$)}
\Symb{I}{The identity operator,}
\Symb{\R_\pm,\N,\R,\C}{$\R_\pm:=\pm[0,\infty)$, %
 $\N:=\{0,1,2,\ldots\},\ \R:=$real, $\C:=$complex numbers}
\Symb{\C_\omega^+}{The right half-plane $\{z\in\C\I \re z>\omega\}$;
  $\C^+:=\C_0^+$}
\Symb{\wlim}{The weak limit: $\wlim_{n\to\infty}D_n=D\ \IFF\ \p{D_n x,y}\to\p{Dx,y}\ \all x,y$}
\Symb{\Box}{``End of proof''; or at the end of a theorem/result: ``no formal proof follows, see the following text for a proof/reference''}

``Iff''$:=$''if and only if'', ``s.t.''$:=$''such that'', 
 ``w.l.o.g.'':=''without loss of generality'',
  ``w.r.t.'':=''with respect to'', ``one-to-one''$:=$''injective''
 (i.e., $f(x)=f(y)\THEN x=y$). 

We try to explain the rest of the notation as it appears,
 hence the reader may skip the rest of this section at this stage
 and use it later to find forgotten symbols or terms.

See the following pages (or formulas) for the following symbols:
 $\ABCD$ \pageref{lABC}\&\pageref{dReg}, (\ref{exy=ABCDintro}),
 $\ssyspm{A\|B\crh C\|D}$ \pageref{pageABCD},
 $\qABCD$ \pageref{dWPLS0}, (\ref{exy=qABCDintro});
 $(A,B)$,$\bsyspm{A\|B}$,$\qAB$ \pageref{pageAB};
 $H_1,H_{-1}$ \pageref{pageH1},
 $\Hstrong^2$ \pageref{pageH2strong}, %
 $\H^\infty$ \pageref{pageHoo},
 $\H^\infty_\infty$ \pageref{pageHooHoo},
 $H_B$ \pageref{pageHB},
 $J,\cJ$ (\ref{eJ}),
 $\qKF$, $K$, $F$ \pageref{eALSExt}\&\pageref{elALSClLgen}\&\pageref{dIRE},
 $\L_\omega^2,\L^1_\omega$ \pageref{pageL2+}\&\pageref{pageL2omega},
 $\Lc^2:=\{u\in\L^2\I u$ has a compact support$\}$;
 $\qM,\qN$ \pageref{pageqM}\&\pageref{dIRE},
 $\Ric$ \pageref{pageRiccatiOp}\&\pageref{SIRE}\&\pageref{dARE}\&\pageref{OptIRE}\&\pageref{dIRE};
 $\Refl f)(t):=f(-t)$;
 $\PTO$ (Popov Toeplitz operator) \pageref{pageqS},
 $\dUexp$ (\ref{eUexp}), %
 $\dUout$ (\ref{eUout}), %
 $\dUstr$ \pageref{pageUstr},
 $\gUstar$ \pageref{egUstar}, %
 $\|\cdot\|_{\gUstar}$ \pageref{lOptCost0}\&\pageref{pageUexpnorm},
 $u,x,y,x_0$ \pageref{dStateOutput},
 $\W^{1,2}_\omega$ \pageref{pageW12},
 $\sZ,\uZ,\tY,\qQR,\qRR$ \pageref{pagesZ},
 $\pi_\pm$ \pageref{pagepiplus},
 $\pi_E$ \pageref{pagepiplus},
 $\rho(A):=\sigma(A)^c$, $\rho_\infty(A)$ \pageref{pagerhoinfty},
 $\tau$ \pageref{pagepiplus},
 $\khi_E$ \pageref{pagepiplus},
 $\omega_A$ \pageref{pageomegaA}.

By $A,B,C,D,K,F,M,N,X$ we denote the generators
 (pp.\ \pageref{pagegenerators1}\&\pageref{pagegenerators2}\&\pageref{pagegenerators3}) of 
 $\qA,\qB,\qC,\qD,\qK,\qF,\qM,\qN,\qX$ respectively; %
 similarly for other pairs of capital and script letters
 (and sub- and superscripts).

Subscripts: 
 $\hqD_\ALS$ \pageref{pageCharFct},
 $\hqXc$ \pageref{pagehqXALSext};
 $\ALSp$ \pageref{eALSp},
 $\ALSClL$ \pageref{eALSb}\&\pageref{dIRE},
 $\Cc,\Dc$ \pageref{lComp},
  $\ALSExt$ \pageref{eALSExt},
 $\ALS_L$ \pageref{eALS_L},
 $\ALSopt$ \pageref{eALSopt};
  $\Bw^*$ \pageref{pageBw}\&\pageref{eCw},
 $\Cw$, $\Kw$ \pageref{eCw}.

Superscripts:
 $\hat{\cdot}$ \pageref{eLaplace}\&\pageref{lCharFct},
 $\check{\cdot}$ \pageref{lCharFct}, %
 $\hqA;\hqB;\hqC,\hqK$ \pageref{lCharFct},
 $\hqD,\hqX,\hqF,\hqM$ \pageref{TransferFct1}\&\pageref{lCharFct},
 $B^*;C^*,K^*$ \pageref{pageB*}; %
 $\ALS^\rmd$: see ``dual system'' below;
 $\ALStau$ \pageref{pageALStau};
 $\qA^t;\ \qBt:=\qB\tau^t\pi_+;\ \qCt:=\piOt\qC,\ \qKt;\
  \qDt:=\piOt\qD\piOt,\ \qFt,\ \qXt,\ \qMt,\ \qNt$ %
 (\ref{eWPLSdiscr}).
Non-generic symbols having superscripts:
 $P^t:=\piOt + \tau^{-t}\oqK\qBt$ \pageref{pagePt}, %
 $\qSt$ (\ref{eS2XSX=}) \& \pageref{pageqSt2},
 $\hqS$ (\ref{eS2hXSX=}).

Acronyms:
 ARE$:=$Algebraic RE \pageref{dARE}, 
 \BwARE\ \pageref{pageBwARE},
 FCC means that $\gUstar(x_0)\ne\tyhja\ \all x_0\in H$
 (cf.\ pp. \pageref{pageFCC}\&\pageref{eFCCUexp}\&\pageref{pageoutput-FCC});
 \optIRE\ \pageref{pageoptIRE},
 \hoptIRE\ \pageref{pagehoptIRE},
 IRE$:=$integral RE \pageref{pageIRE};
 q.r.c., r.c., d.c. \pageref{drcf}\&\pageref{pageqrc-}\&\pageref{pageexprc};
 RE$:=$Riccati equation, 
 $\MTIC$ \pageref{pageMTIC},
 $\WTIC$ \pageref{pageWTIC},
 RCC \pageref{pageRCC},
 $\rconn$ \pageref{pagerconn}, 
 \SIRE,\hSIRE\ \pageref{pageSIRE},
 $\TIC:=\TIC_0$,
 $\TIC_\omega$ \pageref{pageTIC},
 WPLS \pageref{dWPLS0},
 WR, SR, UR, ULR \pageref{dReg}.

Terms:
 {\em adjoint} see dual;
 {\em admissible} \pageref{pageadmissible1}\&\pageref{pageadmissible3}\&\pageref{pageadmissible2}\&\pageref{pageadmissibleL},
 {\em bounded $B,C$} \pageref{pageboundedBC}, %
 {\em $B$ not maximally unbounded} \pageref{pageBnotmax},
 {\em characteristic function} \pageref{pageCharFct},
 {\em closed-loop system} \pageref{pageClL-L}\&\pageref{pageClL-KF},
 {\em control in WPLS form} \pageref{pagecontrolWPLSform}, %
 {\em coprime} \pageref{drcf},
 {\em cost function $\cJ$} \pageref{pagecostfunction},
 {\em detectable} \pageref{pageExpDet}\&\pageref{pageExpDet2},
 {\em discrete subset} \pageref{pagediscrete0}, %
 {\em dual system}
 ($\ALS^\rmd=\smash{\big(\ssysm{A^*\|B^*\crh C^*\|}\big)}$,
  $\smash{\shat{\qD^\rmd}}(s)=\hqD(\bar s)^*$):
 p.~\pageref{pagedualsystem} and [M02], %
 {\em dynamic feedback controller} \pageref{pageDFC},
 {\em estimatable} \pageref{pageEst},
 {\em exponentially stable} \pageref{pageExpStab}\&\pageref{pageExpStable2},
 {\em exponentially stabilizing} \pageref{pageExpStab0}\&\pageref{pagestabilizing},
 {\em external input} $\uc$, $u_L$ 
 \pageref{pageextinput1}\&\pageref{pageextinputIntro}\&\pageref{pageextinputuL},
 {\em factorization} \pageref{pagercf}\&\pageref{pageSpF},
 {\em feedthrough} \pageref{pagefeedthrough}, %
 {\em generators} \pageref{pagegenerators1}\&\pageref{pagegenerators2}\&\pageref{pagegenerators3},%
\\{\em $J$-coercive} = $\PTO\in\cG\BL$ = invertible {\em Popov} Toeplitz operator
 (= ``no invariant zeros'' = $\|\qD u\|_{\gUstar}\ge \|u\|\ (u\in\gUstar(0))$
   if $J\gg0$) \pageref{J-coercive}\&\pageref{Assumptions}\&\pageref{pageqS},\\
 {\em $J$-optimal} = ``optimal'' (= ``minimizing'' if $J\ge0$) %
 \pageref{dJopt0}\&\pageref{pageJ-optimalK}, %
 {\em jointly stabilizable and detectable} \pageref{pageJointly}, %
 {\em meromorphic} \pageref{pagemeromorphic1}\&\pageref{pagemeromorphic}, %
 {\em nondiscrete} \pageref{pagediscrete0}, %
 {\em optimizable} \pageref{pageoptimizable},
 {\em output-FCC} \pageref{pageoutput-FCC},
 {\em output-stabilizing} \pageref{pageoutput-stable}\&\pageref{pageoutput-stabilizable},
 {\em Popov} \pageref{pagePTO}, %
 {\em Pritchard--Salamon systems} \pageref{pagePS},
 {\em realization} \pageref{pageRealization},
 {\em regular} \pageref{dReg},
 {\em Riccati operator $\Ric$} ($J$-optimal cost operator) \pageref{pageRiccatioperator}, %
 {\em signature} \pageref{pagesignature}, 
 {\em SOS-stable} \pageref{pageSOS-stabilizing}\&\pageref{pagestabilizing}\&\pageref{pageSOS-stable},
 {\em stabilizing} \pageref{pagestabilizing}, see also ``$\gUstar$-stabilizing'';
 {\em stabilizable} \pageref{pagestabilizable},
 {\em stable} \pageref{pageStable}\&\pageref{pagestable}\&\pageref{pageStableIntro}\&\pageref{pagestableTIC}, %
 {\em state feedback} \pageref{pagestate-feedback},
 {\em state-FCC} \pageref{pagestate-FCC},
 {\em transfer function} 
 \pageref{TransferFct1}\&\pageref{pageTrFctIntro}\&\pageref{rTrFct},
 {\em $\gUstar$-stabilizing} \pageref{pagegUstarstabilizing3}\&\pageref{pagegUstarstabilizing2}\&\pageref{pagegUstarstabilizing4},
 {\em Yosida extension} \pageref{pageYosidaExt}.

Most of the notation and terminology
 and some proofs and further results are presented 
 in greater detail in [M02]
 (under the replacements $\gUstar\mapsto \gUg$,
 opt$\mapsto$crit, $J$-optimal$\mapsto J$-critical,
 ARE$\mapsto$[e]CARE, IRE$\mapsto$IARE,
 $\qA,\qB,...\mapsto{\mathbb A,\mathbb B,...}$).
\section{Well-posed linear systems (WPLSs)}\label{sWPLS} %

If the generators of the system (\ref{exy=ABCDintro}) are {\em bounded},
 i.e., $\ABCD\in\BL(H\times U,H\times Y)$,
 then the unique solution of (\ref{exy=ABCDintro})
 is obviously given by the system
\begin{eqnarray}\label{exy=qABCDintro}
\left\{\bigmatrix{
  x(t)&=&\qA^t x_0 + \qB\tau^t u\\
  y&=&\qC x_0 + \qD u,
}\right.
\end{eqnarray}
 where\vspace{-2ex}
\begin{eqnarray}
   \label{eWPLSfromGensintro}
 \begin{aligned}
   \qA^t  &=\efn^{At},\ & \qB\tau^t u&=\int_0^t \qA^{t-s}Bu(s)\,ds,\\
   (\qC x_0)(t)&=C\qA^t x_0,\ & (\qD u)(t)&=C\qB\tau^t u+Du(t).
 \end{aligned}
 \end{eqnarray}
This is illustrated in Figure~\ref{DiagramFigintro}.
 
\begin{figure} %
\[
        \begin{picture}(40,40)(0,-10)

        \thicklines

        \put(1,-15){
        \framebox(50,33){ %
        $
        \bsysm{\qA \| \qB \tau \crh 
        \qC \| \qD\phantom{\tau}} %
        $}}

        \thinlines

        \put(13,41.5){\vector (0,-1){23}}
        \put(2,32){$x_0$}

        \put(40,41.5){\vector (0,-1){23}}
        \put(31,32){$u$}

        \put(2.5,10){\vector (-1,0){30}}
        \put(-13,13){$x$}
        \put(69,8){$x=\qA x_0+\qB\tau u$}

        \put(2.5,-6){\vector (-1,0){30}}
        \put(-13,-3){$y$}
        \put(69,-8){$y=\qC x_0+\qD u$}

        \end{picture}
        \]
        \caption{Input/state/output diagram of a WPLS {${\btwotwo{\qA}{\qB}{\qC}{\qD}}$}}
        \label{DiagramFigintro}
\end{figure}

The formulae (\ref{exy=ABCDintro}),
 (\ref{exy=qABCDintro}) and (\ref{eWPLSfromGensintro}) are actually valid
 for rather unbounded generators.
Therefore, the WPLSs %
 are defined %
 by requiring $\qA$ to be a strongly continuous semigroup,
 $\qD$ to be time-invariant and causal,
 and $\qB$ and $\qC$ to be compatible with $\qA$ and $\qD$;
 in addition, one requires that
 $\sbm{\qA^t&\qB\tau^t\cr \qC&\qD}$ is linear and
 continuous $H\times \Lloc^2(\R_+;U)\to H\times \Lloc^2(\R_+;Y)$
 for each $t\ge0$,
 equivalently, that %
 \begin{equation}
   \label{}
   \|x(t)\|_H^2+\int_0^t\|y(s)\|_Y^2\,ds\le K_t\big(\|x_0\|_H^2+\int_0^t\|u(s)\|_U^2\,ds\big)
 \end{equation}
 for some (equivalently, all) $t>0$, where $K_t$ depends on $t$ only.
An equivalent formulation (due to Olof Staffans)
 is given in Definition \ref{dWPLS0},
 where we use the unique natural extensions of $\qB$ and $\qD$
 that allow the inputs to be defined on the whole real line,
 thus simplifying several formulae.

We use the notation
 $\L^2_\omega=\efn^{\omega \cdot}\L^2   \label{pageL2omega}
  =\{f\I \efn^{-\omega\cdot}f\in\L^2\}$
 (similarly, $\L^1_\omega:=\efn^{\omega\cdot}\L^1$), %
 $(\tau^t u)(s):=u(t+s)$ and $\pi_\pm u:=\khi_{\R_\pm} u$,\label{pagepiplus}
 where $\khi_E(t):=
 {\scriptsize\begin{cases}
   1,& t\in E;\cr 
   0,& t\not\in E
 \end{cases}}$.
(Similarly, $\pi_E u:=\khi_E u$ when $E\sub\R$.)
We also consider $\pi_+$ as the projection $\L^2(\R;U)\to\L^2(\R_+;U)$
 or as its adjoint. %

Throughout this article, we assume that $\ALS=\qABCD$ is a WPLS on $(U,H,Y)$,
 i.e., that 1.--4.\ below hold for some $\omega\in\R$:
\begin{defin}[WPLS and stability]\label{dWPLS0}
Let $\omega \in \R$. 
An %
 {\em $\omega$-stable well-posed linear system
 on $(U,H,Y)$} %
\index{WPLS|emph}\index{stable!omega-@$\omega$-|emph}
 is a quadruple %
  $\ALS = \SmallbSystem $,
 where $\qA^t$, $\qB$, $\qC$, and $\qD$ are bounded linear operators of the
following type:

\begin{enumerate}

\item[1.] $\qA^\cdot \colon H \to H$ is a strongly continuous semigroup of
bounded linear operators on~$H$ satisfying
 $\sup_{t\ge0} \|\efn^{-\omega t} \qA^t\|_H < \infty$;

\item[2.] $\qB\colon  \L^2_\omega(\R;U) \to H$ satisfies 
        $
        \qA^t \qB u = \qB \tau^t  \piminus u
        $
        for all $u \in \L^2_\omega(\R;U)$ and $t \in \rplus$;

\item[3.] $\qC\colon  H \to \L^2_\omega(\R;Y)$ satisfies
        $
        \qC \qA^t x = \piplus \tau^t \qC x
        $
        for all $x \in H$ and $t \in \rplus$;

\item[4.] $\qD\colon \L^2_\omega(\R;U) \to \L^2_\omega(\R;Y)$ satisfies
        $\tau^t \qD u = \qD \tau^t u$, 
        $\piminus \qD \piplus u = 0$, and
        $\piplus \qD \piminus u = \qC\qB u$
        for all $u \in \L^2_\omega(\R;U)$ and $t \in \R$.

\end{enumerate}

The different components of $\ALS=\SmallbSystem$ are named as follows:
  $U$ is the \indtermem{input space},
 $H$ the \indtermem{state space},
 $Y$ the \indtermem{output space},
 $\qA$ the \indterm{semigroup},
 $\qB$ the \indtermem{reachability map},
 $\qC$ the \indtermem{observability map},
 and $\qD$ the \indtermem{I/O map} (input/output map) of $\ALS$.  

We say that $\qA$ (resp.\ $\qB$, $\qC$, $\qD$) is {\em $\alpha$-stable} %
 if 1. (resp.\ 2., 3., 4.) holds for $\omega=\alpha$.
{\em Stable}\label{pageStable} means $0$-stable;
 {\em exponentially stable}\label{pageExpStable2} means $\omega$-stable for some $\omega<0$.
The system is {\em output stable}\label{pageoutput-stable} (resp.\ {\em SOS-stable}\label{pageSOS-stable})
 if $\qC$ (resp.\ $\qC$ and $\qD$) is stable.
We set $\ALStau\label{pageALStau}:=\qABtauCD$ (cf.\ (\ref{StateOutputDef})). 
\end{defin}

(A {\em SOS} (Stable-Output System) satisfies $y\in\L^2$
 for all $x_0\in H,\ u\in\L^2$, where $y:=\qC x_0+\qD u$.)

Any sub- or superscripts of a system are inherited by its
 parts and generators (see Lemma~\ref{lABC} and Definition~\ref{dReg});
 e.g., $\qA_L,\qB_L,\qC_L,\qD_L$ denote the maps 
 and $A_L,\ B_L,\ C_L,\ D_L$ the generators of $\ALS_L$
 (in Lemma~\ref{lALSL0}).
Practically all conventions above and below follow [S04], [M02] etc.

Exponential stability of a system is equivalent to that of its
 semigroup, hence Datko's Theorem leads to the following:
\begin{lemma}\label{lExpStable} %
A WPLS is  $\omega$-stable for any $\omega>
 \omega_A\label{pageomegaA}:=\inf_{t>0}[t^{-1}\log\|\qA^t\|]$.
In particular, it is exponentially stable
 iff $\qA x_0\in \L^2(\R_+;H)$ for all $x_0\in H$.\noproof
\end{lemma}

(See Lemmas 6.1.10(a1) and A.4.5 of [M02].)
\begin{defin}[State and output]\label{dStateOutput}
With initial time zero,
 \indtermem{initial value} $x_0 \in H$, and \indtermem{control}
 (or \indtermem{input})
$u \in \L^2_\omega(\R_+;U)$, the controlled
 \indtermem{state} $x(t)\in H$ at time $t \in
\rplus$ and the \indtermem{output} $y\in \L^2_\omega(\rplus,Y)$ of $\ALS$
 are given by (cf.\ Figure \ref{DiagramFigintro})
        \begin{equation}
        \label{StateOutputDef}
        \bbm{ x(t) \cr y}
        = \bbm{\qA^t & \qB\tau^t \cr \qC & \qD} \ \ 
        \bbm{x_0 \cr  u}
        = \bbm{
        \qA^t x_0 + \qB\tau^t u \cr 
        \qC x_0 + \qD  u}.
        \end{equation}

Sometimes we use the
 equivalent notation %
\begin{equation}
  \label{eWPLSdiscr} %
 \!\!\!\!\!\!\!
  \bbm{\qA^t & \qB^t\cr 
       \qC^t & \qD^t} 
  := \bbm{\phantom{\piOt} \qA^t & \phantom{m} \qB\tau^t\piOt\cr 
                          \piOt\qC&\piOt\qD\piOt}:
    \bbm{x_0\cr u}\mapsto \bbm{x(t)\cr \piOt y}.
\end{equation}
\end{defin}

G. Weiss et al.\ use symbols\label{pageWeissWPLSdef}
 $\sbm{\bbbT_t & \Phi_t\cr \Psi_t & \bbbF_t}
  :=  \sbm{\rlap{$\scriptstyle \qA^t$}\phantom{\bbbT_t} & \qB^t\cr 
           \rlap{$\scriptstyle\qC^t$}\phantom{\Phi_t} & \qD^t}$
 and a different but equivalent definition of WPLSs.

By causality, the state and output (in particular, $\qD$ and $\qB\tau$)
 are well defined for
 any $u\in\L^2_\loc(\R_+;U)$  (with $y\in\L^2_\loc(\R_+;Y)$),
 or even $u\in\L^2_\omega(\R;U)+\L^2_\loc(\R_+;U)$.

The existence of a feedthrough operator (``$D$'') is 
 equivalent to regularity (Definition~\ref{dReg}), but
 a WPLS always has generators $\ABC$
 that satisfy the rest of (\ref{eWPLSfromGensintro}):
\begin{lemma}($A,B,C$)\label{lABC}
Let $A$ be the generator of $\qA$ and let $\alpha\in\rho(A)$.\footnote{%
The exact value of $\alpha$ is insignificant, since resulting norms
 on $H_1$ or $H_{-1}$ are equivalent, by the resolvent equation.}

We set $H_1:=\Dom(A)$\label{pageH1} with $\|x\|_{H_1}:=\|(\alpha-A)x\|_H$ 
 (this is equivalent to the graph norm),
 and define $H_{-1}$ to be the completion of $H$ under the norm $\|(\alpha-A)^{-1}\cdot\|_H$
 (thus $H_1\sub H\sub H_{-1}$).

The following hold:
\begin{itemlist}
 \item[(a)] 
$\qA$ can be isometrically extended to $H_{-1}$
 and restricted to $H_1$. 
We identify the three semigroups (``$\qA$'')
 and their generators (``$A$'');
 thus, the map $\alpha-A$ is an isometric isomorphism of $H_{n}$ onto $H_{n-1}$
 ($n=0,1$).

 \item[(b)]
There is a unique \indtermem{input operator} $B\in\BL(U,H_{-1})$ s.t.\
 $(u\in\L^2_\loc(\R_+;U), \ \ t\ge0)$
 \begin{equation}
   \label{eqB} %
 \qB\tau^t  u %
   =\int_{0}^t \qA^{t-s}Bu(s)\,ds\in H
 \end{equation}
 (the integration is carried out in $H_{-1}$
  but the integral belongs to $H$).
Moreover, $x:=\qA x_0+\qB\tau u$  %
 satisfies $x'=Ax+Bu$ in $H_{-1}$ a.e.\ on $\R_+$
 and $x(t)-x_0=\int_0^t (Ax+Bu)\,dm$
 for all $t\ge0$, $x_0\in H$, $u\in\Lloc^2(\R_+;U)$.

 \item[(c)]
There is a unique \indtermem{output operator} $C\in\BL(H_1,Y)$ s.t.\
\begin{equation}
  \label{eqC}
  (\qC x_0)(t)=C\qA^t x_0\ \ \ (\all x_0\in H_1,\ \ t\ge0).
\end{equation}
Moreover, $(\qC x_0)(t)=\Cw\qA^t x_0$ for a.e.\ $t>0$ and all $x_0\in H$
 (see (\ref{eCw}) for $\Cw$). %
\end{itemlist}

\noindent
We say that {\em $\ALS$ is generated by $\ssysbm{A\| B\crh C\| }$},\index{generated}
 and we call $\ssysbm{A\| B\crh C\| }$
 the \indtermem{generators}\label{pagegenerators2} of $\ALS$;
 they are independent of $\alpha$ (and $\omega$).
Also the following hold:
\begin{itemlist}
  \item[(d)] $\ssysbm{A\| B\crh C\| }$ determine $\ssysbm{\qA\| \qB\crh \qC\| }$ uniquely and $\qD$ modulo an additive constant from $\BL(U,Y)$.
\end{itemlist}\itemlistnoproof%
\end{lemma}

We consider $H$ as the pivot space, so that $H_{-1}=\Dom(A^*)^*$,
 $B^*\in\BL(\Dom(A^*),U)$,\label{pageB*} and
 $C^*\in\BL(Y,\Dom(A)^*)$ (see Definition 6.1.17 of [M02] for details).

Let $\omega\in\R$.
We define $\TIC_\omega(U,Y)$\label{pageTIC}
 to be the (closed) subspace of operators 
 $\qD\in\BL(\L^2_\omega(\R;U);\L^2_\omega(\R;Y))$ that are
 {\em causal} (i.e., $\pi_-\qD\pi_+=0$) and 
 \indtermem{time-invariant}, i.e.\ \ $\tau^t \qD=\qD\tau^t $ for all $t\in\R$.
The I/O maps of WPLSs are exactly all such operators
 ($\TIC_\infty(U,Y):=\cup_{\omega\in\R}\TIC_\omega(U,Y)$,
 often called ``the well-posed I/O maps'').
In fact, they can be identified with proper transfer functions
 (i.e., functions bounded and holomorphic on some right half-plane,
 which we denote by
 $\H^\infty_\infty(U,Y)\label{pageHooHoo}
   :=\cup_{\omega\in\R}\H^\infty(\C_\omega^+;\BL(U,Y))$):
\begin{theorem}[Transfer functions $\pmbold{\hqD}$]\label{TransferFct1} %
For each $\qD\in\TIC_\omega(U,Y)$, there is a unique function
  $\widehat{\qD}\in \H^\infty(\C_\omega^+;\BL(U,Y))$,
 called the {\em transfer function} (or symbol) of $\qD$,
 s.t.\ $\widehat{\qD u}=\widehat{\qD}\hat u$ on $\C_\omega^+$
 for all $u\in \L^2_\omega(\R_+;U)$.
The mapping $\qD\mapsto\hqD$ is an isometric isomorphism
 of $\TIC_\omega(U,Y)$ onto $\H^\infty(\C_\omega^+;\BL(U,Y))$.\noproof
\end{theorem}

Here $\BL(U,Y)$ denotes the space of bounded linear operators $U\to Y$,
 $\Hoo(\C_\omega^+;\BL(U,Y))$\label{pageHoo}
 denotes the Banach space %
 of bounded holomorphic functions $\C_\omega^+\to\BL(U,Y)$,
 and $\hu$ denotes the Laplace transform of $u$:
 \begin{equation}
   \label{eLaplace}
   \hu(s):=\int_\R \efn^{-st}u(t)\,dt \ \ \ \ \
       (s\in\C_\omega^+:=\{s\in\C\I \re s>\omega\}).
 \end{equation}

If $f$ is holomorphic on $\C_\omega^+$, and $\Omega\sub\C_\omega^+$ is open,
 then we identify $f$ and $f\raj{\Omega}$. %
In fact, we do this whenever $f$ is holomorphic on $\C_\omega^+\pois E$,
 where $E$ does not have limit points on $\C_\omega^+$. %
Since any holomorphic extensions to right half-planes are unique,
 this does not cause problems
 (not even with $E$ if we remove removable singularities). %

A {\em realization}\label{pageRealization} of $\qD$ or $\hqD$
 means a WPLS whose I/O map is $\qD$.

If $\hqD$ has a limit at infinity
 (along the positive real axis), then the system is called regular:
\begin{defin}[$D$, Regularity]\label{dReg} %
We call $\qD\in\TIC_\omega(U,Y)$  (and $\hqD$ and $\SmallbSystem$)
 {\em weakly (resp.\ strongly, uniformly) regular (WR (resp.\ SR, UR))} 
 with {\em feedthrough\label{pagefeedthrough} operator} $\hqD(+\infty):=D\in\BL(U,Y)$
 if $\hqD(s)\to D$ weakly (resp.\ strongly)
 as $s\to+\infty$ on $(\omega,+\infty)$.

We call $\qD$ {\em ULR} (uniformly line-regular) %
 if $\|\hqD(s)-D\|\to0$ as $\re s\to+\infty$
 (uniformly with respect to $\im s$).
\end{defin}
If $\ALS$ is WR, then we say that $\ABCD$ are the
 {\em generators}\label{pagegenerators3}
 of $\ALS$, %
 since they determine the system uniquely,
 and we sometimes denote $\ALS$ by $\ssyspm{A\|B\crh C\|D}\label{pageABCD}$.
Any WPLS with {\em bounded}\label{pageboundedBC}
 $B$ or $C$ (i.e., $B\in\BL(U,H)$ or $C\in\BL(H,Y)$)
 is ULR, by Lemma 6.3.16 of [M02].
An equivalent condition for the weak regularity of $\ALS$ is that
 $(\alpha-A)^{-1}BU\sub \Dom(\Cw)$, where
\begin{equation}
  \label{eCw}
 \!\!\!\!\!\!\!\!\!\!
 \Dom(\Cw) :=\{x\in H \I
 \Cw x:=\wlim_{s\to+\infty}\,Cs(s-A)^{-1}x \ \text{exists}\!\!\!\}.
\end{equation}
(Here $\wlim$ is the weak limit (in $Y$).
The above condition %
   is independent of $\alpha\in\rho(A)$.)
The map $\Cw:\Dom(\Cw)\to Y$ is called the {\em weak Yosida\label{pageYosidaExt} extension} of $C$).
If $\ALS$ is WR and $\omega$-stable, then $\hqD(s)=D+\Cw(s-A)^{-1}B$
 when $\re s>\omega$,
 and $y=\Cw x +Du$ a.e.\ for all $x_0\in H$ and all $u\in\Lloc^2(\R_+;U)$.
Similar claims hold for $\Cs$, $\slim$ and ``SR''.

Using Lemma~\ref{lABC},
 one can show that any $\tsysbm{A\|B\crh C\|D}\in\BL(H\times U,H_{-1}\times Y)$
 are the generators of a WR WPLS iff  $\tsysbm{\qA^t\| \qB\tau^t\crh \qC\|\qD}$
 defined by (\ref{eWPLSfromGensintro})
 a.e.\ (with $\Cw$ in place of $C$)
 are bounded $H\times \L^2([0,t];U)\to H\times L^2([0,t];Y)$
 for some (hence all) $t>0$.
In (\ref{eWPLSfromGensintro}), ``$\qAt=\efn^{At}$'' must be interpreted 
 as the requirement
 that $A$ generates a $C_0$-semigroup $\qA^\cdot$.

The dual system $\tsyspm{A^*\|C^*\crh B^*\|D^*}$ can be defined for arbitrary WPLSs:
\begin{lemma}[Dual system \pmbold{$\ALS^\rmd$}]\label{lDualALS} %
If $\ALS$ is an $\omega$-stable WPLS, then so is 
 its {\em dual system}\index{dual system ($\ALS^\rmd$)|emph}
\begin{equation}
  \label{elDualALS}
 \ALS^\rmd:=\bsysbm{\qA^\rmd\|\qC^\rmd\crh \qB^\rmd \| \qD^\rmd}
        :=\bsysbm{\qA^*\|\qC^*\Refl\crh \Refl\qB^* \| \Refl\qD^*\Refl}
\end{equation}
 (over $(Y,H,U)$), where $(\Refl u)(t):=u(-t)$.
Moreover, $(\ALS^\rmd)^\rmd=\ALS$,
 and $\smash{\sbm{A^*&C^*\cr B^*&}}$
 ($\smash{\sbm{A^*&C^*\cr B^*&D^*}}$ if $\ALS$ is WR)
 are the generators of $\ALSd$,
 and $\smash{\shat{\qD^\rmd}}(s)
  =\hqD(\bar s)^*\ \all s\in\C_\omega^+$.\noproof
\end{lemma} 

(This is well-known, see Lemmas 6.1.4, 6.2.2 and 6.2.9(b) of [M02].)
We use $\L^2$ as the pivot space (p.~898 of [M02]);
 e.g., $\int_\R \p{\qC x_0,\ty}(t)\,dt = \p{x_0,\qC^*\ty}_H$.
Thus, $\ALS^\rmd$ is independent of $\omega$
 (and $\qC\in\BL(H,\L^2_\omega(\R;Y))
   \IFF \qC^*\in\BL(\L^2_{-\omega}(\R;Y),H)$).

\NotesFor{Section \ref{sWPLS}}\label{pageNotes-sWPLS} %
Everything in this section is well known; see, e.g., [W94a] and [W94b]
 (or Sections 6.1--6.2 of [M02]).
Much more on WPLSs can be found in [M02] too, 
 but [S04] is the most thorough book on the subject
 and also covers $\L^p$ signals for $p\ne2$ and for general Banach
 spaces in place of $U,H,Y$.

The {Lax--Phillips scattering theory} %
 and the operator-based model theory
 of Béla Sz.-Nagy and Ciprian Foia{\c s} %
 have been shown equivalent to WPLSs (see Chapter 11 of [S04]).
The former has been extensively developed in the (ex--) Soviet Union
 area by Damir Z. Arov and others (cf.\ [AN96]), independently of WPLSs.
See pp.\ 23 and~167 of [M02] for further details and references.

\section{State feedback}\label{sFeedback} %

In this section we first define (static) output feedback
 (Lemma~\ref{lALSL0}).
Then we extend state feedback
 (the formula $u(t)=Kx(t)$) to WPLSs,
 first in a ``generalized'' sense
 (Definition \ref{dWPLSform})
 and then in the standard sense (Definition \ref{dAdmKF0}).
For the former
 one can more easily generalize 
 Theorems \ref{IntroTh-Uexp} and~\ref{IntroTh-Uout},
 but the latter is more desirable in the applications.

Output feedback means feeding the output $y$ back to the input $u$
 through some feedback operator $L\in\BL(Y,U)$,
 i.e., $u=L y+u_L$, where $u_L$ is the external\label{pageextinputuL} input, as in
 Figure~\ref{fOutputFeedback}.
Obviously, the closed loop formulas
 $\ssysbm{\qA_L\|\qB_L\tau\crh\qC_L\|\qD_L}:\sbm{x_0\cr u_L}\to\sbm{x\cr y}$
 can be uniquely solved iff $I-L\qD$ is invertible
 (equivalently, $I-L\qD\in\cG\TIC_\infty$).
The solution is the following:
\begin{lemma}[\pmbold{$\ALS_L$}]\label{lALSL0} %
Let $L\in\BL(Y,U)$ be an {\em admissible}\label{pageadmissibleL} output %
 feedback operator for $\ALS$ %
 (i.e., $I-L\qD\in\cG\TIC_\infty(U)$).
Then also the {\em closed-loop system}\label{pageClL-L} 
 $\ALS_L$ is a WPLS over $(U,H,Y)$, where
\begin{align}\label{eALS_L}
        \ALS_L 
        :&=
        \bsysbm{\qA_L \| \qB_L \crh \qC_L \| \qD_L}
        := \bsysbm{
        \qA+ \qB \tau L\Par{I - \qD L}^{-1}\qC \|
        \qB \Par{I - L \qD}^{-1} \crh
        \Par{I - \qD L}^{-1}\qC \|
        \qD\Par{I - L \qD}^{-1}}.
\end{align}\noproof %
\end{lemma}

(See, e.g., Section~6 of [W94b] for the proof.)

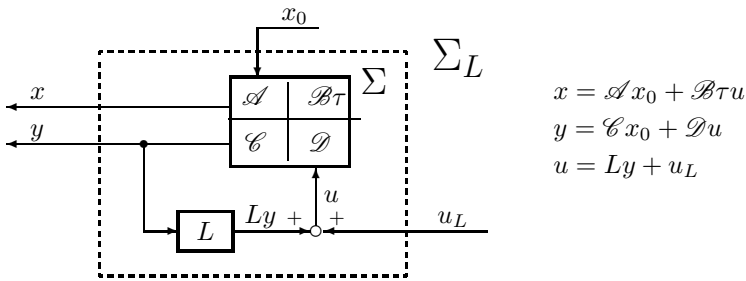
\begin{figure} %
        \[
        \begin{picture}(40,70)(40,-40)

        \thicklines

        \put(0,-15){
        \framebox(44,33){
        $
        \bsysm{\qA \| \qB\tau \crh 
        \qC \| \qD\phantom{\tau}}
        $}}

        \put(-20,-47){
        \framebox(20,14){$L$}}
        \put(52,12){\Large$\ALS$}

        \put(-50,-57){ \dashbox{2}(116,85){}}
        \put(79,21){\LARGE$\ALS_L$}

        \thinlines

        \put(36,37){\line (-1,0){23}}
        \put(13,37){\vector (0,-1){19}}
        \put(22,40){$x_0$}

        \put(3,7){\vector (-1,0){85}}
        \put(-73,10){$x$}

        \put(3,-7){\vector (-1,0){85}}
        \put(-73,-4){$y$}

        \put(125,10){$x=\qA x_0+\qB\tau u$}
        \put(125,-4){$y=\qC x_0+\qD u$}
        \put(125,-18){$u=Ly+u_L$}

        \put(100,-40){\vector (-1,0){62}} %
        \put(81,-37){$ u_L$}

        \put(-30,-7){\circle*{3}} %
        \put(-30,-7){\line (0,-1){33}}
        \put(-30,-40){\vector (1,0){13}}
        \put(4,-40){\vector (1,0){29}} %
        \put(8,-37){$Ly$}

        \put(35,-40){\circle{4}} %
        \put(24,-37){$\scriptstyle +$} 
        \put(40,-37){$\scriptstyle +$} 

        \put(35,-38){\vector (0,1){22}}
        \put(38,-29){$u$}

        \end{picture}
        \]
        \caption{Static output feedback}
        \label{fOutputFeedback}
\end{figure}

Next we define an important generalized form of state feedback.
Given a WPLS and a control law $\oqK:x_0\mapsto u$,
 the corresponding function $x_0\mapsto\sbm{x\cr y\cr u}$ is called
 a controlled WPLS form iff it is (the left column of) a WPLS
 (equivalently, iff $\oqAK:x_0\mapsto\sbm{x\cr u}$ is): %
\begin{defin}[$\oqK$, $\oALS$, WPLS form]\label{dWPLSform} %
We call the control $x_0\mapsto \oqK x_0$ %
 a {\em control for $\ALS$ in WPLS form}\index{WPLS form!control in|emph}\label{pagecontrolWPLSform} 
 (and $\oALS$ a {\em controlled WPLS form for $\ALS$})  %
 if $\oqK:H\to\Lloc^2(\R_+;U)$ is s.t.\
 $\oALS$ is a WPLS\footnote{%
Like here, we sometimes omit a zero input column (or output row) from a WPLS.%
}
 (on $(\{0\},H,Y\times U)$),
 where
\begin{equation}
  \label{eoACKdef}
 \oALS:=
  \bsysbm{\oqA\|\crh \oqC \| \cr \oqK\|}
  :=\bsysbm{\qA +\qB\tau\oqK\|\crh
              \qC +\qD\oqK\|\cr
              \oqK\|}. %
\end{equation}
\end{defin}

A control in WPLS form need not be of (well-posed) state-feedback form
 unless, e.g., $B$ is bounded
 (see p.~374 of [M02]).
However, it can be considered as being of non-well-posed state-feedback form,
 since $u(t)=(\oK)_\w x(t)$ a.e., by, e.g., (5.6) of [W94b]. %

Controls in WPLS form can be easily characterized in the frequency domain too:
\begin{lemma}[$\pmbold{\oALS}$]\label{lWPLSform} %
A triple $\oALS:=\oqACK$ is a controlled WPLS form for $\ALS$
 iff there exist $\omega\in\R$ and linear operators $\oA$ on $H$
 and $\oK:\Dom(\oA)\to U$ 
 s.t.\ $\oqK\in\BL(H,\L_\omega^2(\R_+;U))$,
 $\oqC=\qC+\qD\oqK,\ \oqA=\qA+\qB\tau\oqK$,
 $\sbm{\shat{\oqA x_0}\cr \shat{\oqK x_0}}(s)
 =\sbm{I\cr \oK}(s-\Ao)^{-1}x_0
 \ \all x_0\in H\ \all s\in\C_\omega^+$. 
\end{lemma}

\begin{proof}
``Only if'' is quite obvious, so we prove ``if''.
Assume, w.l.o.g., that $\ALS$ is $\omega$-stable
 (increase $\omega$ if necessary).
One easily verifies that $\oqC\in\BL(H,\L_\omega^2(\R_+;Y))$,
 $\oqA x_0\in\cC(\R_+;H),\ \oqA^t\in\BL(H)\ (t\ge0),\ \qA^0=I,\
 \|\oqA^t\|\le M\efn^{\omega t}$ (use (2.2) of [M02]).
By Lemma~\ref{lResSG}, $\oqA$ is a semigroup.
By Lemma 6.3.15 of [M02], $\ssysbm{\oqA\crh \oqK}$ is a WPLS
 (note that $\oK = \hoqK (s)(s-\oA)\in\BL(\Dom(\oA),U)$),
 so we can complete the proof by computing that
 (use 3.\&4. of Definition~\ref{dWPLS0} for $\qC$ and $\oqK$)
 \begin{align}
   \pi_+\tau^t\oqC &=\pi_+\tau^t\qC+\pi_+\tau^t\qD\oqK
 &&=&& \qC\qA^t + \pi_+\qD(\pi_+ + \pi_-) \tau^t\oqK\\
 &= \qC\qA^t + \pi_+\qD\oqK\oqA^t    + \qC\qB \tau^t\oqK
 &&=&& \qC\oqA^t +\qD\oqK\oqA^t
  \ = \ \oqC\oqA^t.
 \end{align}
\end{proof}

Obviously, $\oqK$ is a control in WPLS form for $\ALS$ iff
 it is  a control in WPLS form for $\qAB$. 
The dual condition is given below:
\begin{lemma}[$\oALS^\rmd$]\label{lWPLSformDual} %
$\oqK$ is a control in WPLS form for $\ALS$ iff
 $-\qB^\rmd$ is a control in WPLS form for $\bsysbm{\oqA^\rmd\| \oqK^\rmd}$.
\end{lemma}

(The latter condition contains the requirement that 
 $\bsysbm{\oqA^\rmd\| \oqK^\rmd}$ is a WPLS,
 i.e., that $\oqAK$ is. See (\ref{eoACKdef}) for $\oqA$
 and (\ref{elDualALS} for $()^\rmd$.)

\begin{proof}
$1^\circ$ ``Only if'': $\tqA:=(\oqA+(-\qBd)\tau\oqK^\rmd)^\rmd
  :=(\oqA+(-\qBd)\tau\oqK^\rmd)^*
 =\oqA^*-\qB^*\Refl\tau\Refl\oqK^* %
 = (\qA+\qB\tau\oqK)^*-\qB^*\tau^*\oqK^* = \qA^*=:\qA^\rmd$,
 hence $\sbm{\tqA\cr -\qB^\rmd}$ is a WPLS (since $\ALSd$ is).

$2^\circ$ {\em ``If'':} By $1^\circ$ (applied to 
 $\bsysbm{\oqA^\rmd\| \oqK^\rmd}$),
 $-\oqK$ is a control in WPLS form for $\bsysbm{\qA\| -\qB}$.
\end{proof}

Often one uses state feedback of form $u(t)=Kx(t)$ (pure),
 or $u(t)=Kx(t)+Fu(t)$ (non-pure) to stabilize and/or optimize
 the system, as in  Figure~\ref{fALSClL} or Theorem~\ref{IntroTh-Uexp}. 
Thus, we add an extra output signal $\qK x_0+\qF u$
 (which can be written as $\Kw x+Fu$ if the feedback $\qF$ is WR)
 that is fed back to the input ($u$).
This leads to the equation $u=\qK x_0+\qF u + \uc$ (cf.\ (\ref{ehxhuhyClL})),
 where $\qK$ and $\qF$ are to be chosen so that
 the solution $u=(I-\qF)^{-1}\qK x_0$ is the optimal input
 given any initial state $x_0$
 (when the {\em external}\label{pageextinput1} perturbation (input) $\uc$ is zero).
\begin{figure} %
        \[
        \begin{picture}(90,75)(20,-45)

        \thicklines

        \put(0,-29){ %
        \framebox(66,53){ %
        $\ExttauSystem$}}

        \thinlines
        \put(77,14){\Large$\ALSext$}

        \put(-44,-62){ \dashbox{2}(146,92){}}
        \put(111,24){\Huge$\ALSClL$}

        \put(42,33){\line (-1,0){20}}
        \put(22,33){\vector (0,-1){9}}
        \put(28,36){$x_0$}

        \put(3,12){\vector (-1,0){63}}
        \put(-50,15){$x$}

        \put(3,-4){\vector (-1,0){63}}
        \put(-50,-1){$y$}

        \put(3,-17){\vector (-1,0){63}}
        \put(-53,-15){$\qK x_0+\qF u$}
        \put(-15,-17){\circle*{3}}
        \put(-15,-17){\vector (0,-1){31}} %

        \put(-15,-50){\circle{4}}
        \put(-22,-42){$\scriptstyle +$} 
        \put(-28,-47){$\scriptstyle +$} 

        \put(-60,-50){\vector (1,0){43}}
        \put(-55,-47){$ \uc$}

        \put(-13,-50){\vector (1,0){250}}
        \put(62,-47){$u=(I-\qF)^{-1}\uc+(I-\qF)^{-1}\qK x_0$}
        \put(56,-50){\circle*{3}} %
        \put(56,-50){\vector (0,1){21}} %

        \end{picture}
        \]
        \caption{State-feedback connection $u(t)=Kx(t)+Fu(t)$}
        \label{fALSClL} 
\end{figure}
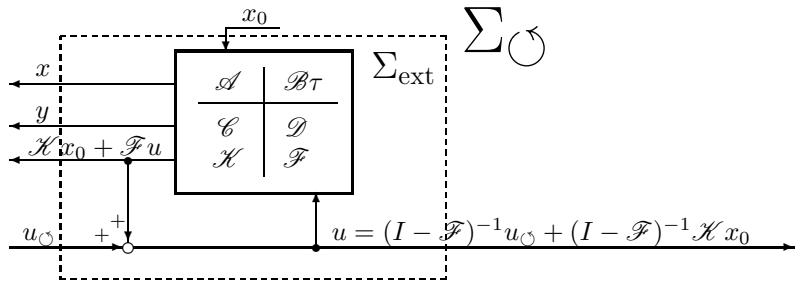

For the above solution to exist, $I-\qF$ must have a well-posed inverse,
 equivalently, 
 $I-\hqF$ must be boundedly invertible on some right half-plane;
 this makes 
 the map $\uc\to u$ from the external input $\uc$
 in Figure~\ref{fALSClL} is well-posed. 
Thus, a pair $\qKF$ is called admissible for $\qABCD$ iff $\qABKF$ is a WPLS
 and $I-\qF$ is invertible:
\begin{defin}[\pmbold{$\ALSClL,K,\qKF$}]\label{dAdmKF0} %
A pair $\bsysbm{\qK \| \qF}$ is called an
 {\em admissible\label{pageadmissible1} state-feedback\label{pagestate-feedback} pair} 
 for $\ALS$ %
 if the extended system
 \begin{equation}
   \label{eALSExt}
        \ALSExt := \bExtSystem
 \end{equation}
 is a $\WPLS$ %
 and $I-\qF\in\cG\TIC_\infty(U)$.

We set $\qX:=I-\qF,\ \qM\label{pageqM}:=\qX^{-1},\ \qN:=\qD\qM$
 and denote the corresponding closed-loop\label{pageClL-KF} system
 (see Figure~\ref{fALSClL}) %
\begin{align}
  \label{eALSb}
  \ALSClLtau 
  &=\bsysbm{\qAClL\|\qBClL\tau\crh \qCClL\|\qDClL\cr \qKClL\|\qFClL}
  =\bsysbm{\qA+\qB\tau\qM\qK\|\qB\qM\tau\crh \qC+\qD\qM\qK \| \qD\qM\cr
            \qM\qK\| \qM-I}\\
  \label{eALSbuLu}
  &=\ALSexttau \bbm{I&0\cr -\qK&I-\qF}^{-1}
  =\ALSexttau \bbm{I&0\cr \qM\qK&\qM} %
   :\bbm{x_0\cr \uc}\mapsto \bbm{x\cr y\cr u-\uc}.
\end{align}

If $\qF$ is weakly regular and $F=0$, then we call
 the generator $K$ (or $\Kw$) of $\qK$ a weakly regular
 {\em state-feedback operator} for $\ALS$.

We call $\qKF$ {\em stabilizing}\label{pagestabilizing} 
 if $\ALSClL$ is stable. We add ``[q.]r.c.-''\label{pageqrc-} if $\qN,\qM$ are [q.]r.c.
 (Definition~\ref{drcf} below).
If there exists a stabilizing state-feedback pair for $\ALS$,
 then $\ALS$ is called {\em stabilizable}\label{pagestabilizable}
 (similarly for exponentially, SOS- or output-stabilizing etc.).
\end{defin}

Obviously, $I-\qF\in\cG\TIC_\infty(U)$ iff $L=\bbm{I&0}$
 is an admissible output feedback operator for $\ALSext$.
By Lemma~\ref{lALSL0}, $\ALSClL$ is then indeed a WPLS (on $(U,H,Y\times U)$).
If $\qD$ and $\qF$ are strongly regular %
 with feedthrough operators $D$ and $F=0$, 
 then the generators of the two systems are as follows:
 \begin{align}\label{elALSClLgen}
   \ALSext=\bsyspm{A\|B\crh C\|D\cr K\|0}, \ \ \ 
   \ALSClL=\bsyspm{A+BK\|B\crh C+DK\|D\cr K\|0},
 \end{align}
 by Proposition 6.6.18(d4) of [M02] (or [W94b]). 
Observe that $\ALSClL\sbm{I\cr0}$ %
 is a controlled WPLS form.

We can reduce most output feedback results to dynamic feedback results:
\begin{remark}[``\pmbold{$\ALSClL=\ALS_L$}'']\label{routputfeedback=statefeedback} %
Any static output feedback can be written as (part of) %
 state feedback and vice versa.
\end{remark}

\begin{proof}
We observed above that
 the state feedback $\qKF$ for $\ALS$
 corresponds to the static output feedback
 $L=\bbm{0&I}$ for $\ALSext$,
 i.e., $\ALSClL=(\ALSExt)_I$. 
Conversely, static output feedback can be written as a special case of
 state feedback (set $\qKF=\bsysbm{L\qC\| L\qD}$ %
 and drop the bottom row of $\ALSClL$ to obtain $\ALS_L$). 

Moreover, given a WPLS $\ALS=\qABCD$,
 a pair $\qKF$ is admissible for $\ALS$ iff $\qKF$ is admissible for $\qAB$,
 and this is the case iff $\ssysbm{\qC\| 0&\qD\cr \qK\| 0&\qF}$
 is an admissible state-feedback pair for $\bsysbm{\qA\| 0&\qB}$
 (the closed-loop system equals $\ALSClL$ with a column of zeros
  inserted to the middle). %

(Similarly, a ``flow inverse'' (see, e.g., Section 6.2 of [S04])
 of $\ssysbm{\qA\|\qB\crh -\qK\|\qX}$
 means $\ssysbm{\qAClL\|\qBClL\crh \qKClL\|\qFClL+I}
 =\ssysbm{\qA+\qB\qKClL\|\qB\qM\crh \qM\qK\|\qM}$.)  %
\end{proof}

The state-feedback map $\qKClL$ determines the pair $\qKF$
 uniquely modulo $E\in\cG\BL(U)$:
\begin{lemma}[All \pmbold{$\qKF$}]\label{lAllqKF} %
Let $\qKF$ be an admissible state-feedback pair for $\ALS$.
Then all admissible state-feedback pairs $\tqKF$ leading
 to same control $\qKClL$ are given by
 \begin{equation}
   \label{eAllqKF}
   \tqKF = \bsysbm{E\qK \| I-E(I-\qF)}\ \ \ \ (E\in\cG\BL(U)).
 \end{equation}
\end{lemma}

Mnemonic: $\tqK=E\qK$, $\tqX=E\qX$, where $\qX:=\qM^{-1}$.

The following follows from a straight-forward computation:
\begin{lemma}\label{luc} %
Let $\qKF$ be an admissible state-feedback pair for $\ALS$.
Then $\xc=x$ and $\yc=y$ for any $x_0\in H$ and $u\in\Lloc^2(\R_+;U)$
 if $\uc=-\qK x_0+\qX u$, equivalently, if $u=\qKClL x_0+\qM \uc$.

Moreover, $\qKClLo$ is a control in WPLS form for $\qABCDClL$ %
 iff $\oqK:=\qKClL+\qM\qKClLo$ is a control in WPLS form for $\ALS$.
We have $\qKClLo=-\qK+\qX\oqK$.\noproof
\end{lemma}

Here $\xc:=\qAClL x_0+\qB\tau\uc$ and $\yc:=\qCClL x_0+\qDClL\uc$
 are the state and output of $\ALSClL$
 with input $\uc\in\Lloc^2(\R_+;U)$ and initial state $x_0$.

\NotesFor{Section \ref{sFeedback}}\label{pageNotes-sFeedback} %
Definition \ref{dWPLSform} and Lemma~\ref{lWPLSform} are from [M02],
 and they are  necessary tools for a complete Riccati equation
 theory, as shown by Theorem~\ref{SIRE}
 (and Example 8.4.13 of [M02], cf.\ ``3c.'' on p.~\pageref{page3c.}). 
Similar structures have implicitly been used in, e.g., [FLT88] and [Z96].
The rest of this section is essentially well known,
 mostly due to [W94b] (and [S98a]);
 see [S04] or [M02] for further results.

\section{Optimal control and $J$-coercivity}\label{sOpt} %

We shall present our main results on stabilization and
 factorization in Section~\ref{sStabOpt},
 the ARE theory in Section~\ref{sARE},
 and the IRE theory in Section~\ref{sIRE},
 thus generalizing the results of Section~\ref{sIntro}.

In this section
 (from Chapter~8 of [M02])
 we present the optimization setting and certain tools
 on which those results are based.
First we need to generalize ``minimal'' or ``optimal'' (``$J$-optimal'')
 control so as to cover
 1. all WPLSs, 2. alternative optimization domains to $\dUexp$ (\ref{eUexp}),
  and also 3. indefinite problems ($\cJ(0,\cdot)\not\ge0$).
Then we need a general coercivity assumption on cost functions
 (``$J$-coercivity'')
 that covers much more (\ref{ecJx2u2}) and (\ref{ecJy2u2})
 (in fact, all nonsingular control problems) %
 and yet guarantees the existence of a unique optimal control
 (under the corresponding FCC).

A reader who wants to avoid technical details may %
 consider $B$, $C$ and $D$
 bounded (so that they constitute a WPLS with any 
 $C_0$-semigroup $\qA$ on $H$)
 and $\gUstar=\dUexp$ (and $J=I$), as in Section~\ref{sIntro}.
Hypothesis~\ref{shgUstar} is redundant in that setting.

In Theorem~\ref{IntroTh-Uexp} we optimized over the set 
\begin{equation}
  \label{eUexp}
  \dUexp(x_0):=\{u\in\L^2(\R_+;U)\I x\in\L^2(\R_+;H)\},
\end{equation}
 sometimes called the set of exponentially
 (or internally) stabilizing controls
 (see (\ref{StateOutputDef}) for $x=x_{x_0,u}$).
Recently it has become popular to study optimization over a larger
 set of controls than $\dUexp(x_0)$, namely over the set
\begin{equation}
  \label{eUout}
  \dUout(x_0):=\{u\in\L^2(\R_+;U)\I y\in\L^2(\R_+;H)\}
\end{equation}
 of (externally or) output-stabilizing controls,
 as in  Theorem~\ref{IntroTh-Uout}.
By discretization (Lemma 7.2 of [WR00]), %
 one can show that $\dUexp\sub\dUout$
 (this means that $\dUexp(x_0)\sub\dUout(x_0)$ for all $x_0\in H$),
 i.e., that $u,x\in\L^2\ \THEN\ y\in\L^2$.
Sometimes the set
 $\dUstr(x_0)\label{pageUstr}:=\{u\in\dUout(x_0)\I \|x(t)\|_H\to0$ as $t\to+\infty\}$ 
 of strongly stabilizing controls is used,
 and in certain proofs
 (see, e.g., Corollaries \ref{cExpStabB1} and~\ref{cUout-B1}, %
 [M02], [M03b])
 we need a very special domain of optimization. 
In general, we shall denote the chosen domain of optimization
 by $\gUstar(\cdot)$
 (which a reader who wants to avoid technical details
    may read as $\dUexp(\cdot)$).
\begin{standinghypothesis}[$\pmb{\ALS,J,\gUstar,x,y}$]\label{shgUstar} %
Throughout this article, we assume that
 $\ALS=\qABCD$ is a WPLS on Hilbert spaces $(U,H,Y)$, that $J=J^*\in\BL(Y)$,
 and that $\gUstar$ and its parameters $\vartheta,\qQR,\qRR,\sZ,\uZ$
 are of the form explained in the following paragraph.
For any $u,x_0$, we define $x,y$ by (\ref{StateOutputDef}).
\end{standinghypothesis}

We assume that $\vartheta\in\R$, $\qABQR$ is a WPLS on $(U,H,\tY)$
 for some Hilbert space $\tY$;
  $\sZ$\label{pagesZ}
     is a Banach space and $\uZ$ is a topological vector space
 (e.g., a normed space);
 $\sZ\sub\uZ$ continuously;
 $\qQR\in\BL(H,\uZ),\ \qRR\in\BL(\L^2(\R_+;U),\uZ)$;
 and $\pi_+\tau^t z\in\sZ\IFF z\in\sZ\ (z\in\uZ,\ t>0)$.
Moreover, we set
\begin{equation}
  \label{egUstar}
 \gUstar(x_0):=
 \gUstar^\ALS(x_0):=
  \{u\in\L^2_\vartheta(\R_+;U) \I 
      \sbm{\qC & \qD\cr \qQR &\qRR} \sbm{x_0\cr u}\in\L^2\times \sZ\}.  
\end{equation}

Thus, Standing Hypothesis~\ref{shgUstar} equals
 Hypothesis 9.0.1 of [M02] (plus (\ref{StateOutputDef})).
Note that we can make $\gUstar$ equal to $\dUexp$
 (resp.\ to $\dUout$)
  by setting
 $\qQR=\qA$, $\qRR=\qB\tau$
 (resp.\  $\qQR=\qC$, $\qRR=\qD$), %
 $\tY=H$, $\sZ=\L^2$,
 $\vartheta=0$.
The following is obvious:
 \begin{lemma}\label{lgUstar} %
$\gUstar(\alpha x_0+\beta x_1)=\alpha\gUstar(x_0)+\beta\gUstar(x_1)$
 whenever $\alpha,\beta\in\C\pois\{0\},\ \gUstar(x_0)\ne\tyhja$.\noproof
 \end{lemma}

See Section 8.3 (and 9.0) of [M02] for further details and results.

Formula (\ref{StateOutputDef}) is equivalent to (\ref{exy=qABCDintro}) 
 as well as to $\sbm{x(t)\cr \piOt y}=\sbm{\qAt &\qBt\cr \qCt& \qDt}
 \sbm{x_0\cr u}$ (see (\ref{eWPLSdiscr})).

For any optimization domain $\gUstar$,
 the results are quite similar to those for $\dUexp$.
The main difference is that, instead of the exponential stability,
 we must require some other kind of stability
 for the closed-loop semigroup if $\gUstar\ne\dUexp$
 (cf.\ Theorem \ref{IntroTh-Uexp}(ii)). 
Further details will be given in Sections \ref{sARE}
 and~\ref{sIRE}. %

We want to have our theory applicable to any quadratic cost functions,
 hence we define the {\em cost function}\label{pagecostfunction}
 by
\begin{equation}
  \label{eJ}
  \cJ(x_0,u):=\p{y,Jy}_{\L^2(\R_+;Y)}\ \ \ (x_0\in H,\ u\in\dUexp(x_0)).
\end{equation}
If, e.g., we want to study the cost function (\ref{ecJy2u2})
 or $\cJ(x_0,u)=\|y\|_2^2+\|u\|_2^2$,
 given a system $\ALS$,
 we can achieve this by
 taking $J=\sbm{I&0\cr 0&I}\in\BL(Y\times U)$ %
 and replacing $\qC$ by $\sbm{\qC\cr 0}$ and $\qD$ by $\sbm{\qD\cr I}$
 (and $Y$ by $Y\times U$),
 as in the proof of Theorem~\ref{QRCF-Uout}.
See the proof of Corollary~\ref{cOptExpstab1}
 for the cost function (\ref{ecJx2u2}).

Optimization theory is needed, e.g., in minimization problems
 (as in Theorems \ref{IntroTh-Uexp} and~\ref{IntroTh-Uout}) and in 
 $\H^\infty$ control problems, %
 where $J$ is indefinite and a saddle point of the cost function is sought
 (``best control for the worst disturbance''),
 since it leads to a formula for the desired controller.
Since a minimum or a saddle point is necessarily 
 a (often unique)
 zero of the derivative of the cost function,
 in such problems the goal is to find a control that
 is optimal in the following sense:
\begin{defin}[$\pmbold{J}$-optimal]\label{dJopt0} %
Let $x_0\in H$. A control $u\in\gUstar(x_0)$ 
 is called {\em $J$-optimal} for $x_0$ %
 if the Fréchet derivative (on $\gUstar(x_0)$)
 of the cost function $u\mapsto \p{y,Jy}$ %
 is zero at $u$.

[Generalized] state-feedback $\qKF$ or $K$ [or $\oqK$ or $\oALS$] 
 (Definition \ref{dAdmKF0} [\ref{dWPLSform}])
 is called {\em $J$-optimal} if, for all $x_0\in H$,
 the control $\qKClL x_0$ [or $\oqK x_0$]
 is $J$-optimal for $x_0$.
\end{defin}

By Lemma \ref{lOptCost0}(b), $u$ is $J$-optimal iff
 $\p{\qD\eta,J y}_{\L^2}=0$ (i.e., ``$\p{\Delta y,Jy}_{\L^2}=0$'')
 for all $\eta\in\gUstar(0)$ (recall that $y:=\qC x_0+\qD u$).

The optimal cost is always unique, although an optimal control
 might be nonunique:
\begin{lemma}%
[Optimal cost $\pmbold{\cJ(x_0,\uopt)}$]%
\label{lOptCost0} %
\index{J-optimal cost@$J$-optimal cost} 
{\bf (a)} 
If $u$ and $\tu$ are $J$-optimal for $x_0\in H$,
 then the cost $\p{y,Jy}$ is the same for both $y=\qC x_0+\qD u$
 and for $\ty=\qC x_0+\qD\tu$.

{\bf (b)} For any $x_0\in H$ and $u\in\gUstar(x_0)$,
 the following are equivalent:
 \begin{itemlist}
\advance\leftskip4em
  \item[(i)] $u$ is $J$-optimal for $x_0$;
  \item[(ii)] $\p{\qD \eta,Jy}=0\ \ \all \eta\in\gUstar(0)$;
  \item[(iii)] $\cJ(x_0,u+\eta)
    =\p{y,Jy} + \p{\qD \eta,J\qD \eta}$
    \ \ \ $\all\eta\in\gUstar(0)$;
  \item[(iv)]  $\p{\qC \tx_0+\qD (\tu+\teta),J(\qC x_0+\qD (u+\eta))}
    =\p{\ty,Jy} + \p{\qD \teta,J\qD \eta}$
  whenever $\eta,\teta\in\gUstar(0)$ and $\tu$
 is $J$-optimal for $\tx_0\in H$, $\ty:=\qC \tx_0+\qD\tu$. %
 \end{itemlist} %

{\bf (c)} If there is at most one $J$-optimal control for $x_0=0$,
 then there is at most one $J$-optimal control for any $x_0\in H$.

{\bf (d)} If(f) $\cJ(0,\cdot)\ge0$, 
 then a control is minimizing iff it is $J$-optimal.\noproof %
\end{lemma}
 
\begin{proof} %
(a) 
We have $\p{\ty,J\ty}-\p{y,Jy}=\p{Jy,\ty-y}+\p{\ty-y,J\ty}=0$,
 because $\ty-y=\qD\eta$, 
 where $\eta:=\tu-u\in\gUstar(0)$.

(b) Obviously, ``(ii') $\re\p{\qD \eta,Jy}=0$ for all $\eta\in\gUstar(0)$''
 is equivalent to (ii) and to (iii) (use $i\eta$). %
But  $\frac{\rmd\cJ(x_0,\uopt+t\eta)}{\rmd t}(0)=2\re\p{y,JD\eta}$,
 hence (ii') is equivalent to (i).
Trivially, (iv) implies (iii); 
 conversely, by using (iii) three times to compute %
 $\cJ(x_0+\tx_0,u+\tu+\eta+\teta)$, we obtain $2\re$(iv), %
 hence then (iv) holds.

(c) If $u,u+\eta$ are $J$-optimal for $x_0$,
 then $\eta$ is $J$-optimal for $0$, by (iv)
 (set $\tx_0=0$, $\tu=0$).

(d) This follows from (iii), because $\cJ(0,\eta)=\p{\qD\eta,J\qD\eta}$.
\end{proof}

To define $J$-coercivity, we need a natural norm on $\gUstar(0)$:
\begin{lemma}[$\|\cdot\|_{\gUstar}$]\label{lgUnorm} %
The set $\gUstar(0)$ is a Banach space under the norm\linebreak
 $\|u\|_{\gUstar}:=\max\{\|u\|_{\L^2_\vartheta},\|\qD u\|_2,\|\qRR u\|_{\sZ}\}$.\noproof
\end{lemma}

(This is straight-forward, %
 because $\L^2_\vartheta,\L^2,\sZ$ are Banach spaces.)
 
Obviously, the norms 
 $\|u\|_{\dUexp}^2\label{pageUexpnorm} :=\|u\|_2^2+\|x\|_2^2$ and
  $\|u\|_{\dUout}^2:=\|u\|_2^2+\|y\|_2^2$ %
 are equivalent to
 those defined in (a) above.
Moreover, $\qD\in\BL(\gUstar(0),\L^2(\R_+;Y))$ %
 (and $\qB\tau\in\BL(\dUexp(0),\L^2(\R_+;H))$).

In many of our results, we shall require that
  the {\em Popov\label{pagePTO} Toeplitz operator} $\PTO:=\qD^*J\qD$
 is (boundedly) invertible $\gUstar(0)\to \gUstar(0)^*$.
In the case $J=I$, this is true %
 iff there exists $\eps>0$ s.t.\ 
 $\|\qD u\|_2\ge \eps\|u\|_{\gUstar(0)}$ for all $u\in\gUstar(0)$;
 for $\gUstar=\dUout$ we can equivalently write this as
 $\|\qD u\|_2\ge\eps'\|u\|_2\ \all u\in\dUout(0)$.

This generalizes all coercivity assumptions that we have met
 in the literature (except those for ``singular control''),
 including so called ``no transmission zeros'' and
 ``no invariant zeros'' conditions 
 (see Theorem~\ref{Assumptions}).

If $\ALS$ is exponentially stable and $\gUstar=\dUexp$
 (or $\ALS$ is SOS-stable and $\gUstar=\dUout$),
 then an equivalent condition
 is that $\qD^*J\qD$ is (boundedly) invertible on $\L^2(\R_+;U)$,
 equivalently, that 
 the {\em Popov function} $\hqD^*J\hqD$ is uniformly invertible
 on the imaginary axis $i\R$.

See the proofs of the corollaries in Section~\ref{sStabOpt}
 to observe how our condition is satisfied in various applications.
E.g., for the ``LQR'' cost function $\cJ(x_0,u)=\|x\|_2^2+\|u\|_2^2$
 (i.e., $C=\sbm{I\cr0}$, $\qD=\sbm{0\cr I}$, $J=I$),
 obviously, $\PTO\gg0$ %
 on $\dUexp(0)$,
 equivalently, $\ALS$ is positively $J$-coercive over $\dUexp$,
 which leads to the existence of a unique optimal control:
\begin{theorem}[$J$-coercive $\THEN \ex! J$-optimal control]\label{J-coercive}
Assume that $\ALS$ is {\em $J$-coercive} over $\gUstar$, i.e.\ 
 that $\PTO:=\qD^*J\qD\in\BL(\gUstar(0),\gUstar(0)^*)$ is (boundedly) invertible.
Then, for each $x_0$ s.t.\ $\gUstar(x_0)\ne\tyhja$,
 there exists a unique $J$-optimal control.

If(f), in addition, $\PTO\ge0$
 (or $\p{\qD \cdot,J\qD \cdot}\ge0$), i.e.\ 
 $\ALS$ is {\em positively $J$-coercive},
 then the unique $J$-optimal control is (strictly) minimizing.
\end{theorem}

(The proof is given on p.~\pageref{pageproof-J-coercive}.
Note that $\dUexp(0)$ and $\dUout(0)$ are Hilbert(izable) spaces,
 hence $\dUexp(0)^*=\dUexp(0)$ $\dUout(0)^*=\dUout(0)$. %
See Lemma~\ref{lJcoerc-tD} for more.)

Thus, under the standard coercivity condition and the {\em FCC}\label{pageFCC} 
 ($\gUstar(x_0)\ne\tyhja\ \all x_0\in H$),
 there exists a unique optimal control for each initial state $x_0$.
Also the converse is true if, e.g.,
 $\dim U<\infty$ and $\gUstar=\dUexp$ (see pp. \pageref{pageiii}
  and~\pageref{pagesignature}).

Even better, a unique optimal control can be given
 in WPLS form (Definition~\ref{dWPLSform}),
 i.e., as an output of a system:
\begin{theorem}[\pmbold{$\ex!J$}-optimal \pmbold{$\THEN\ \ex\ALSopt$}]\label{ALSopt} %
Assume that there is a unique $J$-optimal control $\uopt(x_0)$ 
 for each $x_0\in H$. Then $\qKopt:x_0\mapsto\uopt(x_0)$
 is a control in WPLS form, i.e., 
\begin{equation}
\label{eALSopt}
 \ALSopt:=
  \bsysbm{\qAopt\|\crh \qCopt \| \cr \qKopt\|}:x_0\mapsto
  \bsysbm{\qA x_0+\qB\tau \uopt(x_0)\|\crh
              \qC x_0+\qD\uopt(x_0)\|\cr
              \uopt(x_0)\|}
\end{equation}
 is a WPLS
 (on $(\{0\},H,Y\times U)$).
We call $\Ric\label{pageRiccatiOp} :=\qCopt^*J\qCopt$ %
 the {\em $J$-optimal cost operator} (or the Riccati operator)\label{pageRiccatioperator}.
It satisfies $\Ric=\Ric^*\in\BL(H)$,
  and the minimal cost equals $\cJ(x_0,\uopt(x_0))=\p{x_0,\Ric x_0}$
 for all initial states $x_0\in H$.

If $\gUstar\sub\dUexp$, then $\ALSopt$ is exponentially stable;
 if $\gUstar\sub\dUout$, then $\ALSopt$ is output stable.
\end{theorem} 

(The proof is given on p.~\pageref{pageproof-ALSopt}.
We call $\Ric:=\qCopt^*J\qCopt$ the {\em $J$-optimal cost operator}
 (for $\ALS,J,\gUstar$)
 whenever $\qKopt$ is a $J$-optimal control in WPLS form,
 even if it were not unique.)

Obviously, the state and first output of $\ALSopt$ with initial state $x_0$
 are those of $\ALS$ with initial state $x_0$ and input $\uopt(x_0)$.
The $J$-optimal control $u$ also satisfies $u(t)=(\Kopt)_\w x(t)$ a.e.\ 
 for certain $\Kopt\in\BL(\Dom(\Aopt),U)$,
 by Lemma~\ref{lABC}(c). %
Moreover,  $\Aopt=A+B\Kopt$, %
 where $\Aopt,\Copt,\Kopt$ are the generators of $\ALSopt$.

Since $u(t)=\Koptw x(t)$ for a.e.\ $t\ge0$,
 the (Yosida extension $\Koptw$ of the) operator $\Kopt$
 is a ``generalized state-feedback operator'' for $\ALS$ in certain sense.
However, the feedback loop may be ill-posed
 (this is not the case if $\PTO\gg0$, %
  by Theorem~\ref{PosJcKF}, or if the system is sufficiently regular). %

If, e.g., $B$ is bounded, then $K=\Kopt$ can be computed from
 (\ref{eAREminB}),
 and $\ALSopt$ is the left column of the (well-posed) closed-loop
 system $\ALSClL$ (see (\ref{elALSClLgen}), p.~\pageref{elALSClLgen}).
In Sections \ref{sARE}--\ref{soptIRE} we explain in detail
 when $\Kopt$ is well-posed and how the optimal feedback
 is determined by different AREs and IREs,
 in principle as in Theorems \ref{IntroTh-Uexp} and~\ref{IntroTh-Uout}.
\NotesFor{Section~\ref{sOpt}}\label{pagesOptnotes} %
In Chapter~8 of [M02], everything above and much more is presented,
 the only exception being that  in Theorem~\ref{J-coercive}
 we no longer require $\sZ$ to be reflexive.

For the case $\gUstar=\dUout$, %
 Theorems \ref{J-coercive} and~\ref{ALSopt}
 are known for the cost function $\|y\|_2^2+\|u\|_2^2$ %
 [FLT88] [Z96]
 ([Z96] seems to be the only unstable optimization result on WPLSs
   before [M02])
 and for a general $J$-coercive cost function in the stable case [S98c].

(To be exact, in [S98b] the $\dUout$ minimization problem for
 jointly stabilizable and detectable WPLSs was reduced to the stable case.
It has previously been very difficult to verify the joint assumption,
 but now %
 Theorem~\ref{JointlyUout} can be used for effectively that purpose.
However, 
 thanks to Theorem~\ref{QRCF-Uout}(iii)
  (see Theorem 8.4.5(e)\&(g1) of [M02]),
 now any problems over $\dUout$ can be reduced to the stable case
 (use Corollaries \ref{cUout-B1}, \ref{cOptExpstab1} and~\ref{cExpStabB1}
  for partial control ($\H^\infty$) and/or for $\dUexp$).
On the other hand, our theory also provides a direct solution.)

In the stable case (with $\L^2(\R_+;U)$ in place of $\gUstar(x_0)$),
 Definition~\ref{dJopt0} is rather old,
 and for WPLSs it was first used in [S97].

$J$-coercivity was defined in [S98c] for stable WPLSs
 (in [M02] for general ones), but
 equivalent definitions have been very popular
 for finite-dimensional or other very special systems,
 as explained in Section~\ref{sJcoerc}.
See Chapter~11 of [M02] for applications of indefinite $\PTO$
 to $\H^\infty$ problems.

The sets $\dUout$, $\dUexp$ and $\dUstr$ have been used
 (at least implicitly) for decades, 
 but we have not seen a unified approach before [M02],
 nor (indefinite unstable versions of) any of the results
 of this section (not even for finite-dimensional systems).
See the notes in Chapter~8 of [M02] for further comments.

\section{Minimizing control, stabilization and coprime factorizations}\label{sStabOpt} %

In this section, we shall study uniformly positive cost functions
 ($\PTO\gg0$) only. 
By Theorems \ref{J-coercive} and~\ref{ALSopt}, 
 we already know that in this case 
 the FCC leads to the existence of a unique minimizing control
 in WPLS form. %
In Theorem~\ref{PosJcKF} we shall
 show that this control is actually given by a
 (minimizing, well-posed) state-feedback pair.
The remainder of the section consists of corollaries to that theorem:
 we derive numerous simple but important consequences on
 stabilization and coprime factorizations.
They are the main results of this article along with the Riccati equation
 theory of Sections \ref{sARE}--\ref{sIRE}.

As mentioned above, in the uniformly positive case 
 ($\p{\qD u,J\qD u}\ge\eps\|u\|_{\gUstar}^2\ (u\in\gUstar(0))$)
 with the FCC,
  the unique optimal control is always
 given by a (well-posed) state-feedback pair:
\begin{theorem}[\pmbold{$\PTO\gg0\ \THEN\ \ex\qKF$}]\label{PosJcKF} %
Assume that $\PTO\gg0$ and $\vartheta=0$.
Then the FCC is satisfied iff there is a $J$-optimal state-feedback pair.
\end{theorem}

(The proof is given on p.~\pageref{pageproof-PosJcKF}.
Recall that $\vartheta=0$ when $\gUstar=\dUout$ or $\gUstar=\dUexp$.)

By Lemma \ref{lOptCost0}(d), 
 here a state-feedback pair is $J$-optimal iff it is minimizing.
We shall show in Corollary~\ref{cPosIRE} that
 also the existence of a $\gUstar$-stabilizing
 solution of the Riccati equation
 is equivalent to the FCC.

By setting $\gUstar=\dUexp$ and making the system coercive without
 affecting $\qA$ or $\qB$,
 we can obtain the perfect generalization of a classical
 finite-dimensional result (and Corollary~\ref{cIntroOptim}), 
 thus solving the well-known open problem:
\begin{cor}[Optimizable \pmbold{$\IFF$} Exp.\ stabilizable]\label{cOptExpstab1} %
A WPLS is optimizable iff it is exponentially stabilizable.  
\end{cor}

The WPLS (or the pair $\qAB$) is called
 {\em optimizable}\label{pageoptimizable}
 iff $\dUexp(x_0)\ne\tyhja\ \all x_0\in H$ %
 (i.e., the state-FCC holds).
In addition to the corresponding IRE
 (see Corollary~\ref{cPosIRE}(c)),
 one equivalent condition is that 
 a certain (non-integral) Riccati equation
 has a nonnegative solution, as will be shown
 in [M03b]
 (if $0\in\rho(A)$, then the equation becomes
   $\Ric^2=(A_-^*+\Ric)(I+B_-B_-^*)^{-1}(A_-+\Ric)$, %
  where $A_-:=A^{-1},\ B_-:=A^{-1}B$ are bounded).

By duality, the corollary implies that a WPLS is
  estimatable iff it is exponentially detectable.
($\ALS$ is called {\em estimatable}\label{pageEst}
 (resp.\ {\em exponentially detectable}\label{pageExpDet})
 iff $\ALS^\rmd$ is optimizable
  (resp.\ {\em exponentially stabilizable})
 i.e., iff $\bsyspm{A^*\|C^*}$ is optimizable.)

As above, by $(A,B)$, $\bsyspm{A\|B}$\label{pageAB} or $\qAB$
 we refer to a system with zero output %
 (to $\ssysbm{\qA\|\qB\crh0\|0}$),
 although the concepts of Corollaries \ref{cOptExpstab1} and~\ref{cExpStabB1}
 are independent of the second row ($\qCD$) of the system.

\begin{proof}[Proof of Corollary~\ref{cOptExpstab1}:]
Set $\qC=\sbm{\qA\cr 0}$,\ $\qD=\sbm{\qB\tau\cr I}$, $J=I$, $\gUstar=\dUexp$.
Then, by Theorem~\ref{PosJcKF}, there is a $J$-optimal state-feedback
 pair iff $\qAB$ is optimizable.
By, e.g., Theorem~\ref{ALSopt}, the pair is exponentially stabilizing
 (equivalently, $\qAClL$ is exponentially stable;
  recall that $\ALSopt=\ALSClL\sbm{I\cr0}$).
\end{proof}

If $(A,\bbm{B_1&B_2})$ is optimizable through the first input,
 then it is exponentially stabilizable through the first input:
\begin{cor}[\pmbold{$(A,B_1)$} exp.stab.]\label{cExpStabB1} %
If $(A, \bbm{B_1&B_2})$ is a WPLS and $(A,B_1)$ 
 is optimizable, then
 $(A, \bbm{B_1&B_2})$ has an exponentially stabilizing state-feedback pair
 of form $\qKF=\ssysbm{\qK_1\|\qF_{11}&\qF_{12}\cr 0\|0&0}$.
\end{cor}

Note that an arbitrary exponentially stabilizing state-feedback
 pair $\bsysbm{\qK_1\|\qF_{11}}$ for $\bsysbm{\qA\| \qB_1}$
 need not be extendable for  $\bsysbm{\qA\| \qB_1& \qB_2}$
 (i.e., $K_1$ and $B_2$ might be ``incompatible'',
  that is, ``$K_1(\cdot-A)^{-1}B_2$'' need not be well-posed),
 by Example 6.6.23 of [M02]
 (there no pair of form $\bsysbm{\qK\|\qF &*}$
   is admissible for the WPLS $\ssysbm{\qA\|\qB&\qH\crh}$,
  if $\qKF$ and $\qHG$ are the exponentially stabilizing 
  pairs (``lower row'' and ``right column'')
   of the example).

Nevertheless, one can conclude from the proof below
 (as in the proof of Corollary~\ref{cUout-B1})
 that the %
 $(\|x\|_2^2+\|u_1\|_2^2)$-minimimizing pair $\bsysbm{\qK_1\|\qF_{11}}$
 for $\bsyspm{A\|B_1}$
 (which is unique modulo (\ref{eAllqKF})) %
 is necessarily admissible with any $B_2$ s.t.\ $\bsyspm{A\|B_2}$
 is a WPLS
 (even though $K_1$ depends on $B_1$),
 i.e., that it satisfies the requirements of the corollary
 with some $\qF_{12}\in\TIC_\infty(U_2,U_1)$.

Corollary~\ref{cExpStabB1}
 clarifies the assumptions for $\H^\infty$ control problems
 (see [M02], Chapters 11--12).

\begin{proof}[Proof of Corollary~\ref{cExpStabB1}:] %
Let $\bsysbm{\qA\|\bbm{\qB_1&\qB_2}}$ be a WPLS on $(U_1\times U_2,H,-)$
 (it might have a second row, but 
  it has no influence on the problem).
Define $\qC:=\sbm{\qA\cr0},\ \qD:=\sbm{\qB_1\tau&\qB_2\tau\cr I&0}$,
 $\qQR:=0,\ \qRR:=\bbm{0&I}$, $\uZ:=\L^2,\ \sZ:=\{0\},\ \vartheta=0$.
It follows that $y=\sbm{x\cr u}$
 and $\gUstar(x_0)=\{\sbm{u_1\cr0}\in\L^2(\R_+;U_1\times U_2)\I x\in\L^2\}$
 (note that Standing Hypothesis~\ref{shgUstar} is satisfied).

The norm $\|\qD u\|_2=\|\sbm{x\cr u}\|_2$ 
 is obviously equivalent to $\|u\|_{\gUstar}:=\max\{\|u\|_2,\|\qD u\|_2,\|0\|\}$,
 (which is complete, by Lemma \ref{lOptCost0}(a)),
 hence $\PTO\ge\eps I$ for some $\eps>0$ (when we set $J:=I$).
We have $\gUstar(x_0)\ne\emptyset\ \all x_0\in H$,
 by the optimizability of $\bsysbm{\qA\|\qB_1}$.
Thus, Theorem~\ref{PosJcKF} implies that there is a
 $J$-optimal state-feedback pair $\qKF$.
Set $\qX:=I-\qF\in\cG\TIC_\infty(U)$. 
Fix $\alpha$ big enough, so that $\hqX(\alpha)\in\cG\BL(U)$.
By Lemma~\ref{lAllqKF}, we can redefine $\qKF$
 so that $\hqF(\alpha)=0$ (without affecting $\qKClL$). %
But $\qKClL x_0\in \gUstar(x_0)\ \all x_0\in H$
 implies that $(\qK_{2})_\ClL=0$. %
Therefore, $\qM:=I+\qFClL=\sbm{*&*\cr M_{21}&M_{22}}$,
 where $M_{21},M_{22}$ are constants, by, e.g., (\ref{eCharFct}).
Since $\hqM(\alpha)=(I-\hqF(\alpha))^{-1}=\sbm{I&0\cr 0&I}$,
 we have $M_{21}=0,\ M_{22}=I$, hence $\hqF=\sbm{*&*\cr 0&0}$.
But $\qK=\qM^{-1}\qKClL=\sbm{*\cr 0}$,
 hence $\qKF$ is as required
 (since $\qKClL x_0\in \gUstar(x_0)\sub\dUexp(x_0)\ \all x_0\in H$,
  the pair $\qKF$ is exponentially stabilizing).
\end{proof}

Before going on, we need a few concepts on coprimeness and factorization.
(Recall from Definition~\ref{dWPLS0}4.\ 
 that the maps in $\TIC:=\TIC_0$ are called {\em stable}\label{pagestableTIC}.)
\begin{defin}%
[[q.{]}r.c., [q.{]}r.c.f., d.c.f.]\label{drcf} %
{\bf (a)} We call $\qN\in\TIC(U,Y)$, $\qM\in\TIC(U)$ {\em right coprime (r.c.)}
 if $\sbm{\tqX &-\tqY}\sbm{\qN\cr \qM}=I$ for some $\tqX,\tqY\in\TIC$;
 {\em quasi--right coprime (q.r.c.)}
 if $u\in\L^2 \IFF \sbm{\qN\cr \qM}u\in\L^2$ whenever %
 $u\in\L^2_\omega(\R_+;U),\ \omega\in\R$. %

{\bf (b1)} Let $\qD\in\TIC_\infty(U,Y)$.
We call $\qN\qM^{-1}$ a {\em right factorization} of $\qD$
 if $\qN,\qM\in\TIC$, $\qM\in\cG\TIC_\infty(U)$ and $\qD=\qN\qM^{-1}$.
It is called a {\em [quasi--]right-coprime factorization\label{pagercf} ([q.]r.c.f.)}
 if, in addition, $\qN,\qM$ are [q.]r.c.

{\bf (b2)} Let $\qD\in\TIC_\infty(U,Y)$.
We call $\qN\qM^{-1}=\tqM^{-1}\tqN$
 a {\em doubly coprime factorization (d.c.f.)} of $\qD$
 if $\sbm{\qM&\qY\cr \qN&\qX}^{-1}=\sbm{\tqX& -\tqY\cr -\tqN &\tqM}\in\cG\TIC(U\times Y)$
 for some $\qY,\qX,\tqY,\tqX\in\TIC$,
 $\qM\in\TIC_\infty(U)$ and $\qD=\qN\qM^{-1}$.

{\bf (c)} By the coprimeness of $\hqN,\hqM:\C^+\to\BL$
 we refer to the coprimeness of $\qN$ and $\qM$
 (see Theorem~\ref{TransferFct1}) etc.

{\bf (d)} Replace all maps by their adjoints (and $U$ by $Y$ and $Y$ by $U$) %
 to obtain the ``left'' definitions (e.g., ``l.c.'')
 corresponding to (a) and (b1).
\end{defin} %

(The minus signs are due to historical reasons.
Under (b2), we have $\tqM\in\cG\TIC_\infty(Y)$,
 and $\qD=\tqM^{-1}\tqN$ is a l.c.f.,
 by Lemma 6.5.9 of [M02].)

We recall some basic properties of coprimeness from [M02]:
\begin{lemma}\label{lrcf}
{\bf (a1)} If $\qN,\qM$ are r.c., then
 $\hqN^*\hqN+\hqM^*\hqM\ge\eps I$ on $\C^+$ for some $\eps>0$.

If $\dim U<\infty$, then also the converse is true,
 and any r.c.\ pair $\qN,\qM$ can be extended
 to an invertible element $\sbm{\qM&\qY\cr \qN&\qX}\in\cG\TIC$
 (which is a d.c.f.\ iff $\qM\in\cG\TIC_\infty$).

{\bf (a2)} If $\hqN^*\hqN+\hqM^*\hqM\ge\eps I$ on $\C^+$ for some $\eps>0$,
 then $\qN,\qM$ are q.r.c.

{\bf (b)} If $\qN,\qM$ are r.c., then they are q.r.c.

{\bf (c)} If $\qN,\qM$ are q.r.c., then
 $\|\hqN u_0\|_Y+\|\hqM u_0\|_U>0$ on $\C^+$ for all $u_0\in U$,
 and $\qN^*\qN+\qM^*\qM\gg0$.

{\bf (d)} Let $\qN_0\qM_0^{-1}$ be a q.r.c.f.\ of $\qD\in\TIC_\infty(U,Y)$.
Then all right factorizations %
 $\qD=\qN\qM^{-1}$ %
 are given by
 $\qN=\qN_0\qE$, $\qM=\qM_0\qE$ with
 $\qE\in\TIC(U)\cap \cG\TIC_\infty(U)$.
Moreover, $\qE\in\cG\TIC(U)$ iff $\qN\qM^{-1}$ is a q.r.c.f.

In particular, if $\qD$ has a r.c.f.,
 then any q.r.c.f.\ of $\qD$ is a r.c.f.

{\bf (e)} $\qD$ has a d.c.f.\ iff $\qD$ and $\qD^\rmd$ 
 have a r.c.f.
\end{lemma}

Thus, q.r.c.\ transfer functions do not have common zeros in $\C^+$
 (in the sense of (c)),
 nor on the imaginary axis
 ($\|\hqN u_0\|+\|\hqM u_0\|\ge\eps\|u_0\|$ a.e.\ on $i\R$ for all $u_0\in U$, by (c));
 see also %
 the comments below Example~\ref{exaqrcfNotrc}. %
By Theorem~\ref{DCFiff}, ``and $\qD^\rmd$'' can be removed from (e).
See [M02], Sections 6.4--6.5 for further results and details.

In Corollary~\ref{cQRCFiff} we will show that
 any map having a right factorization has a q.r.c.f.
By Example~\ref{exaqrcfNotrc}, a q.r.c.f.\ need not be a r.c.f.\ (cf.\ (d)).
Nevertheless, any rational
 q.r.c.f.\ is a r.c.f.\ (by Lemma 6.5.3(b) of [M02]);
 a generalization of this result can be derived from
 Lemma~\ref{lNMrcUfinite}.
However, %
 not all well-posed maps have a right factorization:
\begin{example}\label{exanotrf}
There exists $\qD\in\TIC_\infty(\C)$ that does not have a right factorization. 
\end{example}

Indeed, set $\hqD(s):=(s-1)^{-1/2}$,
   so that $\qD\in\TIC_\omega(\C)\ \all \omega>1$.
Then $\hqD$ has an essential singularity at $s=1$,
 whereas maps in $\TIC_\infty(\C)$ having a right factorization
 have meromorphic extensions to $\C^+$,
 by Corollary~\ref{cUout-Meromorphic}.

By Corollary~\ref{cQRCFiff} and Theorem~\ref{DCFiff},
 no realization of the above $\qD$ is output-stabilizable nor 
 dynamically stabilizable. %
See those corollaries also for equivalent conditions for
 the existence of [quasi-]coprime factorizations.

Constructive formulas (from the solutions of AREs or IREs) %
 for [q.]r.c.f.'s and d.c.f.'s
 are given in and below Corollary~\ref{cPosIRE}.
For that purpose one has to use an output-stabilizable realization of $\qD$;
   to get a d.c.f.\ also the dual condition is required;
   cf.\ Theorem \ref{QRCF-Uout} (or \ref{cOptExpstab1} with state-FCC)
   and Theorem~\ref{JointlyUout} (or \ref{cJointlySD})).
Corresponding formulas also give corresponding stabilizing
 controllers or state-feedback pairs etc.

\begin{proof}[Proof of Lemma~\ref{lrcf}:]
(a1) Take $\eps:=1/\|\sbm{\tqX&\tqY}\|^2$.
The converse follows from the Corona Theorem and
 the extension from Tolokonnikov's Lemma,
 but neither holds when $\dim U=\infty$;
 see Theorem 4.1.6 and Lemma 6.5.3(b) of [M02] for details
 (and for similar results for $\WTIC$ or other sets in place of $\TIC$).

(a2) This follows from Lemma 4.1.8(g) of [M02].

(b) Now $u=\sbm{\tqX&\tqY}\sbm{\qN\cr \qM}u\in\L^2$
 when $\sbm{\qN\cr \qM}u\in\L^2$. (Alternative proof: (a1)\&(a2).)

(c) Assume that $\sbm{\hqN\cr \hqM}(s_0)u_0=0$
 for some $s_0\in\C^+,\ u_0\in U$.
Set $\omega:=\re s_0+1$, $u(t):=\efn^{s_0 t}u_0$
 (i.e., $\hu(s):=(s-s_0)^{-1}u_0$).
Then $u\in\L^2_\omega(\R_+;U)$ but
 $\sbm{\hqN\cr \hqM}\hu\in\H^2(\C^+;Y\times U)$
 (because it is holomorphic and bounded and $\le M/|\im s|$ for big $|\im s|$),
 i.e., $\sbm{\qN\cr \qM}u\in\L^2(\R_+;U\times Y)$.

(d) This is Lemma 6.4.5(b)\&(c) of [M02]
 (set $\qE:=\qM_0^{-1}\qM\in\cG\TIC_\infty(U)$
  and use [q.]r.c.).

(e) This is Lemma~4.3(iii) of [S98a].
\end{proof}

As noted below Corollary~\ref{cOptExpstab1},
 we know that %
 $\ALS$ and $\ALS^\rmd$
 satisfy the state-FCC
 iff
 $\ALS$ is exponentially stabilizable and exponentially detectable;
 in fact, then it is 
 exponentially {\em jointly} stabilizable and detectable
 (the terminology will be explained below the corollary):
\begin{cor}[$\dUexp\ne\tyhja\ne\dUexp^{\ALS^\rmd}\ {\IFF}$ jointly stab.\&det.]\label{cJointlySD} %
The following are equivalent: %
 \begin{itemlist}
  \item[(i)] $\ALS$ is exponentially jointly stabilizable and detectable.
  \item[(ii)] $\ALS$ and $\ALS^\rmd$ satisfy the state-FCC.
  \item[(iii)] $\ALS$ satisfies the output-FCC and $\ALS^\rmd$ 
 the state-FCC.
  \item[(iv)] $\ALS$ satisfies the state-FCC and $\ALS^\rmd$
 the output-FCC.
  \item[(v)] There is an exponentially stabilizing 
 {\em dynamic feedback controller}\label{pageDFC}
  for $\ALS$ with internal loop.
 \end{itemlist}

Moreover, any output-stabilizing state-feedback pair for
 an estimatable system is exponentially r.c.-stabilizing. 
Any exponentially jointly stabilizing pairs for $\ALS$
 define (through (\ref{eS4A:4.4}))
 an exponential doubly coprime factorization of
 the I/O map $\qD$ of $\ALS$.
\end{cor}

As before, output-FCC (FCC for $\dUout$)
 means that $\dUout(x_0)\ne\tyhja\ \all x_0\in H$.
Condition (i) means that $\ALS$ can be extended to a WPLS
 \begin{equation}
   \label{eJointly1}
  \ALStotal:=    \bsysbm{\qA\|\qH&\qB\crh \qC\|\qG&\qD\cr
                                             \qK\|\qE&\qF}
 \end{equation}
 (on $(Y\times U,H,Y\times U)$)
 s.t.\ $(\ALStotal)_L$ and $(\ALStotal)_{\tL}$
 are exponentially stable, where
 $L=\sbm{0&0\cr 0&I}$ and $\tL=\sbm{I&0\cr 0&0}$
 (see Lemma~\ref{lALSL0}).
This says that $\ALStotal$
  becomes exponentially stable when the added output is connected
 to the original input %
 or the original output is connected to the added input.
({\em Jointly\label{pageJointly} stabilizable and detectable} means the same except
 that $(\ALStotal)_L$ and $(\ALStotal)_{\tL}$ need be merely stable.)

By Lemma~\ref{lExpStable}, the two closed-loop systems are exponentially stable
  iff $\qA_L=\qA+\qB\tau(I-\qF)^{-1}\qK$
 and $\qA_\tL=\qA+\qH\tau(I-\qG)^{-1}\qC$ are
 exponentially stable (i.e., iff they map $H$ into $\L^2(\R_+;H)$).
In this case, we call $\qKF$ and $\qHG$
 {\em exponentially jointly stabilizing pairs} for $\ALS$.
It follows that $\qKF$ is (an) exponentially stabilizing 
 (state-feedback pair)
 and $\qHG$ is (an) {\em exponentially detecting}\label{pageExpDet2}  %
 ({\em output injection} pair) for $\ALS$
 (see Section 6.6 of [M02] for further explanations and results).

Under (i), the $\cG\TIC(U\times Y)$ maps %
 \begin{equation}
   \label{eS4A:4.4}
     \bbm{\qM&\qY_1\cr \qN&\qX_1} := \bbm{I+\qF_L &-\qE_L\cr \qD_L& I-\qG_L}
 \ \ \ \text{and} \ \ \  %
    \bbm{\tqX&-\tqY\cr -\tqN& \tqM}  :=  \bbm{I-\qF_\tL&\qE_\tL\cr -\qD_\tL&I+\qG_\tL}
 \end{equation}
 are the inverses of each other
 (by a direct computation,
   see Theorem 4.4 of [S98a] for the details;
 actually these maps are the inverses of each other even when
 they are unstable, it suffices that $L,\tL$ are admissible).
It follows that (\ref{eS4A:4.4}) defines a d.c.f.\ of $\qD$
 (actually, an {\em exponential d.c.f.}, %
  which means that (\ref{eS4A:4.4})$\in\cG\TIC_\omega(U\times Y)$
  for some $\omega<0$).

As noted below Corollary~\ref{cExpStabB1},
 arbitrary  exponentially stabilizing
 and detecting pairs $\smash{\qKF}$ and $\smash{\qHG}$ for $\ALS$ need not be
 jointly admissible for $\ALS$
 (i.e., no $\qE$ makes (\ref{eJointly1}) a WPLS).
However, the  $\|x\|_2^2+\|u\|_2^2$-minimizing pair $\smash{\qKF}$
 is jointly admissible with any admissible $\smash{\qHG}$,
 as noted in $2^\circ$ below.
By duality, any admissible $\smash{\qKF}$ is jointly admissible
 with certain exponentially stabilizing $\smash{\qHG}$
 (if any exists, i.e., if $\smash{\ALS^\rmd}$ satisfies the state-FCC)).

The ``moreover'' claim means that if $\ALS$ is estimatable
 and an admissible state-feedback pair $\qKF$ makes $\qCClL$ and $\qKClL$
 stable, then it actually makes $\ALSClL$ exponentially stable
 and $\qN,\qM$ {\em exponentially r.c.}\label{pageexprc}%
\footnote{The maps $\qN,\qM$ are called {\em exponentially r.c.}
 if there exist $\omega<0$,
   $\smash{\big[\tqY\ \tqX\big]}\in\TIC_\omega(Y\times U,U)$ s.t.\
 $\qN,\qM\in\TIC_\omega$ and $\smash{\big[-\tqY\ \tqX\big]}\sbm{\qN\cr \qM}=I$.
Recall that $\qM:=\smash{(I-\qF)^{-1}},\ \qN:=\qD\qM=\qDClL$,
 so that $\qD=\qN\qM^{-1}$.}

Condition (v) means roughly a system in place of $L$ 
 in Figure~\ref{fOutputFeedback} (p.~\pageref{fOutputFeedback})
 s.t.\ the connection stabilizes both systems exponentially.
It is further explained in Section 7.2 of [M02]. %
By Theorem~\ref{DCFiff}, ``with internal loop'' may be removed
 if $\dim U,\dim Y<\infty$
 (take any jointly exponentially stabilizable and detectable realization
   of any $\qT$ s.t.\ $\sbm{I&-\qT\cr -\qD&I}^{-1}$ is exponentially stable;
   cf.\ Theorem 7.2.3(d)\&(c1) of [M02]),
 but the general case is open.

\begin{proof}[Proof of Corollary~\ref{cJointlySD}:]
The last claim was shown above
 (actually, in Theorem 4.4 of [S98a]).
By the dual of $2^\circ$,
 any output-stabilizing state-feedback pair for $\ALS$
 is exponentially jointly (coprime) stabilizing 
 with some $\qHG$ and ({\em interaction operator}) $\qE$.
This proves the ``moreover'' claim, 
 hence only the equivalence remains to be proved.

$1^\circ$ {\em (i)$\THEN$(iii)$\THEN$(ii):}
The first implication is obvious. Assume then (iii),
 so that there exists $\qKF$ s.t.\ $\qCClL,\qKClL$ are stable
 and $\ALS$ is estimatable ($\dUexp^{\ALS^\rmd}(x_0)\ne\tyhja\ \all x_0\in H$),
 hence so is $\ALSext$,
 hence $\ALSClL$, hence $\ALSClL$ is exponentially stable,
 by Theorem 7.3 and Proposition 6.2 of [WR00].

$2^\circ$ {\em (ii)$\THEN$(i):}
By the dual of Corollary~\ref{cOptExpstab1},
 $\ALS$ can be extended to a WPLS $\ALS_2:=\ssysbm{\qA\|\qB&\qH\crh \qC\|\qD&\qG}$
 that becomes exponentially stable under the static output feedback
 through $L=\sbm{0\cr I}$ (i.e., under input $u=u_L+Ly$,
 where $u_L:\R_+\to U\times Y$
 is an external input and $y=\qC x_0 + \bbm{\qD&\qG} u$;
 see Definition 6.6.21 of [M02] or Lemma~\ref{lALSL0} for details).

By Corollary~\ref{cExpStabB1}, there is an 
 exponentially stabilizing state-feedback pair
 $\ssysbm{\qK\|\qF&\qE\cr 0\|0&0}$ for $\ALS_2$
 (the $\|x\|_2^2+\|u\|_2^2$-minimizing one,
  as noted below Corollary~\ref{cExpStabB1}).
Obviously, (\ref{eJointly1}) is a WPLS and $\qA_L$ and $\qA_\tL$
 are exponentially stable.

$3^\circ$ {\em (i)--(iv) are equivalent:}
By $1^\circ$--$2^\circ$, (i)--(iii) are equivalent.
But (iv) is exactly (iii) applied to $\ALS^\rmd$.

$4^\circ$ {\em (i)$\THEN$(v)$\THEN$(ii):} 
This was given in Theorem 7.2.4(b)\&(a) of [M02].
(See [M02] for the definition and further results and notes.)
\end{proof}

We also conclude the equivalence of the weak and strong forms
 of the standard assumption for the $\Hoo$ Four-Block Problem
 (``stabilizable through $u_1$ and detectable through $y_2$''):
\begin{remark}\label{rHooAB1AC2} %
Assume that $B=\sbm{B_1\cr B_2}$ and $C=\bbm{C_1&C_2}$.
Then there are exponentially jointly stabilizing and detecting
 pairs through $B_1$ and $C_2$ (as in (12.76)--(12.77) of [M02])
 iff $(A,B_1)$ is exponentially stabilizable and $(A,C_2)$
 is exponentially detectable
\end{remark}

By Corollary~\ref{cOptExpstab1}, 
 a third equivalent condition is that
 $(A,B_1)$ and $(A^*,C_2^*)$ satisfy the state-FCC.
From the proof of Lemma 12.5.4 of [M02] we observe that
 Hypothesis 12.5.1 is exponentially satisfied iff
 the above conditions hold and $\qD_{11}$ and $\qD_{22}^\rmd$
 are $I$-coercive (over $\dUexp=\dUout$).

To obtain similar results on non-exponentially stabilizing  %
 $\H^\infty$ controllers, one should use Corollary~\ref{cUout-B1}
 and work as in the proof of Theorem~\ref{JointlyUout}.

\begin{proof}[Proof of Remark~\ref{rHooAB1AC2}:]
Choose first $\qKF$ as in  Corollary~\ref{cExpStabB1}.
Then work as in  $2^\circ$ of the proof of Corollary~\ref{cJointlySD}
 but choose the (permuted dual $\ssysbm{\qH_2^\rmd\| \qG_{12}^\rmd&\qG_{22}^\rmd\cr 0\|0&0}$ of the) pair $\qHG$ as in
 Corollary~\ref{cExpStabB1},
 so that its first column is zero.
\end{proof}

Using Theorem~\ref{PosJcKF}, we can deduce that
 the output-FCC implies the existence of an output-stabilizing %
 state-feedback pair for $\ALS$
 (namely the $\|u\|_2^2+\|y\|_2^2$-minimizing one),
 thus generalizing Corollary~\ref{cQRCF-intro}.
Actually, we can show that this specific pair is 
 {\em SOS-stabilizing}\label{pageSOS-stabilizing} 
 (which means that $\qCClL,\qDClL,\qKClL,\qFClL$ are {\em stable}\label{pagestable},
 i.e., that they map $H$ or $\L^2$ into $\L^2$) 
 and leads to a (normalizable)
 quasi--right coprime factorization of $\qD$: 
\begin{theorem}[\pmbold{$\dUout$}: FCC \pmbold{$\IFF \ex\qKF$}]\label{QRCF-Uout} %
The following are equivalent:
 \begin{itemlist}
  \item[(i)] $\dUout(x_0)\ne\tyhja\ \all x_0\in H$.
  \item[(ii)] There is an output-stabilizing state-feedback pair $\qKF$ for $\ALS$.
  \item[(iii)] There is a SOS-stabilizing state-feedback pair $\qKF$ for $\ALS$
 s.t.\ $\qD=\qN\qM^{-1}$ is a q.r.c.f.\ and $\qN^*\qN+\qM^*\qM=I$.
 \end{itemlist}
\end{theorem}

(The proof is given on p.~\pageref{pageproof-QRCF-Uout}.)

Conversely, any map having  a q.r.c.f.\ 
 (equivalently, a right factorization)
 has a realization satisfying (i)--(iii),
 by Corollary \ref{cQRCFiff}(i).
Recall that the output-FCC (i)
 means that
 for all $x_0\in H$, there exists $u\in\L^2(\R_+;U)$ s.t.\ $y\in\L^2$.

A q.r.c.f.\ is unique modulo an element of $\cG\TIC(U)$.
If(f) $\qD$ has a %
 right-coprime factorization,
 then any q.r.c.\ factorization of $\qD$ is right-coprime.
See Lemma 6.4.5(c) of [M02] for proofs.

By Lemma~\ref{lSgg0}(c), we have $\sbm{\qD\cr I}\dUout(0)=\sbm{\qN\cr \qM}\L^2(\R_+;U)$.

The maps $\qN,\qM$ are actually r.c.\ in (iii)
 if, e.g., $\dim U<\infty$ and $\sigma(A)$ is nice
 (see Lemma~\ref{lNMrcUfinite} below).
Any q.r.c.f.\ %
 can be ``normalized''
 to satisfy $\qN^*\qN+\qM^*\qM=I$,
 by Lemma \ref{lrcf}(c)\&(d) and Theorem \ref{SpF1}(a).

If the input space is finite-dimensional, then the FCC implies
 that $\qD$ is {\em meromorphic}\label{pagemeromorphic1}
 (i.e., for any $s_0\in \C^+$, there is $n\in\N$ s.t.\
  $s\mapsto (s-s_0)^n\hqD(s)$ is holomorphic on a neighborhood of $s_0$):
\begin{cor}[\pmbold{$\hqD$} is meromorphic]\label{cUout-Meromorphic} %
Assume that $\dUout(x_0)\ne\tyhja\ \all x_0\in H$ and $\dim U<\infty$.
Then $\hqC,\hqD$ are meromorphic on $\C^+$ %
 (and so are $\hqK,\hqF,\hqM^{-1}$ for any output-stabilizing $\qKF$).

If $\dUexp(x_0)\ne\tyhja\ \all x_0\in H$,
 then  $\hqA,\shat{\qB\tau},\hqC,\hqD$
 are meromorphic on $\C_{-\del}^+$ for some $\del>0$
 (and so are $\hqK,\hqF,\hqM^{-1}$ for any exponentially stabilizing $\qKF$).
\end{cor}

In particular, $\hqD=\hqD_\ALS$ and $\hqC=C(\cdot-A)^{-1}$
 a.e.\ on $\C^+$ (or on $\C_{-\del}^+$),
 by Lemma \ref{lCharFct}(f)\&(a)\&(b1)
 (which defines the above symbols).

If we would define ``meromorphic'' as quotient of analytic maps,
 assumption $\dim U<\infty$ would be redundant (by the proof below)
 but now it is not the case, by Example~\ref{exadimUinftyRCF}.

\begin{proof}
$1^\circ$ We obtain $\hqD=\hqN\hqX$, 
 $\hqC=\hqCClL-\hqD\hqKClL$, $\hqK=\hqX\hqKClL,\ \hqF=I-\hqX$
 on some right half-plane %
 from Theorem~\ref{QRCF-Uout}.
Since $f(s):=\det\hqM\not\equiv 0$,
 the only singularities of $f^{-1}$
 (and of $\hqX=\hqM^{-1}$ and of $\hqD=\hqN\hqX$)
 on $\C^+$ are isolated poles %
 (cf.\ p.~112 of [M02]). %
It follows that also $\hqD$ and $\hqC$ have meromorphic
 extensions to $\C^+$ (cf.\ Remark~\ref{rTrFct}).

$2^\circ$ If $\qKF$ is exponentially stabilizing,
 i.e., $-\del:=\omega_{\AClL}<0$,
 then $\shat{\ALSClL}$ is holomorphic
 and $\hqX$ is meromorphic on $\C_{-\del}^+$,
 hence $\hqD,\hqC,\hqK$ have meromorphic
 extensions to $\C_{-\del}^+$,
 as above, and so do $\hqB=\hqBClL\hqX,\ \hqA=\hqAClL-\hqB\hqKClL$.
\end{proof}

As mentioned above, the poles of $\hqD$ need not be isolated
 when $\dim U=\infty$,
 not even when $\dUexp(x_0)\ne\tyhja\ \all x_0\in H$:
\begin{example}{(${\hqD}$ is not meromorphic).}\label{exadimUinftyRCF} %
Let $\dim U=\infty$.
Then there is an exponentially
 (r.c.-)stabilizable WPLS $\ALS=\qABCD$
 s.t.\ $\qD\in\TIC_\infty(U)$
 and $\qD=I\qM^{-1}$ is an exponential r.c.f.\ (hence q.r.c.f.)
 but all points of $\{z\in\C\I |z-5|<1\}$ are poles of $\hqD$.
\end{example}

Note also that $\qD$ also has a normalized exponential r.c.f.\
 $\qD=\tqN\tqM^{-1}$
 (by Lemma 6.4.7(a)\&(c), for some $\del>0$
  there is $\qX\in\cG\TIC_{-\del}(U)$ s.t.\ $\qX^*\qX=I^*I+\qM^*\qM$;
 set $\tqN:=\qX^{-1},\ \tqM:=\qM\qX^{-1}$).

\begin{proof}
Take $s_0=-1$ and choose an infinite compact $K\sub\C^+$ 
 (e.g., $K=\{z\in\C\I |z-5|<1\}$)
 to obtain, from Lemma 3.3.9 of [M02],
 a function $\hqM\in\H^\infty(\C^+;\BL(U))$
 (actually, $\hqM\in\H^\infty(\C_{-\del}^+;\BL(U))$ for any $\del<1$)
 s.t.\ $\qM\in\TIC\cap\cG\TIC_\infty(U)$
 but $\hqD:=\hqM^{-1}$ has an infinite number of poles
 (all points of $K$) on $\C^+$.

By Corollary~\ref{cQRCFiff}(iii), we already have a ``counter-example''
 to Corollary~\ref{cUout-Meromorphic}, 
 but to make it even more striking, we use a shifted
 version of Lemma 6.6.29 of [M02]
 (i.e., we take an exponentially stable realization
  ``$\ALSClL$'' of $\sbm{\qN\cr \qM-I}$ and apply static feedback 
 $L:=\bbm{0&I}$ to open it, thus obtaining a realization $\ALSext$
 of $\sbm{\qD\cr I-\qD}$; then we drop the bottom row
 (which is an exponentially r.c.-stabilizing state-feedback pair
  for $\ALS$))
 to obtain an exponentially (r.c.-)stabilizable realization
 of $\qD$.

(``R.c.-'' means that $\qN:=\qD\qM$ and $\qM$ are {\em r.c.}\label{pagerc} %
 in Definition~\ref{dAdmKF0},
 i.e., that $\qN,\qM\in\TIC$ are s.t.\ $\tqX\qM-\tqY\qN=I$
   for some $\tqX,\tqY\in\TIC$;
 ``exponential'' means that this holds with $\TIC_{-\del}$
in place of $\TIC$ for some $\del>0$. See Definition 6.6.10 of [M02] for more.)
\end{proof}

Any $\qKF$ making $\qN,\qM$ q.r.c.\ actually makes them r.c.\
 if $\sigma(A)$ is nice and $\dim U<\infty$:
\begin{lemma}[Nice $A$: q.r.c.$\IFF$r.c.]\label{lNMrcUfinite} %
Assume that $\qKF$, $A$ and $U$ are as in Lemma~\ref{lNMholomUfinite}
  and that $\hqM(s)$ converges as $s\in\C^+,\ |s|\to\infty$.
Then $\qN,\qM$ are q.r.c.\ iff they are r.c.
Moreover, there is a rational $g\in\H^\infty(\C^+;\C)$ %
 s.t.\ $gI_U$ and $g\hqD$ form a r.c.f.
\end{lemma}

Recall from Theorem 4.1.6(d) of [M02] that
 any r.c.f.\ can be extended to a d.c.f.\ when $\dim U<\infty$.

\begin{proof}
$1^\circ$ {\em R.c.f.:}
Now $M=\lim_{s\to+\infty}\hqM(s)\in\cG\BL(U)$
 (since $\hqM(s)^{-1}=\hqX(s)$ is uniformly bounded for big $s$).
Choose $\eps_K>0$ s.t.\ $M^*M >2\eps_K^2 I$
 and a compact  $K\sub\cRHP$ s.t.\
 $\hqM^*\hqM\ge\eps_K^2$ on $\cRHP\pois K$.

Set $\qE:=\sbm{\qN\cr \qM}$.
By Lemma \ref{lrcf}(c),
 $f(s):=\min_{\|u_0\|_U=1}\|\hqE(s)u_0\|>0\ \all s\in\C^+$
 and  there is $\eps\in(0,\eps_K)$ s.t.\ 
 $\hqE^*\hqE\ge 2\eps^2$ a.e.\ on $i\R$,
 hence everywhere on $i\R$
 (use the extensions of Lemma~\ref{lNMholomUfinite}).
We conclude that $\eps_1:=\inf_{s\in K}f(s)>0$.
Therefore, $f(s)\ge \eps_2:=\min\{\eps_1,\eps_K\}\ \all s\in\cRHP$.
By Lemma~\ref{lrcf}(a1),
 $\qN,\qM$ are r.c.

$2^\circ$ {\em $g$:} %
The poles of $\hqD$ on $\cRHP$ are on $K$, hence their number is finite;
 denote them by $s_1,...,s_n$. %
Set $g:=\prod_{k=1}^n (s-s_k)/(s+s_k+1)$ to have $\|g\|_\infty\le1$.
Then $g\hqD\in\H^\infty$ and $\det(g\hqD)$
 has no common zeros with $g$ on $\cRHP$,
 hence $g\hqD$ and $g I_U$ are r.c.\ (as in $1^\circ$).
\end{proof}

If an I/O map can be written as the quotient of two stable maps,
 then these stable maps can be chosen to be quasi--right coprime
 and normalized:
\begin{cor}[q.r.c.f.]\label{cQRCFiff} %
Let $\qD\in\TIC_\infty(U,Y)$.
Then following are equivalent:
 \begin{itemlist}
  \item[(i)] $\qD=\qN\qM^{-1}$, where $\qN,\qM\in\TIC$, $\qM\in\cG\TIC_\infty(U)$.
  \item[(ii)] $\qD=\qN\qM^{-1}$, where $\qN,\qM$ are q.r.c.,
 $\qM\in\cG\TIC_\infty(U)$, and $\qN^*\qN+\qM^*\qM=I$.

  \item[(iii)] There is a realization of $\qD$ s.t.\
 $\dUout(x_0)\ne\tyhja\ \all x_0\in H$.

  \item[(iv)] There is a stabilizable realization of $\qD$. %

  \item[(v)] For some $\omega\in\R$ %
 and any $v\in\L_\omega^2(\R_-;U)$, $\qD\in\TIC_\omega$ and
 there exists
 $u\in\L^2(\R_+;U)$ s.t.\ $\pi_+\qD (v+u)\in\L^2$.
 \end{itemlist}

Assume (ii). Then all solutions of (ii) are given by
 $\sbm{\tqN\cr \tqM}=\sbm{\qN\cr \qM}E\ (E\in\cG\BL(U),\ E^*E=I)$.
Assume that (ii) holds and $\dim U<\infty$.
Then also $\qD^\rmd$ satisfies (ii)
 (i.e., $\qD$ has both a q.r.c.f.\ and a q.l.c.f.).
If $\qN,\qM\in\WTIC$, %
 then $\qN,\qM$ are actually r.c.\ %
 and can hence be extended to a d.c.f.\ in $\WTIC$. 
\end{cor}

See Theorem~\ref{QRCF-Uout} for further equivalent conditions.
Condition (iii) (hence (i)--(v))
 holds iff $\qD$ is the I/O map of some system having
 output-stabilizing inputs --- in the negative case
 no reasonable control problems for $\qD$ have solutions.

Condition (v) says that the range of the Hankel operator $\pi_+\qD\pi_-$
 (restricted to some $\L^2_\omega$) %
 is contained in the sum of $\L^2$
  and the range of the Toeplitz operator $\pi_+\qD\pi_+$.
If (v) holds for some $\omega\in\R$, then it holds
 for any $\omega'>\omega$
 (because $\L^2_{\omega'}(\R_-;U)\sub\L^2_{\omega}(\R_-;U)\sub$).

The function $\hqD(s):=(s-1)^{-1/2}$ satisfies $\qD\in\TIC_\omega(\C)$
 for any $\omega>1$, by Theorem~\ref{TransferFct1},
 but does not satisfy any of (i)--(v),
 as noted below Example~\ref{exanotrf}.

The corollary can be applied to any quadratic minimization problems
 of  even more general systems than WPLSs as long
 as the stabilizability assumption (v) is satisfied.
\begin{proof}[Proof of Corollary~\ref{cQRCFiff}:]
$1^\circ$ {\em (ii)$\THEN$(i):}
By the definition of ``q.r.c.'', $\qN,\qM\in\TIC$.

$2^\circ$ {\em (i)$\THEN$(v):}
Let $\omega\ge0$ be s.t.\ $\qM\in\cG\TIC_\omega$.
Since $\tv:=\pi_-\qM^{-1} v\in\L^2_\omega(\R_-;U)\sub\L^2$,
 we have $u,y\in\L^2$, where $u:=\pi_+\qM \tv,\ y:=\pi_+\qN\tv$.
But $v=\pi_- \qM\qM^{-1}v=\pi_-\qM \tv$,
 hence $\pi_+\qD (v+u)=\pi_+\qD(\pi_-\qM\tv+\pi_+\qM\tv)
 =\pi_+\qD\qM\tv=y$, so (v) holds.

$3^\circ$ {\em (v)$\THEN$(iii):}
Condition (v) (without $\qD\in\TIC_\omega$)
 is exactly condition (iii) for the
 $\omega$-stable exactly reachable realization
  $\ssysbm{\tau\pi_-\|\pi_-\crh\pi_+\qD\pi_-\|\qD}$ %
 on $(U,\L^2_\omega(\R_-;U),Y)$
 (which is a WPLS when $\qD\in\TIC_\omega$). 

$4^\circ$ {\em (iii)$\THEN$(ii):}
This follows from Theorem~\ref{QRCF-Uout}(iii).

$5^\circ$ {\em (ii)$\THEN$(iv)$\THEN$(iii):}
The latter implication is trivial, and the 
 former is from Lemma 6.6.29 of [M02]
 (in fact, the realization is strongly q.r.c.-stabilizable).

$6^\circ$ {\em All solutions formula:} %
It is from from Lemma 6.4.5(e) of [M02].

$7^\circ$ {\em Case $\dim U<\infty$:} %
Choose $f$ as in the proof of Corollary~\ref{cUout-Meromorphic},
 so that $\hqX=f^{-1}F$ for some $F\in\Hoo$.
But $\det \hqX=(\det\hqM)^{-1}=f^{-1}$ a.e.,
 hence $\det F=1$ on $\C^+$, hence $F^{-1}\in\Hoo$. %
Thus, $fI_U=\hqM F^{-1}=F^{-1}\hqM$ and $G:=f\hqD=\hqN F^{-1}$ are q.r.c.,
 and $fI_Y$ and $G$ form a left factorization of $\hqD$,
 since $f\in\cG\H_\omega^\infty$ for some $\omega>0$.
Thus, (i) (hence (ii)--(v) too) holds for $\qD^\rmd$ too.
$8^\circ$ {\em R.c.:}
By the proof of Lemma~\ref{lNMrcUfinite},
 $\qN,\qM$ are r.c.\
 when $\hqN,\hqM$ are q.r.c.\ and continuous on $\CRHP$
 (the latter holds if $\qN,\qM\in\WTIC$, see p.~\pageref{pageWTIC2}),
 $\qM\in\cG\TIC_\infty(U)$ and $\dim U<\infty$.
By Theorem 4.1.6(d) of [M02], there is an extension (a d.c.f.)
 $\sbm{\qM&*\cr \qN&*}\in\cG\WTIC(U\times Y)$.
\end{proof}

However, a q.r.c.f.\ need not be a r.c.f., hence ``q.''
 is not redundant in Theorem~\ref{QRCF-Uout}
 nor in Corollary~\ref{cQRCFiff}:
\begin{example}\label{exaqrcfNotrc}(q.r.c.f.$\not\IFF$r.c.f.). %
Let $\hqM$ and $\hqN$ be the Blaschke products
 with zeros
 $\{n^{-2}\I n=2,3,...\}$ and $\{(n^2+1)^{-1}\I n=2,3,...\}$,
 respectively.
Then $\qN\qM^{-1}$ is a q.r.c.f.\
 and $\qN^*\qN+\qM^*\qM=2$ 
 (multiply $\qN$ and $\qM$ by $2^{-1/2}$ to normalize them
   as in Corollary~\ref{cQRCFiff}(ii)),
 but $\qN,\qM$ are not r.c.,
 hence $\qN\qM^{-1}$ does not have a r.c.f.,
 by Lemma 6.4.5(c) of [M02].
\end{example}

In particular, there is no d.c.f.\ although
 $\qN\qM^{-1}=\qM^{-1}\qN$ are a q.r.c.f.\ and a q.l.c.f.

\begin{proof}
$1^\circ$ {\em $\qM$ and $\qN$ are q.r.c.:}
If $\hqM f,\hqN f\in\H^2$,
 then $f$ cannot have singularities on $\C^+$
 (since $\hqM$ and $\hqN$ have no common zeros).
Thus, the zeros of $\hqM f$ equal those of $f$ combined
 with those of $\hqM$, %
 hence $B_{\hqM f}=B_{\hqM} B_f$
 (where $B_f$ is the Blaschke product formed with the zeros of $f$ etc.).
But $\hqM=B_\hqM$, %
 hence $\H^2\owns \hqM f/B_{\hqM f}=f/B_f$,
 by pp.\ 132--133 of [H62], %
 hence $\H^2\owns B_f \cdot f/B_f=f$.

$2^\circ$ One easily verifies that $\hqN(k^{-2})\to0$ as $k\to+\infty$,
 hence $\hqN,\hqM$ are not r.c., by Lemma~\ref{lrcf}(a1).
Moreover, $\hqM(s)=\prod_{n=2}^\infty |1-2/(1+s/n^{-2})|
 \ge\prod_{n=2}^\infty |1-2n^{-2}| <\infty$ when $\re s>1$,
 because $\sum_n 2n^{-2}<\infty$. %
Thus, $\|\hqM(s)^{-1}\|$ is bounded on $\C_1^+$.
\end{proof}

We note that a right factorization $\qN\qM^{-1}$ is a q.r.c.f.\ %
 iff any $\cG\TIC_\infty(U)$ %
   common right factor of $\qN$ and $\qM$ %
 is a unit
 (i.e., iff $\sbm{\qN\cr\qM}=\sbm{\qN_0\cr \qM_0}\qE,\ 
 \qN_0,\qM_0,\qE\in\TIC,\ \qE\in\cG\TIC_\infty(U) \TTHEN
 \qE\in\cG\TIC$).
(Proof: apply Lemma~\ref{lrcf}(d) to any q.r.c.f.\ of $\qN\qM^{-1}$.)

If the requirement $\cG\TIC_\infty$ were dropped, the above condition
 would be the definition of ``weakly coprime'' in [S89].
(Thus, any ``w.r.c.f.'' is a q.r.c.f.;
  let $\qE$ be the right shift on $U:=\ell^2(\N)$
  to observe that no maps are ``w.r.c.'' if $\dim U=\infty$.) %
If the system has more poles than its transfer function,
 no $\dUexp$-stabilizing state feedback can be right coprime,
 as illustrated in Example~\ref{exaUoutUexpstab} below.
However, if the system is estimatable,
 then this is not the case (and $\dUout=\dUexp$),
 by Corollary~\ref{cJointlySD}.

The ``$\dUexp$-variant'' of Theorem~\ref{QRCF-Uout} is contained
 in Corollary~\ref{cOptExpstab1}
 except that (iii) must be dropped, by Example~\ref{exaUoutUexpstab}.
Similarly,  Corollary~\ref{cQRCFiff} has an $\dUexp$-variant:
\begin{cor}[$\dUexp:\ \qN\qM^{-1}$]\label{cQRCF-Uexp} %
The following are equivalent for any $\qD\in\TIC_\infty(U,Y)$:
 \begin{itemlist}
  \item[(i)] $\qD=\qN\qM^{-1}$, where $\qN,\qM\in\TICexp$, $\qM\in\cG\TIC_\infty(U)$.

  \item[(iii)] $\qD$ has an optimizable realization
 (i.e., one with $\dUexp(x_0)\ne\tyhja\ \all x_0\in H$).

  \item[(iv)] For some $\omega\in\R,\ \del<0$,
 and any $v\in\L_\omega^2(\R_-;U)$, there exists
 $u\in\L^2_\del(\R_+;U)$ s.t.\ $\pi_+\qD (v+u)\in\L^2_{\del}$.
 \end{itemlist}
\end{cor}

Here $\TICexp:=\cup_{\omega<0}\TIC_\omega$.
Note that by shifting (Remark 6.1.9 of [M02])
 we obtain some kind of ``exponential''
 version of any of the $\dUout$ results of (e.g.) this section.

\begin{proof}
By Lemma 6.6.29 of [M02] (which was explained in 
 the proof of Example~\ref{exadimUinftyRCF}),
 (i) implies (iii).
The converse follows from Corollary~\ref{cOptExpstab1} %
 (set $\qM:=(I-\qF)^{-1}$).
We get ``(iv)$\IFF$(i)'' from
 Corollary~\ref{cQRCFiff} applied to $\hqD(\cdot+\del)$
 (or $\efn^{-\del\cdot}\qD\efn^{\del\cdot}$)
 (for each $\del<0$). %
\end{proof}

Next we present the %
 $\dUout$-variant of Corollary~\ref{cExpStabB1}.
The ``optimal'' output-stabilizing feedback for $\ALS$
 also output-stabilizes any extension of $\ALS$:
\begin{cor}[$\dUout$ through $B_1$]\label{cUout-B1} %
Assume the output-FCC and choose $\qKF$ as in Theorem \ref{QRCF-Uout}(iii).
If $\tALS:=\ssysbm{\qA\|\qB& \qH\crh \qC\|\qD & \qG}$ is a WPLS
 (say, on $(U\times W,H,Y)$, for some $\qH,\qG$),
 then there is $\qE\in\TIC_\infty(W,U)$ s.t.\
 $[\tqK|\tqF]
 :=\ssysbm{\qK\|\qF& \qE\cr 0\|0&0}$
 is a SOS-stabilizing state-feedback pair for $\tALS$.

Moreover, $\tqN,\tqM$ are q.r.c., and so are 
 $\qDClL=\tqDClL\sbm{I\cr 0}$ and $\qM=\tqM_{11}$.
\end{cor}

(The proof is given on p.~\pageref{pageproof-cUout-B1}.
Note that $\tqKClL=\tbm{\qKClL\cr 0}$,
 $\tqM=\tbm{\qM&\qM\qE\cr 0&I}$, $\tqFClL=\sbm{*&*\cr 0&0}$.)
If $\qKF$ is given by some state-feedback operator $K:\Dom(A)\to U$,
 then the above corollary surprisingly tells us that
 not only is $K$ ``compatible'' with $\Hgen$
 but it also makes $\qG_\ClL$ stable.

The above corollary leads to 
 the $\dUout$-variant %
 of Theorem~\ref{cJointlySD},
 showing that the output-FCC for $\ALS$ and $\ALSd$ is
 sufficient for the existence of a d.c.f.\ of $\qD$:
\begin{theorem}[$\dUout\&\dUout^{\ALSd} \TTHEN$ d.c.f.]\label{JointlyUout} %
Assume that $\ALS$ and $\ALSd$ satisfy the output-FCC. %
Let $\qKF$ and $\qHGd$ be the corresponding optimal state-feedback pairs.
Then they are jointly externally stabilizing
 and define a doubly coprime factorization of $\qD$,
 namely (\ref{eS4A:4.4}).

Moreover, then any SOS- (resp.\ I/O-)stabilizing state-feedback pair
 for $\ALS$ is externally (resp.\ I/O-)r.c.-stabilizing.

Finally, the equivalence of Corollary~\ref{cJointlySD} %
 also holds after replacements
 ``external''$\mapsto$``exponential'' and
 ``state-FCC''$\mapsto$``output-FCC''
 (again ``with internal loop'' is extraneous if $\dim U,\dim Y<\infty$).
\end{theorem}

(The proof is given on p.~\pageref{pageproof-JointlyUout}.
A system is {\em externally stable} if its components are stable
 except possibly the semigroup.
Thus, the two pairs are {\em jointly externally stabilizing}
 if there exists $\qE\in\TIC_\infty(Y,U)$ s.t.\ (\ref{eJointly1}) is a WPLS, 
 and $(\ALStotal)_L$ and $(\ALStotal)_\tL^\rmd$ are 
 externally stable (i.e., their components, except possibly $\qA_L,\qA_\tL$, 
 are stable).)

In particular, 
 the $\TIC(U\times Y)$ %
 maps in (\ref{eS4A:4.4})
  are stable and the inverses of each other.

In the theorem, one can replace $\qKF$ by any other SOS-stabilizing pair
 (or by any other I/O-stabilizing pair if  $(\ALStotal)_L$
   is required to be merely I/O-stable,
   i.e., to have its I/O map in $\TIC$), %
 as one observes from the proof.

We finish this section %
 by giving ``generalizations'' of  Corollaries \ref{cOptExpstab1}
 and~\ref{cExpStabB1} and further observations.
The following result shows that the optimization over 
 a typical domain of optimization 
 can be completely reduced to the optimization of a SOS-stable
 system over $\L^2(\R_+;U)$
 (a similar claim on partial control is given in Proposition~\ref{pExKFstabB1}):
\begin{proposition}[$\gUstar:$ FCC$\IFF$stabilizable]\label{pExKFstab} %
Assume that $\vartheta=0$ and $\sZ=\L^2(\R_+;\tY)$. 
Then the FCC is satisfied iff there is a state-feedback pair
 $\qKF$ s.t.\ $\qKClL x_0\in\gUstar(x_0)\ \all x_0\in H$.
Assume the FCC and choose $\qKF$ as in the proof.

{\bf (a)}
Then $\qKClL x_0 +\qM \uc\in\gUstar(x_0) \IFF \uc\in\L^2(\R_+;U)\ \all x_0\in H\ \all \uc$;
 in particular, $\qKF$ is SOS-stabilizing,
 $\qX\in\cG\BL(\gUstar(0),\L^2(\R_+;U))$.

{\bf (b)}
If $\gUstar=\dUout$ (resp.\ $\gUstar=\dUexp$),
 then $\qN,\qM$ are q.r.c.\ (resp.\ $\ALSClL$ is exponentially stable),
 and also the pair of Theorem~\ref{QRCF-Uout} 
  (resp.\ Corollary~\ref{cOptExpstab1})
  satisfies (a)--(d). 

{\bf (c)}
The system
 $\ALS_2:=\qABCDClL$ is
 [positively] $J$-coercive over $\dUout^{\ALS_2}=\L^2(\R_+;U)$
 iff $\ALS$ is  [positively] $J$-coercive over $\gUstar$.

{\bf (d)}
A control $\uc$ is $J$-critical for $x_0$, $\ALS_2$ and $J$ over $\dUout^{\ALS_2}$
 iff $u:=\qKClL x_0+\qM \uc$ is $J$-critical for $x_0$, $\ALS$ and $J$
 over $\gUstar$.
Moreover, $\cJ(x_0,u)=\cJ_\ClL(x_0,\uc):=\p{\yc,J\yc},\ \yc:=\qCClL x_0+\qDClL \uc\ \all x_0,\uc$,
 hence $\Ric=\Ric_2$
 (if either, hence both exist). %
\end{proposition}

This shows that the inputs $u\in\gUstar(x_0)$
 correspond 1-1 to the stable inputs to the stabilized system.
Sometimes this allows one to reduce the problem to
 Theorem~\ref{SpF1} or other results for the stable case.
This is particularly useful when one can show that the smoothness
 is preserved in the uniformly positive case (cf.\ Section~\ref{sMTIC}),
 even if the original problem were indefinite.
However, often one prefers to use the original data instead of $\ALS_2$.

\begin{proof}[Proof of Proposition~\ref{pExKFstab}:]\label{pageproof-pExKfstab}
$1^\circ$ {\em ``Iff'', SOS, $\qX$, q.r.c.:}
Define $\tALS$ by setting $\tqA:=\qA,\ \tqB:=\qB,\
 \tqC:=\tbm{\qC\cr \qQR\cr 0},\ 
 \tqD:=\tbm{\qD\cr \qRR\cr I}$.
Since, obviously, $\dUout^{\tALS}=\gUstar$,
 we obtain the equivalence from Theorem~\ref{QRCF-Uout}
 (whose proof shows that $\tRic\ge0$, $\tS=I$),
 by which $\tALSClL$ is also SOS-stable,
 hence so is $\ALSClL$ (being contained in $\tALSClL$).

By Lemma~\ref{lSgg0}(c), $\qX\in\cG\BL(\gUstar(0),\L^2(\R_+;U))$
 and $\tqDClL$ and $\qM$ are q.r.c., 
 hence so are $\tbm{\qN\cr \qRRClL}$ and $\qM$,
  because $\tqDClL=\tbm{\qN\cr \qRRClL\cr \qM}$
 (hence so are $\qN$ and $\qM$ if $\gUstar=\dUout$ so that $\qRRClL=\qN$).

$2^\circ$ {\em ``$\IFF$'':}
Given $\uc\in\Lloc^2(\R_+;U)$, define $u$ as in Lemma~\ref{luc}.
If $\uc\in\L^2$, then $u=\tyc\tbm{0\cr 0\cr I}\in\L^2$.
If $u\in\gUstar(x_0)$, then $u=\qKClL x_0+\qM \uc$,
 hence $\qM\uc\in\gUstar(0)$, by Lemma~\ref{lgUstar},
 hence $\uc=\qX\qM\uc\in\L^2$.

(b) The ``q.r.c.'' claim was proved in $1^\circ$.
If $\gUstar=\dUexp$,
 then $\tALSClL\sbm{I\cr0}$ is exponentially
 stable, by Theorem~\ref{ALSopt}, hence so is $\tALSClL$,
 by Lemma~\ref{lExpStable}.
For the pair of the theorem or the corollary, the proofs are
 similar than to that above (and mostly given in Theorem 8.4.5 of [M02]).

(c)\&(d) These follow easily from Lemma~\ref{luc} and
 the claims on $\qX$ and $\uc$ in (a). %
\end{proof}

If $\ALS$ is formally stabilizable through the first input, then
 it is (state-feedback) stabilizable through the first input:
\begin{proposition}[$\gUstar$: stabilizable through $B_1$]\label{pExKFstabB1} %
Assume that $\vartheta=0$ and $\sZ=\L^2(\R_+;\tY)$
 and that $\qKF$ is as in Proposition~\ref{pExKFstab}.
If $\tsysbm{\qA\|\qB&\qH\crh \qC\|\qD&\qG\cr \qQR\|\qRR&\qT}$
 is a WPLS (for some $\qH,\qG,\qT$),
 then $\qKF$ can be extended as in Corollary~\ref{cUout-B1},
\end{proposition}

Note that $\tqKClL x_0+\tqM \sbm{\uc\cr w}=\sbm{u_1\cr w}$,
 and $w=0 \TTHEN (u_1\in\gUstar(x_0)\IFF \uc\in\L^2(\R_+;U))$.
Obviously, $\sbm{u_1\cr w}\in\dUout^{\tALS}(x_0)
 \IFF \sbm{\uc\cr w}\in\L^2$ %
 (this is useful for $\H^\infty$ problems, 
  whether over $\dUout$, $\dUexp$ or something else).
To be brief, we shall postpone the obvious further equivalents of (a)--(d) of 
 Proposition~\ref{pExKFstabB1} to an $\H^\infty$ article. %
 
\begin{proof}[Proof of Proposition~\ref{pExKFstabB1}:]\label{pageproof-pExKFstabB1}
Apply Corollary~\ref{cUout-B1} to the $\tALS$ of the proof
 of Proposition~\ref{pExKFstab}.
Then $\tqN=\tbm{\qD&\qG\cr \qRR&\qT\cr I&0}\tqM$ and $\tqM$
 are q.r.c., hence so are
 $\tbm{\qD&\qG\cr \qRR&\qT}\tqM$ and $\tqM$
 (which are the ``$\tqN,\tqM$'' of Corollary~\ref{cUout-B1}
  applied to the system in Proposition~\ref{pExKFstabB1}).
\end{proof}

It is known that a matrix-valued transfer function has a stabilizing 
 dynamic feedback controller iff it has a d.c.f.
Using Corollary~\ref{cQRCFiff}, we can extend ``only if''
 to operator-valued proper transfer functions:
\begin{lemma}\label{lDFCdcf} %
If $\qD$ has a stabilizing dynamic feedback controller 
 (without an internal loop),
  i.e., %
    $\sbm{I& -\qT\cr -\qD&I}^{-1}\in\TIC$ for some $\qT\in\TIC_\infty(Y,U)$,
 then $\qD$ has a d.c.f.
\end{lemma} 

It follows that all stabilizing dynamic feedback controllers
 for $\qD$ are given
 by the standard
 Youla parametrization formula
 (p.~290 of [M02]).
Conversely, if $\qD$ has a d.c.f.\ 
 and $\dim U,\dim Y<\infty$,
 then $\qD$ has a {\em stable} stabilizing dynamic feedback controller,
 by Corollary~6.6 of [Q03].
The case for general $U,Y$ is still open (whether ``stable'' removed or not).
We have named the above result a lemma, since it is a special
 case of Theorem~\ref{DCFiff} (and needed for its proof).

\begin{proof} %
We have the right factorizations $\qD=\qN_0\qM_0^{-1},\ \qT=\qY_0\qX_0^{-1}$,
 where $\qM_0=(I-\qT\qD)^{-1}$, $\qX_0=(I-\qD\qT)^{-1}$, by (7.5) of [M02].
Therefore, there are q.r.c.f.'s $\qD=\qN\qM^{-1},\ \qT=\qY\qX^{-1}$,
 by Corollary~\ref{cQRCFiff}.
By Lemma 7.1.5(b) of [M02], $\sbm{\qM&\qY\cr \qN&\qX}\in\cG\TIC$,
 i.e., $\qD$ and $\qT$ have a (joint) d.c.f. 
(We used above the fact that the proof of Lemma 7.1.5 (and 6.6.6)
 obviously applies to q.r.c.f.'s in place of r.c.f.'s.)
\end{proof}

This leads to the following equivalence
 for any $\qD\in\TIC_\infty(U,Y)$ 
 (the terminology will be explained below):
\begin{theorem}[D.c.f.$\IFF$...]\label{DCFiff} %
The following are equivalent for any $\qD\in\TIC_\infty(U,Y)$:
 \begin{itemlist} %
   \item[(i)] $\qD$ has a d.c.f. 
   \item[(ii)] $\qD$ has a r.c.f. %
   \item[(iii)] $\qD$ has a l.c.f. %
   \item[(iv)] $\qD$ has a realization $\ALS$ s.t.\ %
 the output-FCC holds for $\ALS$ and $\ALS^\rmd$.
   \item[(v)] $\qD$ has a stabilizable and detectable realization. %
   \item[(vi)] $\qD$ has a jointly stabilizable and detectable realization.
   \item[(vii)] $\qD$ has a stabilizing controller with internal loop. 

   \item[(viii)] $\qD$ has a stabilizing canonical controller. %

   \item[(ix)] Some realization of $\qD$ has  %
 a stabilizing controller with internal loop.  %
   \item[(x)] $\sbm{\qD&0\cr 0&I}$ has a d.c.f.\
 (or r.c.f.\ or l.c.f.). %
 \end{itemlist} %

If $\dim U,\dim Y<\infty$,
 then we have three more equivalent conditions:
\begin{itemlist}
   \item[(xi)] $\qD$ has a stabilizing controller.
   \item[(xii)] Some realization of $\qD$ has a stabilizing controller.
   \item[(xiii)] %
 $\hqD=FG^{-1}$
 with $F,G\in\H^\infty$, $F^*F+G^*G\ge\eps I$ on $\C^+$ for some $\eps>0$,
 $\det G\not\equiv0$.
\end{itemlist}

Given a d.c.f., all stabilizing controllers with internal loop for $\qD$
 are obtained from the standard Youla parameterization %
 $(\qY+\qM\qE)(\qX+\qN\qE)^{-1}$, where
 $\qE\in\TIC(Y,U)$ is arbitrary
 (the controller is proper iff $\qX+\qN\qE\in\cG\TIC_\infty(U)$).

In particular, any stabilizing controller with internal loop
 (for any well-posed map)
 is equivalent to a canonical controller.
\end{theorem}

(Any dimension of $I$ will do in (x).
In general (xi) and (xii) are sufficient
 and the Corona condition (xiii) necessary but not sufficient.)

We say that $\qO$ is a {\em stabilizing controller with internal loop}
 for $\qD$ if $\qO\in\TIC_\infty(Y\times \Xi,U\times\Xi)$
 for some Hilbert space $\Xi$ %
 and $(I-\qDo)^{-1}\in\TIC$,
 where $\qDo=\tbm{0&\qO_{11}&\qO_{12}\cr 
            \qD&0&0\cr 0&\qO_{21}&\qO_{22}}$.
Note from Figure~\ref{fILDF-stab}
  that $\qD^o_I:\tbm{u_L\cr y_L\cr \xi_L}  \mapsto \sbm{u\cr y\cr \xi}$,
  where $\qD^o_I:=(I-\qDo)^{-1}-I$; cf.\ Definition~\ref{lALSL0}
  with $L=I$.
Thus, $\qO$ is stabilizing iff the maps
 $\tbm{u_L\cr y_L\cr \xi_L}  \mapsto \sbm{u\cr y\cr \xi}$
 are well-posed and stable.

Condition (ix) is formally stronger: 
 it means (the existence of $\ALS$ and $\tALS$ for this fixed $\qD$ such)
 that all 25 maps from initial states and external inputs
 to states and outputs in Figure~\ref{fILDF-stab} %
 are stable
 (i.e., that $\ALS^o_I$ is stable, where $\ALS^o$ is given by (7.21) of [M02]).
See Section 7.2 of [M02] for further details.
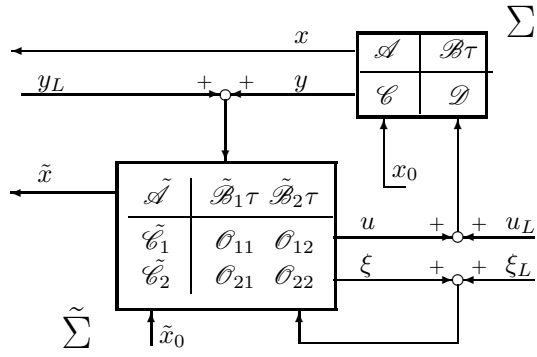
\begin{figure} %
        \[
        \begin{picture}(80,108)(-40,-91)
        \thicklines
        \put(6,-15){ \framebox(48,31){
        $ \bsysm{\qA \| \qB\tau\crh 
        \qC \| \qD}
        $}}
        \put(65,15){\LARGE$\ALS$}

        \put(-103,-103){\LARGE$\tALS$}

        \put(-85,-88){ \framebox(82,55){
        $ \bsysm{\tqA \| \tqB_1\tau \!\!\!\! &\tqB_2\tau\crh 
        \tqC_1 \| \qO_{11}\phantom{} \!\!\!\! &  \qO_{12}\phantom{}\cr
        \tqC_2 \| \qO_{21}\phantom{} \!\!\!\! &  \qO_{22}\phantom{}}
        $}}

        \thinlines

        \put(-41,-7){\circle{4}} 
        \put(-52,-4){$\scriptstyle +$} 
        \put(-36,-4){$\scriptstyle +$} 

        \put(-41,-9){\vector (0,-1){23}}

        \put(8,-7){\vector (-1,0){47}}
        \put(-15,-4){$y$}

        \put(-112,-4){$y_L$}
        \put(-118,-7){\vector (1,0){75}}
        \put(0,-61){\vector (1,0){45}} %
        \put(10,-58){$u$}

        \put(47,-61){\circle{4}} 
        \put(36,-58){$\scriptstyle +$} 
        \put(52,-58){$\scriptstyle +$} 

        \put(76,-61){\vector (-1,0){27}} 
        \put(65,-58){$ u_L$}

        \put(47,-59){\vector (0,1){43}} %
        \put(22,-38){$x_0$}
        \put(19,-42){\vector (0,1){26}} 
        \put(19,-42){\line (1,0){8}} 

        \put(-15,12.5){$x$}
        \put(8,9.5){\vector (-1,0){130}}

        \put(0,-77){\vector (1,0){45}}
        \put(10,-74){$\xi$}

        \put(47,-77){\circle{4}} 
        \put(36,-74){$\scriptstyle +$} 
        \put(52,-74){$\scriptstyle +$} 

        \put(65,-74){$\xi_L$}
        \put(76,-77){\vector (-1,0){27}} 

        \put(47,-79){\line (0,-1){22}}
        \put(47,-101){\line (-1,0){60}}
        \put(-13,-101){\vector (0,1){12}}

        \put(-66,-101){$\tx_0$}
        \put(-69,-102){\vector (0,1){13}}

        \put(-112,-40){$\tx$}
        \put(-82.4,-44){\vector (-1,0){40}} %
        \end{picture}
        \]
        \caption{DF-controller $\tALS$ with internal loop for $\ALS\in\WPLS(U,H,Y)$}
        \label{fILDF-stab}
\end{figure}

If $\qY\in\TIC(Y,U)$ and $\qX\in\TIC(U)$ are r.c., then
 $\qO:=\sbm{0&\qY\cr I&I-\qX}$ 
 is called
 a {\em canonical controller}
 (see [CWW01] or [M02]);  %
 in [M02], the term {\em controller with a coprime internal loop}
 was used.
Sometimes we denote it by $\qY\qX^{-1}$,
 as in the Youla parameterization above.
(It is equivalent to $\tbm{0&I\cr \tqY&I-\tqX}$
 for certain l.c.\ $\tqY$ and $\tqX$.)
A (dynamic feedback) controller $\qO$ (resp.\ $\tALS$) with internal loop
 is {\em proper} or well-posed
 (i.e., ``with internal loop'' can be dropped)
 iff $I-\qO_{22}\in\cG\TIC_\infty$.
In that case we can redefine $\qO$ (resp.\ $\tALS$) so as to 
   have $\qO_{12},\qO_{21},\qO_{22}=0$
 (resp.\  $\qO_{12},\qO_{21},\qO_{22},\tqB_2,\tqC_2=0$),
  as in the classical definition of a controller.

The Youla parameterization covers all stabilizing controllers
 with internal loop in the sense that
 any other controller with internal loop defines
 the same closed-loop map $\sbm{u_L\cr y_L}\mapsto \sbm{u\cr y}$
 as exactly one of these
 (modulo $\sbm{\qY'\cr \qX'}=\sbm{\qY\cr \qX}\qE$ for some $\qE\in\cG\TIC$),
 although the maps from $\xi_L$ and to $\xi$ 
 (internal loops) may differ.
In particular, this parameterization contains all well-posed 
 stabilizing controllers.

Any map having a right factorization
 has a realization that satisfies the output-FCC,
 by Lemma 6.6.29 of [M02].
Thus, the map $\qN\qM^{-1}=\qM^{-1}\qN\in\TIC_1(\C)$ %
 of Example~\ref{exaqrcfNotrc}
 has realizations that satisfy the output-FCC
 and ones whose dual satisfies the output-FCC.
However, none of those realizations satisfies both,
 by (iv) and (ii) above.

\begin{proof}[Proof of Theorem~\ref{DCFiff}:] %
$1^\circ$ {\em We have (vi)$\IFF$(vi')$\IFF$(v)$\IFF$(iv)$\IFF$(i):}
By Theorem~\ref{JointlyUout}, we have (vi')$\IFF$(iv)$\THEN$(i),
 where we have added ``externally'' to (vi) to define (vi').
The equivalence of (i) and (vi) was established in Theorem 4.4 of [S98a].
The implications ``(vi')$\IF$(vi)$\THEN$(v)$\THEN$(iv) are obvious.

$2^\circ$ {\em (vii)$\THEN$(x):}
Assume (vii).
By Lemma 7.2.6 of [M02], %
 some $\qO\in\TIC_\infty$ (I/O-)stabilizes $\lqD:=\sbm{\qD&0\cr 0&I}$.
By Lemma~\ref{lDFCdcf}, (x) follows.

$3^\circ$ {\em (x)$\THEN$(iv):}
Apply ``(i)$\IFF$(v)'' to obtain a stabilizable and
 detectable realization of $\lqD$,
 and then remove the last row and column to satisfy (iv) for $\qD$.

$4^\circ$ {\em (i)$\THEN$(ii)$\THEN$(viii)$\THEN$(vii):}
Implications ``(i)$\THEN$(ii)'' and ``(viii)$\THEN$(vii)''
 are trivial,
 and ``(ii)$\THEN$(viii)'' is from Corollary 7.2.13(b) of [M02].

$5^\circ$ From the above we see that (i), (ii), (iv), (v), (vi), (vii),
 (viii) and (x) are equivalent.
By duality, we get ``(i)$\IFF$(iii)''.

$6^\circ$ {\em (vii)$\IFF$(ix):}
Implication ``(ix)$\THEN$(vii)'' is trivial.
Conversely, if (vii) holds, then $\qD$
 and $\qO$ (see $2^\circ$) have d.c.f.'s,
 hence (vi) holds to both of them.
Therefore, (ix) follows from Theorem 7.2.3(b)(3.) of [M02].

$7^\circ$ {\em (ii)$\THEN$(xiii)$\THEN$(vii):}
We have (ii)$\THEN$(xi) in general, 
 by Lemma \ref{lrcf}(a1).
Conversely, if $\dim U<\infty$ and (xi) holds
 (it tacitly requires that $G\in\H^\infty(\C^+;\BL(U))$),
 then $\sbm{G\cr F}$ can be extended to
 $R:=\tbm{G&\hqY\cr F&\hqX}\in\cG\H^\infty(\C^+;\BL(U\times Y))$,
 hence $\hqY\hqX^{-1}$ satisfies (vii) %
 (since $(I-\qD_I^o)^{-1}\in\TIC \IFF (\hqX-FG^{-1}\hqY)\in\TIC
 \IFF R\in\cG\H^\infty$). %
$8^\circ$ {\em (xi)\&(xii):}
Obviously, (xi) or (xii) implies (ix) or (vii).
Assume then (i). As noted below Lemma~\ref{lDFCdcf},
 (vii) holds even without ``with internal loop'';
 so does (ix) too, by $2^\circ$ (slightly modified).

$9^\circ$ {\em All stabilizing controllers:}
This follows from Theorem 7.2.14(ii) of [M02].

$10^\circ$ {\em Canonical controllers:}
The last claim follows now from Corollary 7.2.13(a1) of [M02].
\end{proof} %

The ``$\dUexp$-variant'' of Theorem~\ref{QRCF-Uout} is contained
 in Corollary~\ref{cOptExpstab1}
 except that (iii) must be dropped, by Example~\ref{exaUoutUexpstab}.
Similarly,  Theorem~\ref{DCFiff} has an $\dUexp$-variant:
\begin{cor}[Exponential d.c.f.]\label{cDCF-Uexp} %
The following are equivalent for any $\qD\in\TIC_\infty(U,Y)$:
 \begin{itemlist}
  \item[(i)] $\qD$ has an exponential d.c.f. %

  \item[(ii)] $\qD$ has an exponentially stabilizing 
 controller with internal loop. %

  \item[(iii)] $\qD$ satisfies the exponential version of
 any (hence all) of (i)--(x) of Theorem~\ref{DCFiff}. %

  \item[(iv)] $\qD$ has a realization that satisfies
  any (hence all) of (i)--(v) of Corollary~\ref{cJointlySD}. %
 \end{itemlist}\itemlistnoproof %
\end{cor}

An {\em exponential d.c.f.} is defined by Definition~\ref{drcf}(b2)
 with $\cG\TICexp$ in place of $\cG\TIC$
 (here $\TICexp:=\cup_{\omega<0}\TIC_\omega$).
By (iii)(ii), it is equivalent to an exponential r.c.f.\ (or l.c.f.).

Thus, (i) holds iff there exists $\omega<0$ s.t.\
 $\efn^{-\omega\cdot}\qD\efn^{\omega\cdot}$
 (which is the I/O map corresponding to  $\hqD(\cdot+\omega)$)
 has a d.c.f.
That is the exponential version of Theorem~\ref{DCFiff}(i)
  (see Remark 6.1.9 of [M02] for details on shifting),
 hence equal to (iii)(i). Similarly, (ii) equals (iii)(vii)
 and (iii)(vi) equals (iv)(i); this proves Corollary~\ref{cDCF-Uexp}.

For matrix-valued transfer functions, one typically
 allows for any
 controllers $\hqT:=fg^{-1}$ or $g^{-1}f$ %
 (``$\H^\infty/\H^\infty$'' fraction controllers, possibly improper),
 where $f,g\in\H^\infty$ and $\det g\not\equiv0$.
We recall from Remark 7.2.8 of [M02]
 that such controllers (and more) are covered
 by controllers with internal loop:
\begin{remark}[$\H^\infty/\H^\infty$ controllers]\label{rImproper} %
Let $\Xi$ be a Hilbert space, and let
 $f\in\H^\infty_\infty(Y\times\Xi,U\times \Xi)$ %
 be s.t.\ $f_{22}$ is invertible at some $s_0\in\C^+$.
The map $\hqT:=f_{11}+f_{12}f_{22}^{-1}f_{21}$
 ``stabilizes $\hqD$''
 (i.e., $\tbm{I&-\hqT\cr -\hqD&I}$ equals the inverse of 
    a $\TIC(U\times Y)$ map near $s_0$)
 iff %
 the map $\qO\in\TIC_\infty$, defined by
 $\hqO:=\tbm{f_{11}&f_{12}\cr f_{21}&I-f_{22}}$,
 is a stabilizing controller with internal loop for $\qD$.
If this is the case, then $\qD$ has a d.c.f.\
 and $\qO$ is ``equivalent to $\qT$''.
Naturally, this remark also holds with ``DPF-stab''
 in place of ``stab'' (see the end of Corollary~\ref{cDPF})
 if we remove the ``i.e.''-comment in parenthesis. %
\end{remark}

(This follows from the computations of the proof of Lemma 7.2.7 of [M02]
 as in Remark 7.2.8,
 except that the d.c.f.\
 (which is ``joint with $\qT$'' due to
    the last claim of Theorem~\ref{DCFiff}) %
 is from Theorem~\ref{DCFiff}.)

E.g., the finite-dimensional unstable
 SISO plant $\hqD(s)=1+1/s$ 
 is (exponentially) stabilized by the controller $\qO=\sbm{0&1\cr -1&1}$
 with internal loop
 (since $I-\qO_{22}=0$ is nowhere invertible, this is not equivalent to
   any proper nor to any ``$\H^\infty/\H^\infty$'' controller,
  by Remark 7.2.8 of [M02]). %
(As noted in p.~7 of [WC97], %
 this example is physically meaningful.)

{\em Dynamic partial feedback (DPF)} 
 of $\qD\in\TIC_\infty(U\times W,Z\times Y)$,
 where also $W$ and $Z$ are Hilbert spaces,
 means that, in Figure~\ref{fILDF-stab},
 there is an additional first output (``$z$'') and second input (``$w$'')
 that are not connected to the controller.
Thus, $\qO$ is a stabilizing {\em DPF-controller} for $\qD$ with internal loop
 iff
 $\qO_{\rm DF}:=\tbm{0&\qO_{11}&\qO_{12}\cr 0&0&0\cr 0&\qO_{21}&\qO_{22}}$
 is a stabilizing (DF-)controller for $\qD$ with internal loop
 (see Section 7.3 of [M02] for details).
\begin{cor}[Partial feedback]\label{cDPF} %
The following are equivalent for $\qD\in\TIC_\infty(U\times W,Z\times Y)$:
\begin{itemlist}
  \item[(i)] $\qD$  has a stabilizing DPF-controller with internal loop
  \item[(ii)] $\qD$ has a r.c.f.\ of the form
      $\qD=\sbm{\rN_{11}&\rN_{12}\cr \rN_{21}&\rN_{22}}
                \sbm{\rM_{11}&\rM_{12}\cr 0&I}^{-1}$
     s.t.\ $\qN_{21}$ and $\qM_{11}$ are r.c.

  \item[(iii)] $\qD$ has a l.c.f.\ of the form
      $\qD        =\sbm{I &\lM_{12}\cr 0 &\lM_{22}}^{-1}
          \sbm{\lN_{11}&\lN_{12}\cr \lN_{21}&\lN_{22}}$
     s.t.\ $\tqN_{21}$ and $\tqM_{22}$ are l.c.
\end{itemlist}

If this is the case, then
 such controllers are exactly
 the stabilizing controllers with internal loop for $\qD_{21}$;
 thus, any of them is equivalent to a canonical controller
 given by the Youla parameterization
 (in particular, some of them are proper if $\dim U,\dim Y<\infty$). 
\end{cor}

Note that $\qD_{21}:(u+u_L)\mapsto y$
 is the control-to-measurement part of $\qD$,
 and that $\qD_{21}=\qN_{21}\qM_{11}^{-1}$ is a r.c.f.\ (under (ii)).
By Lemma 7.3.8 of [M02], %
 two controllers with internal loop
 are equivalent %
 as DPF for $\qD$
 iff they are are equivalent %
 as DF for $\qD_{21}$.

In particular, the stabilizing DPF-controllers for $\qD$ are exactly
 the canonical controllers $\tqX^{-1}\tqY$ for
 $\tqX,\tqY\in\TIC$ s.t.\ $\tqX\qM_{11}-\tqY\qN_{21}=I$,
 equivalently, $\qY\qX^{-1}$ for which
 $\tbm{\qM_{11}&\qM_{12}&\qY\cr \qM_{21}&\qM_{22}&0\cr \qN_{21}&\qN_{22}&\qX}
   \in\cG\TIC(U\times W\times Y)$ %
 for some (hence all) r.c.f.\ $\qN\qM^{-1}$ of $\qD$
 (Lemma 7.3.22 of [M02]). %

The coprimeness condition cannot be weakened:
 the (exponential) r.c.f.\
 $\hqN\hqM^{-1}:=\sbm{1&0\cr 0&0}\tbm{s/(s+1)&0\cr 0&1}^{-1}$
 is %
 of the form $\sbm{*&*\cr *&*}\tbm{*&*\cr 0&I}^{-1}$
 (hence $\qD$ and $\qD_{21}$ both have a d.c.f.\
  and are thus DF-stabilizable with internal loop),
 but yet $\qD$ is not DPF-stabilizable with internal loop.
However, if $\qD$ is DPF-stabilizable, then any r.c.f.\ of 
 the form $\sbm{*&*\cr *&*}\sbm{*&*\cr 0&I}^{-1}$
 has $\qN_{21},\qM_{11}$ r.c., 
 by Corollary 7.3.17 of [M02].

By Corollary~\ref{cDPF}, Hypothesis 7.3.15 of [M02]
 holds iff $\qD$ is DPF-stabilizable with internal loop.
These results simplify significantly Section 7.3 of [M02]
 like Theorem~\ref{DCFiff} did for 7.1 and 7.2.

(Note: the proof of ``(ii)$\THEN$(iii)'' in Lemma 7.3.6(b2) of [M02]
 has been written down incompletely.
Perhaps the shortest way to prove Lemma 7.3.6 is to choose
 pair $\qKF$ for $\ALS_{21}$ (use Theorem~\ref{JointlyUout})
 and extending it as in Corollary~\ref{cUout-B1},
 to obtain (iii) above, hence (i) above,
 so that 7.3.6(b2)(iii) follows from 7.3.11(b)(1.). %
Alternatively, work as in Corollary 2.2 of [G92]
 and use (i)$\IFF$(ii) of Theorem~\ref{DCFiff}.) %

\begin{proof}[Proof of Corollary~\ref{cDPF}:]
If $\qO$ DPF-stabilizes $\qD$ with IL, then it DF-stabilizes $\qD_{21}$
 with IL, by Lemma 7.3.5 of [M02],
 hence then $\qD_{21}$ has a d.c.f., by Theorem~\ref{DCFiff}.
Thus, (i) is equivalent to (i) (hence to (ii) and (iii) too)
 of Proposition 7.3.14 of [M02], whose proof provides the equivalence.
The remaining claims follow from Theorem 7.3.19 (and 7.3.20) of [M02]
 (except case $\dim U,\dim Y<\infty$ from Theorem~\ref{DCFiff}).
\end{proof}

By combining the results of this section with Chapter~7
 (including Lemmas 7.3.5 and 7.3.6(b1)) of [M02],
 we get the following:
\begin{cor}\label{cDPFsystem} %
A WPLS $\ALS$ on $(U\times W,H,Z\times Y)$
 is exponentially DPF-stabilizable with internal loop
 iff $\ALS_{21}:=\tsysbm{\qA\|\qB_1\crh \qC_2\|\qD_{21}}$ 
 and its dual 
 satisfy the state-FCC,
 in which case the exponentially DPF-stabilizing controllers with internal
 loop for $\ALS$ equal the (DF-)stabilizing controllers with internal
 loop for $\ALS_{21}$.

A similar claim holds with ``exponentially'' removed and
 ``output-FCC'' in place of ``state-FCC''
 (cf.\ the proof of Theorem~\ref{DCFiff}), %
 except that then also $\ALS$ and $\ALS^\rmd$ must satisfy the output-FCC.\noproof
\end{cor}

Theorem~\ref{PosJcKF} solved positively $J$-coercive problems
 by reducing them to the stable (positive) spectral factorization\label{pageSpF}
 result given below.
Indefinite problems for stable I/O maps of $\MTIC$ 
 (convolutions with measures) type can 
 be solved through spectral factorization as well, as explained in [M02].
\begin{theorem}[SpF]\label{SpF1} %
{\bf (a)} If %
 $\qD^*J\qD\gg0$,
 then $\qD^*J\qD=\qX^*\qX$ for some $\qX\in\cG\TIC(U)$.

{\bf (b)} If  $\pi_+\qD^*J\qD\pi_+\in\cG\BL(\L^2(\R_+;U))$
 and $\qD\in\WTIC(U,Y)$
 then $\qD^*J\qD=\qX^*S\qX$, 
 where $\qX\in\cG\WTIC(U,Y),\ S\in\cG\BL(U)$.
\end{theorem}

The claim $\pi_+\qD J\qD\pi_+\ge\eps I$ (on $\L^2(\R_+;U)$)
 is equivalent to $\qD^*J\qD\ge\eps I$ (on $\L^2(\R;U)$)
 (by Lemma~13(i) of [S97]). 

By $\qD\in\WTIC(U,Y)$\label{pageWTIC2} we mean that %
 $\qD u=Du+f*u\ \all u\in\L^2$ 
 for some $D\in\BL(U,Y),\ f\in\L^1(\R_+;\BL(U,Y))$
 (i.e., $\qD$ consists of an $\L^1$ impulse response plus a feedthrough).
Similar results hold when $\qD$  also has delays
 (see Theorems 5.2.7--5.2.8 of [M02]). %
Even in the positive case, it is often important
 to use the fact that the $\MTIC$ classes are closed under spectral
 factorization; see, e.g., %
 Sections 5.2, 8.4 and 9.1 of [M02] and Section~\ref{sMTIC}
 for details. %

\begin{proof}[Proof of Theorem~\ref{SpF1}:]
(Instead of our standing hypothesis (\ref{shgUstar}), it would suffice
 that $\qD\in\TIC(U,Y)$, $J=J^*\in\BL(Y)$.)
The results are contained in Theorems 5.2.7--5.2.8 of [M02].
Alternatively,
 claim (a) can also be found from 
Lemma 18(ii) of [S97] 
 (or from Lemma 5.2.1(a) of [M02]), and
 claim (b) is essentially given in [GL73].
\end{proof}

In Corollary~\ref{cPosIRE} we shall show roughly that,
 in most results of this section,
 one more equivalent condition is that the corresponding IRE or ARE
 has a nonnegative solution,
 and that the smallest
 such solution provides the desired feedback or factorization.
See also Sections \ref{sMTIC} and~\ref{sARE} for the smoothness 
 (or regularity) of the factorizations and closed-loop systems.

\NotesFor{Section~\ref{sStabOpt}} %
Theorem~\ref{SpF1}(a) is essentially from [RR85] 
 and (b) from [GL73]. %
We presented them and similar spectral factorization results
 (some of which were new) in Chapter~5 of [M02].
Definition~\ref{drcf} and Lemma~\ref{lrcf}
 are from [M02]; except for quasi-coprimeness, they are well known
 (cf.\ [S98a]).
Example~\ref{exanotrf} %
 is due to Olof Staffans %
 and Example~\ref{exaqrcfNotrc} due to Sergei Treil.

We extended (the classical finite-dimensional version of)
 Corollary~\ref{cJointlySD}
 to WPLSs having a smoothing semigroup
 ($\qA B,\qA^*C^*\in\Lstrongloc^1$, $\qD\in\ULR$)
 in Theorem 7.2.4 of [M02], 
 showing that then the dynamic controller in (v)
 does not need an %
 internal loop
 (and giving its constructive formula).
For finite-dimensional $U$ and $Y$,
 Lemma~\ref{lDFCdcf} is implied by Theorem~1 of [S89], 
 as shown in Lemma 7.1.4 of [M02].

All the other results of this section are new
 (except that the claim below (\ref{eS4A:4.4})
   is from Theorem~4.4 of [S98a], as mentioned there).
In particular, even for the cost $\|y\|_2^2+\|u\|_2^2$
 with a bounded output operator ($C\in\BL(H,Y)$),
 as in [FLT88] (with $\gUstar=\dUout$)
 it has not been known that the optimal state-feedback is well-posed,
 not even that it is admissible for the open-loop system. 

The only exception is that we presented Corollary~\ref{cOptExpstab1},
 Theorem \ref{QRCF-Uout}(i)\&(ii)
 and a weak version of Theorem~\ref{PosJcKF}
 already in [M03];
 at that time our proof was based on resolvent REs
 (a generalization of reciprocal REs, see [M03b]).

An thorough treatment of optimizability (and estimatability)
 is given in [WR00], although the concept 
 (under the name FCC) is very old.
Much of Propositions \ref{pExKFstab} and \ref{pExKFstabB1} was presentied in [M02] (e.g., in Theorem 8.4.5).

We defined the concept ``q.r.c.'' in [M02],
 because q.r.c.-SOS-stabilization is the weakest form
 of stabilization that allows one to reduce problems (over $\dUout$)
 to the stable case.
(It is implied by r.c.-stabilization which in turn is
 implied by joint stabilization and detection.)
Later it came to us as a %
 surprise that
 q.r.c.-SOS-stabilization can be applied to all WPLSs
 for which $\dUout$-optimization makes sense %
 (i.e., that (i) implies (iii) in Theorem~\ref{QRCF-Uout}).

\section{Algebraic Riccati Equations (AREs)}\label{sARE} %

Traditionally, one finds the optimal state-feedback by solving an ARE,
 such as (\ref{eAREminBsimpleI}) or (\ref{eAREminBsimpleCC}).
In this section we shall generalize this to weakly regular WPLSs
 (Definition~\ref{dReg}).
Since the equation becomes rather complicated in the general case,
 we shall show how it can be simplified in some special cases,
 the simplest of which is
 the (essentially known) case 
 where $B$ is bounded:
\begin{theorem}[Unique minimum $\IFF$ ARE]\label{INTROcvKLQR0} %
Assume that %
 $D^*JD\gg0$ and $B\in\BL(U,H)$.
Then the following are equivalent:
\begin{itemlist}
  \item[(i)] There is a unique minimizing control over $\dUexp(x_0)$
  for each initial state $x_0\in H$. %

  \item[(ii)] The (algebraic) Riccati equation (ARE)
 \begin{equation}
 \label{eAREminB}
  \left\{\begin{aligned}
     K^*SK &=A^*\Ric+\Ric A+C^*JC,\\
     S &= D^*JD,\\
     K &=-S^{-1}(B^*\Ric+D^*JC),
  \end{aligned}\right.
 \end{equation}
   has a solution $\Ric=\Ric^*\in\BL(H)$ that is
  exponentially stabilizing.

  \item[(iii)] The state-FCC (\ref{eFCCUexp}) holds, %
 and there is $\eps>0$ s.t.\ for all $x_0\in H,\ u_0\in U,\ r\in\R$ we have
\begin{equation}
  \label{eJc1INTRO}
 \!\!\!\!\!\!\!\!\!\!\!\!\!\!\!\!\!\!\!\!\!\!\!\!\!
 (ir-A)x_0=Bu_0
  \THEN \p{Cx_0+Du,J(Cx_0+Du)}_Y\ge\eps \|x_0\|_H^2
\end{equation}
\end{itemlist}

Assume that (ii) has a solution. Then this solution is unique,
 and the minimizing control is given by the state feedback $u(t)=\Kw x(t)$.
The minimal cost equals $\p{x_0,\Ric x_0}_H$.\noproof
\end{theorem} 

(This is a special case of Theorem 1.2.6 of [M02]. 
See Theorem~\ref{Assumptions} for equivalent conditions
 to (\ref{eJc1INTRO}), one of which is that $\PTO\gg0$.) %

By setting $C:=\sbm{I\cr 0},\ D:=\sbm{0\cr I}$, $J=I$,
 we can make the cost equal to (\ref{ecJx2u2})
 and thus obtain Theorem~\ref{IntroTh-Uexp} as a special case.
The reader is invited to carry out %
 the same simplification to most
 AREs presented in the sequel.

Naturally, without regularity the limit $D:=\hqD(+\infty)$
 does not exist
 and hence the ARE (\ref{eAREminB}) becomes meaningless.
Fortunately, all physically relevant WPLSs seem to be regular.
When, e.g., $\PTO\gg0$ and that the FCC holds (and $\vartheta=0$),
 then there exists a $J$-optimal state-feedback pair $\qKF$,
 by Theorem~\ref{PosJcKF}.
In a generalization of Theorem~\ref{INTROcvKLQR0},
 this pair should be given by a state-feedback operator,
 i.e., $\qF$ should be WR and $F=0$
 (or $I-F\in\cG\BL(U)$ so that we can normalize $F$ to zero,
  as in Lemma~\ref{lAllqKF};
  this is necessarily the case if $\qF$ is UR).

However, the optimal state-feedback for a regular WPLS
 is not always regular, by Example 11.5 of [WW97], 
 and it is not known whether this holds for all physically relevant systems.
Before presenting sufficient conditions to prevent this problem,
 we state the most general ARE result %
 where we circumvent the problem by dropping (iii) and reformulating (i)
 of Theorem~\ref{INTROcvKLQR0}.
Note that this new equivalence holds for arbitrary
 (even indefinite and noncoercive) cost functions:
\begin{theorem}[Optimal $K$ $\IFF$ ARE]\label{ARE} %
Let $\ALS$ be WR.
Then the following are equivalent:
\begin{itemlist}
  \item[(i)] There is a $J$-optimal WR state-feedback operator
 $K\in\BL(\Dom(A),U)$;

  \item[(ii)] The algebraic Riccati equation (ARE)
 \begin{subequations}\label{eARE}
   \begin{align}
    \label{eARE1} 
  K^*SK &=A^*\Ric+\Ric A+C^*JC,\\
    \label{eARE2} 
      S &= D^*JD
         + \wlim_{s\to+\infty}    B^*_\w\Ric(s-A)^{-1}B,\\
    \label{eARE3}
      SK &=-(B^*_\w\Ric+D^*JC),
   \end{align}
 \end{subequations}
   has a solution $\Ric=\Ric^*\in\BL(H)$, $S=S^*\in\BL(U)$,
     $K\in\BL(\Dom(A),U)$
  s.t.\ the feedback $u(t)=\Kw x(t)$  %
  is %
  $\gUstar$-stabilizing (Definition~\ref{dARE}).
 \end{itemlist}

Moreover, the following hold:
 \begin{itemlist}
  \item[(a)] Any solution $\Ric$ of (ii) is unique
 (and $\Ric=\qCClL^*J\qCClL$). %

The corresponding operators $K$ in (\ref{eARE}) are exactly the
 WR $J$-optimal state-feedback operators over $\gUstar$.
Thus, the minimizing control is then
 given by the state feedback $u(t)=\Kw x(t)$,
 leading to the cost $\p{x_0,\Ric x_0}$.

  \item[(b)] There is a WR minimizing state-feedback operator over $\gUstar$
  iff (ii) holds and $\cJ(0,u)\ge0$ for all $u\in\gUstar(0)$. 
 \end{itemlist}\itemlistnoproof
\end{theorem} %

(This follows from Lemma~\ref{lAREIRE} and Theorem \ref{IRE}(a2)\&(b).
See Theorem~\ref{IRE} for further properties on the solution.)

This motivates us to call a (unique, by (a))
 solution $\Ric$ of (ii) the $\gUstar$-stabilizing solution of the ARE:
\begin{defin}[ARE]\label{dARE} %
We call $\Ric=\Ric^*\in\BL(H)$ (or $(\Ric,S,K)$)
 a {\em solution of the algebraic Riccati Equation (ARE)}
 (induced by $\ALS$ and $J$) iff the ARE (\ref{eARE})
 is satisfied (with $K\in\BL(\Dom(A),U)$, $S=S^*\in\BL(U)$).

We call $\Ric$ (or $K$ or $(\Ric,S,K)$)
 {\em WR (resp.\ admissible\label{pageadmissible3}, $\gUstar$-stabilizing\label{pagegUstarstabilizing3}, ...)}
 if $\smash{\ssysbm{A\|B\crh K\| 0}}$ generates
 a weakly regular WPLS $\qABKF$
 (resp.\ and $\qKF$ is admissible,
    $\gUstar$-stabilizing, ...).
We call $\qKF$ {\em $\gUstar$-stabilizing}\label{pagegUstarstabilizing2a}
 (with $\Ric$)
 if it is admissible,
 $\qKClL x_0\in\gUstar(x_0)\ \all x_0\in H$,
 and the following condition
 (the {\em RCC}\label{pageRCC}, residual cost condition)
 holds: %
\begin{equation}
  \label{eRCC2} %
 \p{\qBt u+\qAClLt x_0,\Ric\qAClLt x_0}\to0,
     \text{as} t\to+\infty \ \ (\all x_0\in H\ \all u\in\gUstar(0)).
\end{equation}
\end{defin}

The ARE (\ref{eARE})
 is given on $\BL(\Dom(A),\Dom(A)^*) \times \BL(U) \times \BL(\Dom(A),U)$
 (just like (\ref{eAREminB});
 see Section 9.8 of [M02] for details).
If $\ALS$ is $J$-coercive and $\Ric$ is a $\gUstar$-stabilizing solution,
 then $S$ is necessarily invertible
 (and hence then $K$ and $\uopt$ are unique).
We shall show in Theorem \ref{OptIRE}(b1) that
 $\dUexp$-stabilizing means exponentially stabilizing.
See below Theorem~\ref{OptIRE} for more on $\gUstar$-stabilizing and the RCC.

As explained in Definition~\ref{dAdmKF0},
 $K$ being a {\em WR state-feedback operator} for $\ALS$ means that
 $\ssysbm{A\|B\crh C\|D\cr K\|0}$ generate a WR WPLS $\qABCDKF$
 s.t.\ $I-\hqF$ is boundedly invertible on some right half-plane;
 all this is redundant if, e.g., $K\in\BL(H,U)$
 or if $B$ and the ARE are as in Corollary \ref{cPosIRE}(b)\&(c).

Such a $K$ is {\em $J$-optimal}\label{pageJ-optimalK} if %
 the corresponding closed-loop input $\qKClL x_0$
 (i.e., the one given by $u(t):=\Kw x(t)$ a.e.)
 is $J$-optimal for each initial state $x_0\in H$.
As the sections to follow will reveal, 
 the left column of the closed-loop system $\ALSClL$
 is exactly like $\ALSopt$ of Theorem~\ref{ALSopt} 
 except that it is unique iff $S$ is one-to-one. %

For (\ref{eARE2}) and (\ref{eARE3}) to be defined,
 we must have $\Ric[H_B]\sub \Dom(\Bw^*)$,
 where $\Dom(\Bw^*)\label{pageBw}
 :=\{x_0\in H\I \wlim_{s\to+\infty}B^*s(s-A^*)^{-1}x_0$ exists$\}$
 and $H_B\label{pageHB}:=(\alpha-A)^{-1}BU+\Dom(A) \sub H$
 (this set is independent of $\alpha\in\rho(A)$).
By Theorem~\ref{ARE}, this (and the ARE)
 is satisfied by the Riccati operator $\Ric:=\qCClL^*J\qCClL$
 when there is a WR $J$-optimal state feedback $\qF$ 
 (with no feedthrough)
 and $\qD$ is WR.
See Remark 9.8.3 of [M02] for further details on, e.g., 
 $K$ satisfying the above requirements,
 and the rest of Chapter~9 
 for simplifications of the equation and %
 for further results.

Under certain additional smoothness, 
 any unique optimal control is given by regular state feedback,
 and in some cases we even have $S=D^*JD$ and
 $\Bw^*\Ric\in\BL(H,U)$, %
  as in Theorem~\ref{BwARE} below.
For general regular systems, %
 the Riccati operator need not satisfy $\Ric[H]\sub\Dom(\Bw^*)$,
 not even if there is a WR $J$-optimal state-feedback operator
 (see, e.g., Example 9.13.8 of [M02]),
 and we do not know a priori whether an optimal control is even
 well-posed
 (by Example 8.4.13 of [M02];
  cf.\ the difference between ``1.'' and ``2.'' on p.~\pageref{page1.}).

As in the discrete-time case (where $S=D^*JD+B^*\Ric B$),
 the definiteness of 
 the indicator or {\em signature operator} %
 $S$ 
 is inherited from 
 the Popov Toeplitz operator $\PTO$ or $\cJ(0,\cdot)$,
 i.e., from the underlying optimal control problem  %
 (this is not true for $D^*JD$!). See p.~\pageref{pagesignature} for details.
Moreover, the cost
 becomes $\p{y,Jy}=\p{x_0,\Ric x_0}_H+\p{\uc,S\uc}$
 if we add an external input $\uc\in\Lc^2(\R_+;U)$
 to the optimally controlled closed-loop system.
(In fact, this paragraph is true even if $D$ does not exist
 (i.e., if $\qD$ is irregular); see Section~\ref{sIREmore} for details.)

Next we show that the RCC is not redundant for $\dUout$
 (otherwise $\Ric=2$ would be $\dUout$-stabilizing,
   hence $K=-2$ would be $J$-optimal over $\dUout$):
\begin{example}{(RCC; exp.\ stabilizing cannot be q.r.c.).}\label{exaUoutUexpstab} %
Let $\ALS=
 {\ssyspm{1\|1\crh 0\|1}}$,
 $J=1$.
Obviously, $\cJ(x_0,u)=\|u\|_2^2$, hence $K=0$
 is the unique $J$-optimal state-feedback operator over $\dUout$.
The $\dUout$-stabilizing solution 
 $\Ric=0$ (``no feedback needed to minimize $\|y\|_2^2$ over $\dUout$''
  because $\ALS$ is already output-stable)
 of the ARE $(-\Ric)^2=1\Ric+\Ric1+0$, $S=1$, $K=-\Ric$
 differs from the $\dUexp$-stabilizing solution $\Ric=2$
 (``feedback $u(t)=-2x(t)$ (leading to cost $2|x_0|^2$)
   needed to minimize $\|y\|_2^2$ over $\dUexp$'').

Trivially, $1=1\cdot1^{-1}$ is a q.r.c.f.\ of $\qD=1$.
A coprime stabilization (such as the zero feedback above)
 means (in the finite-dimensional case) that
 ``$\hqN$ and $\hqM$ have no common zeros on $\cRHP$'',
 i.e., that one stabilizes as little as possible
 (only the poles of $\hqD$).
The semigroup $A=1$ has more poles (namely $s=1$)
 than the transfer function $\hqD=1$ (which has none),
 hence one must introduce additional zeros to $\hqM$
 (and hence to $\hqN=\hqDClL=\hqD\hqM$ too: $\hqM(1)=0=\hqN(1)$)
 to stabilize the semigroup too ($\dUexp$ vs.\ $\dUout$).
Thus, no exponentially stabilizing state-feedback for the
 system $\ssyspm{1\|1\crh0\|1}$ can be q.r.c.-stabilizing.
\end{example}

(See Example 9.13.2 of [M02] for further details.)

If $B$ is not maximally unbounded, then any state-feedback is UR, 
 hence then Theorem~\ref{PosJcKF} implies that,
 for any positively $J$-coercive cost function, %
 the FCC holds iff there is a UR minimizing state-feedback operator
 (see (v)):
\begin{lemma}\label{lBnotmax} %
Assume that $B$ is {\em not maximally unbounded}\label{pageBnotmax}, i.e., that
 there are $M,R,\eps>0$ s.t.\
 $\|(s-A)^{-1}B\|_{\BL(U,H)}\le Ms^{-\frac12-\eps}$ for $s\in(R,\infty)$.
Then $\ALS$ is uniformly regular (UR).

Consequently, the following are equivalent
 \begin{itemlist}
  \item[(i)] There is a $J$-optimal state-feedback pair over $\gUstar$.
  \item[(ii)] There is a UR $J$-optimal state-feedback operator over $\gUstar$.
  \item[(iii)] The IRE
 has a $\gUstar$-stabilizing solution.\footnote{See Definition~\ref{dIRE}.}
  \item[(iv)] The ARE has a $\gUstar$-stabilizing solution.
 \end{itemlist}

Moreover, the $\wlim$ in the ARE converges uniformly to zero
 and the optimal UR state-feedback is given by $u(t)=\Kw x(t)$ a.e.,
 with cost $\p{x_0,\Ric x_0}$.

For positively $J$-coercive problems 
 (having $\vartheta=0$),
 a fifth equivalent condition is %
 \begin{itemlist}
 \item[(v)] The FCC holds.
 \end{itemlist}
\end{lemma}

(The proof is given on p.~\pageref{pageproof-lBnotmax}.)

(The inequality can always be established for $\eps=0$;
 for $\eps=1/2$ it holds iff $B$ is bounded
 (in which case (i)--(v) are equivalent for any
  $J$-coercive problems
  and the $\wlim$ condition becomes redundant, by Theorem~\ref{BwARE}).
A sufficient condition is that $A$ is analytic
 and $(s_0-A)^{-\beta}B$ is bounded for some $\beta<1/2$, $s_0\in\rho(A)$,
 by Lemma 9.4.2(k) of [M02].)

It follows that in  the results of Section~\ref{sStabOpt},
 when $B$ is not maximally unbounded,
 the stabilizability condition (or FCC) is equivalent 
 to the solvability of the corresponding ARE,
 whose solution provides 
 the desired (UR) stabilizing state-feedback operator,
 as explained in Corollary \ref{cPosIRE}(b)\&(c).

We now apply the above equivalence of (i)--(v) to a detectable LQR problem,
 so that ``$\gUstar$-stabilizing'' can be ignored
 (as long as $\Ric\ge0$):
\begin{cor}[LQR,\ $\pmbold{B}$]\label{cBARE} %
{\bf (a)}
Assume that $B$ is not maximally unbounded,
 and let $R,T\gg0$, $Q\ge0$. %
Then, for each initial state $x_0\in H$, there is a control
 $u\in\L^2(\R_+;U)$ s.t.\ the cost
\begin{equation}
  \cJ(x_0,u):=\int_0^\infty \left(\p{y,Qy}_Y+\p{x,Tx}_H+\p{u,Ru}_U\right)\,dm,
\end{equation}
 is finite iff the ARE
 \begin{subequations}\label{eBARE}
   \begin{align}
    \label{eBARE1} 
  K^*SK &=A^*\Ric+\Ric A+C^*QC+ T,\\
    \label{eBARE2} 
      S &= D^*QD + R,\\
    \label{eBARE3}
      SK &=-(B^*_\w\Ric+D^*QC),
   \end{align}
 \end{subequations}
 has a nonnegative solution $\Ric\in\BL(H)$
 satisfying $\lim_{s\to+\infty}    B^*_\w\Ric(s-A)^{-1}B=0$. %

Assume that $(\Ric,S,K)$ is such a solution.
Then $K$ is the unique uniformly regular $J$-optimal state-feedback operator,
 and it is exponentially stabilizing
 and leads to the minimal cost, which equals $\p{x_0,\Ric x_0}$.

{\bf (b)} %
Instead of $T\gg0$, assume that $T\ge0$ and $Q\gg0$.
Then everything in (a) still holds except that
 ``Then $K$ ...'' holds for the smallest nonnegative solution $\Ric$ only
 (which exists whenever there are any solutions,
  equivalently, whenever the FCC holds),
 $K$ is SOS-stabilizing
 and $\qN,\qM$ become q.r.c. %
\end{cor}

(The proof is given on p.~\pageref{pageproof-cBARE}.
Note that Theorems \ref{IntroTh-Uexp} and~\ref{IntroTh-Uout}
 are special cases of this
 and that $\Ric$ is unique in (a), being the $J$-optimal cost operator.)

The above FCC ``$\all x_0\ \ex u\in\L^2$ s.t.\ $\cJ<\infty$''
 is obviously equivalent to the state-FCC (\ref{eFCCUexp}) in (a)
 (to the output-FCC in (b) if $T=0$).

When can one remove the above $\wlim$ condition?
If $H_B\sub Z\sub H$ continuously,
 and $Z$ is a Banach space with $(s-A)^{-1}B\to0$
 in $Z$ as $s\to+\infty$ (this is true for $Z=H$),
 then $\Ric [Z]\sub\Dom(\Bw^*)$ is a sufficient condition;
 this also applies to indefinite problems.
In Section 9.4 of [M02] we give sufficient conditions in the
 case of an analytic semigroups; below we study the case $Z=H$.

Under certain assumptions, the ARE becomes equivalent to
 the following conditions (the {\em \BwARE}):\label{pageBwARE}
 $\Ric=\Ric^*\in\BL(H),\ \Ric[H]\sub\Dom(B_\w^*)$,
 and
 \begin{equation}
   \label{eBwARE} 
 (B^*_\w\Ric+D^*JC)^*(D^*JD)^{-1}(B^*_\w\Ric+D^*JC) = A^*\Ric+\Ric A+C^*JC. 
 \end{equation}
Moreover, then a unique $J$-optimal control is necessarily 
 given by an ULR state-feedback operator:
\begin{theorem}[\BwARE\ \pmbold{$\IFF J$}-optimal]\label{BwARE} %
Assume that at least one of (1.)--(4.) below holds:
    \begin{itemlist} %
      \item[(1.)] $B$ is bounded (i.e., $B\in\BL(U,H)$);

      \item[(2.)] $\qA B\in\L^1([0,1];\BL(U,H))$ and $C\in\BL(H,Y)$;
      \item[(3.)] $\qA Bu_0\in\L^2([0,1];H)$
       and $\Cw\qA Bu_0\in\L^2([0,1];Y)$ for all $u_0\in U$;
      \item[(4.)] {\bf (Stable case)} 
  $C\in\BL(H,Y)$, $D^*JC=0$, %
 $\qD\in\BL(U,Y)+\BL(U,\L^1(\R_+;Y))*$, and
 $\gUstar=\dUexp$ and $\qB\tau$ is stable
 (or $\gUstar=\dUout$ and $\qC$ is stable).
    \end{itemlist}
Then $\qD$ is ULR.
If $D^*JD\in\cG\BL(U)$, then the following are equivalent:
\begin{itemlist}
  \item[(i)] There is a unique $J$-optimal control over $\gUstar(x_0)$
 for each $x_0\in H$.
  \item[(ii)] There is a $J$-optimal state-feedback pair over $\gUstar$.
  \item[(iii)] The IRE or the ARE or the \BwARE\
 has a $\gUstar$-stabilizing solution.
\end{itemlist}
If (iii) holds, then the IRE, ARE and \BwARE\ have the same
 $\gUstar$-stabilizing solution (with $S=D^*JD$),
 hence then Theorem \ref{IRE} applies;
 moreover, then $K:=-(D^*JD)^{-1}(\Bw^*\Ric+D^*JC)$
 is the unique ULR $J$-optimal state-feedback operator.\noproof
\end{theorem}

(This follows from Theorems 9.2.9 and 9.2.3 of [M02];
 in the same section also
  further alternatives for (1.)--(4.) 
 and numerous further results are given.
As an example, if
 $\qA B u_0\in\L^1([0,1];H)\ \all u_0\in U$,
 then the state-FCC holds iff %
 there is $\Ric\ge0$ s.t.\ $\Ric [H]\sub\Dom(\Bw^*)$
 and $ (B^*_\w\Ric)^* B^*_\w\Ric = A^*\Ric+\Ric A+I$.
Moreover, then $K:=-\Bw^*\Ric\in\BL(H,U)$ is ULR and exponentially
 stabilizing.)

In contrast to Lemma~\ref{lBnotmax}, 
 we note that here
 {\bf (a)} we do not need positive $J$-coercivity
 to guarantee the existence of $K$
 (although $J$-coercivity and the FCC is sufficient for (i),
  by Theorem~\ref{J-coercive}),
 in particular, also the indefinite case is covered;
 {\bf (b)} the condition  $\wlim \Bw^*\Ric(s-A)^{-1}B=0$
 is replaced by the stronger assumption that
 $\Ric[H]\sub\Dom(\Bw^*)$,
 or equivalently, $\wlim_{s\to+\infty}B^*s(s-A)^{-1}\Ric x_0\ \all x_0\in H$
 must exist for all $x_0\in H$. %

\NotesFor{Section~\ref{sARE}} %
The necessity of equations (\ref{eARE}) for 
 SR stable $J$-coercive problems over $\dUout$
 was shown by Olof Staffans [S98b] (see Remark~5.2 of [S98c]). 
At the same time, (\ref{eARE1}) and (\ref{eARE3})
 were discovered independently by Martin Weiss and George Weiss [WW97]. %
In the same setting, we proved the sufficiency in [M97].

The above (new) frequency-domain proof for Theorem~\ref{ARE}
 is significantly shorter and simpler
 than our original time-domain proof of [M02]
 (Section 9.11).
However, the latter, technically more demanding 
 but closer to
 finite-dimensional ones, %
 can more easily be generalized to 
 finite-horizon, time-variant and/or nonlinear settings.

A number of further results, special cases and notes
 are given in Chapters 9--10 of [M02]
 (see, e.g., Section 10.1 for LQR results),
 including Riccati inequalities and
 relations to spectral and coprime factorizations.
Corresponding results on discrete-time AREs are presented
 in Chapter~14 of [M02].
See Section 9.13 of [M02] for examples where, e.g., 
 $\qD$ and $\qF$ are regular but $\Ric[H]\not\sub\Dom(\Bw^*)$
 (although  $\Ric[H_B]\sub\Dom(\Bw^*)$)
 or where $\qD$ is very regular but $\qF$ not regular at all.

Under mild assumptions, a minimizing state-feedback operator
 also solves the  ``$\H^2$ problem''
 (see Section 10.4 of [M02] for definition and proofs).

The fact that $\ALS$ is UR when $B$ is not maximally unbounded
 is due to G. Weiss [WC99], who applied it to the stable LQR problem. %
For exponentially detectable %
 systems with analytic semigroups,
 the results in [LT00]
 allow for significantly more unbounded $B$'s than
 Corollary~\ref{cBARE} does
 (they have the corresponding indefinite result too,
 both for highly coercive cost functions). 
However, there do not seem to exist similar results for non-analytic 
 semigroups, and Lemma~\ref{lBnotmax} covers more general cost functions.
Further optimization and ARE results for as general cost functions
 can be found in [LR95] and [IOW99], for finite-dimensional systems.

For Pritchard--Salamon\label{pagePS} systems\footnote{P--S systems
 are exactly the WPLS with a bounded input operator ($B$)
 that can be written as WPLS with a bounded output operator ($C$)
 by changing the state space,
 as shown in [M02], Lemma 6.9.4.} that are smooth, %
 most of Theorem~\ref{INTROcvKLQR0} was proved in [vK93], Theorem 3.10.
Theorem~\ref{BwARE}(1.) extends those results.
See also Theorem~\ref{Assumptions}.

\section{Integral Riccati equations (IREs) and optimal control}\label{sIRE} %

By Theorem~\ref{ALSopt}, a unique optimal control can always
 be given in WPLS form (i.e., as a ``generalized state feedback'',
 see Definition \ref{dWPLSform}).
Traditionally, this control is determined by finding the stabilizing 
 solution of the 
 corresponding (infinitesimal) algebraic Riccati equation (ARE);
 this was illustrated in the previous section.

However, without significant regularity assumptions, such as those above,
 the feedthrough operator (often normalized to $F=0$, as above)
 of the optimal state-feedback loop
  (``$u(t)=\Kw x(t)+Fu(t)$ for a.e.\ $t\ge0$'')
 need not exist.
In fact, sometimes this loop is even ill-posed!
Nevertheless, we can use certain integral Riccati equations
 (IREs) to characterize the optimal control. 

In Theorem~\ref{SIRE} we shall show that a unique optimal control
 is the one given by the $\gUstar$-stabilizing solution of the
 \SIRE\ (or \hSIRE). 
The ``generalized state-feedback loop'' ($u(t)=\Kw x(t)$ a.e.\ $t\ge0$)
 of this control %
 is well posed iff the IRE %
 has a $\gUstar$-stabilizing solution
 (equivalently, any of (i)--(vi) of Theorem~\ref{GenSpF} holds).
We reduce this condition to a stable spectral factorization problem
 (Theorem~\ref{GenSpF}(iv)).

These results form a direct generalization of the classical
 (algebraic) RE theory to an extent that cannot be covered by the
 (standard) ARE.
In addition, they will be used to prove the stabilization and
 factorization results of Section~\ref{sStabOpt}.

We start by noting that a control $\qKopt$ in WPLS form is optimal
 and $\Ric$ is the optimal cost operator
 iff $\Ric,\qKopt$ satisfy the \SIRE:
\begin{theorem}[\boldedSIRE\ \&\ \boldedhSIRE]\label{SIRE} %
Let $\qKopt$ be a control in WPLS form for $\ALS$,
 and let $\Ric=\Ric^*\in\BL(H)$, $\omega\ge\max\{\omega_A,\omega_{\Aopt}\}$.

Then $\qKopt x_0$ is $J$-optimal
 and $\Ric$ is its Riccati operator $\qCopt^*J\qCopt$
 iff $\Ric,\qKopt$ is a $\gUstar$-stabilizing solution
 of the following equations ({\em the \SIRE})\label{pageSIRE}
 for all $t\ge0$:
 \begin{subequations}\label{eSIRE}
   \begin{align}
    \label{eS2KSK=} %
  \qKoptt^*\qSt\qKoptt&={\qA^t}^*\Ric\qA^t-\Ric+\qCt^*J\qCt,\\ %
    \label{eS2XSX=}
  \qSt&:= \qDt^*  J\qDt
     +\qBt^*\Ric\qBt,\\
    \label{eS2XSK=}
  \qSt\qKoptt&=-\left(\qDt^*J\qCt+\qBt^*\Ric\qAt\right)
   \end{align}
 \end{subequations}
  
Moreover, equations (\ref{eSIRE}) hold iff the following equations
 ({\em the \hSIRE}) hold for some
 (equivalently, all) $s,z\in\C_{\omega}^+$: %
 \begin{subequations}\label{ehSIRE}
   \begin{align}
    \label{eS2hKSK=} 
  \hqKopt(s)^*\hqS(s,z)\hqKopt(z)
 &= (s-A)^{-*}\left(A^*\Ric+\Ric A+C^*JC\right)(z-A)^{-1},\\
    \label{eS2hXSX=} 
   \hqS(s,z)
 &:= \hqD(s)^*J\hqD(z) + (z+\bar s)B^*(s-A)^{-*}\Ric (z-A)^{-1}B,\\
    \label{eS2hXSK=}
  \hqS(s,z)\hqKopt(z)
 &= - \hqD(s)^*JC(z-A)^{-1}
  - B^*(s-A)^{-*}\Ric(s^*+A)(z-A)^{-1}.
   \end{align}
 \end{subequations}\itemlistnoproof
\end{theorem}

(This follows from Lemma~\ref{lhSIRE} and Theorem~\ref{OptIRE}.)

By {\em $\gUstar$-stabilizing}\label{pagegUstarstab} 
 we mean that $\qKopt x_0\in\gUstar(x_0)\ \all x_0\in H$
 and the {\em RCC} %
 (\ref{eRCC2}) holds (with $\qAopt$ in place of $\qAClL$). %
By Theorem~\ref{OptIRE}(b1),
 {\em $\dUexp$-stabilizing} is equivalent to
 ``$\ALSopt$ is exponentially stable''
 (equivalently, to $\qAopt x_0\in\L^2(\R_+;H)\ \all x_0\in H$).

Note that we have $\qKopt x_0\in\L^2_\omega(\R_+;U)$
 for some $\omega\in\R$ (Definition~\ref{dWPLS0}),
 hence some (unique) holomorphic 
$\shat{\smash{\rlap{$\qKopt$}}\phantom{ddd}} %
 :\C_\omega^+\to\BL(H,U)$
 satisfies $\smash{\shat{\qKopt x_0}}=\smash{\hqKopt} x_0$ on $\C_\omega^+$
 for all $x_0\in H$.

For a fixed $t>0$, the \SIRE\ (\ref{eSIRE}) 
 coincides with the (discrete-time) algebraic Riccati equation
 for the discretized system $\smash{\sbm{\qAt&\qBt\cr \qCt&\qDt}}$; %
 this fact provides an alternative proof for the theorem
 (see Theorem 14.1.6 and Proposition 9.8.7 of [M02];
  it also follows that ``all $t\ge0$'' is equivalent to ``some $t>0$'').

Thus, given the FCC and $J$-coercivity ($\PTO\in\cG\BL$),
 there is a unique optimal control, it is given in the WPLS form
 (i.e., as generalized state feedback,
  by Theorems \ref{J-coercive} and~\ref{ALSopt}),
 and it satisfies the \SIRE\
 and the \hSIRE.
But is it given by (well-posed) state-feedback?

The answer is ``not always'' (unless $\PTO\gg0$ or the system is rather smooth),
 by Example 8.4.13 of [M02]. %
The answer is positive iff the spectral factorization problem (iv) below
 has a solution, 
 equivalently, 
 iff the (optimally truncated Popov Toeplitz)
 operator $\qSt$ can be factorized as $\qXt^*S\qXt$,
 again equivalently,
 iff $\hqS(s,s)$ can be factorized as $\hqX(s)^*S\hqX(s)$:
\begin{theorem}[$\pmbold{\hqS=\hqX^*S\hqX\ \IFF\ \ex\qKF}$]\label{GenSpF} %
Assume that there is a unique $J$-optimal control for each $x_0\in H$.
Define $\Ric$ and $\qSt$ as in Theorem~\ref{SIRE}.
Then the following are equivalent:
\begin{itemlist}
  \item[(i)] There is a $J$-optimal state-feedback pair $\qKF$. %

  \item[(ii)] There are $\hqX\in\cG\H^\infty_\infty(U)$, $S\in\BL(U)$ s.t.\ 
 \begin{equation}
   \label{eXGenSpF}
  \hqX(s)^*S\hqX(s)=\hqD(s)^*J\hqD(s)+2\re sB^*(s-A)^{-*}\Ric(s-A)^{-1}B 
 \end{equation} %
  on some right half-plane
 (equivalently, on a strip $\C_\alpha^+\pois\C_\beta^+$,
   where $\omega_A\le\alpha<\beta<\infty$).

  \item[(iii)] There are $\qX\in\cG\TIC_\infty(U),\ S\in\BL(U)$
 that satisfy $\qXt^*S\qXt=\qSt\ \all t>0$.
  \item[(iv)] There are $\qXp\in\cG\TIC(U)$, $S=S^*\in\BL(U)$  %
 s.t.\ $S$ is one-to-one and $\qXp^*S\qXp=\qDp^*J_+\qDp$
 for some %
 $\alpha>\max\{0,\omega_A\}$. %

Here $\qD_+:=\sbm{\efn^{-\alpha\cdot}\qD\efn^{\alpha\cdot}\cr
 \efn^{-\alpha\cdot}\qB\tau\efn^{\alpha\cdot}}\in\TIC_{-\del}$ for some $\del>0$
 and $J_+:=\sbm{J&0\cr 0&2\alpha\Ric}$.

  \item[(v)] The IRE (\ref{eIRE}) has a $\gUstar$-stabilizing solution.

  \item[(vi)] The \hIRE\ (\ref{ehIRE}) has a $\gUstar$-stabilizing solution.

  \item[(vii)] There is an admissible state-feedback pair $\qKF$
 s.t.\ $\qKopt=\qKClL$. %

\end{itemlist}

Moreover, the following hold:
 \begin{itemlist} 
  \item[(a)] The solutions of (i)--(vii) are equal 
 (with $\qF=I-\qX$, $\qK=\qX\qKopt$,
  $\qXp:=\efn^{-\alpha\cdot}\qX\efn^{\alpha\cdot}$).

  \item[(b)] Given one solution $(\qX,S)$,
 all solutions are given by $(E\qX,E^{-*}SE^{-1})\ (E\in\cG\BL(U))$,
 and the operator $S$ is one-to-one.
If $\PTO$ is invertible, then so is $S$.
Also the rest of Theorem~\ref{IRE} applies.
\end{itemlist}\itemlistnoproof
\end{theorem} %

(The proof is given by Lemma~\ref{lGenSpFproof}.
See (\ref{eqK+}) for $\qK$ (and $\qKClL$) in terms of
 $\qX$, $S$, $\Ric$, %
 $\ALS$ and $J$.)
If $\gUstar$ equals $\dUout$ or $\dUexp$, then
 one more equivalent condition is that $\qD$ has a
 ``$J$-optimal factorization'' %
 (a generalization of spectral factorization),
 by Theorem 9.14.3 of [M02].
Another equivalent condition is that
  $s\hqK(s)B_-,s\hqKClL(s) B_-\in\H^\infty_\infty$
 (as functions $s\to\BL(U)$), as will be shown in [M03b];
 here $B_-:=A^{-1}B\in\BL(U,H)$
 and $\hqK(s)=K(s-A)^{-1}$, where $K$ is determined by
 the so called reciprocal ARE.

The spectral factorization condition (iv) seems independent of $\gUstar$.
Of course, that cannot be the case: the information on $\gUstar$
 is carried by $\Ric$.

When $\PTO\gg0$ (and, e.g., $\gUstar=\dUexp$ or $\gUstar=\dUout$),
 we can show that condition (iv) can be satisfied whenever the FCC holds
 (p.~\pageref{pageproof-PosJcKF}).
This will establish Theorem~\ref{PosJcKF}
 and the other results presented Section~\ref{sStabOpt}.

The {\em signature operator}\label{pagesignature} %
 (indicator) $S$ has obviously
 the same definiteness as $\qSt$,
 which in turn inherits (a restriction of) that of
 the Popov Toeplitz operator $\PTO$,
 as noted in Lemma~\ref{lqSt=PTO.Pt}. %
In particular, $\PTO\ge0$ (resp.\ $>0$, $\gg0$, $\in\cG\BL$, is one-to-one) $\THEN$
 $S\ge0$ (resp.\ $>0$, $\gg0$, $\in\cG\BL$, is one-to-one). 
In fact, at least if $\gUstar=\dUexp$
 (or $\gUstar=\dUout$ and $\qN,\qM$ are q.r.c.), %
 then  also the converse implications hold
 and $\qX u\in\L^2 \lland 
   \cJ(0,u)=\p{\qX u,S\qX u}=\p{u,\qSt u}\ \all u\in\gUstar(0)$.
See Theorem 9.9.1(f2)\&(h)\&(k), Lemma 9.10.3, %
 Theorem 8.4.5(d) and pp.\ 482\&387 of [M02] for details.

By Theorem~\ref{GenSpF}, an admissible state-feedback pair $\qKF$ for $\ALS$
 is $J$-optimal iff it is a $\gUstar$-stabilizing solution 
 of the IRE (with some $\Ric,S$):
\begin{defin}[A $\gUstar$-stabilizing solution of the IRE (or \hIRE)
 (\pmbold{$(\Ric,S,\qKF),\qX,\qM,\qN,\ALSClL$})]\label{dIRE} %
We call $\Ric$ (or $(\Ric,S,\qKF)$)
 a {\em solution of the Integral Riccati Equation (IRE)}\label{pageIRE}
\index{IRE|emph}\index{IRE|emph}\index{solution!of the IRE}%
\index{regular!P@$\Ric$}\index{stabilizing!P@$\Ric$}%
\index{admissible!P@$\Ric$}\index{J-optimal@$J$-optimal!P@$\Ric$}\index{stable!P@$\Ric$}
 (induced by $\ALS$ and $J$)
 iff the IRE
 \begin{subequations}\label{eIRE}
   \begin{align}
    \label{eKSK=} %
  \qKt^*S\qKt&={\qA^t}^*\Ric\qA^t-\Ric+\qCt^*J\qCt,\\ %
    \label{eXSX=}
  \qXt^* S\qXt&= \qDt^*  J\qDt
     +\qBt^*\Ric\qBt,\\
    \label{eXSK=}
  \qXt^*S\qKt&=-\left(\qDt^*J\qCt+\qBt^*\Ric\qAt\right)
   \end{align}
 \end{subequations}
 (here $\qX:=I-\qF$)
  is satisfied for all $t>0$, %
 and $\Ric =\Ric^*\in\BL(H)$, $S=S^*\in\BL(U)$,
 $\qK\in\BL(H,\L^2_\loc(\R_+;U))$, and $\qF\in\TIC_\infty(U)$.

We call $\Ric$ {\em admissible}\label{pageadmissible2} or {\em $\gUstar$-stabilizing}\label{pagegUstarstabilizing2}
 if $\qKF$ is (see Definition~\ref{dARE}).

{\em Solutions of the \hIRE} are defined in the same way, except %
 that instead of (\ref{eIRE}) we require 
 that $\qABKF$ is a WPLS and that %
 the following \hIRE\ is
 satisfied for some $s=z\in\C_{\omega_A}^+$:
 \begin{subequations}\label{ehIRE}
   \begin{align}
    \label{ehKSK=} 
  K^*SK &=A^*\Ric+\Ric A+C^*JC,\\
    \label{ehXSX=} 
   \hqX(s)^*S\hqX(z)
 &= \hqD(s)^*J\hqD(z) + (z+\bar s)B^*(s-A)^{-*}\Ric (z-A)^{-1}B,\\
    \label{ehXSK=}
  \hqX(s)^*S K(z-A)^{-1}  
 &=- \hqD(s)^*JC(z-A)^{-1}
  - B^*(s-A)^{-*}\Ric(s^*+A)(z-A)^{-1}.
   \end{align}
 \end{subequations}
\end{defin}

(By Lemma~\ref{lhIRE}, this implies that (\ref{ehIRE})
 actually holds for all $s,z\in\rho(A)$.
Note from the definition that we only study the self-adjoint solutions.)

As in Definition~\ref{dAdmKF0},
 for admissible $\Ric$,  we denote
 the corresponding closed-loop system by $\ALSClL$ %
 and set $\qX:=I-\qF\in\TIC_\infty(U)$, $\qM:=\qX^{-1}\in\cG\TIC_\infty(U)$,
 $\qN:=\qDClL:=\qD\qM\in\TIC_\infty(U,Y)$.

It suffices to require (\ref{eIRE}) for some $t>0$
 in Theorem~\ref{GenSpF}(v),
 by the comments below Theorem~\ref{OptIRE}.

From Theorem~\ref{ARE} (or Definition~\ref{dARE})
 we observe that any admissible
 (resp.\ $\gUstar$-stabilizing) solution of the ARE
 is an admissible
 (resp.\ $\gUstar$-stabilizing) solution of the IRE
 (the converse holds iff $\qD,\qF$ are WR and $F=0$).

If $B$ is bounded, $C=\sbm{\tC\cr 0},\ D=\sbm{0\cr I},\ J=\sbm{I&0\cr 0&I}$
 (hence $\cJ(x_0,u)=\|u\|_2^2+\|\tC x\|_2^2$),
 then, by (\ref{eBARE}),
 we get $S=I,\ K=-B^*\Ric$,
 hence then the ARE reduces to $\Ric BB^*\Ric=A^*\Ric+\Ric A+\tC^*\tC$.
This ARE is equivalent to (\ref{eKSK=}), which becomes
\begin{equation}
  \label{eIRE-Gibson}
\Ric x_0 = \qAt^*\Ric \qAt x_0
   + \int_0^t {\qA^s}^*(\tC^*\tC-\Ric BB^*\Ric)\qA^s x_0\,ds
    \ \all x_0\in H,
\end{equation}
 familiar from classical results,
 such as equation (4.26) of [G79] %
 (take $s=0,\ D=\tC^*\tC,\ Q=I$).

If $B$ is bounded and $D^*JD$ is invertible,
 then all the above REs (and the ones
 presented in Theorem~\ref{OptIRE}) are equivalent,
 and it suffices to verify the equations
 (since every solution generates an admissible state-feedback pair $\qKF$):
\begin{lemma}[Bounded $B$: ARE $\IFF$ \SIRE]\label{lBboundedARE} %
Assume that $B\in\BL(U,H)$.

{\bf (a)}
If $(\Ric,S,K)$ is a WR solution of the ARE, then it is admissible
 and ULR and $(\Ric,\qKClL)$ solve the \SIRE, \hSIRE, \hIRE\ and IRE.
If $D^*JD\in\cG\BL(U)$, then
 any solution of the ARE is WR.

{\bf (b)}
Conversely, if $\qKopt$ is a control in WPLS form 
 and $(\Ric,\qKopt)$ solve the \SIRE\ or the \hSIRE,
 then $\qKopt=\qKClL$ for some $K$ which is as in (a).
\end{lemma}

(The proof is given on p.~\pageref{pageproof-lBboundedARE}.)

Most results of Section~\ref{sStabOpt} provide equivalent
 conditions for the existence of a certain kind of
 stabilizing state-feedback pair $\qKF$.
By (c), one more equivalent condition is that the IRE
 has a nonnegative admissible solution:
\begin{cor}[$\Ric\ge0 \IFF $minimizing]\label{cPosIRE} %
{\bf (a)}
In any of the results mentioned in Corollary~\ref{cMTICinftyKF+}(a)[(b)]
 below,
 one more equivalent condition is that the corresponding IRE[s]
 (equivalently, \SIRE[s])
 has a $\gUstar$-stabilizing solution.

{\bf (b)}
If $B$ [and $C^*$] is not maximally unbounded, then a further equivalent
 condition is that the corresponding ARE[s]
 has a $\gUstar$-stabilizing solution.
Moreover, $S=D^*JD\gg0$,
 and we can have $\qKF$ generated by $\bsysbm{K\|0}$,
 where $K$ is from the ARE
 [and $\qHG$ by $\ssysbm{\Hgen\crh0}$, $\Hgen=\tK^*$, where $(\tRic,\tS,\tK)$
  is the solution of the dual ARE].

{\bf (c)} Exclude Theorem~\ref{PosJcKF} from the results mentioned above. %
Then the existence of an admissible
 nonnegative solution[s] of the IRE[s] 
 (or any nonnegative solution[s] of the ARE[s] in (b)) 
 is another equivalent condition.
Moreover, any nonnegative UR solution of the ARE is admissible;
 if $\dim U<\infty$ [$\dim Y<\infty$ for the dual IREs]
 or $B$ [and $C^*$] is not maximally unbounded,
 then any nonnegative solution of the ARE is admissible.
Any admissible solution is SOS-stabilizing.

(A nonnegative admissible solution of the state-IRE (or of the state-ARE)
 is $\dUexp$-stabilizing (hence unique);
 the $\dUout$-stabilizing solution of the output-IRE
 (or output-ARE), if any,
 is the smallest admissible nonnegative solution.)
\end{cor}

Naturally, ``$C^*$ not maximally unbounded''
 means that
 $\|(s-A^*)^{-1}C^*\|_{\BL(Y,H)}\le Ms^{-\frac12-\eps}$ for $s\in(R,\infty)$
 and some $R,M<\infty$.
See Corollary \ref{cMTICinftyKF+}(c) and Remark~\ref{rcA}
 for (b) and (c) under alternative assumptions on $\ALS$.
Similar claims also hold for the \BwARE,
 since any of its solutions is an admissible solution of the IRE,
 by Proposition 9.2.7 of [M02].
If $\qD$ is UR, then a solution of the ARE is UR
 iff the limit in the ARE converges uniformly
 (this is the case in most applications), by Lemma 9.11.5(e) of [M02].

Here $\gUstar$ and the IRE
 should be the same as in the proof of that result
 (hence $\gUstar=\dUout$ and $J=I$
   except possibly for Theorem~\ref{PosJcKF};
  moreover, $\dUout=\dUexp$ for, e.g., Corollary~\ref{cOptExpstab1}).
Thus, the IRE or ARE is determined by the system
 (sometimes ``$\tALS$'' instead of $\ALS$)
 and the $J$ %
 used in the proof.
Note that, e.g., in the proof of Corollary~\ref{cOptExpstab1}
 we have $\gUstar=\dUexp$ and $\qCD=\ssysbm{\qA\|\qB\tau\cr 0\|I}$
  (i.e., $\bsyspm{C\|D}=\ssyspm{I\| 0\cr 0\|I}$),
 hence the ARE in (b) becomes %
 the {\em state-FCC ARE}
 \begin{equation}
   \label{eUexpAREBnotmax}
   (\Bw^*\Ric)^*\Bw^*\Ric = A^*\Ric + \Ric A + I
 \end{equation}
 (and $\lim_{s\to+\infty}\Bw^*\Ric(s-A)^{-1}B=0$, see Lemma~\ref{lBnotmax})
 and $K=-\Bw^*\Ric$ (and $S=I$);
 thus, there is a nonnegative solution $\Ric\in\BL(H)$
 to this problem iff the state-FCC is satisfied.

Note that the results mentioned in Corollary~\ref{cMTICinftyKF+}(b)
 [the above text in brackets corresponds to those results;
 such text must all be included or all excluded]
 correspond to two IREs (or AREs) each;
 e.g., Corollary~\ref{cJointlySD}(i)
 to (\ref{eUexpAREBnotmax}) and to %
 the {\em dual state-FCC ARE} (or filter ARE)
 $(\Cw\tRic)^*\Cw\tRic = A\tRic + \tRic A + I$,
 whose unique nonnegative solution $\tRic\in\BL(H)$
 provides $\Hgen=\tK^*=(-I^{-1}\Cw\tRic)^*=-(\Cw\tRic)^*$.
Replace $B$ by $B_1$ for Corollary~\ref{cExpStabB1}
 or Remark~\ref{rHooAB1AC2}
 (in the latter, use also the dual with $C_2$ in place of $C$).

In Theorem~\ref{QRCF-Uout}(i) the ARE becomes
 the {\em output-FCC ARE}
 \begin{equation}
   \label{eoutput-FCC-ARE}
   K^*SK=A^*\Ric+\Ric A+C^*C
 \end{equation}
 with $K=-\Bw^*\Ric$, $S=D^*D$
 (this leads to (iii) with $\qN^*\qN+\qM^*\qM=S$;
   use $\ssysbm{A\| B\crh S^{1/2}K\| I-S^{1/2}}$
 to generate the $\qABKF$ satisfying (iii) completely (cf.\ (\ref{eAllqKF}))).
Thus, $\hqM(s)=I+\Kw(s-\AClL)^{-1}B$, %
 $\hqN(s)=D+(\CClL)_\w(s-\AClL)^{-1}B$, $\AClL=A+B\Kw,\ (\CClL)_\w=\Cw+D\Kw$
 (use $S^{-1/2}\qN,S^{-1/2}\qM$ for (iii)),
 by Proposition 6.6.17(d4) of [M02].

Note that $\hqM(s)=I+\Kw(s-\AClL)^{-1}B$, %
 $\hqN(s)=D+\CClL(s-\AClL)^{-1}B$, $\AClL=A+B\Kw,\ \CClL=\Cw+D\Kw$.
Naturally, in Theorem \ref{JointlyUout}
 also the dual of (\ref{eoutput-FCC-ARE}) is used.

Above we gave the AREs corresponding to (b)
 (not maximally unbounded $B$);
 to obtain the corresponding general AREs
 ($B^*$ [or $C^*$] possibly maximally unbounded),
 we have to add the $\wlim$ terms to $S$ [and $\tS$].

Naturally, the results of (b) and (c) apply also to general WPLSs,
 if we use the resolvent AREs (see [M03b])
 (or reciprocal AREs if $\rho(A)\cap i\R\ne\tyhja$)
 instead of the ordinary ones.
Those AREs have bounded coefficients (e.g., $(s-A)^{-1}$ in place of $A$)
 and are equivalent to corresponding \SIRE{}s;
 in particular, any of these equations give constructive formulas
 for the feedback and factors. 

\begin{proof}[Proof of Corollary~\ref{cPosIRE}:]\label{pageproof-cPosIREab}
(a) %
In each of the results (or proofs), [two] some kind of ``FCC condition[s]''
 is shown to be equivalent to the existence of certain
 $J$-optimal (possibly for modified $\ALS$ and $J$, see the proofs)
 state-feedback pair[s]. %
By Theorem~\ref{GenSpF}(i)\&(v), this holds iff the corresponding IRE[s]
 (i.e., that corresponding to the possibly modified $\ALS$ and $J$
   in the proofs)
 has a $\gUstar$-stabilizing solution.
But a $\gUstar$-stabilizing solution of the IRE is that of the \SIRE,
 which in turn implies the FCC ($\qKopt x_0\in\gUstar(x_0)\ \all x_0$).
Conversely, here the FCC is also sufficient, by Theorem~\ref{PosJcKF}.

(b) This follows from (a) and Lemma~\ref{lBnotmax}.   

(c)\&(d) These will be proved on p.~\pageref{pageproof-cPosIRE}.
\end{proof}

By (c) above, %
 the $J$-optimal state-feedback pair over $\dUout$
 often corresponds to the smallest nonnegative solution of the IRE.
Much more generally,
 the $J$-optimal state-feedback pair over $\dUexp$ (or $\dUstr$)
 corresponds to the greatest solution of the IRE:
\begin{theorem}%
[Maximal solution $\Ric_{\rm max}$]\label{GreatestSol} %
A strongly internally stabilizing solution of the IRE
 is unique and greater than any admissible solution having $S\ge0$.
\end{theorem}

(The proof is given on p.~\pageref{pageproof-GreatestSol}.
{\em Strongly internally stabilizing}
 means that ($\qKF$ is admissible and)
 $\qAClL^t x_0\to0$ as $t\to+\infty$;
 thus, any exponentially (or $\dUstr$-)stabilizing solution
 is strongly internally stabilizing. Similar results hold
 for the \SIRE.)

If, e.g., $J\ge0$ %
 then any nonnegative admissible solution has $S\ge0$,
 by (\ref{eXSX=}),
 hence then an exponentially stabilizing solution is
 the greatest admissible nonnegative solution.
In fact, it is then the greatest nonnegative solution of the IRE.
(By discretization ([M02], Section~13.4),
 one obtains similar results
 on all solutions of the IRE
 (including the WR solutions of the ARE), %
 regardless of admissibility, because in discrete-time all solutions
  have a bounded ``$K$'' and are hence admissible.)
\NotesFor{Section~\ref{sIRE}} %
For decades, the Riccati operator has been shown to satisfy
 numerous integral equations including the three
 appearing in the IRE; 
 see [S98b] for the case of jointly stabilizable and detectable WPLSs
 and p.~481 of [M02] for a list of earlier ones.
Our contribution in [M02] was
 1. to pick these three and to label them as the IRE,
 2. to prove the sufficiency
 (and define the $\gUstar$-stabilizing solutions),
 3. to generalize the necessity and sufficiency
  to arbitrary WPLSs and $\gUstar$'s,
 4. to observe that the IRE is exactly the discrete-time ARE
 (for the discretized system 
  $  \big[\sm{\qA^t & \qB^t\cr   \qC^t & \qD^t}\big]$)
  and to use the connection for several uniqueness-type results
 (the discrete-time ARE is technically significantly simpler
   than the continuous-time one due to bounded ``generators'').
This then allowed us to derive similar results on the ARE.

We presented the equivalence ``(i)$\IFF$(v)'' of Theorem~\ref{GenSpF}
 in Theorem~9.9.1 of [M02].
The other conditions in Theorem~\ref{GenSpF} seem to be new
 and so does Theorem \ref{SIRE}
 (although, with the \optIRE\ in place of the \SIRE,
 Theorem~\ref{SIRE} is essentially Theorem 9.7.1 of [M02];
 cf.\ the notes to Section~\ref{soptIRE}).
However, the literature on Riccati equations is so abundant,
 that probably some special cases of most of the equations have
 appeared before; 
 e.g., while we were writing these notes, 
 it was pointed out to us that
 recently in [MSW03] (equation (3.4)) and [WST01] (equation (38)) %
 it was shown that a \hSIRE-resembling equation
 ($0,R,0$ on the left-hand-sides)
 holds iff the WPLS is ``$(R,\Ric,J)$-energy preserving''.

Except for coprimeness,
 most of Corollary~\ref{cPosIRE} has been known for
 Pritchard--Salamon systems (see Theorems 3.3 and 3.4 of [PS87]).

For the ARE (\ref{eoutput-FCC-ARE}) with bounded $B,C$,
 (``$\tS\ge0$'' is redundant and)
 it was already known that an exponentially stabilizing solution is maximal
 (see the notes on p.~853 of [M02]).
Theorem~\ref{GreatestSol} generalizes this
 but its proof does not apply to Riccati inequalities unlike 
 that of Theorem 9.8.13 of [M02].
Example 8.4.13 of [M02] is due to Ilya Spitkovsky.

\section{Smooth WPLSs}\label{sMTIC} %

In this section we shall study systems for which
 $\qA B$ and $\Cw\qA B$ are in $\L^1_\omega$, or (slightly) more generally,
 for which
 $\qB\tau:u\mapsto x$ and $\qD:u\mapsto y$ are convolutions
 with $\L^1_\omega$ functions (plus the feedthrough $D$)
 for some $\omega\in\R$
 (``$\qB\tau,\qD\in\WTIComega$'').
This is typically the case if $\qA$ is smoothing (e.g., analytic).

For such systems, one more equivalent condition in most results
 of Section~\ref{sStabOpt} is that the ARE has a nonnegative solution.
Moreover, the solution determines desired factorizations and
 (optimal) stabilizing state-feedback operators.
The resulting closed-loop maps are also of the same form
 (hence ULR, by Theorem~\ref{MTICinftyKF}(c)).
For similar results under alternative assumptions,
 see Section~\ref{ARE} and Corollary~\ref{cPosIRE}.

If $\qD u=Du + f*u\ \all u\in\L^2$,
 where $D\in\BL(U,Y)$ and $f\in\L^1(\R_+;\BL(U,Y))$,
 then we say that $\qD\in\WTIC(U,Y)$\label{pageWTIC}.
When $\cA=\WTIC$, $\cA=\TIC$ or similar,
 we set $\cA_\infty:=\cup_{\omega\in\R}\cA_\omega$,
 where $\cA_\omega:=\{\efn^{\omega \cdot}\qD\efn^{-\omega\cdot}\I
 \qD\in\cA\}$,
 so that $\cA_{\omega'}\sub\cA_\omega\sub\TIC_\omega\
  \all \omega\in\R\cup\{\infty\},\ \omega'\le\omega$.
Thus, $\qD\in\WTICinfty(U,Y)$ if $\qD=D+f*$,
 where $\efn^{-\omega\cdot} f\in\L^1(\R_+;\BL(U,Y))$ for some $\omega\in\R$.
Naturally, $\qE\in\cA$ means that $\qE\in\cA(U,Y)$
 for some Hilbert spaces $U,Y$, and $\cA(U)$ stands for $\cA(U,U)$.

We start by noting that $\WTICinfty$ smoothness is inherited by the 
 optimal closed-loop system, hence the IRE becomes
 equivalent to the ARE:
\begin{theorem}\label{MTICinftyKF} %
Let $\cA=\WTIC$. %
Assume that $\qB\tau,\qD\in\cA_\infty$.
 \begin{itemlist}
  \item[(a)] If  there is a $J$-optimal state-feedback pair $\qKF$,
 and $S\in\cG\BL(U)$, then the following hold:

   \begin{itemlist}
  \item[(a1)] We have $\qF,\qFClL,\qX,\qM,\qN,\qDClL,\qBClL\tau\in\cA_\infty$,
 and $S=D^*JD$. %

   \item[(a2)] 
If $\qC^\rmd\tau\in\cA_\infty$,
 then $\qK^\rmd\tau,\qCClL^\rmd\tau,\qKClL^\rmd\tau\in\cA_\infty$,
 and 
  $\qCClL^\rmd\tau,\qKClL^\rmd\tau,
  \qBClL\tau,\qDClL,\qFClL,\qM,\qN\in\cA_\omega$
 for any $\omega>\omega_A$.
If $\gUstar\sub\dUexp$, then $\omega_A<0$.

   \end{itemlist}
  \item[(b)] 
Assume that $\PTO\in\cG\BL$.
Then the following condition is equivalent to 
 conditions (i)--(vi) of Theorem~\ref{GenSpF}:
 \begin{itemlist}
 \item[(vii)] The ARE (\ref{eARE}) has a $\gUstar$-stabilizing solution.
 \end{itemlist}
Moreover, if (vii) holds, then the $\wlim$ in the ARE converges in norm
 to zero, hence then $S=D^*JD\in\cG\BL(U)$.

  \item[(c)] Any map in $\cA_\infty$ is ULR.
\end{itemlist}
\end{theorem}

(The proof is given on p.~\pageref{pageproof-MTICinftyKF}.)

In (a), $S$ is the signature operator %
 of the problem (e.g., the one appearing in any of
  (ii)--(vi) of Theorem~\ref{GenSpF});
 recall from Theorem~\ref{GenSpF}(b)
 that $\PTO\in\cG\BL \TTHEN S\in\cG\BL$.

Note that always $\qB\tau\in\TIC_\infty(U,H)$.
We have $\qB\tau\in\WTICinfty(U,H)$ iff $\piOI\qA B\in\L^1([0,1];\BL(U,H))$,
 by (\ref{eqB}) (and Lemma 6.8.1(c) of [M02]).
However, $\qB\tau,\qD\in\WTICinfty$ does not imply
 that $\Cw \qA^t B$ is defined for any $t\ge0$. %
Nevertheless, $\Cw\qA^t B\in\L^1_\omega(\R_+;\BL(U,H))$ implies that
 $\qD\in\WTIComega$.

If $\Cw\qA,\qA B,\Cw\qA B$ are all locally $\L^1$,
 then the assumptions of the theorem and the corollary below
 are satisfied:
\begin{lemma}\label{lcAsystem} If $\omega>\omega_A$ and $\cA=\WTIC$,
 then the following are equivalent:
  \begin{itemlist}
    \item[(i)] $\qB\tau,\qC^\rmd\tau,\qD\in\cA_\infty$.
    \item[(ii)] $\qB\tau,\qC^\rmd\tau,\qD\in\cA_\omega$.
    \item[(iii)] $\qA B,\Cw\qA,\Cw\qA B$ are integrable over $[0,1]$.
  \end{itemlist}

\end{lemma}

(By $\qB\tau\in\cA_\infty$ we mean that $\qB\tau\in\cA_\infty(U,H)$
  (not $\cA_\infty(U,H_{-1})$);
 similarly for $\qC^\rmd\tau,\ \qD$, (ii), (iii) and (iii').) %

\begin{proof}
We have ``(iii)$\THEN$(ii)'', by Lemma 6.8.5(a) of [M02],
 ``(ii)$\THEN$(i)'' is trivial,
 and ``(i)$\THEN$(iii)'' is given in Lemma 6.8.3
 (with a slight modification in the proof of (c)). %
\end{proof}

As in Corollary \ref{cPosIRE}(b)\&(c),
 also this kind of systems are stabilizable
 iff the corresponding ARE has a nonnegative solution:
\begin{cor}\label{cMTICinftyKF+} %
Assume that $\qB\tau,\qD\in\cA_\infty:=\WTICinfty$.

{\bf (a)}
In Theorems \ref{PosJcKF} and~\ref{QRCF-Uout}(iii) and
 Corollaries \ref{cOptExpstab1}, \ref{cExpStabB1}
 and \ref{cUout-Meromorphic}, %
 the pair $\qKF$ (if any exists) can be chosen so that
 $\qF,\qFClL,\qX,\qM,\qN,\qDClL,\qBClL\tau\in\cA_\infty$.

For this pair,
 $\qC^\rmd\tau\in\cA_\infty \TTHEN
    \qK^\rmd\tau,\qCClL^\rmd\tau,\qKClL^\rmd\tau\in\cA_\infty$.

{\bf (b)}
If $\qC^\rmd\tau\in\cA_\infty$,
 then $\qKF$ and $\qHG$ (if such exist)
 in Corollary \ref{cJointlySD}(i), Theorem~\ref{JointlyUout}
  and Remark~\ref{rHooAB1AC2}
 can be chosen so that
 $\qF,\qX,\qM,\qN,\qD,\qB\tau,
 \qC^\rmd\tau,\qK^\rmd\tau,
 \qE,\qG,\qH\tau,\qX_1,\qY_1,\tqX,\tqY\in\cA_\infty$,
 and the same holds with the subindex $L$ or $\tL$ added
 (not defined for $\qX,\qM,\qN,\qX_1,\qY_1,\tqX,\tqY$).

{\bf (c)} 
In (a) and (b), one more equivalent condition in any of the 
 results mentioned above is that the corresponding ARE(s) have
 $\gUstar$-stabilizing solutions.
Except for Theorem~\ref{PosJcKF},  %
 another equivalent condition is that the corresponding ARE(s) have
 nonnegative solutions with the $\lim$ converging uniformly to zero
 (the last paragraph of Corollary~\ref{cPosIRE} applies).
Moreover, Corollary~\ref{cBARE} applies
 (even if $B$ is maximally unbounded).

{\bf (d)} %
Assume that $\qC^\rmd\tau\in\cA_\infty$.
In Corollaries \ref{cOptExpstab1} and \ref{cExpStabB1},
 the pair mentioned in (a) also satisfies
 $\qBClL\tau,\qDClL,\qFClL,\qN,\qM,\qCClL^\rmd\tau,\qKClL^\rmd\tau
   \in\cA_\omega$
 for some $\omega<0$.
In Corollary \ref{cJointlySD}(i) and Remark~\ref{rHooAB1AC2},
 the subindexed maps mentioned in (b) belong to $\cA_\omega$
 for some $\omega<0$.
\end{cor} 

(The proof is given on p.~\pageref{pageproof-cMTICinftyKF+}.
Note that $\cA_\omega\sub\cA$ for $\omega<0$.)

We observe from (d), Lemma~\ref{lcAsystem} and Corollary~\ref{cOptExpstab1}
 that if $\Cw \qA B,\qA B,\Cw\qA$ are $\L^1$ over $[0,1]$
 and the state-FCC is satisfied,
 then there is an %
  exponentially stabilizing state-feedback operator $K$
 for which the closed-loop system has an $\L^1(\R;\BL(U,Y))$
 impulse response ($\Cw\qAClL B$).
Theorem \ref{BwARE}(2.)
 shows that we can have $K\in\BL(H,U)$ etc.

The assumption $\qB\tau,\qD\in\WTICinfty$ is equivalent to
 $\qB\tau,\qD\in\MTIC_\infty$, where $\MTIC\label{pageMTIC}\sub\TIC$ is the bigger class
 allowing for delays too, as one can deduce from Section 6.8 of [M02]
 ($\qB\tau$ cannot contain delays, and if it is $\WTICinfty$,
  then neither can $\qD$).
What if $f\in\Lstrong^1$ (i.e., $fu_0\in\L^1\ \all u_0$) instead of $\L^1$,
 where $f=\qA B,\Cw\qA B$?
We do not know (unless, e.g., $C$ is bounded; see Hypothesis 9.2.2 of [M02]),
 but $f\in\Lstrong^2$ 
 is sufficient (see also Theorem~\ref{BwARE}(3.)): 
\begin{remark}[$\cA=\Hstrong^2$]\label{rcA} %
The class $\cA=\{\qD\I \hqD(\cdot-\eps)\in\Hstrong^2(\C^+;\BL)$
 for some $\eps>0\}=:\cA_{\H^2}$
 will also do in Theorem~\ref{MTICinftyKF}(a1)\&(b)\&(c),
 (and in corresponding parts of Corollary~\ref{cMTICinftyKF+}(a)\&(c)),
 and $\cA=\cA_{\H^2}\cap\cA_{\H^2}^\rmd=:\cA_2$ 
 will do in the whole theorem and corollary.
Since $\qB\tau,\qD\in\cA_{\H^2}$ iff Theorem \ref{BwARE}(3.) holds,
 for either of these two classes the ARE can be replaced by the \BwARE\ 
 (under $D^*JD\in\cG\BL$), see p.~\pageref{pageBwARE}.

Another valid choice in the theorem and corollary is
 $\cA=\WTICLC:=\{D+f*\in\WTIC\I f(t)$ is compact for all $t$,
 and $\qD\in\BL\}$
 (or with ``finite-dimensional'' in place of ``compact'').
\end{remark}

(The proof is given on p.~\pageref{pageproof-rcA}.)
Here $F\in\Hstrong^2(\C^+;\BL(U,Y))$ iff $F:\C^+\to\BL(U,Y)$
 is holomorphic and
 $\|F\|_{\Hstrong^2} \label{pageH2strong} %
 :=\sup_{u_0\in U,\ r>0} \|F(r+i\cdot)u_0\|_{\L^2(\R;Y)} <\infty$.

\NotesFor{Section~\ref{sMTIC}} %
For more on $\WTIC$, $\MTIC$ and the other classes, see [M02],
 e.g., Sections 2.6, 6.8 and 9.2, which also provide further results
 on optimization and closed-loop smoothness
 (but do not cover those presented here) and notes.
The results there cover also the case
 where $\ALS\in\SOS$ and $\qD\in\cA$ (no assumptions on $\qB\tau$).
Note that $\WTIC$ is often called the Wiener class
 and $\MTIC$ the Callier--Desoer class.
These classes seem to have been studied mainly in 
 the Pritchard--Salamon setting or in less general settings.

\section{Generalized IREs and the $\Dom(\Aopt)$-ARE}\label{soptIRE} %

This far we have mainly presented the setting and the main results;
 most proofs and accompanying minor results still remain.
In this section, we shall study equivalent conditions
 for the existence of a $J$-optimal control in WPLS form
 (see Theorems~\ref{ALSopt} and~\ref{J-coercive} for sufficient conditions).
In particular, we shall
 (1.) prove Theorem~\ref{SIRE},
 (2.) generalize the $\Dom(A+BK)$-ARE theory of [FLT88], and
 (3.) provide further results and tools for subsequent sections.

The key to (1.) is the  \optIRE\ 
 (Theorem~\ref{OptIRE}), %
 which we show to be equivalent to the
 \hoptIRE, $r$-shifted \optIRE\ (\ref{e0yJyL2r})--(\ref{e0ya=-2rxa}),
 \SIRE\ and \hSIRE.
Each of these five equivalent (systems of) equations
 leads to further results on the $J$-optimal control,
 such as the results in the previous sections
 (e.g., the ARE and the IRE) or as the resolvent RE of [M03b].

In Theorem~\ref{GenARE} we generalize the ARE theory of
 [FLT88] (by Flandoli, Lasiecka and Triggiani)
 by deriving (from the \hoptIRE) %
 the (infinitesimal) ARE on $\Dom(\Aopt)$,
 i.e., on the domain of the closed-loop semigroup generator
 (assuming only the regularity of the original system,
  not that of the optimal control).

We start by showing that a control in WPLS form
 (see Definition~\ref{dWPLSform})
 is optimal over $\gUstar$ iff it is $\gUstar$-stable
 and satisfies the RCC and the \optIRE\ 
 (\ref{e3DJC+BPA=0})--(\ref{e3APA+CJCmix}):
\begin{theorem}[\bfoptIRE]\label{OptIRE} %
Assume that $\oqK$ is a control in WPLS form,
 and $\Ric=\Ric^*\in\BL(H)$. 

 \begin{itemlist}
   \item[(a)] 
Then $\oqK x_0$ is $J$-optimal %
 and $\Ric=\oqC^*J\oqC$
 iff $\oqK x_0\in\gUstar(x_0)$ for all $x_0\in H$
 and the following hold (for all $t>0$):
 \end{itemlist}
 \begin{align}\label{eRCC} %
 \p{\qBt u+\oqAt x_0,\Ric\oqAt x_0}&\to0,
     \text{as} t\to+\infty \ \ (\all x_0\in H,\ u\in\gUstar(0)),\\
  \label{e3DJC+BPA=0}
   0&=(\qDt)^*J\oqCt +(\qBt)^*\Ric\oqAt \, \in\BL(H,\L^2([0,t];U)),\\
    \label{e3APA+CJCmix}
  \Ric&={\qAt}^*\Ric\oqAt+{\qCt}^*J\oqCt\in\BL(H).
 \end{align}

We can make the following enhancements in (a):
\begin{itemlist}
  \item[(a1)] We may replace (\ref{e3APA+CJCmix}) above by 
\begin{align}
    \label{e3APA+CJC}
  \Ric&={\oqAt}^*\Ric\oqAt+{\oqCt}^*J\oqCt \, \in\BL(H).
\end{align}

  \item[(a2)] Equations (\ref{e3APA+CJC}) and  %
  (\ref{e0yJyL2r}) are equivalent. %
  \item[(b1)] The RCC %
 (\ref{eRCC}) is redundant
   if $\gUstar\sub\dUexp$ or $\gUstar\sub\dUstr$.

  \item[(b2)] The limit in (\ref{eRCC}) exists whenever 
 $\qCo x_0\in\L^2$ and (\ref{e3DJC+BPA=0})--(\ref{e3APA+CJCmix}) hold
 (but it need not be zero).

  \item[(c)] If $\gUstar=\dUexp$ (resp.\ $\gUstar=\dUstr$), %
  then $\oALS$ is $J$-optimal iff
  $\oALS$ is exponentially (resp.\ strongly) stable
    and (\ref{e3DJC+BPA=0})--(\ref{e3APA+CJCmix}) hold.

  \item[(d1)] Equation (\ref{e3APA+CJCmix}) holds iff
    \begin{equation}
      \label{e33AP+PAmix}
  -\Ao^*\Ric =\Ric A+\Co^*JC
         \ \ \ \ \in\BL(\Dom(A),\Dom(\Ao)^*).
    \end{equation}

  \item[(d2)] Assume (\ref{e3APA+CJCmix})
 and let $\omega\ge\max\{\omega_A,\omega_{\Ao}\}$. %
  Then  (\ref{e3DJC+BPA=0}) holds iff
  \begin{equation}
    \label{ehcIRE-DJC+BP}
    0=\hqD(s)^*J\hqCo(z)+B^*(s-A)^{-*}\Ric(\bar s+\Ao)(z-\Ao)^{-1}
   =:T(s,z) 
  \end{equation}
 for some (equivalently, all) $s,z\in\C_\omega^+$.

  \item[(e)] If we would allow for any $\gUstar$ satisfying Definition 8.3.2
 of [M02] (instead of Standing Hypothesis~\ref{shgUstar}),
 then (a2), (b1), (b2), (d1) and (d2) %
 still hold
 (and (\ref{e3APA+CJC}) would be equivalent to 
   (\ref{e3APA+CJCmix}) under (\ref{e3DJC+BPA=0}).

  \item[(f)] $\oqK$ is $J$-optimal %
 and $\Ric=\oqC^*J\oqC$
 iff (RCC) holds, %
 $\oqK x_0\in\gUstar(x_0)\ \all x_0\in H$,
 and %
 for some (hence all) $r>0$ %
 we have
 \begin{align}\label{e0yJyL2r}
   \p{x_0,\Ric x_1}_H
   &= \p{\oqC x_0,J\oqC x_1}_{\L^2_r} + 2r\p{\oqA x_0,\Ric\oqA x_1}_{\L^2_r}
 \ \ \ \ \ \ \ \ (\all x_0,x_1\in H),\\\label{e0ya=-2rxa}
  \p{\oqC x_0,J\qD\eta}_{\L^2_r} &= -2r \p{\oqA x_0,\Ric\qB\tau\eta}_{\L^2_r}
 \ \ \ \ \ \ \ \ (\all x_0\in H,\ \eta\in\gUstar(0)).
 \end{align}
(Above we may replace $\gUstar(0)$ by $\L^2_r(\R_+;U)$
 if $r>\max\{0,\omega_A,\omega_{\oA}\}$
  (and $r\ge\vartheta$ unless we give up sufficiency).)
\end{itemlist}

We call (\ref{e3DJC+BPA=0})--(\ref{e3APA+CJCmix}) the
 {\em \optIRE\label{pageoptIRE}\ for $\ALS$ and $J$}. 
We call (\ref{e33AP+PAmix})--(\ref{ehcIRE-DJC+BP}) the
 {\em \hoptIRE\label{pagehoptIRE}\ for $\ALS$ and $J$}.
We call $\qKo$ (or $\Ric$) {\em $\gUstar$-stabilizing}\label{pagegUstarstabilizing4}
 if $\qKo x_0\in\gUstar(x_0)\ \all x_0\in H$ and the RCC holds.
\end{theorem}

(The proof is given on p.~\pageref{pageOptIREproof}.)

Note that \optIRE\ is equivalent to \hoptIRE.
In the theorem we require the \optIRE\ to hold for all $t>0$,
 but if it holds for some $t>0$ and the RCC holds,
 then it holds for all $t\ge0$,
 as one observes from Proposition 9.8.7 and Lemma 14.2.1
 (and Theorem 13.4.4(f2) and Remark 13.4.6) %
 of [M02].

By (b1), ``$\dUexp$-stabilizing'' means the same as ''exponentially stabilizing''.
Similarly, ``$\dUstr$-stabilizing'' means that 
 $\ALSopt$ is strongly stable. %

However, ``$\dUout$-stabilizing'' means that $\ALSopt$
 is output-stable ($\qCopt x_0,\qKopt x_0\in\L^2\ \all x_0\in H$)
 {\em and} the RCC holds (see Example~\ref{exaUoutUexpstab}).
Intuitively, this ``extra'' condition is because
 now we have more candidate controls to be ruled out
 ($\dUexp\subsetneq\dUout$). %

For stable problems, the RCC takes the simple form
  $\p{\qAt x_0,\Ric\qAt x_0}\to0$,
  as shown in [M97] (and in Proposition 9.8.11 of [M02]).

In Theorem~\ref{GenARE}
 we shall use (d1)--(d2) to develop a nonintegral form of
 the Riccati equation (the ``$\Dom(\Aopt)$-ARE'').
In [M03b] we shall use (f) to create the resolvent RE theory;
 (f) is also needed for the proof of Lemma~\ref{lGenSpFproof},
 which is used to obtain Theorem~\ref{GenSpF}(iv) 
  and hence all the results of Section~\ref{sStabOpt}!

Next we list several lemmas that are needed in the proofs of
 Theorem~\ref{OptIRE} and many other results.

Algebraic (infinitesimal) Lyapunov-type equations
 can be equivalently written in 
 integral forms and vice versa, %
 as described below: %
\begin{lemma}\label{lAPA+CJC} %
Let $\ssysbm{\qA_k\crh \qC_k}\in\WPLS(\{0\},H,Y)$  $(k=1,2)$,
 $P\in\BL(H)$ and $\tJ\in\BL(Y)$.
Then
\begin{align}
\label{eAPA=CJCdiff}
  \p{A_1x_1,P x_2}_{H_1} +  \p{x_1,P A_2 x_2}_{H_1}
  + \p{C_1x_1,\tJ C_2x_2}_{Y_1}&\ge0 && (x_1\in\Dom(A_1),\ x_2\in\Dom(A_2))\\
  \label{eAPA=CJC}
 \IIIFF \ \ \ \ \ 
  (\qA_1^t)^*P\qA_2^t + (\qC_1^t)^*\tJ \qC_2^t&\ge P \ \ && (t\in[0,+\infty)).
\end{align}

Moreover, we can, equivalently, replace ``$(t\in[0,+\infty))$''
 by ``$(t\in(0,\eps))$'', for any $\eps>0$,
 or require (\ref{eAPA=CJCdiff}) only for $x_k\in %
 \cap_{n\in\N}\Dom(A_k^n)$.
All this also holds with ``$=$'' or ``$\le$'' in place of ``$\ge$''.
\end{lemma}

Equation (\ref{eAPA=CJCdiff}) is equivalent to
\begin{equation}
  \label{eAP+PAgeCJC}
  A_1^*P+PA_2 +C_1^*\tJ C_2 \ge0 
        \ \ \ \text{(in $\BL(\Dom(A_2),\Dom(A_1)^*)$)},
\end{equation}
 where $\Dom(A_k)$ is equipped with the graph topology
 and $\Dom(A_k)^*$ is its dual w.r.t.\ the pivot space $H$.
(This is the standard convention;
 see p.~464 and Lemma A.3.24 of [M02] 
 or Lemma~\ref{lABC} for details.)

\begin{proof} 
(Further details and references are given on p.~464 of [M02].)

$1^\circ$ {\em ``$\IF$'':}
Let $x_k\in\Dom(A_k)$ ($k=1,2$). Then
\begin{equation}
  \label{}
  (\qA_k x_k)'=A_k\qA_k x_k
     =\qA_k A_k x_k\in\cC(\R_+;H_k) \ \ \ (k=1,2),
\end{equation}
 in particular, $\qA_k x_k\in\cC(\R_+;\Dom(A_k))$.
Consequently, $\qC_k x_k=C_k\qA_k^\cdot x_k\in\cC(\R_+;Y)$ ($k=1,2$).

Since $f:=\p{x_1,gx_2}$,
  where $g:=(\qA_1^t)^*P\qA_2^t + (\qC_1^t)^*\tJ \qC_2^t-P$,
 satisfies $f(0)=0$, $f\ge0$ and $f\in\cC^1(\R_+)$,
 we have $f'(0)\ge0$,
 which implies that (\ref{eAPA=CJCdiff}) holds.

$2^\circ$ {\em ``$\THEN$'':}
Assume that (\ref{eAPA=CJCdiff}) holds on $\Dom(A_1^\infty)\times \Dom(A_2^\infty)$.
Let $a_k\in \Dom(A_k^\infty):=\cap_{n\in\N}\Dom(A_k^n)$ and $t\ge0$.
Set $x_k:=\qA_k^t a_k\in \Dom(A_k^\infty)$, so that $\qC_k a_k=C_k x_k$ 
 and $(\qA_k a_k)'(t)=A_k x_k$ ($k=1,2$), as in $1^\circ$.
By substituting these into (\ref{eAPA=CJCdiff}), we obtain
\begin{eqnarray*}
 0&\!\!\!\!\le\!\!\!\!&\p{\qA_1'(t)a_1,P \qA_2(t)a_2}+\p{\qA_1(t)a_1,P \qA_2'(t)a_2}
        +\p{C_1\qA_1(t)a_1,\tJ C_1\qA_2(t)a_2}\\
 &\!\!\!\!\!=\!\!\!\!\!&\frac{d}{dt}\left[\p{\qA_1(t)a_1,P \qA_2(t)a_2}_H
        +\int_0^t\p{C_1\qA_1(t)a_1,\tJ C_2\qA_2(t)a_2}_Y\,dt\right] \\
 &\!\!\!\!\!=\!\!\!\!\!&\frac{d}{dt}\left[\p{a_1,\qA_1(t)^*P \qA_2(t)a_2}_H
        +\p{a_1,(\qC_1^t)^*\tJ \qC_2^t a_2}_H\right]. 
\end{eqnarray*}

Thus, the expression in brackets must be increasing, hence
 for any $t>0$, we have
\begin{eqnarray*}
\lefteqn{\p{a_1,\qA_1(t)^*P \qA_2(t)a_2}+\p{a_1,\qC_1^*\tJ \pi_{[0,t]}\qC_2 a_2}} \\
 & & \ge\p{a_1,\qA_1(0)^*P \qA_2(0)a_2}+\p{a_1,\qC_1^*\tJ \pi_{[0,0]}\qC_2 a_2}
        =\p{a_1,P a_2} - 0. 
\end{eqnarray*}
The same holds for $a_1,a_2\in H\times H$ too,
 because $\Dom(A_k^\infty)$ is dense in $H$. %

$3^\circ$ The ``moreover'' claim can be observed from the above proofs;
 the claim on ``$\le$'' follows by replacing $P$ by $-P$
 and $\tJ$ by $-\tJ$;
 the claim on ``$=$'' follows from ``$\le$'' and ``$\ge$''.
\end{proof}

When using the ``dynamic programming principle'', we often
 need the following:
\begin{lemma}\label{ltaugUstar} %
Let $x_0\in H$ and $u\in\Lloc^2(\R_+;U)$.
Then $u\in\gUstar(x_0)$ iff $\pi_+\tau^t u\in\gUstar(\qA^t x_0+\qB^t u)$
 for some (equivalently, all) $t\ge0$.
\end{lemma}

This says that $u$ is admissible for some initial state $x(0)=x_0$
 iff at some (hence any) moment $t$ the remaining part of $u$ is admissible
 for the current state $x(t)$.
The proof of this lemma is where we explicitly use
 the hypothesis that $\qABQR$ is a WPLS.

\begin{proof}
Given $t\ge0$, set $u':=\piOt u$, $u'':=\pi_+\tau^t u$,
 so that $u=u'+ \tau^{-t}\pi_+ u''$ 
 and $x_t:=\qAt x_0+\qBt u=\qAt x_0+\qBt u'$.
Obviously, $u\in\L^2_\vartheta \IFF u''\in\L^2_\vartheta$.
We have (recall that $\tau^t u''=\pi_- \tau^t u$)
\begin{equation}
  \label{}
  (\qC x_t) + \qD u'' = (\pi_+\tau^t\qC x_0 + \pi_+\qD \tau^t u'')
   + \pi_+ \qD \pi_+ \tau^t u
  = \pi_+\tau^t(\qC x_0 + \qD u)
\end{equation}
 hence $\qC x_t + \qD u'' \in\L^2$ iff 
 $\qC x_0+\qD u\in\L^2$.

Analogously, we can show that 
 $\qQR x_t + \qRR u'' \in\sZ$ iff 
 $\pi_+\tau^t(\qQR x_0+\qRR u)\in\sZ$,
 i.e., iff $\qQR x_0+\qRR u\in\sZ$ 
 (by Standing Hypothesis \ref{shgUstar}). 
Thus, we have shown that $u\in\gUstar(x_0) \IFF
 u''\in\gUstar(x_t)$.
Since $t\ge0$ was arbitrary, this establishes the claim.
\end{proof}

For any ``test function''
 $\teta\in\L^2([0,t);U)$, there is $\eta\in\gUstar(0)$ s.t.\ $\piOt\eta=\teta$
 and the rest of $\eta$ is optimal (if $\oqK$ is):
\begin{lemma}\label{lGengUstar} %
Assume that
 $\oqK$ is a control in WPLS form s.t.\
 $\oqK x_0\in\gUstar(x_0)$ for all $x_0\in H$.

Then, for any $t\ge0$ and $\teta\in\L^2([0,t);U)$,
 we have 
 \label{pagePt}$P^t\teta:=\eta:=\piOt\teta+\tau^{-t}\oqK\qBt\teta\in\gUstar(0)$,
 $\piOt\eta=\teta$ and $\qD\eta=\qDt\teta+\tau^{-t}\oqC\qBt\teta$.\noproof
\end{lemma}

(The claim ``$\in\gUstar(0)$''
 follows from Lemma \ref{ltaugUstar} by setting $x_0=0$,
 since $\oqK\qBt\teta\in\gUstar(\qBt\teta)$.
The claim on $\qD\eta$ is straight-forward.)
The ``dynamic programming principle'' of this lemma will allow us
 to establish the Riccati equation through (\ref{eDynPrg}).
The principle was based (through Lemma~\ref{ltaugUstar})
 on the requirement on $\sZ$ in
 Hypothesis \ref{shgUstar}. %

Note that $P^t u$ keeps $\piOt u$ but replaces $\piti u$ by the optimal input
 (if $\oqK$ is optimal),
 hence $(P^t)^2=P^t$, so that $P^t$ is a projection
 $\Lloc^2(\R;U)\to \Ran(P^t)\sub\gUstar(0)$.

We shall also need the following fact on how integral operator equation
 systems can be written in the frequency domain and vice versa:
\begin{lemma}[$\pmbold{\qDt^*J\oqDt\IFF\hqD^*J\hoqD}$]\label{lIREiffhIREtechnique} %
Assume that $\tU$ is a Hilbert space,
 $\ssysbm{\oqA\|\oqB\crh\oqC\|\oqD}$ is a WPLS on $(\tU,H,Y)$,
 and $\Ric\in\BL(H)$, $J\in\BL(Y)$. %
Let $\alpha\ge\max\{\omega_A,\omega_{\oA}\}$.
\begin{itemlist}
  \item[(a)] We have (\ref{eD-IRE}) %
 iff (\ref{eD-hIRE}) holds. %

  \item[(b)] We have (\ref{eCJC=}) and (\ref{eDJC=}) %
 iff (\ref{ehCJC=}) and (\ref{ehDJC=}) hold. %

\end{itemlist}
Above we referred to the following equations:\vspace{-1ex}
 \begin{subequations}\label{eD-IRE}
   \begin{align}
    \label{eCJC=} 
  \qCt^*J\oqCt&={\qA^t}^*\Ric\oqA^t-\Ric\ \ \ \ \ &&\ \all t\ge0,\\ %
    \label{eDJD=} 
  \qDt^* J\oqDt&= \qBt^*\Ric\oqBt\ \ \ \ \ &&\ \all t\ge0,\\ 
    \label{eDJC=}
  \qDt^*J\oqCt&= \qBt^*\Ric\oqAt\ \ \ \ \ &&\ \all t\ge0.
   \end{align}
 \end{subequations}\vspace{-4ex}%
\begin{subequations}\label{eD-hIRE}
   \begin{align}
    \label{ehCJC=} 
  C^*J\oC &= A^*\Ric+\Ric \oA,&& \phantom{test}\\
    \label{ehDJD=} 
   \hqD(s)^*J\hoqD(z)
 &= (z+\bar s)B^*(s-A)^{-*}\Ric (z-A_0)^{-1}B_0
 \ \ &&\all\ s,z\in\C_\alpha^+\\
    \label{ehDJC=}
  \hqD(s)^*J C_0(z-A_0)^{-1}  
 &=  B^*(s-A)^{-*}\Ric(\bar s+A_0)(z-A_0)^{-1}
  \ \ \ \ \ &&\all\ s,z\in\C_\alpha^+.
   \end{align}
 \end{subequations}

{\bf (c)} We can have ``for some $s,z\in\C_\alpha^+$'' %
 in place of `''$\all s,z\in\C_\alpha^+$'' in (b).
The same applies to (a) if $J=J^*,\ \Ric=\Ric^*$,
 $A_0=A,\ C_0=C,\ B_0=\pm B$ and $\hqD_0=\pm \hqD$.

{\bf (d)} 
In addition to (c),
 $s,z\in\C_\alpha^+$ can be replaced by $s\in\rho(A), z\in\rho(A_0)$
 if we use characteristic functions in place of the transfer functions
 $\hqD,\hqD_0$.

{\bf (e)} Drop the standing hypotheses on $\ALS$ for the moment.
Assume, instead, that $A:H\supset\Dom(A)\to H$
 and $A_0:H\supset\Dom(A_0)\to H$ are linear operators on $H$,
 $s\in\rho(A),\ z\in\rho(\oA)$,
  $B^*\in\BL(\Dom(A^*),U),\ C\in\BL(\Dom(A),Y)$,
  $B_0^*\in\BL(\Dom(A_0^*),\tU),\ C_0\in\BL(\Dom(A_0),Y)$,
 $\hqD(s)\in\BL(U,Y),\ \hoqD(z)\in\BL(\tU,Y)$.
Extend $\hqD,\hoqD$ by setting
 $\hqD(\zeta):=\hqD(s)+(\zeta-s)C(s-A)^{-1}(\zeta-A)^{-1}B\ \all \zeta\in\rho(A)$,
 $\hoqD(\zeta):=\hoqD(z)+(\zeta-z)C_0(z-A_0)^{-1}(\zeta-A_0)^{-1}B_0\ \all \zeta\in\rho(\oA)$.

Then the equations in (\ref{ehCJC=}) and (\ref{ehDJC=}) hold
 for these $s,z$ iff they hold for all $s\in\rho(A),\ z\in\rho(\oA)$.
If $J=J^*,\ \Ric=\Ric^*$,
 $A_0=A,\ C_0=C,\ B_0=\pm B$ and $\hqD_0=\pm \hqD$,
 then the equations in (\ref{eD-hIRE}) hold
 for these $s,z$
 iff they hold for all $s\in\rho(A),\ z\in\rho(\oA)$.
\end{lemma}

When applying (b), one may want to set $\oqB=0$, $\oqD=0$.
Note that the formulas in (e) also hold
 in (a)--(d) except that we should have $\chqD,\choqD$
 (the characteristic functions) in place of $\hqD,\hoqD$.

\begin{proof} %
(Actually (\ref{eDJD=}) and (\ref{ehDJD=}) are equivalent,
  which can be shown as in the proof as in 
 Proposition 9.11.3 of [M02].)

(b)
$1^\circ$ {\em Equations (\ref{eCJC=}) and (\ref{ehCJC=})
 are equivalent}, by Lemma~\ref{lAPA+CJC}.

$2^\circ$ {\em ``If'':}
Let $\omega>\alpha$, $x_0\in H$ and
 $u\in\L^2_\omega(\R_+;U)$ be arbitrary.
Set $F(t):=\qBt u$, $G(t):=\Ric\oqAt x_0$,
 $f(t):=\qDt u,\ g(t):=J\oqCt x_0$.
Write (\ref{ehDJC=}) as
\begin{equation}
  \label{ehDJC2=}
  \hqD(s)^*J\hoqC(z)
  = \hqB(s)^*\Ric\left(z+\bar s -(z-A_0)\right)(z-A_0)^{-1}
  = (z+\bar s)\hqB(s)^*\Ric(z-A_0)^{-1}-\hqB(s)^*\Ric I
\end{equation}
 to observe that
 Lemma~\ref{lhFG=hfg}(i) is satisfied
 (apply (\ref{ehDJC2=}) to $\p{\hu(s),\hqD(s)^*J\hoqC(z) x_0}_Y$),
 hence so is (v); set $r=0$ to obtain (\ref{eDJC=})
 (since $u,x_0$ were arbitrary).
Combine this with $1^\circ$ to obtain ``if''.

$3^\circ$ {\em ``Only if'':}
Let $\omega,u,x_0,F,G,f,g$ be as above,
 so that (\ref{ehDJC=}) follows from
 Lemma~\ref{lhFG=hfg}(i) (since $\omega,u,x_0$ were arbitrary)
 once we establish (v).

$3.1^\circ$ {\em Case $r<0$:}
Since $\piOt\tau^r\qD\pi_+ =\piOt\qD\tau^r=\piOt\qD\piOt\tau^r\pi_+$ 
 for $r<0$,
 we obtain from (\ref{eDJC=}), that
 $\p{\piOt \tau^r\qD u,\piOt J\oqC x_0}
  = \p{ \tau^r u,\qDt^*J\oqCt x_0}
  = \p{ \tau^r u,\qBt^*\Ric\oqAt x_0}
  = \p{ \qB^{t+r} u,\Ric\oqAt x_0}-0
  = \p{ \qB^{t+r} u,\Ric\oqAt x_0}-\p{\qB^r u,\Ric x_0}$,
 i.e., (v) holds for $r<0$.

$3.2^\circ$ Let $r\ge0$. 
Now $\piOt\tau^r\qD u=\piOt\qD\pi_+\tau^r u+\piOt\qD\pi_-\tau^r \pi_+ u
 =\qDt\tau^r u+\piOt\qC\qB\tau^r\pi_+ u$, by 4. of Definition~\ref{dWPLS0}.
Therefore, (\ref{eDJC=}) implies that
 hence $\p{\piOt\qD\tau^r u,J\oqCt x_0}
 = \p{\qDt\tau^r u+\qCt\qB^r u,J\oqCt x_0}
 = \p{\qBt\tau^r u+\qAt\qB^r u,\Ric\oqAt x_0}-\p{\qB^r u,\Ric x_0}
 = \p{\qB^{t+r}u,\Ric\oqAt v}-\p{\qB^r u,\Ric x_0}$, by $1^\circ$,
 hence (v) holds for $r\ge0$ too (see (\ref{ehDJC2=}));
 thus, (\ref{ehDJC=}) holds
 (for all $s,z\in\C_\omega^+$; but also $\omega>\alpha$ was arbitrary,
 hence for all $s,z\in\C_{\alpha}^+$).

(a) The proof is analogous to $2^\circ$--$3^\circ$ of the proof of (b):

$1^\circ$ {\em ``If'':}
Assume (\ref{eD-hIRE}).
From (b) we obtain (\ref{eCJC=}) and (\ref{eDJC=}).
Let $\omega>\alpha$ and
 $u,v\in\L^2_\omega(\R_+;U)$ be arbitrary.
Set $F(t):=\qBt u$, $G(t):=\Ric\oqBt v$,
 $f(t):=\qDt u,\ g(t):=J\oqDt v$, so that (\ref{eDJD=}) follows from
 Lemma~\ref{lhFG=hfg}(v) (since $u,v$ were arbitrary),
 because (i) follows from (\ref{ehDJD=}) 
 (note that $G(0)=0$).

$2^\circ$ {\em ``Only if'':}
Assume (\ref{eD-IRE}).
From (b) we obtain (\ref{ehCJC=}) and (\ref{ehDJC=}).
With $\omega,u,v,F,G,f,g$ as in $1^\circ$, we obtain (\ref{ehDJD=})
 as in $3^\circ$ of the proof of (b).

(c)\&(d)\&(e)
Observe first that,
 by (\ref{eCharFct}) and the resolvent equation,
 \begin{equation}
   \label{eCharFctDiff}
 \hqD_\ALS(z)-\hqD_\ALS(s) =C[(z-A)^{-1}-(s-A)^{-1}]B.   
 \end{equation}

$1^\circ$ {\em (b):}
Assume (\ref{ehCJC=}) and (\ref{ehDJC=}) for some fixed $s\in\C_\alpha^+$,
 so that 
  $f(s)=g(s)$,
  where $f(s):=\hqD(s)^*J\oC$, $g(s):=B^*(s-A)^{-*}\Ric(\bar s+\oA)$.
We have $(z-A)^{-*}\bar z-(s-A)^{-*}\bar s=[(z-A)^{-*}-(s-A)^{-*}]A^*$,
 hence $g(z)-g(s)=B^*[(z-A)^{-*}-(s-A)^{-*}](A^*\Ric + \Ric A_0)\ \all z$.
By (\ref{ehCJC=}), this equals
 $B^*[(z-A)^{-*}-(s-A)^{-*}]C^*JC$,
 which equals $f(z)-f(s)$,
 by (\ref{eCharFctDiff}).
Thus, $f(z)=g(z)$ for all $z\in\C_\alpha^+$,
 hence (c) holds for (b)
 (including the claim on characteristic functions,
  just replace $\C_\alpha^+$ by $\rho(A)$ above).

$2^\circ$ {\em (a):}
Assume that (\ref{eD-hIRE}) holds for some fixed
  $s_0$ in place of $s$ and $z$.
By $1^\circ$,
 equations (\ref{ehCJC=}) and (\ref{ehDJC=}) 
 hold for any $s,z\in\C_\alpha^+$.
But, by (\ref{eCharFctDiff}) and (\ref{ehDJC=}), we get
 (here $T_s:=(s-A)^{-1}$)
 $\hqD(s)^*J[\hqD_0(z)-\hqD_0(s)]
 =\hqD(s)^*J C_0 [T_z -T_s] B_0
 =B^*(z-A)^{-*}\Ric(\bar s+A_0)[T_z-T_s]B_0
 =B^*(z-A)^{-*}\Ric[(\bar s+z)T_z-(\bar s+s)T_s]B_0$. %
(We used the fact that $T_z-T_s=(s-z)T_zT_s=0$ maps $\Ran(B_0)$
 to $T_z[H]=\Dom(A_0)$.) %

Thus, (\ref{ehDJD=}) is equivalent under the change of $z$
 (i.e., for a fixed $s$, it holds for all $z$ or for no $z$).
Take the adjoint of (\ref{ehDJD=}) to observe that it is equivalent
 under the change of $s$ too.
Thus, if it holds for some pair $s,z$, then it holds for all $s,z$.
\end{proof}

\begin{proof}[Proof of Theorem~\ref{OptIRE}:]\label{pageOptIREproof}
Trivially, condition $\oqK x_0\in\gUstar(x_0)$ ($x_0\in H$) is necessary.
For the rest of the proof, we assume that this condition holds.
Consequently, $\oqC$ is stable %
 and Theorem 8.3.9(a2) of [M02] holds.

$1^\circ$ {\em ``Only if'':}
Given $\teta\in\L^2([0,t);U)$, we have
 for $\eta:=\teta+\tau^{-t}\oqK\qBt\teta$
 and any $x_0\in H$ that
 (note that $\eta\in\gUstar(0)$, by Lemma \ref{lGengUstar})
\begin{align} \label{eDynPrg}
  0=\p{\qC x_0+\qD\oqK x_0,J\qD\eta}
  &=\p{(\piOt+\tau^{-t}\tau^t\piti) \oqC x_0,J\qD\eta}\\
  &= \p{\piOt \oqC x_0,J\qD\eta} + \p{\pi_+\tau^t\oqC x_0,J\qD\tau^t\eta}\\
  &= \p{\piOt \oqC x_0,J\qD\eta} + \p{\oqC\oqA^t x_0,J\oqC\qB^t\teta}\\
  &= \p{\oqC^t x_0,J\qDt\teta} + \p{\oqA^t x_0,\Ric\qB^t\teta}.
\end{align}
Thus, (\ref{e3DJC+BPA=0}) holds.

By  Definition \ref{dWPLS0}(3.), 
 we have $\tau^t\oqC=\pi_+\tau^t\oqC+\pi_-\tau^t\oqC
 =\oqC\oqA+\tau^t\oqC^t$,
 hence $\Ric=\oqC^*J\oqC
 = \oqA^*\oqC^*J\oqC\oqA
   + (\oqC^t)^*J\oqC^t$,
 i.e., (\ref{e3APA+CJC}) holds;
 by (a1), it implies (\ref{e3APA+CJCmix}).
Moreover, since $\Ric=\oqC^*J\oqC$, we obtain from Definition \ref{dWPLS0}(3.)
 that
\begin{equation}
  \label{ePoqAto0tod}
  \p{\oqA^t x_0,\Ric\oqA^t x_0}=\p{\oqC\oqA^t x_0, J\oqC\oqA^t x_0}
  =\p{\pi_+\tau^t\oqC x_0,J\pi_+\tau^t\oqC x_0}
\end{equation}
  $=\p{\oqC x_0,\piti J\oqC x_0}\to0$, as $t\to+\infty$, which shows that
 the second term in (\ref{eRCC}) converges to zero.

Let now $x_0\in H$ and $\eta\in\gUstar(0)$ be arbitrary.
Because $\p{\piOt\qD \eta,J\oqC x_0}\to\p{\qD \eta,J\oqC x_0}$,
 as $t\to\infty$, %
 equation (\ref{e3DJC+BPA=0}) implies that 
 \begin{equation}
  \label{eBPatoDJCgen1}
   \p{\qB\tau^t \eta,\Ric\oqA^t x_0}\to -\p{\qD \eta,J\oqC x_0}, \ \ \ 
 \text{as} \ t\to+\infty.
 \end{equation}
Because $\eta\in\gUstar(0)$ was arbitrary,
 $J$-optimality implies that (\ref{eRCC}) holds.

$2^\circ$ {\em ``If'':}
Assume that $\oqK x_0\in\gUstar(x_0)\ \all x_0\in H$
 and that (\ref{eRCC})--(\ref{e3APA+CJC}) hold.

The identity $\Ric=\oqC^*J\oqC$
 follows from (\ref{e3APA+CJC}) by letting $t\to+\infty$
 and using (\ref{eRCC}).
From (\ref{eBPatoDJCgen1}) and (\ref{eRCC})
 we obtain that $\p{\qD\eta,J\oqC x_0}=0$.

{\em Remark:} As in $1^\circ$, we actually obtain that
 $\oqA^t\Ric\oqA^t x_0,\Ric\oqA^t x_0\to0$, as $t\to\infty$.

(a1) 
Equation (\ref{e3APA+CJC}) is equivalent to (\ref{e3APA+CJCmix}),
 because (\ref{e3APA+CJCmix})$^*-$(\ref{e3APA+CJC})$=\oqK^*((\qBt)^*\Ric\oqA^t+(\qDt)^*J\oqC^t)=0$ when (\ref{e3DJC+BPA=0}) holds. 

(a2) %
By Lemma~\ref{lAPA+CJC}, equation (\ref{e3APA+CJC}) is equivalent
 to $(\dag)$  $\oA^*\Ric+\Ric\oA + \oC^*J\oC=0$
 (in $\BL(\Dom(\oA),\Dom(\oA)^*)$),
 and
 (\ref{e0yJyL2r}) is equivalent to
 $0=\tA^*\Ric+\Ric\tA +\tC^*\tJ\tC=\oA^*\Ric+\Ric\oA-2r\Ric
  +\oC^*J\oC+I^*(2r\Ric)I=\oA^*\Ric+\Ric\oA + \oC^*J\oC$
 (set $\tJ:=\sbm{J&0\cr 0&2r\Ric},\ \tC:=\sbm{\oC\cr I},\ \tA:=\oA-r$
  and note that ``$\tqC^*\tJ\tqC = \oqC^*\efn^{-2r\cdot}J\oqC
  +\oqA^*\efn^{-2r\cdot}2r\Ric\oqA$'').

(b1)
Let $x_0\in H$ and $\eta\in\gUstar(0)$.
If $\gUstar\sub\dUstr$,
 then $\oqA x_0,\qB\tau \eta\in\cCo(\R_+;H)$, by Theorem 8.3.9(a2) of [M02],
 hence then (\ref{eRCC}) obviously holds.

If $\gUstar\sub\dUexp$, then $\oqA x_0,\qB\tau \eta\in\L^2(\R_+;H)$,
   hence then the limit in (\ref{eBPatoDJCgen1}) cannot be nonzero in any case,
   so it must be zero (since it exists, by (\ref{eBPatoDJCgen1})).

(b2) Let $t\to+\infty$ in (\ref{e3DJC+BPA=0})--(\ref{e3APA+CJC})
 (see the proof of (a1) for (\ref{e3APA+CJC})).

(c) ``Only if'' follows from $\oqK x_0\in\gUstar(x_0)$ 
 and the Closed Graph Theorem 
 (as shown in Theorem 8.3.9(a2) of [M02])
 and ``if'' from (a)\&(b1).

(d1) This follows from Lemma~\ref{lAPA+CJC}. %

(d2) Apply ((d1) and) Lemma \ref{lIREiffhIREtechnique}(b)\&(c)
 with $-J$ in place of $J$.

(e) %
Hypothesis~\ref{shgUstar} %
 was only used to prove
 the two ``$J$-optimal'' equivalences, hence (e) holds.
(f) 
We shall use below the facts that $\p{\oqA x_0,\Ric \qB\tau u}_H$
 is bounded for $u\in\gUstar$, by the RCC,
 and that $\L^2(\R_+;Y)\sub\L^2_r(\R_+;Y)$.

By (a), (a1) and (a2), we only have to establish 
  (the ``hence all'' and ``replace'' claims and)
 the equivalence
 between (\ref{e3DJC+BPA=0}) and (\ref{e0ya=-2rxa}),
 and we may assume (\ref{e0yJyL2r}), (\ref{e3APA+CJC}) and the RCC,
 hence also that $\Ric=\oqC^*J\oqC$ (use (\ref{e3APA+CJC}) and the RCC).

Let $\eta\in\Lloc^2(\R_+;U),\ t\ge0,\ x_0\in H$ be arbitrary.
By Lemma~\ref{lGengUstar}, we have
 $u:=\piOt\eta+\tau^{-t}\oqK \qBt\eta\in\gUstar(0)$,
 hence $\p{\oqAt x_0,\Ric\qBt u}$ is bounded, by the RCC,
 and $\qD u\in\L^2\sub\L^2_r$.
We also note that $\oqA x_0,\oqC x_0\in\L^2_r$.

$1^\circ$ {\em A useful identity:} %
Since $\p{f,\piti g}_{\L^2_r}=\efn^{-2rt}\p{\tau^t f,\pi_+\tau^t g}_{\L^2_r}$
 $f,g\in\L^2_r$,
 we have (use Definition~\ref{dWPLS0} and (\ref{e0yJyL2r})
 \begin{align} %
   \efn^{2rt}\p{J\oqC x_0,\piti\qD u}_{\L^2_r}
  &= \p{J\pi_+\tau^t\oqC x_0,\pi_+\qD(\pi_++\pi_-)\tau^t u}_{\L^2_r}\\
  &= \p{J\oqC\oqA^t x_0,\pi_+\qD\oqK\qBt u +\qC\qB\tau^t u}_{\L^2_r}\\
  &= \p{J\oqC\oqA^t x_0,\oqC\qBt u}_{\L^2_r}\\
  &= \p{\oqA^t x_0,\Ric\qBt u}_{H}
     -2r \p{\oqA \oqA^t x_0,\Ric\oqA\qBt u}_{\L^2_r}\\\label{OptIRE-f5}
  &= \p{\oqA^t x_0,\Ric\qBt u}_{H}
     -2r \efn^{2rt}\p{\oqA x_0,\piti \Ric\qB\tau u}_{\L^2_r},
 \end{align}
 because  $\qB\tau(\pi_-+\pi_+)\tau^t u = \qA\qBt u + \qB\tau\oqK\qBt u
  = \oqA \qBt u$ and
\begin{align}
  \p{\oqA x_0,\piti\Ric\qB\tau u}_{\L^2_r}
  &= \int_{s=t}^\infty \efn^{-2rs}\p{\oqA^s x_0,\Ric\qB\tau^s u}_H\,ds\\
  &= \int_{v=0}^\infty \efn^{-2rt}\efn^{-2rv}
         \p{\oqA^v\oqA^t x_0, \Ric\qB\tau^v\tau^t u}_H\,dv\\
  &= \efn^{-2rt}\p{\oqA\oqA^t x_0,\Ric\qB\tau \tau^t u}_{\L^2_r}.
\end{align}

$2^\circ$ {\em ``If'':} %
From (\ref{OptIRE-f5}) and (\ref{e0ya=-2rxa}) (with $u$ in place of $\eta$)
 we obtain that
\begin{align} \label{eOptIRE-f7}
  \p{\oqC x_0,J\piOt\qD u}_{\L^2_r}
  &= \p{\oqC x_0,J\qD u}_{\L^2_r} - \p{\oqC x_0,J\piti\qD u}_{\L^2_r}\\
  &= -2r\p{\oqA x_0,\Ric\qB\tau u}_{\L^2_r}
     -\efn^{-2rt}  \p{\oqA^t x_0,\Ric\qBt u}_{H}
     +2r \p{\oqA x_0,\piti \Ric\qB\tau u}_{\L^2_r}\\\label{eOptIRE-f9}
  &=  -2r\p{\oqA x_0,\piOt \Ric\qB\tau u}_{\L^2_r}
     -\efn^{-2rt}  \p{\oqA^t x_0,\Ric\qBt u}_{H}.
\end{align}

Since two terms above are differentiable a.e.,
 so must the third be too.
Differentiate  (\ref{eOptIRE-f9})
 w.r.t.\ $t$ and multiply by $\efn^{2rt}$
 to obtain that a.e.\ 
\begin{align}
 \p{\oqC x_0,J\qD u}_Y(t)
  =  -2r\p{\oqAt x_0,\Ric\qBt u}_H
    +2r\p{\oqAt x_0,\Ric\qBt u}_H
     -  \p{\oqA x_0,\Ric\qB\tau u}_H'(t).
\end{align}

Integrate both sides ($\int_0^t$) to obtain (\ref{e3DJC+BPA=0}).

$3^\circ$ {\em ``Only if'':} %
This follows by going $2^\circ$ backwards.

$4^\circ$ {\em ``Hence all'':}
This follows from   $3^\circ$,
 because (\ref{e3DJC+BPA=0}) is independent of $r>0$.

$5^\circ$ {\em ``Replace'':}
$5.1^\circ$ {\em ``Only if'':} this follows as above
 (our additional assumptions on $r$ imply that
   $\qD u,\qB\tau u,\oqA x_0\in\L^2_r$,
   so that, e.g., (\ref{eOptIRE-f7}) is justified).
(Note that here any $r\ge0$, for which $\oqA,\qB,\qD$ are $r$-stable, will do.)
$5.2^\circ$ {\em ``If'':} With the additional assumption $r\ge\vartheta$,
 we obviously have $\gUstar(0)\sub\L^2_r$,
 hence sufficiency remains.
\end{proof}

Next we shall prove the equivalence of the \SIRE\ and the \hSIRE\
 to the \optIRE:
\begin{lemma}[\boldedSIRE\ \&\ \boldedhSIRE]\label{lhSIRE} %
Make the assumptions of Theorem~\ref{OptIRE}. %
Then the \optIRE\ (\ref{e3DJC+BPA=0})--(\ref{e3APA+CJCmix})
 holds iff $\Ric,\oqK$ solve the \SIRE\ (\ref{eSIRE}).
The \SIRE\ holds for all $t>0$
 iff the \hSIRE\ holds for some $s,z\in\C_{\omega}^+$ %
 (equivalently, for all $s,z\in\C_{\omega}^+$).
We can above replace $\C_{\omega}^+$ by $\rho(A)\cap\rho(\Aopt)$
 if replace $\hqKopt$ by $\Kopt(\cdot-\Aopt)^{-1}$
 and $\hqD$ by $\hqD_\ALS$.
\end{lemma}

By Theorem~\ref{oALSgen}, we have
 $\overline{\C^+}\cap\rho(A)\sub \rho(A)\cap\rho(\Aopt)$
 if $\oqK$ is stable.

\begin{proof}[Proof of Lemma~\ref{lhSIRE}:]
(We write ``$0$'' in place of ``opt'' to shorten the formulas.
Note from $1^\circ$ that ``\optIRE$\IFF$\SIRE''
 actually holds for any single, fixed $t>0$
 (but in $2^\circ$ we only show that \SIRE\ holds for all $t>0$
   iff the \hSIRE\ holds for all $s,z\in\C_{\omega_A}^+$.)

$1^\circ$ {\em \optIRE$\IFF$\SIRE:}
Equation (\ref{eS2XSK=}) is equivalent to (\ref{e3DJC+BPA=0}),
 as one notices by substituting the identities
 $\oqCt=\qCt+\qDt\oqKt$ and 
 $\oqAt=\qAt+\qBt\oqKt$ into (\ref{e3DJC+BPA=0}). %
Similar substitution  into (\ref{e3APA+CJCmix})
 and use of (\ref{eS2XSK=}) shows 
 that (\ref{e3APA+CJCmix}) is equivalent to (\ref{eS2KSK=})
 (under (\ref{eS2XSK=})). %

$2^\circ$ {\em \hoptIRE$\IFF$\hSIRE:}
(Recall from Theorem \ref{OptIRE}(d1)\&(d2)
 that \optIRE\ is equivalent to \hoptIRE\
 holding for some $s,z$, equivalently, for all $s,z$,
 so it suffices to prove ``\hoptIRE$\THEN$\hSIRE''
 for arbitrary, fixed $s,z\in\C_\omega^+$.)

$2.1^\circ$ {\em (\ref{ehcIRE-DJC+BP})$\IFF$(\ref{eS2hXSK=})
 (under (\ref{eS2hXSX=})):}
Substitute $\hoqC=\hqC+\hqD\hoqK$, $(z-\oA)^{-1}=(z-A)^{-1}[I+B\hoqK(z)]$
 into (\ref{ehcIRE-DJC+BP})
 to obtain that (here $\hqB^*:=B^*(s-A)^{-*},\ \hqB:=(z-A)^{-1}B$ etc.)
 \begin{align}
  \hqD^*J\hqC + \hqD^*J\hqD\hoqK
 & = - \hqB^*\Ric [(\bar s+z)(z-A_0)^{-1}-I]\\ %
 & = - \hqB^*\Ric [(\bar s+z)((z-A)^{-1}+\hqB\hoqK)-I],
 \end{align}
 equivalently, (use (\ref{eS2hXSX=}))
 \begin{equation}
   \label{eS2hXSK=b}
     \hqS\hoqK
 = - \hqD^*J\hqC   - \hqB^*\Ric(s^*+z)(z-A)^{-1}-\hqB^*\Ric,
 \end{equation}
 which is a reformulation of (\ref{eS2hXSK=}),
 because 
 \begin{equation}
   \label{e(s+z)(z-A)}
   (\bar s+z)(z-A)^{-1}-I = (\bar s+A)(z-A)^{-1}.
 \end{equation}

$2.2^\circ$ {\em \hoptIRE$\THEN$\hSIRE:}
By Lemma~\ref{lAPA+CJC}, equation (\ref{e3APA+CJC}) is equivalent to
\begin{equation}
  0=A_0^*\Ric+\Ric A_0+C_0^*JC_0,
\end{equation}
 hence so is (\ref{e33AP+PAmix}).
Assume \hoptIRE.
Then (note that $A_0(z-A_0)^{-1}=z(z-A_0)^{-1}-I$)
\begin{align}
  0&=(s-A_0)^{-*}(A_0^*\Ric+\Ric A_0+C_0^*JC_0)(z-A_0)^{-1}\\
 &= \hoqC^*J\hoqC+(\bar s+z)(s-A_0)^{-*}\Ric (z-A_0)^{-1}
  -\Ric(z-A_0)^{-1}-(s-A_0)^{-*}\Ric\\
 &= (\hqC+\hqD\hoqK)^*J(\hqC+\hqD\hoqK)
     + (\bar s+z)(s-A_0)^{-*}\Ric (z-A_0)^{-1}\\
  &\ \ \ \     -\Ric(z-A_0)^{-1}-(s-A_0)^{-*}\Ric\\
 &= \hqC^*J\hqC + \hoqK^*\hqS\hoqK
   + \hoqK^*\left(\hqD J\hqC+(\bar s+z)\hqB\Ric(z-A)^{-1}\right)
   + \Big(\Big)^*\hoqK\\
  &\ \ \ \   +(\bar s+z)(s-A)^{-*}\Ric (z-A)^{-1}
   -\Ric(z-A)^{-1} - \Ric\hqB\hoqK-(s-A)^{-*}\Ric-\hoqK^*\hqB^*\Ric\\
  &= (s-A)^{-*}(C^*JC+A^*\Ric+\Ric A)(z-A)^{-1}
   + \hoqK^*\hqS\hoqK  -2 \hoqK^*\hqS\hoqK,
\end{align}
 by (\ref{e(s+z)(z-A)}) and (\ref{eS2hXSK=b}), %
 hence (\ref{eS2hKSK=}) holds.

$2.3^\circ$ {\em \hoptIRE$\IF$\hSIRE:} 
Use first $2.1^\circ$ and then go $2.2^\circ$ backwards.

$3^\circ$ {\em $\rho(A)\cap\rho(\Aopt)$:}
The above proof still applies %
 (see Theorem~\ref{oALSgen}), mutatis mutandis %
 (note from the proof that in Theorem \ref{OptIRE}(d2)
   we could have ``$s\in\rho(A),\ z\in\rho(\Aopt)$''
 in place of ``$s,z\in\C_\omega^+$'').  
\end{proof}

The following is straight-forward
 (cf.\ Lemma~\ref{luc}):
\begin{lemma}[\hSIRE$\IFF$\hSIRE$_\ClL$]\label{lhSIREClL} %
Make the assumptions of Theorem~\ref{OptIRE}. %
Let $\qKF$ be an admissible state-feedback pair for $\ALS$
 with closed-loop system $\ALSClL$

Then $\qKClLo=-\qK+\qX\oqK$
 satisfies the \hSIRE\ for $\qABCDClL$ and $J$
 iff $\oqK$ satisfies the \hSIRE\ for $\ALS$ and $J$.
\end{lemma}

The relation $\qSt=\PTO P^t$ 
 connects $\qSt$ to the uniqueness of optimal control:
\begin{lemma}[$\qSt=\PTO P^t$]]\label{lqSt=PTO.Pt} %
Let $\qKopt$ be a $J$-optimal control in WPLS form.
Define $\qSt$ by the \SIRE.
\begin{itemlist}
  \item[(a)] Then $\qSt=\PTO P^t=P^t \PTO=P^t\PTO P^t\ \all t\ge0$.

  \item[(b)] %
The $J$-optimal control is unique some (hence all) $x_0\in H$ %
 iff $\qSt$ is one-to-one for some (hence all) $t>0$.

  \item[(c)] If $\qKopt$ is given by a state-feedback pair,
 then $\qSt=\qXt^*S\qXt$, 
 hence then $\qSt$ is one-to-one iff $S$ is. 

\end{itemlist}
\end{lemma}

Whenever there is a $J$-optimal control for each $x_0\in H$,
 we get similar results. %
In [M03b] %
 we shall show that one more equivalent condition in (b) is that
 $\hqS(s,s)$ 
 is one-to-one for some (hence all) $s\in\C_{\max\{0,\omega_A,\vartheta\}}^+$.

\begin{proof}
(a) (This follows from Lemma~\ref{lOptCost0}(iv),
 but we give here a more direct proof.)
Let $\tu,v\in\gUstar(0)$.
Set $u:=P^t\tu,\ \eta:=P^t \qBt v - \pi_+\tau^t v$.
Then $\eta\in\gUstar(\qBt v)-\gUstar(\qBt v)=\gUstar(0)$,
 by Lemmata \ref{ltaugUstar} and~\ref{lgUstar},
 and $\pi_+\qD\tau u=\qCopt\qBt u$, by Lemma~\ref{lGengUstar},
 hence 
\begin{equation}
  \p{\tau^{-t}\eta,\PTO u} = \p{J\qD \eta,\qD \tau^t u}
  = \p{J\qD \eta,\oqC \qBt u}=0,
\end{equation}
 by $J$-optimality.
Thus, $\p{v+\tau^{-t}\eta,\PTO u}=\p{v,\PTO u}$.
But $v+\tau^{-t}\eta=P^t v$. Since $v,\tu\in\gUstar(0)$
 were arbitrary, we have $(P^t)^*\PTO P^t=\PTO P^t$,
 hence $(P^t)^*\PTO P^t=[(P^t)^*\PTO P^t]^*=(P^t)^*\PTO$.

Finally, $\qD P^t u = \qDt u + \tau^{-t}\oqC\qBt u$
 for all $u\in\Lloc^2$, by Lemma~\ref{lGengUstar}, 
 hence 
 \begin{equation}
   \p{P^t v,\PTO P^t u}=\p{\qDt v,J\qDt u}+\p{\qBt v,(\oqC^*J\oqC)\qBt u}
  = \p{v,\qSt u}.
 \end{equation}

(b)
If $u$ is $J$-optimal for $x_0=0$ (i.e., $\PTO u=0$),
 then $\qSt u=(P^t)^*\PTO u\equiv 0\ \all t\ge0$.
Conversely, if $\qSt u=0$ for some $u\in\Lloc^2$ and $t>0$,
 then $P^t u$ is $J$-optimal for $x_0=0$ (since $\PTO P^t u=\qSt u=0$).
Thus, $\PTO$ is one-to-one iff $\qSt$ is.
Now (b) follows from Lemma~\ref{lOptCost0}(d)\&(c).

(c) This holds because $\qXt\qMt=I$ (see the IRE).   
\end{proof}

If $\qD$ is regular and there is a unique optimal control
 (Theorem~\ref{ALSopt}),
 then we can generalize the classical results (and those in [FLT88])
 by showing that the ARE $A^*\Ric+\Ric A+C^*\Cw=\Ric B\Bw^*\Ric$
 is satisfied on  $\Dom(\Aopt)$,
 and that the optimal control is $u(t)=-\Bw^*\Ric x(t)$ a.e.\
 (in the standard LQR problem where $J=I,\ D^*D=I,\ D^*C=0$), by (c)\&(b) below:
\begin{theorem}[\pmbold{$\Dom(\Aopt)$}-ARE]\label{GenARE} %
\index{ARE!generalized (on $\Dom(\Aopt)$)}
\index{generalized ARE}
Let $\qKopt$ %
be a $J$-optimal control for $\ALS$ in WPLS form.
Then
\begin{align}\label{e3AP+PAopt}
   -\Aopt^*\Ric &=\Ric \Aopt+ \Copt^*J\Copt
         &&\in\BL(\Dom(\Aopt),\Dom(\Aopt)^*),\\\label{e3AP+PAmix}
 -\Aopt^*\Ric &=\Ric A+\Copt^*JC
         &&\in\BL(\Dom(A),\Dom(\Aopt)^*),\\\label{e3AP+PAmix2}  %
 -A^*\Ric &=\Ric \Aopt+C^*J\Copt
         &&\in\BL(\Dom(\Aopt),\Dom(A)^*). %
\end{align}
Recall that $\Copt=\Cc+\Dc\Kopt$ and $\Aopt=A+B\Kopt$
 on $\Dom(\Aopt)$. %

Assume, in addition, that $\qD$ is WR. Then
\begin{itemlist}
  \item[(a)] {\bf (``\pmb{$\Kopt=-B^*\Ric$}'')}
          $\Ric\in\BL(\Dom(\Aopt),\Dom(\Bw^*))$,
  $\Bw^*\Ric =-D^*J\Copt$ on $\Dom(\Aopt)$, and %
\begin{equation}\label{eGenAREK}
  (D^*JD)\Kopt = -\Bw^*\Ric  -  D^*J\Cw\in\BL(\Dom(\Aopt),U).
\end{equation}

  \item[(b)] {\bf (``\pmb{$\uopt=-B^*\Ric\xopt$}'')} %
 $\Koptw x(t)=-(D^*JD)^{-1}(\Bw^*\Ric+D^*J\Cw) x(t) =(\qKopt x_0)(t)$ 
   for a.e.\ $t\ge 0$ and all $x_0\in H$,
 if $D^*JD\in\cG\BL(U)$
 (here $x:=\xopt(x_0):=\qAopt x_0$).
In particular, $\Ric x(t)\in\Dom(\Bw^*)$ a.e.

  \item[(c)] {\bf (ARE on Dom(A$_{\bf opt}$))} \  %
 If $\qD$ and $\qD^\rmd$ are SR and $D^*JD\in\cG\BL(U)$,
 then 
\begin{equation}
  \label{eGenAREAopt}
  A^*\Ric+\Ric A + C^*J\Cw = (\Ric B+C^*JD)(D^*JD)^{-1}(D^*J\Cw+\Bw^*\Ric)
\end{equation}
 in $\BL(\Dom(\Aopt),\Dom(\Aopt)^*)$.

  \item[(d)] {\bf (ARE \pmb{$\IFF J$}-optimal)} %
Assume, instead, that $\qKopt$ %
 is a control in WPLS form,
 $\qD\in\WR$,
 and $\Ric=\Ric^*\in\BL(H)$.

Then $\qKopt$ is $J$-optimal and $\Ric=\qCopt^*J\qCopt$ iff
 (\ref{e3AP+PAmix2}) and (\ref{eGenAREK}) hold
 and $\Kopt$ is ``$\gUstar$-stabilizing''
 (i.e., $\qKClL x_0\in\gUstar(x_0)$ for all $x_0\in H$ and the RCC (\ref{eRCC}) holds).
 \end{itemlist}
\end{theorem}

(See Section 9.7 of [M02] for further details, results and notes.)
Since $\Dom(\Aopt)$ is not known a priori,
 we are not satisfied by the above
 but go on to derive the IRE to finally arrive at the ARE
 presented in Section~\ref{sARE}.
However, both (infinitesimal) algebraic REs have their applications; for 
 the above see, e.g., [LT00]. %

\begin{proof}
Apply Lemma \ref{lAPA+CJC} to (\ref{e3APA+CJC}), 
 (\ref{e3APA+CJCmix}) and  (\ref{e3APA+CJCmix})$^*$ to obtain
 (\ref{e3AP+PAopt}), (\ref{e3AP+PAmix}) and (\ref{e3AP+PAmix2}).
The formulae for $\Aopt$ and $\Copt$ are from
 Theorem~\ref{oALSgen}. 

(a) Multiply (\ref{ehcIRE-DJC+BP}) by $z x_0$, where $x_0\in\Dom(\Aopt)$,
 and let $z\to+\infty$ to obtain that
 \begin{equation}
   -\hqD(s)^*J\Copt x_0
 = B^*\bar s(s-A)^{-*}\Ric x_0
 + B^* (s-A)^{-*}\Ric \Aopt x_0
 \end{equation}
Let $\R\owns s\to +\infty$ to obtain that $\Ric x_0\in\Dom(\Bw^*)$
 and $-D^*J\Copt x_0 = \Bw^*\Ric x_0 + 0$.
Since $\Copt=\Cw+D\Kopt$, we obtain (\ref{eGenAREK}).

(b)--(d) Theorem 9.7.3 of [M02] contains a slightly stronger form
 of this theorem. %
Therefore, 
 we refer
 the long proofs of (b)--(d), 
 and only remark that formally (b) and (c) follow from
 (\ref{eGenAREK}) and (\ref{e3AP+PAopt}), %
 and that (d) follows by going the backwards the above proofs.
\end{proof}

\NotesFor{Section~\ref{soptIRE}} %
Theorem \ref{OptIRE}(d2)\&(f)
 and Lemmas \ref{lIREiffhIREtechnique} and~\ref{lhSIRE} 
 (in particular, the \hoptIRE, the \SIRE\  and the \hSIRE)
 seem to be new (see the notes to Section~\ref{sIRE}).
We established
 most of the rest of this section in Sections 8.3 and 9.7 of [M02].

However, the necessity of (\ref{e3DJC+BPA=0})--(\ref{e33AP+PAmix})
 (and essentially Lemma~\ref{lAPA+CJC})
 was already known for some cases;
 see, e.g., [S98b] for jointly stabilizable and detectable 
 $J$-coercive (over $\dUout$) WPLSs.
Similarly, for the case of bounded $C$ and the cost $\|y\|_2^2+\|u\|_2^2$,
 most of Theorem~\ref{GenARE} is contained in [FLT88]
 (with the additional (implicit) assumption that a suitable extension
  of $B^*$ exists; we have shown here that assumption is redundant
 (using $\Bw^*$)). %
See the notes on p.~465 of [M02] for further details.

\section{IRE: details}\label{sIREmore} %

In this section we shall prove Theorem~\ref{GenSpF}
 and further results on the IRE.
We start by a generalization of the theorem
 (dropping the uniqueness requirement): 
 the $J$-optimal state-feedback pairs
 are exactly the ones determined by
 the $\gUstar$-stabilizing solutions of the IRE:
\begin{theorem}[\pmbold{IRE \pmbold{$\IFF$} $J$}-optimal $\pmbold{\qKF}$]\label{IRE} %
The following are equivalent:
 \begin{itemlist}
  \item[(i)] There is a $J$-optimal state-feedback pair over $\gUstar$.
  \item[(ii)] The IRE has a $\gUstar$-stabilizing solution.
  \item[(iii)] The \hIRE\ has a $\gUstar$-stabilizing solution.
 \end{itemlist}

Moreover, the following hold:
 \begin{itemlist}
  \item[(a1)] Problems (ii) and (iii) have same solutions %
 (and (b) if it has any solutions).

  \item[(a2)] A solution $\Ric$ of (ii) is unique
 (and $\Ric=\qCClL^*J\qCClL$),
 and corresponding pairs $\qKF$ are exactly the
 $J$-optimal state-feedback pairs over $\gUstar$.

  \item[(b)] There is a minimizing state-feedback pair over $\gUstar$
  iff (ii) holds and $\cJ(0,u)\ge0$ for all $u\in\gUstar(0)$. 

  \item[(c)] Solutions of (ii) are exactly those $\gUstar$-stabilizing solutions
 of the \optIRE\ that correspond to a state-feedback pair
 (with $\oALS=\ALSClL\sbm{I\cr 0}$).

  \item[(d)] The operator $S$ is one-to-one iff the $J$-optimal control is unique.
If $S$ is one-to-one, then all $J$-optimal pairs are given by (\ref{eAllqKF}).

  \item[(e)] If $\PTO\in\cG\BL$, then $S\in\cG\BL(U)$; %
 moreover, if $\PTO\gg0$, then $S\gg0$. %
 \end{itemlist}
\end{theorem}

(The proof is given on p.~\pageref{pageproof-IRE}.)

As before, $\Ric$ is the $J$-optimal cost operator (over $\gUstar$)
 and $\cJ(x_0,u):=\p{y,Jy}$.
In fact, with perturbation $\uc$ to the closed-loop system 
 (see Figure~\ref{fALSClL}, p.~\pageref{fALSClL}), the cost becomes
 \begin{equation}
   \label{ecJClL}
 \p{y,Jy}_{\L^2(\R_+;Y)}
        =\p{x_0,\Ric x_0}_H+\p{\uc,S\uc}_{\L^2(\R_+;U)}.
 \end{equation}
Here $y:=\qCClL x_0+\qDClL \uc$
  for any  $x_0\in H$ and $\uc\in\L^2(\R_+;U)$ with compact support;
 if $\qN:=\qDClL$ is stable (e.g., $\gUstar=\dUexp$),
 then $S=\qN^*J\qN$, and any $\uc\in\L^2(\R_+;U)$ will do above.
(See Theorem 9.9.1 of [M02] for details and further results.)

We conclude that the $J$-optimal state-feedback pairs over $\dUexp$
 are exactly those exponentially stabilizing state-feedback pairs
 that satisfy the IRE (with $\Ric:=\qCClL^*J\qCClL$ and $S:=\qN^*J\qN$),
 equivalently, that satisfy the \optIRE\ 
 (with $\ALSClL\sbm{I\cr0}$ in place of $\oALS$).
By Example~8.4.13 of [M02], %
 such pairs need not exist
 even if there is a unique $J$-optimal control for each initial state
 (since there the $J$-optimal control in WPLS form is not 
  given by any (well-posed) state-feedback pair, 
 despite $J$-coercivity).
Thus, the \optIRE\ is strictly more general than the IRE.

Note from Lemma~\ref{lhIRE} that in (iii) 
 (and hence in Theorem \ref{GenSpF}(vi) too)
 it suffices to
 have a $\gUstar$-stabilizing pair $\qKF$ that satisfies
 (\ref{ehIRE}) for some $s=z\in\rho(A)$.

The IRE is equivalent to the \hIRE:
\begin{lemma}[$\shat{\bf IRE}$]\label{lhIRE} %
Let $\ALSExt=\qABKF$ be a WPLS,
 $\Ric\in\BL(H)$, $S\in\BL(U)$.
Set $\qX:=I-\qF$.
Then the IRE (\ref{eIRE}) is satisfied iff
 the \hIRE\ holds for all $s,z\in\C_{\omega_A}^+$.

Moreover, when $\Ric=\Ric^*$ and $S=S^*$,
 the \hIRE\ (\ref{ehIRE}) holds for all $s,z\in\rho(A)$
 iff it holds for some 
 $s,z\in\rho(A)$ %
 (use $\hqXc$ (resp.\ $\hqDc$) in place of $\hqX$ (resp.\ $\hqD$)).
\end{lemma}

Naturally, with a slight abuse of notation,
 by $\hqXc(s)\label{pagehqXALSext}$
 we refer to
 ``$\hqX_{\ALS_\qX}$''$:=\hqX(\alpha)+(\alpha-s)(-K)(\alpha-A)^{-1}(s-A)^{-1}B
   =I-\qF_{\ALSext}(s)$,
 the characteristic function of $\ssysbm{\qA\|\qB\crh -\qK\|\qX}$.
Recall from Lemma~\ref{lCharFct}
 that the characteristic functions coincide with the transfer functions
 on $\C_{\omega_A}^+$.

We can write (\ref{ehXSX=}) as
 $\hqX(s)^*S\hqX(z)=\hqD(s)^*J\hqD(z)+(z+\bar s)\shat{\qB\tau}(s)^*\Ric\shat{\qB\tau}(z)$
 on $\C_{\omega_A}^+$
 (the factor $z+\bar s$ is due to the fact that $\qBt^*\Ric\qBt$
   refers to the adjoint (inner product)
  in $H$, not in $\L^2$; see the proof of Lemma~\ref{lhFG=hfg} for details).
Similarly, (\ref{ehXSK=})
 equals
 $\hqX(s)^*S\hqK(z)
  = -\hqD(s)^*J\hqC(z) -(z+s^*)\shat{\qB\tau}(s)^*\Ric\hqA(z) +
    \shat{\qB\tau}(s)^*\Ric$.
In [M03b] we shall show how to prove these equations for any optimal
 control in WPLS form (including the ill-posed ones)
 and how to interpret these as REs for
 a modified system with bounded generators.

In Theorem~\ref{ARE} we showed that the \hIRE\ is equivalent
 to the ARE if(f) $\qD$ and $\qF$ are WR.

\begin{proof}[Proof of Lemma~\ref{lhIRE}:]
This follows from (a)\&(c)\&(d) of Lemma~\ref{lIREiffhIREtechnique}
 through substitutions 
$\qC\mapsto \sbm{\qC\cr -\qK},\ 
 \qD\mapsto \sbm{\qD\cr \qX},\
 \qA_0\mapsto\qA,\ \qB_0\mapsto -\qB,\ 
\qC_0\mapsto \sbm{\qC\cr -\qK},\ 
 \qD_0\mapsto \sbm{-\qD\cr -\qX},\ J\mapsto \sbm{-J&0\cr 0&S}$.
\end{proof}

We shall soon use the fact that a causal self-adjoint (hence static) %
 ``operator'' 
 ``$\qD^*J\qD$'' is an element of $\BL$ even when 
 $\qD$ is unstable, so that the ``operator'' is not well-defined
 on the whole $\L^2$ a priori: %
\begin{lemma}[\pmbold{$\qD^*J\qD=S$}]\label{lNJN=SonL2c} %
Let $\qD\in\TIC_\infty(U,Y)$ and $J=J^*\in\BL(Y)$.
Assume that $\qD u\in\L^2$  %
 and $\p{\qD \pi_+ v,J\qD\pi_- u}=0$ for all $u,v\in\Lc^2$. %
Then there is a unique $S=S^*\in\BL(U)$ s.t.\
 $\p{\qD v,J\qD u}=\p{v,S u}$ for all $u,v\in\Lc^2$.\noproof
\end{lemma}

(This is Lemma 2.3.1 of [M02];
 in the proof it was shown the operators
 $S_t:=(\qD\pitt)^*J\qD\pitt\in\BL(\L^2([-t,t);U))$
 are restrictions of each other and can be extended to a static operator
 (``$S$'').
Note that for $\qD\in\TIC$ the term $\qD^*J\qD$ would be well defined
 and hence the lemma would be a well-known
 simple consequence of the Liouville Theorem.)

Next we list the connections between the IRE and its variants:
\begin{lemma}\label{lOptIRE0} %
Let $S\in\BL(U)$, and $\Ric=\Ric^*\in\BL(H)$.
Let $\qKF$ be an admissible state-feedback pair for $\ALS$,
 and let $\ALSClL:=\SmallbClLExtSystem\in\WPLS(U,H,Y\times U)$
 be the corresponding closed-loop system.
Set $\qM:=(I-\qF)^{-1}$, $\qN:=\qD\qM=\qDClL$.

We consider, for $t\ge0$, the equations
 \begin{align}
  \label{e0DJC+BPA=0}
   0&=\qDt^*J\qCClLt+\qBt^*\Ric\qAClLt,\\ %
  \label{e0NJC+BPA=0}
   0&=\qDClLt^*J\qCClLt+\qBClLt^*\Ric\qAClLt,\\ %
    \label{e0APA+CJC}
  \Ric&=\qAClLt^*\Ric\qAClLt+\qCClLt^*J\qCClLt,\\ %
    \label{e0APA+CJCmix}
  \Ric&=\qAClLt^*\Ric\qAt+\qCClLt^*J\qCt,\\ %
    \label{e0S=NJN+BPB} %
  \piOt S&= \qNt^*  J\qNt +\qBClLt^*\Ric\qBClLt,\\ %
  \label{e0SK=}
   S\qKt&=-\left(\qNt^*J\qCt+\qMt^*\qBt^*\Ric\qAt\right).%
 \end{align}

Claims (a1)--(b3) hold:
 \begin{itemlist}
   
   \item[(a1)] For any $t\ge0$
 we have (\ref{e0NJC+BPA=0})$\IFF$(\ref{e0DJC+BPA=0}), 
 as well as (\ref{eXSX=})$\IFF$(\ref{e0S=NJN+BPB}), %
 and (\ref{eXSK=})$\IFF$(\ref{e0SK=}). %

   \item[(a2)] Any admissible solution of the IRE satisfies
 (\ref{e0DJC+BPA=0})--(\ref{e0SK=}).

   \item[(b1)] Let $t\ge0$ and let (\ref{e0SK=}) hold.
Then (\ref{e0APA+CJCmix})$\IFF$(\ref{eKSK=}). 

   \item[(b2)] Let $t\ge0$ and let (\ref{eXSX=}) hold.
     Then (\ref{e0NJC+BPA=0})$\IFF$(\ref{e0SK=}).

   \item[(b3)] Let $t\ge0$ and let (\ref{e0NJC+BPA=0}) hold.
Then (\ref{e0APA+CJC})$\IFF$(\ref{e0APA+CJCmix}).
 \end{itemlist}

If $\qCClL$ is stable, then (c1)--(c4) hold:

 \begin{itemlist}
   \item[(c1)]  %
  We have $\qN\piOt\in\BL(\L^2)$ for all $t\ge0$.

   \item[(c2)]  %
Assume that $\Ric=\qCClL^*J\qCClL$. Then (\ref{eXSX=}) is equivalent to
  \begin{equation}
    \label{eNJN=uSuL2t}
    \p{\qDClL u,J\qDClL v}_{\L^2(\R_+;U)}=\p{u,Sv}_{\L^2(\R_+;U)}
  \ \ \ (u,v\in\L^2([0,t);U)).
  \end{equation}
 Moreover,  (\ref{eXSX=}) holds for all $t>0$ iff %
 \begin{equation}
   \label{e0NGenSpF}
   \p{\hqN u_0,J\hqN u_0}_Y = \p{u_0,Su_0}\ \ \text{a.e.\ on}\ i\R\ \ (u_0\in U).
 \end{equation}

   \item[(c3)]  %
If $\Ric=\qCClL^*J\qCClL$,
 then (\ref{e0NJC+BPA=0}) is equivalent to %
 \begin{align}
  \label{e0NJC=0}
   \p{ \qDClL\pi_+ u,J\qCClL x_0}_{\L^2(\R_+;U)}=0\ \ \ (u\in\L^2([0,t);U),\ x_0\in H).
 \end{align}

   \item[(c4)]  %
Assume that $\Ric=\qCClL^*J\qCClL$ and
 that (\ref{e0NJC+BPA=0}) holds for all $t>0$.

Then there is a unique $\tS\in\BL(U)$
 s.t.\ $\p{\qN u,J\qN u}=\p{u,\tS u}$ ($u\in\Lc^2$).

Moreover, $\tS=\tS^*\in\BL(U)$, 
 and the IRE (\ref{eIRE})
 and (\ref{e0DJC+BPA=0})--(\ref{e0NJC=0}) are satisfied 
 for all $t\ge0$ with $\tS$ in place of $S$.
 \end{itemlist}
\end{lemma} %

\begin{proof}
We first recall from (\ref{eALSb}) that
\begin{equation}
  \label{eALSb1}
  \qAClLt=\qAt+\qBt\qMt\qKt=\qAt+\qBClLt\qKt,\ \ \ \ \ 
   \qCClLt=\qCt+\qDt\qMt\qKt=\qCt+\qNt\qKt.
\end{equation}

(a1) Multiply by $\qMt$ or $\qXt$ to the left.

(a2) Use (a1) and (b1).

(b1) Insert (\ref{e0SK=}) into (\ref{eKSK=}) to obtain (\ref{e0APA+CJCmix})
 (recall (\ref{eALSb1})).

(b2) From (\ref{eALSb1})
 and (\ref{e0S=NJN+BPB}) (see (a1)) %
 we obtain that
 \begin{equation}
   \label{}
   \qDClLt^*J\qCClLt+\qBClLt^*\Ric\qAClLt
  = S\qKt+\qNt^*J\qCt+\qMt^*\qBt^*\Ric\qAt.
 \end{equation} %

(b3) By (\ref{eALSb1}),
 the difference (\ref{e0APA+CJC})$-$(\ref{e0APA+CJCmix})$^*$
 is equal to
 \begin{equation}
   \label{}
   \qKt^*
     \left(\qDClLt^*J\qCClLt+\qBClLt^*\Ric\qAClLt\right)
     =\qK^*0=0.
 \end{equation} %

(c1) 
Since $\pi_+\qN\pi_-=\qCClL\qBClL$, by Definition~\ref{dWPLS0}4.,
 we have
\begin{align}
  \label{epitiNpiOt}
 \piti\qN\piOt=\tau^{-t}\pi_+\qN\pi_-\tau^t\piOt
   =\tau^{-t}\qCClL\qBClL\tau^t\piOt\in\BL(\L^2).
\end{align}
Since $\qNt\in\BL(\L^2)$,
 we have $\qN\piOt=\qNt+\piti\qN\piOt\in\BL(\L^2)$.

(c2) (From (c1) it follows (see Lemma 2.1.13 of [M02] for more) that
 there is a holomorphic $\hqN:\C^+\to\BL(U,Y)$
 s.t.\ $\shat{\qN u}=\hqN\hu$ for all $u\in\Lc^2(\R_+;U)$
 and that $\hqN u_0$ has a radial (even nontangential)
 limit a.e.\ for each $u_0\in U$
 (indeed, $\hf\hqN u_0=\shat{\qN f u_0}\in\H^2(\C^+;Y)$
   when $\hf\in\Lc^2(\R_+)$).
However, when $\dim U=\infty$, the map $\hqN$
 need not have a boundary function 
 (or it does, but the values are not in $\BL(U,Y)$ anywhere on $i\R$).
(e.g., $\hqN$ could be the Cayley transform of $F$ of Example 3.3.6 of [M02],
 multiplied by, e.g., $\efn^{-s^2/2}$.)) %

$1^\circ$ 
Since $\Ric=\qCClL^*J\qCClL$, we obtain from 
 (\ref{epitiNpiOt}) that
\begin{align}
  \label{elOptAREc2tmp2}
  \p{\piti\qN\piOt u,J\piti\qN\piOt v}
   &= \p{\tau^{-t}\qCClL\qBClL\tau^t\piOt u,J\tau^{-t}\qCClL\qBClL\tau^t\piOt v}\\
   &=\p{\qBClLt u,\Ric\qBClLt v}\ \ \ (u,v\in\L^2).
\end{align}
Consequently, $\p{u,\piOt S v}=\p{\qN\piOt u,(\piOt+\piti) J\qN\piOt v}$ for $u\in\L^2$
 iff (\ref{e0S=NJN+BPB}) holds (equivalently, (\ref{eXSX=}) holds, by (a1)).

$2^\circ$ Assume (\ref{eNJN=uSuL2t})  (equivalently, (\ref{eXSX=}))
 for all $t>0$.
Set $u=fu_0,v=gu_0$, where $f,g$ are scalar
 to observe that $\hf^*\hg$(\ref{e0NGenSpF}) holds for all
 $f,g\in\Lc^2(\R_+)$, hence (\ref{e0NGenSpF}) holds.

$3^\circ$ Assume (\ref{e0NGenSpF}).
Obviously (Lemma A.3.1(g3) of [M02]), the latter $u_0$'s may be replaced
 by any $v_0\in U$.
If $u=\khi_{E}u_0,\ v=\khi_F v_0$,
 then  (\ref{eNJN=uSuL2t}) follows from the Plancherel Theorem.
By linearity, we obtain (\ref{eNJN=uSuL2t}) for simple functions,
 by density, for general $u,v\in\Lc^2$, as required.

(c3) From the identity (use %
 Definition~\ref{dWPLS0})
 \begin{align}
  \label{}
  \piOt \tau^{-t}&\qBClL^*(\qCClL^*J\qCClL)\qAClL(t)
    =\piOt \tau^{-t} \pi_-\qDClL^*\pi_+J\pi_+\tau^t \qCClL\\
    &= \piOt  \qDClL^* J\tau^{-t}\pi_+\tau^t \qCClL
    = \piOt  \qDClL^* J \pi_{[t,\infty)}\qCClL.
 \end{align}
 we obtain that the equation $0=\piOt\qDClL^*J\qCClL
  =\piOt\qDClL^*J(\piOt+\piti)\qCClL$ is equivalent to
 (\ref{e0NJC+BPA=0}), as claimed.

(c4) Now $\p{\qN \pi_+ v,J\qN\pi_- u}=\p{\qN\pi_+ v,J\qCClL\qBClL u}=0$
 for all $u,v\in\Lc^2$, by (c3),
 hence there is a unique $\tS=\tS^*\in\BL(U)$ 
 s.t.\ $\p{\qN u,J\qN u}=\p{u,\tS u}$ ($u\in\Lc^2$),
 by (c1) and Lemma \ref{lNJN=SonL2c}.

By (\ref{epitiNpiOt}),
 we have $\qBClLt^*\Ric\qBClLt=(\piti\qN\piOt)^*J\piti\qN\piOt$.
It follows that (\ref{e0S=NJN+BPB}) holds with $\tS$ in place of $S$,
 for all $t\ge0$.

As observed above (\ref{ePoqAto0tod}),
 the identity $\Ric=\qCClL^*J\qCClL$ leads to (\ref{e0APA+CJC})
 for all $t\ge0$;
 by (c3), (\ref{e0NJC+BPA=0}) holds for all $t\ge0$.
The remaining equations follow from (a1)--(b3).
\end{proof}

\begin{proof}[Proof of Theorem~\ref{IRE}:]\label{pageproof-IRE}
$1^\circ$ {\em (i)$\THEN$(ii):}
Assume (i), so that $\ALSClL\sbm{I\cr 0}$ solves the \optIRE\
 with $\Ric:=\qCClL^*J\qCClL$,
 by Theorem~\ref{OptIRE},
 in particular, (\ref{e0DJC+BPA=0}) holds.
Now Lemma~\ref{lOptIRE0}(a1)\&(c4) provide us
 (\ref{e0NJC+BPA=0}) and
 an $S$ that completes $\Ric,\qKF$
 to a solution of the IRE.

$2^\circ$ {\em (ii)$\THEN$(i):}
Obviously, a solution of (ii)
 is a solution of the \SIRE, 
 hence (i) follows from Lemma~\ref{lhSIRE}
 and Theorem~\ref{OptIRE}.
(See Lemma~\ref{lOptIRE0} for an alternative proof.)

$3^\circ$ {\em (ii)$\IFF$(iii) and (a1):}
These follow from Lemma~\ref{lhIRE}
 (the term $\gUstar$-stabilizing is defined for the \hIRE\ as for the IRE)
 and the proof of (b)
 (from which we see that if there are minimizing state-feedback pairs,
   then any $J$-optimal pairs are minimizing).

(a2) Uniqueness of $\Ric$ follows from Lemma~\ref{lOptCost0}(a)
 and the rest from $1^\circ$ and $2^\circ$.

(b) This follows from Lemma~\ref{lOptCost0}(d).

(c) This is obvious (see $1^\circ$).

(d) (Recall Lemma~\ref{lOptCost0}(c).)
The ``iff'' holds because,
 by equation (9.175)
 (note: in (the last line of) Proposition 9.10.2(b3), one should 
   the assumptions of (b4) (and apply (b1) in the proof)),
 we have $\p{\qD u,J\qD\qX^{-1}\eta}
 =\p{\qX u,S\eta}$ for all $\eta\in\Lc^2(\R_+;U)$ and $u\in\gUstar(0)$
 (hence S is one-to-one iff only $u=0$ is $J$-optimal for $x_0=0$).
Formula (\ref{eAllqKF}) follows from Lemma~\ref{lAllqKF} %
 (and we get $E^{-*}SE^{-1}$ in place of $S$). %

(e) (For $\gUstar=\dUexp$ also the converse holds, by Proposition 9.9.12
 of [M02].) %
Fix $t>0$.
We have $\qXt\in\cG\BL(\L^2([0,t);U))$ (with inverse $\qMt$) 
 and $\qSt=\qXt^*S\qXt$.
If $\PTO\in\cG\BL$, then $\qSt\in\cG\BL$,
 by Lemma~\ref{lqScoercive}(a),
 hence $S\in\cG\BL(U)$
 (since $\qXt\in\cG\BL(\L^2([0,t);U))$ (with inverse $\qMt$), 
 and $\qSt=\qXt^*S\qXt$).
Similarly, $\PTO\gg0\ \THEN\ S\gg0$ (cf.\ Lemma~\ref{lSgg0}).
(Note that the proof of Lemma~\ref{lqScoercive} --- 
 actually, the whole Section~\ref{sJcoerc} ---
 is independent of this section).
\end{proof}

\begin{lemma}[\SIRE\ \&$\qKF \TTHEN$IRE]\label{lSIREqKF=IRE} %
The admissible solutions of the IRE, \hIRE, \SIRE, \hSIRE, \optIRE\
 and \hoptIRE\ are the same.
\end{lemma}

This means that if $\qKF$ is an admissible state-feedback pair for $\ALS$
 and $(\Ric,\qKClL)$ solves the \SIRE\ (or the \optIRE) for all $t>0$
 (or the \hSIRE\ or the \hoptIRE for some $s=z\in\C_\omega^+$), %
 then there is $S\in\BL(U)$ s.t.\ $(\Ric,S,\qKF)$ is a solution of the IRE
 for all $t>0$ (and of the \hIRE). %
Conversely, if $(\Ric,S,\qKF)$ is an admissible solution of the IRE
 (or of the \hIRE), then $(\Ric,\qKClL)$ solves the \SIRE, \hSIRE,
 \optIRE\ and \hoptIRE.

\begin{proof} %
(By Lemma~\ref{lhSIRE} and Theorem~\ref{OptIRE},
 the \SIRE, \hSIRE, \optIRE\ and \hoptIRE\ are equivalent. 
By Lemma~\ref{lhIRE}, so are the IRE and the \hIRE.)

Since an admissible solution of the IRE is obviously one of the \SIRE,
 it suffices to prove the converse. 
Let $(\Ric,\qKClL)$ be an admissible solution of the \SIRE.

Discretization of $\qK$ and $\qX:=I-\qF$
 yields a solution of (14.10)--(14.12) of [M02] for  
 $\tbm{\qAt&\qBt\cr \qCt&\qDt}$ for a fixed $t>0$ (see p.~816 of [M02]),
 hence for $nt$, $n\in1+\N$. %
By dediscretizing, from (14.11) we obtain that
 $(\qX^{nt})^* S^t\qX^{nt}=\qS^{nt}$,
 i.e., $(\qX^{-nt})^* \qS^{nt}\qX^{-nt}=S^t$ on $[0,nt)$, for any $n\in\N$,
 where $S^t u:=\sum_{k=0}^\infty \tau^{-k}S_t\tau^{k}\piOt u$
 and $S_t\in\cG\BL(\L^2([0,t);U))$ is the operator in (14.11).
Obviously, $\|S^t\|_{\BL(\L^2(\R_+;U))}=\|S\|$
 and $\tau^{-nt}S^t=S^t\tau^{-nt}\ \all n\in\N$.

Since the same holds with $t/m$ in place of $t$, for any $m\in1+\N$,
 the corresponding 
 we have $\pi_{[0,t)}S^{t/m}=\pi_{[0,t)}(\qX^{-mt/m})^* \qS^{mt/m}\qX^{-mt/m}
 =\pi_{[0,t)}S^t$, hence $S^t=S^{t/m}$,
 hence 
 $\tau^{-nt/m}S^t=S^t\tau^{-nt/m}\ \all n,m\in1+\N$,
 hence $\tau^{-T}S^t=S^t\tau^{-T}\ \all T\ge0$, by continuity.
By %
  Lemma 2.1.3 of [M02],
 $S^t$ has a unique extension to an element of $\TIC(U)$.
By continuity, $(\qX^{T})^*S^t\qX^T=\qS^T\ \all T>0$.
Since $S^t=(S^t)^*$ and $\piti S^t\piOt=0$,
 it follows from Lemma 2.3.2 of [M02]
 that $S^t\in\BL(U)$; thus, $(\Ric,S^t,\qKF)$ solve the IRE.
\end{proof}

In Lemma~\ref{lGenSpFproof} we shall prove the remaining
 part of Theorem~\ref{GenSpF}.
For the lemma, we need the following auxiliary result:
\begin{lemma}[Generalized SpF]\label{lGenSpF} %
Let $\oqK$ %
 be a control in WPLS form for $\ALS$
 and $\Ric=\Ric^*\in\BL(H)$.
Assume the \SIRE\ (\ref{eSIRE}) (or \optIRE) for all $t>0$. %
Then the following are equivalent
 (for this fixed $\Ric$):
\begin{itemlist}
  \item[(i)] There is a solution of the IRE (\ref{eIRE}).
  \item[(ii)] Problem (\ref{eXGenSpF}) %
 has a solution $\hqX\in\H^\infty_\infty(U)$, $S=S^*\in\BL(U)$
 on some right half-plane. %

  \item[(iii)] There are $\qX\in\TIC_\infty(U),\ S=S^*\in\BL(U)$
 satisfying $\qSt=\qXt^*S\qXt$ for all $t>0$.
\end{itemlist}

Moreover, the following hold:
\begin{itemlist}
  \item[(a)] The solutions (if any) of (i), (ii) and (iii) are the same
 (set $\qK:=\qX\oqK$, $\qF:=I-\qX$,
 or, conversely, $\qX:=I-\qF$).

  \item[(b)] If (ii) holds, $S$ is one-to-one and $\hqX\in\cG\H^\infty_\infty(U)$,
 then $\qKF$ is an admissible state-feedback pair for $\ALS$
 and $\oALS=\ALSClL\sbm{I\cr 0}$.
\end{itemlist}
\end{lemma}

\begin{proof}[Proof of Lemma~\ref{lGenSpF}]
(Actually, it would suffice to assume the \SIRE\ and (iii)
 on any unbounded subset of $[0,\infty)$,
 since then it would still hold for all $t\ge0$,
 as one observes from $3^\circ$--$4^\circ$ below.
By Lemma~\ref{lhSIRE}, the \SIRE\ and the \optIRE\ are equivalent.)

$1^\circ$ {\em (i)$\THEN$(iii):} This is trivial.

$2^\circ$ {\em (iii)$\THEN$(i):}
Set $\qK:=\qX\oqK,\ \qF:=I-\qX$ to obtain (i) from the \SIRE\ 
 (because $\qSt\oqKt=\qXt^*S\qKt,\ \oqKt^*\qSt\oqKt=\qKt^*S\qKt$,
 by (\ref{eXSX=})).

$3^\circ$ {\em (iii)$\THEN$(ii):}
(Note that this would follow from Lemma~\ref{lhIRE} 
 (and (\ref{ehXSX=}))
 if we assumed that $\qKF$ extends $\ALS$ to another WPLS.)
As in the proof of Lemma~\ref{lqScoercive}(b),
 we observe that (\ref{eXGenSpF}) holds for all $s\in\C_\omega^+$,
 where $\omega\ge\omega_A$ is s.t.\ $\qX\in\TIC_\omega$. %

$4^\circ$ {\em (ii)$\THEN$(iii):}
Let $\omega>\omega_A$ be s.t.\ $\hqX\in\H^\infty_\omega$.
Let $u,v\in\W^{1,2}_\omega(\R_+;U)$\label{pageW12}
  (i.e., $u,u'\in\L^2_\omega(\R_+;U)$ and $u(t)=u(0)+\int_0^t u'(r)\,dr\ \all t>0$;
 similarly for $v$).

Set $g_1(t):=(\qB\tau^\cdot v)'(t)=\qB\tau^t v'\in\L^2_\omega$
 (by Theorem 3.1.5 of [M02], since $\qB\tau\in\TIC_\omega$). 
Then $\hg_1(s)=s(s-A)^{-1}B\hv(s) - 0$, by
 Lemma~\ref{lhf'} and Lemma~\ref{lCharFct}(d). %
Set $f_1(t):=\Ric\qB\tau^t u$, $f_2(t):=(\Ric\qB\tau^\cdot u)'(t)$,
 $g_2(t):=\qB\tau^t v$,
 so that 
 \begin{equation} 
   \p{\qBt u,\Ric \qBt v}_H=\int_0^t\p{\qBt u,\Ric \qBt v}_H'(t)\,dt
  = \int_0^t\left(\p{f_1(t),g_1(t)}_H + \p{f_2(t),g_2(t)}\right)\,dt.
 \end{equation}
Set $F:=\sbm{-J\qD u\cr S\qX u},\ G:=\sbm{\qD v\cr \qX v}$
 to obtain from Lemma~\ref{lhfhg=0fg=0}
 (for which it suffices to have (\ref{eXGenSpF}) on $\C_\alpha\pois\C_\beta$)
 that $\p{F,G}_{Y\times U}=\p{f,g}_{H\times H}$ a.e.
Take $\int_0^t$ of both sides to obtain (\ref{eXSX=})
 (since $\piOt\W^{1,2}_\omega(\R_+;U)=\W^{1,2}([0,t);U)$
  is dense in $\L^2([0,t);U)$, by Theorem B.3.11(b1) of [M02]).

(a) By $1^\circ$--$4^\circ$, any solution of (i), (ii) or (iii)
 is a solution of all of them.

(b) $1^\circ$ {\em Useful equations:}
Since $\pi_{[t,s)}\tau^T=\tau^T\pi_{[t+T,s+T)}$ for all $t,s,T\in\R$,
 we have for all $T,t\ge0$ that
 \begin{align}\label{elIREisWPLS-tod1}
   \piOt\tau^T((\qX^{T+t})^*S\qX^{T+t})\tau^{-T}\piTO
 &= \piOt\qX^* S\tau^T\pi_{[0,T+t)}\tau^{-T}\qX\piTO\\
 &= \qXt^* S\piOt\qX\piTO, %
 \end{align}
 because $\tau^T\pi_{[0,T+t)}\tau^{-T}=\pi_{[-T,t)}$ 
 and $\pi_+\qX^*=\pi_+\qX^*\pi_+$.
Since $\tau^{T+t}\pi_+\tau^{-T}\piTO=\tau^t\piTO$
 and $\piOt\tau^T\pi_+\tau^{-t-T}=\piOt\tau^{-t}$,
 (\ref{elIREisWPLS-tod1}) equals
 (substitute $t+T$ in place of $t$ in (\ref{eXSX=}))
 \begin{equation}\label{elIREisWPLS-tod2}
  \qXt^* S\piOt\qX\piTO = \qDt^* J\piOt\qD\piTO
  +\qBt^*\Ric\qB\tau^t\piTO.
 \end{equation}

From (\ref{eXSK=}) we obtain that
\begin{equation}
  \label{}
  \qXt^*S\qKt\qB= -\qDt^*J\piOt\qC\qB-\qBt^*\Ric\qA^t\qB
 =-\qDt^*J\piOt\qD\pi_- -\qBt^*\Ric\qB\tau^t\pi_-
\end{equation}
 (use 2.\&4. of Definition~\ref{dWPLS0}).
By  (\ref{elIREisWPLS-tod2}), it follows that
\begin{equation}
  \label{elIREisWPLS-tod4}
 - \qXt^*S\qK\qB\piTO =   \qXt^* S\pi_+\qX\piTO.
\end{equation} %

$2^\circ$ {\em We have $-\qK\qB =\pi_+\qX\pi_-$ on $\Lc^2$:} 
This follows from   (\ref{elIREisWPLS-tod4})
 (given $u\in\Lc^2(\R;U)$, choose $T$ s.t.\ $\pi_- u=\piTO u$),
  because $\qXt^*S$ is one-to-one
  (obviously, $\piOt\qX^{-*}\piOt=(\qXt)^{-*}$).

$3^\circ$ {\em  We have $\pi_+\tau^t\qK=\qK\qAt\ (t\ge0)$:}
By $2^\circ$, for each $t\ge0$ we have
\begin{align} %
  \label{}
  \pi_+\tau^t\qK
  &=\pi_+\qX(\pi_++\pi_-)\tau^t\oqK
  = \qX\pi_+\tau^t\oqK+\pi_+\qX\pi_-\tau^t\oqK\\
  &= \qX\oqK\oqA^t-\qK\qB\tau^t\oqK
  =\qK(\oqA^t-\qB\tau^t\oqK)=\qK\qA^t.
\end{align} %

$4^\circ$ By $2^\circ$, we have $-\qK\qB=\pi_+\qX\pi_-$ on $\L^2_\omega$,
 by density (See Theorem B.3.11 of [M02]).
From this and $3^\circ$ we observe that Definion~\ref{dWPLS0}
 is satisfied.
\end{proof}

In Theorem~\ref{GenSpF} we gave six equivalent conditions
 for the IRE. 
Now we shall prove them and give a (partial) seventh one:
\begin{lemma}[$\pmbold{\hqS=\hqX^*S\hqX\ \IFF\ \qKF}$]\label{lGenSpFproof} %
Theorem~\ref{GenSpF} holds.
Moreover, 
 a solution of (ii) is a solution (viii). Conversely,
 a solution of (viii) is a solution of (ii)
 if, e.g., $\qN$ and $\qX^{-1}$ are q.r.c.\ and
 $\gUstar=\dUout$.

 \begin{itemlist}
  \item[(viii)] There are $\qX\in\cG\TIC_\infty(U),\ S\in\BL(U)$ s.t.\ 
 for all $u\in\Lc^2(\R_+;U)$ we have
 $\qN u\in\L^2$ (here $\qN:=\qD\qX^{-1}$), and
 \begin{equation} %
   \label{eNGenSpF}
   (\hqN u_0)^*J(\hqN u_0) = S\ \ \text{a.e.\ on}\ i\R\ \ (u_0\in U).
 \end{equation} 
 \end{itemlist}
\end{lemma}

In (\ref{eNGenSpF}), $\hqN u_0$ denotes the boundary function
 of $\hqN u_0\in\H^2(\C^+;Y)$.
If $\hqD,\hqX\in\Hoo$ (or if $\sigma(A)\cap\cRHP$ is at most countable),
 then (\ref{eNGenSpF})
 is equivalent to $\hqD^*J\hqD=\hqX^*S\hqX$ in $\Loostrong(i\R;\BL(U))$
 (equivalently, a.e.\ on $i\R$ for each $u_0$,
   not necessarily pointwise a.e.\ in $\BL(U)$ unless $U$ is separable;
  see Chapter~3 of [M02] for details).
Condition (\ref{eNGenSpF}) is equivalent to 
\begin{equation}
  \label{eNJN=uSuSpF}
  \p{\qN u,J\qN u}=\p{u,Su}\ \ \ \ (u\in\Lc^2(\R_+;U)), 
\end{equation}
 by the proof of Lemma \ref{lOptIRE0}(c2).

Obviously, for any $\qX\in\cG\TIC_\infty(U)$, $S\in\cG\BL(U)$,
 we get (\ref{eKSK=}) and (\ref{eXSK=})
 (and (\ref{ehKSK=}) and (\ref{ehXSK=})) from the \SIRE\ (and \hSIRE)
 by setting $\qK:=\qX\oqK$;
 the additional condition above is equivalent to the middle equation
 of the IRE (and \hIRE) 
 (which in turn can be used to show that $\qKF$ is an admissible
   state-feedback pair if $\oALS$ is a WPLS).

\begin{proof}[Proof of Lemma~\ref{lGenSpFproof}:]
$1^\circ$ {\em (vi)$\IFF$(v)$\IFF$(i)$\THEN$(ii)\&(viii):} 
The equivalence (vi)$\IFF$(v)$\IFF$(i)$\IFF$(vii)
 follows from Theorem~\ref{IRE}
 and Lemma~\ref{lSIREqKF=IRE}. %
For the rest, assume (i)
 (so that $\ALSopt=\ALSClL\sbm{I\cr0}$ in Theorem~\ref{ALSopt}, by uniqueness).
Then the IRE %
 and the \hIRE\ hold and $S$ is one-to-one,
 by Theorem~\ref{IRE}. %
Claim 
 (ii) %
 follows from (\ref{ehXSX=}),
 and (viii) from Lemma \ref{lOptIRE0}(c2).

$2^\circ$ {\em (ii)$\THEN$(i):}
Assume (ii). %
By Theorem~\ref{ALSopt},
 there is a unique $J$-optimal control $\oqK$ %
 in WPLS form.
Apply Lemma~\ref{lGenSpF} to obtain a solution of the IRE
 (with $\qF=I-\qX,\ \qK=\qX\oqK$).

By Lemma~\ref{lGenSpF}(b) and $2.2^\circ$,
 $\qABKF$ is a WPLS.
Since $\qX\in\cG\TIC_\infty$,
 this means that the pair $\qKF$ is an admissible state-feedback pair
 for $\ALS$;
 since $\oqK=\qX^{-1}\qK$, we have $\oALS=\ALSClL\sbm{I\cr 0}$
 and the pair is $J$-optimal.

$2.2^\circ$ {\em $S$ is one-to-one:} %
As noted below Lemma~\ref{lhSIRE}, the \SIRE\ equals the DARE,
 hence $\qSt$ is one-to-one, by the discretized Theorem 9.9.1(f2) of [M02],
 hence so is $S$, by (\ref{eXSX=}). %

$3^\circ$ {\em (iii)$\IFF$(ii):} This follows from Lemma~\ref{lGenSpF}
 (with $\oALS:=\ALSopt$).

$4^\circ$ {\em (ii)$\THEN$(iv):}
Assume (ii) (on $\C_\alpha^+$;
 the claim on $\C_\alpha^+\pois\C_\beta^+$ follows from $2^\circ$,
 which is otherwise unnecessary).
Increase $\alpha$ if necessary
 (from $1^\circ$ and the $\vartheta$-stability of $\oqK$
   we could deduce that any $\alpha>\max\{\omega_A,\vartheta\}$ will do)
 to have $\qDp,\qXp,\qXp^{-1}\in\TIC_{-\del}$ for some $\del>0$,
 where $\qXp:=\efn^{-\alpha\cdot}\qX\efn^{\alpha\cdot}$
 (i.e., $\hqXp(s)=\hqX(s+\alpha)$).
Then $\hqXp^*S\hqXp=\hqDp^*J_+\hqDp$
 on $i\R$, hence (iv) holds ($S$ is one-to-one by (b)).

$6^\circ$ {\em The claims on (viii):}
By $1^\circ$ above, (i) (hence also (ii) and (iii)) implies (viii).
Assume then that (viii) holds and that $\qN$ and $\qM:=\qX^{-1}$ are q.r.c.

$6.1^\circ$ It obviously follows that $\qM[\L^2(\R_+;U)]\sub\dUout(0)\sub\qM[\L^2(\R_+;U)]$.

$6.2^\circ$ {\em We have $\pi_+\qX\pi_-=\qK\qB$,
 where $\qK:=\qX\qKClL$:} %
Let $u\in\L^2([-T,0);U),\ T>0,\ v\in\dUout(0)$,
 so that $\vc:=\qX v,f:=\piOT\qX\tau^{-T}u\in\L^2(\R_+;U)$.
Set $\qT:=\pi_+\qM\pi_-\qX$ to obtain that
\begin{align}
  \p{\qD(u+\qT u),J\qD v}
   &= \p{\qN\pi_-\qX u,J\qN\vc}
   = \p{\qN f,J\qN\tau^{-T}\vc}\\\label{eGenSpF-c2b}
   &= (2\pi)^{-1}\p{\hqN \hf,J\hqN\shat{\tau^{-T}\vc}}\\
  &= (2\pi)^{-1}\p{\hf,S\shat{\tau^{-T}\vc}}
   = \p{f,S\tau^{-T}\vc}=0,
\end{align}
 since $\tau^{-T}\vc$ is supported on $[T,+\infty)$.
Given $u\in\Lc^2(\R_-;U)$, we have $\p{\pi_+\qD(u+\qT u),\qD v}=0$
 for all $v\in\dUout(0)$.
But $\pi_+\qD=\qC\qB u$, hence $\qT u$ must be the
 unique $J$-optimal control for $x_0:=\qB u$,
 i.e., $\qT u=\qKClL\qB u$
 (we have $\qT u\in\dUout(x_0)$,
 because $\qT u\sub\qM\Lc^2\sub\L^2$ and $\qC x_0+\qD\qT u=\qN\pi_-\qX u\sub\qN\Lc^2\sub\L^2$,
  because $\pi_-\qX u\in\Lc^2$).
Consequently, $\qK\qB u=\qX\qT u=-\pi_+\qX\pi_- u$
 (because $\qT=\pi_+\qM\pi_-\qX\pi_- = \pi_+ I\pi_- - \pi_+\qM\pi_+\qX\pi_-
  = -\pi_+\qM\pi_+\qX\pi_-$). %

$6.3^\circ$ {\em Claim (i) holds:}
Set $\qF:=I-\qX$, so that $\pi_+\qF\pi_-=\qK\qB$
 (on $\Lc^2$, hence on $\L_\omega^2$ for $\omega$ big enough,
  by continuity),
 and $\qK\qAt=\qX\oqK(\oqA^t-\qBt\oqK)=...=\pi_+\tau^t\qK$
 (see (8.56) of [M02]), hence $\qABKF$ is a WPLS.
Obviously, by using $\qKF$ for $\ALS$, we get $\oALS=\ALSClL\sbm{I\cr 0}$.

{\em Remarks:}
1. A similar claim holds for any $\gUstar\sub\dUout$.
\\2. By Example 9.13.2 of [M02], condition (viii) is not sufficient
 without an additional assumption connecting $\qM$
 to $\gUstar$, to $\Ric$ or to $\qKClL$.

$5^\circ$ {\em (iv)$\THEN$(i):}
Define ``the extended shifted systems'' $\ALSp$ and $\ALSopt^+$ as follows:
\begin{equation}
  \label{eALSp} %
 \qABCDp %
  := \bsysbm{\efn^{-\alpha\cdot}\qA\| \qB\efn^{\alpha\cdot}\crh
  \efn^{-\alpha\cdot}\qC\| \efn^{-\alpha\cdot}\qD\efn^{\alpha\cdot}\cr
      \efn^{-\alpha\cdot}\qA\| \efn^{-\alpha\cdot}\qB\tau\efn^{\alpha\cdot}},
 \ \  %
 \bsysbm{\qAopt^+\crh \stackrel{\phantom{X}}{\qCopt^+}\cr
                   \stackrel{\phantom{X}}{\qKopt^+}}
  :=\efn^{-\alpha\cdot}
       \bsysbm{\qAopt\crh \qCopt\cr \qAopt\cr \qKopt}
\end{equation}
(These two systems %
 equal $\ALS$ and $\ALSopt$ with the third row added and
  $A$ replaced by $A-\alpha$ and $\Aopt$ by $\Aopt-\alpha$,
  by Lemma 6.2.9(c) of [M02].)
These systems are exponentially stable, since $-\alpha,\omega_A-\alpha<0$.
Set $r:=\alpha$.
Since $\ALSopt$ is $J$-optimal and $\Ric=\qCopt^*J_+\qCopt$,
 equations (\ref{e0yJyL2r}) and (\ref{e0ya=-2rxa}) %
 hold, 
 by Theorem~\ref{OptIRE}(f). %
But (\ref{e0yJyL2r}) is exactly $\Ric=(\qCopt^+)^*J\qCopt^+$,
 and (\ref{e0ya=-2rxa}) is exactly $0=(\qCopt^+)^*J_+\qD_+$,
 which means that $\ALSopt^+$ is $J_+$-optimal for $\ALSp$
 over $\dUexpp=\dUoutp$
 (by Theorem~\ref{OptIRE}(c)). 
(By $5.1^\circ$, it is the only one.)

$5.1^\circ$ {\em Uniqueness over $\dUoutp(x_0)$:}
By Lemma~\ref{lOptCost0}(ii),
 a control $u\in\dUoutp(0)$ is $J_+$-optimal for $0$ iff
 $0=\p{\qDp \eta,J_+\qDp u}=\p{\qXp \eta,S\qXp u}\ (\eta\in\L^2(\R_+;U))$
 (recall that $\qDp^*J_+\qDp=\qXp^*S\qXp$),
 i.e., iff $\qXp^*S\qXp u=0$.
Since $S$ is one-to-one, this implies that $u=0$.
By Lemma~\ref{lOptCost0}(c), %
 it follows that there is at most one $J_+$-optimal control
 over $\dUoutp(x_0)$ for each $x_0\in H$.

$5.2^\circ$ Set $\qMp:=\qXp^{-1}\in\cG\TIC(U)$, $\qNp:=\qDp\qMp$.
Trivially, $\qMp u\in\L^2\ \THEN\ u=\qXp\qMp u\in\L^2$,
 hence we can apply condition (viii) of Lemma~\ref{lGenSpFproof} to $\ALSp$
 (recall that (\ref{eNJN=uSuSpF}) implies (\ref{eNGenSpF})
  and note that ``(viii)$\THEN$(i)'' was established above in $6^\circ$)
 to obtain $\ALSoptp$
 in state-feedback form (i.e., a $J_+$-optimal pair $\qKFp$
 over $\dUoutp$;
 by $5.1^\circ$, we must have $\ALSpClL\sbm{I\cr0}=\ALSoptp$).
Set $\qK:=\efn^{\alpha\cdot}\qKp,\ \qF:=\efn^{\alpha\cdot}\qFp\efn^{-\alpha\cdot}$
 to obtain $\ALSopt$ in state-feedback form,
 i.e., a $J$-optimal state-feedback pair for $\ALS$.

{\em Remark:} %
Assume (iv).
Apply  (\ref{e0SK=}) to $\ALSp$ and let $t\to+\infty$ 
 to obtain that 
\begin{equation}
  \label{eqK+}
  S\qK_+ = -\pi_+(\qN_+)^*J_+\qC_+
  = -\pi_+\efn^{\alpha\cdot}
   \bbm{\qN\cr \qBClL\tau}^* \efn^{-2\alpha\cdot}
   \bbm{J\qC\cr 2\alpha\Ric \qA}. %
\end{equation}
Since $\qK=\efn^{\alpha\cdot}\qK_+$,
 this determines also $\qK$ uniquely
 (recall from (b) that $S$ is one-to-one)
 modulo the constant $E$ mentioned in (b).

(a) This follows from $1^\circ$--$5^\circ$ above
 (for any $\alpha>\max\{\vartheta,\omega_A\}$).

(b) This follows from Theorem~\ref{IRE}.
\end{proof}

\NotesFor{Section~\ref{sIREmore}} %
We defined the IRE and presented the corresponding theory
 in Theorem 9.9.1 of [M02];
 that contained Theorem~\ref{IRE} except for (iii).
Lemmas \ref{lNJN=SonL2c} and~\ref{lOptIRE0}
 are from [M02], and many of the computations for the latter are from [S98b]
 (see the notes for Section~\ref{sIRE}).
Otherwise the results seem to be new.

\section{$J$-coercivity}\label{sJcoerc} %

In this section, we present the (generalization to WPLSs of)
 $J$-coercivity, the standard coercivity condition for
 control problems, and derive results that lead to
 the theory of Section~\ref{sStabOpt}.

As explained before Theorem~\ref{J-coercive}, 
 {\em $J$-coercivity} means that the {\em Popov Toeplitz operator} $\PTO:=\qD^*J\qD$\label{pageqS}
 is boundedly invertible $\gUstar(0)\to\gUstar(0)^*$
 (actually, $\PTO=\pi_+\qD^*J\qD\pi_+$, but the condition remains the same
  since $\gUstar(0)\sub\L_\vartheta^2(\R_+;U)$).
If(f) $\cJ(0,u)\ge0\ (u\in\gUstar(0))$, then a control is minimizing
 iff it is $J$-optimal; moreover, then $J$-coercivity is equivalent to
 the existence of $\eps>0$ s.t.\
 \begin{equation}
   \cJ(0,u)\ge \eps \|u\|_{\gUstar}^2\ \ \ (u\in\gUstar(0)).
 \end{equation}
In the general (indefinite) case, the above condition becomes
 more complicated (see (v)) but still
 nicely applicable to $\H^\infty$ control problems
 (see Chapter~11 of [M02]):
\begin{lemma}[J-coercivity]\label{lJcoerc-tD} %
The following are equivalent:
\begin{itemlist}
  \item[(i)] $\qD$ is $J$-coercive;
  \item[(ii)] $\qD^*J\qD\in\BL(\gUstar(0),\gUstar(0)^*)$ is coercive;
  \item[(iii)] $\qD^*J\qD\in\BL(\gUstar(0),\gUstar(0)^*)$ is (boundedly) invertible;
  \item[(iv)] $\qD\raj{\gUstar(0)}$ and $\tJ\raj{\qD[\gUstar(0)]}$ are coercive;
  \item[(v)] There is $\eps>0$ s.t.\ for all nonzero $u\in\gUstar(0)$ %
 there is a nonzero  $v\in\gUstar(0)$ s.t.\
 \begin{equation}
   \label{eJ-coerciveAbs}
   \p{\qD v,J\qD u}_{\L^2}\ge\eps \| u\|_{\gUstar} \| v\|_{\gUstar}.
 \end{equation}
\end{itemlist}

Moreover, if (ii) holds (and $\gUstar(0)\ne\{0\},\tyhja$),
 then $\eps=\|\PTO^{-1}\|_{\BL(\gUstar(0)^*,\gUstar(0))}$ is the
 maximal value of $\eps$ in (v), and $\vartheta\ge0$.
\end{lemma}

Note from Lemma~\ref{lOptCost0}(a)
  that $\|u\|_{\dUout}$ is equivalent to $\max\{\|u\|_2,\|\qD u\|_2\}$
 and $\|u\|_{\dUexp}$ to $\max\{\|u\|_2,\|\qB\tau u\|_2\}$.

\begin{proof}
(A linear map $D:X\to Y$ is {\em coercive} iff there is $\eps>0$ %
 s.t.\ $\|Dx\|\ge \eps\|x\|\ (x\in X)$.
In (ii) and (iii), the symbol $\qD^*$ refers to 
 the adjoint of $\qD\raj{\gUstar(0)}$, hence
 $\p{v,\qD^*J\qD u}:=\p{\qD v,J\qD u}$.

$1^\circ$ {\em ``(i)$\IFF$(iii)$\THEN$(ii)$\IFF$(v)$\IF$(iv)'':}
In Theorem~\ref{J-coercive} we used (iii) as the definition.
Obviously, (ii) follows from (iii),
 and (ii) is equivalent to (v).
Similarly, (iv) implies (v)
 ($\p{\qD v,J\qD}\ge \eps' \|\qD v\|\|\qD u\|
  \ge \eps'(\eps'')^2\|u\|_{\gUstar}\|v\|_{\gUstar}$).

$2^\circ$ {\em ``(ii)$\THEN$(iii)\&(iv)'':}
Define $\dY:=\qD[\gUstar(0)]\sub\L^2(\R_+;Y)$, 
 let $P$ be the orthogonal projection $\L^2(\R_+;Y)\to\bar\dY$,
 $\tJ:=PJP^*\in\BL(\bar\dY)$,
 $D\in\BL(\gUstar(0),\bar\dY)$,
 so that $D^*\in\BL(\bar\dY,\gUstar(0)^*)$. %
Assume (ii), i.e., that $D^*\tJ D$ is coercive.
Then so are $D$ and $\tJ$, 
 hence $\dY=\bar \dY$ and $D\in\BL(\gUstar(0),\dY)$ is
 an isomorphism onto, hence invertible.
Being self-adjoint and coercive, also $\tJ$ is invertible
 (see A.3.5(c2) and A.3.4(N5) of [M02]).
Thus, (iii) and (iv) hold.

(Note: by the above, $\gUstar(0)$ is a Hilbert space
 (and can thus be identified with its dual when (i) holds;
 of course, $\dUexp$ and $\dUout$ have natural inner products
 even without $J$-coercivity.)

$3^\circ$ {\em On $\eps=\|\PTO^{-1}\|$:}
Obviously, $\inf_{u\ne0}\|\PTO u\|/\|u\|=\|\PTO^{-1}\|$,
 hence $\eps$ cannot be any larger.
Conversely, there is $\La\in\gUstar(0)^{**}$ s.t.\ $\|\La\|\le1$ %
 and $\La \PTO u=\|\PTO u\|_{\gUstar(0)^*}$. %
By reflexivity (which obviously follows from (iii)), 
 we have $\La =\p{v,\cdot}$ for some $v\in\gUstar(0)$.
Obviously, $\|v\|=\|\La\|=1$.
Thus, $\p{v,\PTO u}=\|\PTO u\|\ge \eps\|u\|\|v\|$
 for $\eps:=\|\PTO^{-1}\|$.

$4^\circ$ {\em $\vartheta\ge0$:}
Whenever $\gUstar(0)\ne\{0\}$ and $\|\qD u\|_2\ge\eps\|u\|_{\L^2_\vartheta}$
 for some $\eps>0$ and each $u\in\gUstar(0)$ %
 we have $\vartheta=0$,
 because otherwise
 $\|\qD u\|_2=\|\tau^{-t}\qD u\|_2=\|\qD\tau^{-t}u\|_2
  \ge \eps\|\tau^{-t}u\|_{\L^2_\vartheta}
  = \eps\efn^{t\vartheta}\|u\|_{\L^2_\vartheta}\to+\infty$,
 as $t\to+\infty$
 (we have $\tau^{-t}u\in\gUstar(0)\ (t\ge0)$, by Lemma~\ref{ltaugUstar}).
\end{proof}

Many special cases of $J$-coercivity are 
 commonly used in the study of finite-dimensional
 systems, Pritchard--Salamon systems or other special cases of WPLSs.
Therefore, we now recall from [M02] that 
 for, e.g., systems having
 smoothing semigroups or bounded input operators,
 $I$-coercivity over $\dUexp$, i.e., condition
 \begin{itemlist}
   \item[(i)] $\|\qD u\|_2\ge \eps(\|u\|_2+\|\qB\tau u\|_2)\ (u\in\dUexp(0))$,
 \end{itemlist}
 is equivalent to classical coercivity assumptions:
\begin{theorem}[$\PTO\gg0$]\label{Assumptions} %
Assume that  $J\gg0$ and that the state-FCC is satisfied.
{\bf (a)}
If
 $B\in\BL(U,H)$,
 then also any of (ii)--(vii)
 is equivalent to (positive) $J$-coercivity over $\dUexp$:
\begin{itemlist}
 \item[(ii)] $D^*D\gg0$, and $\|\qD u\|_2\ge \eps \| \qB\tau u\|_2$ %
 for some $\eps>0$ and all $u\in\dUexp(0)$; %
 \item[(iii)] $(ir-A)x_0=Bu_0 \TTTHEN \|\Cw x_0+D u_0\|_Y\ge\eps(\|x_0\|_H+\|u_0\|_U)$
  for some $\eps>0$ and all $x_0\in H,\ u_0\in U,\ r\in \R$; %
 \item[(iv)] $D^*D\gg 0$,
 and $(ir-A)x_0=Bu_0 \TTTHEN \|\Cw x_0+D u_0\|_Y\ge\eps\|x_0\|_H$ %
  for some $\eps>0$ and all $x_0\in H,\ u_0\in U,\ r\in \R$; %
 \item[(v)] %
 $\|\sbm{A-ir &B\cr \Cw& D}\sbm{x_0\cr u_0}\|_{H\times Y} %
 \ge\eps\|\sbm{x_0\cr u_0}\|_{H\times U}$ %
  for some $\eps>0$ and all $r\in\R,\ x_0\in H,\ u_0\in U$;

  \item[(vi)] $D^*D\gg0$,
 and there is a unique minimizing $u\in\dUexp(x_0)$ %
     for each $x_0\in H$; %

  \item[(vii)] $D^*D\gg0$,
 and the \BwARE\ (p.~\pageref{pageBwARE}) 
 has an exponentially stabilizing solution.
\end{itemlist}

{\bf (b)}
If $\qA B\in\L^1([0,1];\BL(U,H))$, $C\in\BL(H,Y)$ and
 ($D^*JC=0$ or $D^*JD\in\cG\BL(U)$),
 then (i)--(vii) are still equivalent
 (in (vii) we must have $\Bw^*$ in place of $B^*$
   and require that $\Ric[H]\sub\Dom(\Bw^*)$).

{\bf (c)}  %
Assume that $\qD$ is ULR.
If $B$ is not maximally unbounded
 or  $\qA B\in\L^1([0,1];\BL(U,H))$,
 then (i)--(v) are equivalent (and imply (vi)). %
\end{theorem}

(The proof is given on p.~\pageref{pageproof-Assumptions}.
Condition (v) is called ``no invariant zeros''.
If $\|(A-ir)x_0+Bu_0\|_H<\infty$, then $x_0\in\Dom(\Cw)$
 (since here $\qD$ is regular),
 as noted below Definition~\ref{dReg}.
See Proposition 10.3.2 of [M02] for more general systems and results.)

Similarly, $I$-coercivity over $\dUout$ 
 (i.e., $\|\qD u\|_2\ge\eps\|u\|_2\ (u\in\dUout(0))$)
 is a generalization of several
 classical assumptions, such as ``no transmission zeros''
 (Proposition 10.3.1(a) of [M02]).
Next we prove Theorems \ref{J-coercive} and~\ref{ALSopt}:

\begin{proof}[Proof of Theorem~\ref{J-coercive}:]\label{pageproof-J-coercive}
(From that of %
 Theorem 8.2.5 of [M02].)

$1^\circ$ {\em $J$-optimal control:}
(We use the results and notation of the proof of Lemma~\ref{lJcoerc-tD},
 in particular, we identify $\gUstar(0)^*$ with $\gUstar(0)$.)
Let $x_0\in H,\ \tu\in\gUstar(x_0)$.
Set $\ty:=\qC x_0+\qD\tu$,
 $v:=-(D^*\tJ D)^{-1}D^*\tJ\ty\in\gUstar(0),\ u:=\tu+v\in\gUstar(x_0),
 \ y:=\qC x_0+\qD u$.
Then 
\begin{equation}
  \label{}
  \p{y,JD\eta}_{\L^2}
  =\p{\ty+D v,\tJ D\eta}_{\L^2}
  =\p{D^* \tJ\ty+D^*\tJ D v,\eta}_{\gUstar(0)}=0
\end{equation}
 for all $\eta\in\gUstar(0)$,
 hence $u$ is $J$-optimal for $x_0$.

$2^\circ$ {\em Uniqueness:}
The difference of two $J$-optimal controls 
 for any $x_0$ is $J$-optimal for $0$,
 hence we can assume that $x_0=0$.
If $u$ is $J$-optimal for $x_0=0$,
 then $\p{D v,JD u}=0$ for all $v\in\gUstar(0)$,
 hence then $\|u\|_{\gUstar}=0$, by (v), hence $u=0$.

$3^\circ$ {\em Case $\PTO\ge0$:}
This follows from Lemma~\ref{lOptCost0}(iii).
\end{proof}

\begin{proof}[Proof of Theorem~\ref{ALSopt}:]\label{pageproof-ALSopt}
(From that of Theorem 8.3.9 of [M02].)

$1^\circ$ {\em $\ALSopt$ is a WPLS:}
Let $x_0\in H$, $t\ge0$.
We first show that $\pi_+\tau^t\qKopt x_0$ is $J$-optimal for $\qAopt^t x_0$,
 i.e., equal to $\qKopt \qAopt^t x_0$:
For $\eta\in\gUstar(0)$ we have $\tau^{-t}\eta\in\gUstar(0)$, hence
 \begin{equation}
   \label{etauCopt1}
 \p{J\pi_+\tau^t \qCopt x_0,\qD \eta}_{\L^2}   
   = \p{J\qCopt x_0,\qD \tau^{-t} \eta}_{\L^2}=0
   \ \ \ (\eta\in\gUstar(0)).
 \end{equation}
But
\begin{align}
  \pi_+\tau^t \qCopt x_0&=\pi_+\tau^t \left(\qC x_0+\qD\qKopt x_0\right)
   = \qC\qA^t x_0+\pi_+\qD(\pi_++\pi_-)\tau^t\qKopt x_0\\
   &= \qC\qA^t x_0+\qD\pi_+\tau^t\qKopt x_0
     +\qC\qB\tau^t\qKopt x_0
    = \qC\qAopt^t x_0+\qD\pi_+\tau^t\qKopt x_0.
\end{align}
This and (\ref{etauCopt1}) imply that $\pi_+\tau^t\qKopt x_0$
 is $J$-optimal for $\qAopt^t x_0$; thus
 \begin{equation}
   \label{eKopt=KA}
 \pi_+\tau^t\qKopt x_0=\uopt(\qAopt^t x_0) =\qKopt \qAopt^t x_0,\ \ 
 \pi_+\tau^t\qCopt x_0=\yopt(\qAopt^t x_0) =\qCopt \qAopt^t x_0.
 \end{equation}
By the dynamic programming principle, $\qA$ is a semigroup;
 a detailed proof of this fact goes as follows, using (\ref{eKopt=KA}):
\begin{align}
  \label{}
  \qAopt^s\qAopt^t
  &= \qA^s(\qA^t  + \qB\tau^t\qKopt) + \qB\tau^s\qKopt\qAopt^t\\
   &= \qA^s\qA^t  + \qB\tau^s\pi_-\tau^t\qKopt + \qB\tau^s\pi_+\tau^t\qKopt
   = \qA^s\qA^t +\qB\tau^{s+t}\qKopt
    =\qAopt^{t+s}.
\end{align}
Obviously, $\qAopt^0=\qA^0=I$,
 and $t\mapsto\pi_-\tau^t u$ is continuous $\R_+\to\L^2_\omega$
 for any $u\in\Lloc^2$, hence
 $\qAopt x_0=\xopt(x_0)$ is continuous for each $x_0\in H$.
Therefore, $\qAopt$ is a $C_0$-semigroup.
This and (\ref{eKopt=KA}) imply that $\ALSopt$ is a WPLS.

$2^\circ$ {\em The rest:} 
The claims on $\Ric$ are obvious.
The continuity of $\qCopt$
 (and $\qKopt:H\to\L^2_\vartheta(\R_+;U)$)
 follows from the closed-graph theorem
 and the exponential stability of $\ALSopt$
 from that of $\qAopt$ (if $\gUstar=\dUexp$).
(See Theorem 8.3.9 of [M02] for details and further results.)
\end{proof}

We shall soon need the following simple fact:
\begin{lemma}\label{lTgeT-1} %
Assume that $0\le T\in\cG\BL(H)$ and set $\eps:=\|T^{-1}\|^{-1}$.
Then $T\ge \eps I$.\noproof
\end{lemma}

(See Lemma A.3.1(b1') of [M02] or use a spectral decomposition 
 or square root of $T$.) %

Naturally, the FCC is necessary for $\Ric,\ \ALSopt$ and $\qSt$ to exist.
Next we show that if the FCC holds and $\PTO\in\cG\BL$,
 then also the ``truncated Popov operators'' $\qSt\label{pageqSt2}:=\qDt^*J\qDt+\qBt^*\Ric\qBt$ are invertible,
 with a uniform (over $t$) norm bound for 
 $(\qSt)^{-1}$ on $\BL(\L^2)$
 (when $\vartheta=0$, as in the case of $\dUexp$, $\dUout$):
\begin{lemma}[\pmbold{$\PTO\in\cG\BL\ \THEN\ \qSt\in\cG\BL$}]\label{lqScoercive}
Assume that $\PTO\in\cG\BL$
 and that $\gUstar(x_0)\ne\tyhja\ \all x_0\in H$.

(a) Then $\qSt\in\cG\BL(\L_\omega^2([0,t);U))$ for all $\omega\in\R,\ t>0$,
 and there are $M_{\omega,t}<\infty$ s.t.\
 \begin{align}
   \label{eStleS}
   \|(\qSt)^{-1}\|_{\BL(\L^2_\vartheta([0,t);U),\L^2_{-\vartheta}([0,t);U))}
  &\le \|\PTO^{-1}\|, \\\label{eMomegatqS}
  \|(\qSt)^{-1}\|_{\BL(\L^2_\omega([0,t);U))}
 &\le
    M_{\omega,t}\|\PTO^{-1}\|\ \ \ \ \ \ (t>0,\ \omega\in\R).
 \end{align}

(b) If $\vartheta=0$ and $\cJ(0,\cdot)\ge0$, %
 then $\hqS(s,s)\ge \eps I$ for $s\in\C_{\omega_0}^+$,
 where $\omega_0:=\max\{0,\omega_A\}$
 and $\eps:=\|\PTO^{-1}\|^{-1}>0$.

(c) If $\cJ(0,\cdot)\ge0$, then $\hqS(s,s)\ge0$ for $s\in\C_{\omega_0}^+$.
\end{lemma} 

All results in Section~\ref{sStabOpt} are based on 
 Theorem~\ref{PosJcKF}, which is a corollary of (b)
 (i.e., of  ``$\PTO\gg0 \ \THEN\ \hqS(s,s)\ge \eps I$'')
 and of Theorem~\ref{GenSpF}.

\begin{proof}
W.l.o.g., we assume that $U\ne\{0\}$.
Let $\oqK$ be the (unique) $J$-optimal control in WPLS form.
Let $t>0$. 

(a) %
Assume that $u\in\L^2([0,t);U)\pois\{0\}$. %
Choose $v$ for $\eps:=\|\PTO^{-1}\|$ and $P^t u$
 as in (\ref{eJ-coerciveAbs}).
Since $\|P^t u\|_{\gUstar}\ge\|P^t u\|_{\L^2_\vartheta}\ge \|u\|_{\L^2_\vartheta([0,t);U)}$,
 we obtain from
 Lemma \ref{lqSt=PTO.Pt}(a) %
 that
\begin{equation}
  \p{v,\qSt u}_{\L^2}=\p{v,\PTO P^t u}_{\gUstar(0),\gUstar(0)^*}
  \ge\eps\|P^t u\|\|v\|_{\gUstar}
 \ge \eps\|u\|_{\L^2_\vartheta([0,t);U)}\|v\|_{\L^2_\vartheta([0,t);U)}.
\end{equation}
Since $t,u,v$ were arbitrary,
 we get (\ref{eStleS}),
 which obviously implies (\ref{eMomegatqS}).
(b) {\em We have $\hqS(s,s)\ge\eps I$:}
Let $s\in\C_{\omega_0}^+$ ,
 and set $u(t):=\efn^{st}u_0$, 
 so that $\pi_- u\in\L^2\cap\L^2_\omega$, and
 \begin{equation}
   \label{eqDu=hqD(s)u0}
 (\qD u)(t)=\efn^{st}\hqD(s) u_0,\ \ \ \ 
 \qB\tau^t u=\efn^{st}(s-A)^{-1}B u_0, \ \ \ \ \ \ \ (t\in\R),
 \end{equation}
 by Lemma 6.10 of [S98c]. 
By time-invariance %
 (a similar computation was used for losslessness in Lemma 6.11 of [S98c]),
\begin{align}
  \p{\qD^t\tau^{-t}u,J\qDt\tau^{-t}u}
  &=\p{\tau^{-t}\qD\pitO u,J\tau^{-t}\qD\pitO u}
  =\int_{-\infty}^0\p{\qD\pitO u,J\qD\pitO u}_Y\,dr\\
  &\to \int_{-\infty}^0\p{\qD u,J\qD u}_Y(r)\,dr
  = \int_{-\infty}^0 \efn^{r(s+\bar s)}\p{\hqD(s) u_0,J\hqD(s) u_0}_Y\,dr\\
  &= \p{\hqD(s) u_0,J\hqD(s) u_0}_Y /2\re s,
\end{align}
 as $t\to+\infty$, because $\pi_-\qD\piit u\to0$ in $\L^2_\omega$
 (because $\qD\in\TIC_\omega$ and $\pi_- u\in\L^2_\omega$),
 hence in $\L^2$ too (because $\pi_-\L^2_\omega\sub\L^2$ continuously), 
 for any $\omega\in(\omega_0,\re s)$.
Therefore, 
\begin{equation} %
  \label{ehqSgeeps5}
  \p{\tau^{-t} u,\qSt \tau^{-t}u}
 \to \p{\hqD(s) u_0,J\hqD(s) u_0}_Y /2\re s 
    + \p{(s-A)^{-1}B u_0,\Ric(s-A)^{-1}B u_0}, %
\end{equation}
 as $t\to+\infty$.
But, by Lemma~\ref{lTgeT-1}, $\qSt\ge\eps I$ on $\L^2([0,t);U)$,
 hence
\begin{equation} %
  \label{ehqSgeeps6}
  \p{\tau^{-t}u,\qSt\tau^{-t}u}
 \ge\eps \int_0^t\|u_0\|^2 \efn^{2(r-t)\re s}\,dr
  \to \eps \|u_0\|^2 /2\re s.
\end{equation}
Since $u_0\in U$ was arbitrary,
 we obtain from (\ref{ehqSgeeps5}) and (\ref{ehqSgeeps6})
 that $\hqS(s,s)\ge\eps$. 

(c) {\em $\hqS(s,s)\ge0$:}
The proof of (b) applies mutatis mutandis.
\end{proof}

\NotesFor{Section~\ref{sJcoerc}} %
The important Lemma~\ref{lqScoercive} seems to be completely new.
Most of the rest we presented in [M02].
See p.~\pageref{pagesOptnotes} for further notes
 and Sections 8.4 and 10.3 of [M02] for further results.

\section{Remaining proofs}\label{sStabOptproofs} %

In this section we give the remaining proofs,
 i.e., %
 those on AREs and those for the theorems of Section~\ref{sStabOpt}.
We start with two auxiliary lemmas.

If part of $J$ is uniformly positive, the corresponding
 part of $\qCClL,\qDClL$ becomes stable:
\begin{lemma}[\pmbold{$(\qC_2)_{\ClL},(\qD_2)_{\ClL}$} are stable]\label{lC2ClLstable} %
Assume that $\ALS=\ssysbm{\qA\|\qB\crh \qC_1\|\qD_1\cr \qC_2\|\qD_2}$,
 and $J=\sbm{J_{11}&0\cr 0&J_{22}}\ge0$, $J_{22}\gg0$.
If the IRE has an admissible solution $(\Ric,S,\qKF)$ 
 with $\Ric\ge0$,
 then $S\ge0$, and $(\qC_2)_\ClL$ and $(\qD_2)_\ClL$ are stable.
\end{lemma}

\begin{proof}
(This is a variant of Proposition 10.7.3 of [M02]. %
Note that we have assumed that $\qCD$ and $J$ have been split
 according to some split $Y=Y_1\times Y_2$.)

Set $\qN:=\qDClL$. 
By (\ref{e0S=NJN+BPB}), %
 for any $t\ge0$ we have 
 $\piOt S= (\qN_1^t)^*J_{11}\qN_1^t + (\qN_2^t)^*J_{22}\qN_2^t+\qBClLt^*\Ric\qBClLt\ge0$,
 hence $S\ge0$ and $\|S\|\|u\|_2^2\ge\|\piOt J_{22}^{1/2}\qN_2 u\|_2^2$
 for all $u\in\L^2(\R_+;U),\ t\ge0$.
Let $t\to\infty$ to obtain that $J_{22}^{1/2}\qN_2$ is bounded
 $\L^2\to\L^2$, hence so is $\qN_2:=(\qD_2)_\ClL$.

Similarly, from (\ref{e0APA+CJC}) %
 we observe that
 $\Ric\ge [\piOt(\qC_2)_\ClL]^*J_{22} \piOt(\qC_2)_\ClL$,
 hence $(\qC_2)_\ClL$ is stable.
\end{proof}

By Theorem~\ref{IRE}(e), the uniform positivity of the Popov
 operator implies that of the signature operator
 ($\PTO\gg0 \ \THEN\ S\gg0$).
We stated that the converse holds for $\gUstar=\dUexp$;
 in fact, it also holds for $\gUstar=\dUout$
  provided that the solution is ($\gUstar$- and) I/O-stabilizing:
\begin{lemma}[\pmbold{$\PTO\gg0 \IFF S\gg0 \TTHEN$ q.r.c.}]\label{lSgg0} %
Assume that the IRE has a $\gUstar$-stabilizing solution $(\Ric,S,\qKF)$
 with $S\gg0$.
Then {\bf (a)} $\qX\in\BL(\gUstar(0),\L^2(\R_+;U))$
 and $\cJ(0,u)=\p{\qX u,S\qX u}\ \all u\in\gUstar(0)$.
Assume also that $\gUstar=\dUout$.
Then {\bf (b)} $\sbm{\qN\cr \qM}v\in\L^2\ \THEN v\in\L^2$
 (for all $v\in\Lloc^2(\R_+;U)$).
Finally, if also $\qN,\qM\in\TIC$,
 then {\bf (c)} $\qN,\qM$ are q.r.c.,
 $\PTO\gg0$ and $\qX\in\cG\BL(\dUout(0),\L^2(\R_+;U))$.
\end{lemma}

This result was applied in Theorem~\ref{QRCF-Uout}.

\begin{proof}
(Recall that $\qX:=I-\qF$, $\qM:=\qX^{-1}$, $\qN:=\qD\qM$.)

(a)
$1^\circ$ Let $v=\qX u$, where $u\in\gUstar(0)$.
Then, by (\ref{eXSX=}), we have 
\begin{equation}
  \p{v,S\piOt v}=\p{\qDt u,J\qDt u}+\p{\qBt u,\Ric\qBt u}.
\end{equation}
Let $t\to+\infty$ to observe that $\int_0^\infty\p{v(t),Sv(t)}_U\,dt
 =\p{\qD u,J\qD u}_{\L^2}=\cJ(0,u)$, by $2^\circ$.
Since $S\gg0$, we conclude that $v\in\L^2$.
Thus, (b) holds (note that $\cJ(0,u)\le \|J\|\|\qD u\|^2_2
 \le \|J\|\|u\|_{\dUout}^2$, hence $\qX$ is continuous). %

$2^\circ$ {\em $\p{\qBt u,\Ric\qBt u}\to0$:}
Set $x_0:=\qBt u$, $\tu:=\pi_+ \tau^t u\in\gUstar(x_0)$
 (Lemma~\ref{ltaugUstar}).
Then (recall 4. of Definition \ref{dWPLS0})
\newcommand{\samakuin}{{\rm -"-}}
\begin{align}
  \p{\qC x_0+\qD\tu,J\samakuin}
 &=\p{\qC\qB\tau^t u + \qD\pi_+\tau^t u,J\samakuin}
 =\p{\pi_+\qD\pi_-\tau^t u + \pi_+\qD\pi_+\tau^t u,J\samakuin}\\
 &=\p{\pi_+\qD\tau^t u,J\samakuin}
 =\p{\pi_+\tau^t\qD u,J\samakuin}
 =\p{\piti\qD u,J\samakuin}. %
\end{align}
But $\p{x_0,\Ric x_0}$ is the minimum
 of $\|\qC x_0+\qD\tu\|_2^2$ over $\tu\in\gUstar(x_0)$,
 hence $\p{x_0,\Ric x_0}\le\p{\qD u,\piti J\qD u}\to0$,
 as $t\to+\infty$. %
(b)
Obviously, $\sbm{\qN\cr \qM}v\in\L^2
 \IFF \sbm{\qD\cr I}\qM v\in\L^2 \IFF \qM v\in\dUout(0)$.
By (a), $\qM v\in\dUout(0)\ \THEN\ \qX \qM v\in\L^2(\R_+;U)$.
But $v=\qX\qM v$.

(c) Now $\qN,\qM$ are q.r.c., and we have $v\in\L^2
 \IFF \sbm{\qN\cr \qM}v\in\L^2\IFF \qM v\in\dUout(0)$, by (b).
Consequently, $\dUout(0)=\qM[\L^2(\R_+;U)]$,
 and $\qX:\dUout(0)\to\L^2(\R_+;U)$ is thus (boundedly) invertible.
Let $S\ge \eps I$, $\eps>0$.
By the proof of (a) and the above, we have 
\begin{equation}
\p{u,\PTO u}:=\cJ(0,u)
 =\p{\qX u,S\qX u}\ge \eps\|\qX u\|_2^2
 \ge \eps\eps'\|u\|_{\dUout}^2,
\end{equation}
 i.e., $\PTO\ge\eps\eps' I\gg0$,
 for some $\eps'>0$.
\end{proof}

The proof of Theorem~\ref{ARE} was based on the following equivalence:
\begin{lemma}[ARE $\IFF$ WR IRE]\label{lAREIRE} %
Assume that $\qD$ is WR.
Then the WR solutions of the ARE are exactly the solutions of the \hIRE\ 
 for which $\qF$ is WR and $F=0$. %
\end{lemma} 

In the proof we also show that ``and $F=0$'' can be removed
 if, in (\ref{eARE}), $S$ is replaced by $X^*SX$ and $SK$ by $X^*SK$,
 where $X=I-F$ %
 (i.e., $\hqX(s)=X-\Kw(s-A)^{-1}B$);
 we call that variant the {\em extended ARE}. 
In Theorem~\ref{ARE} this corresponds to accepting WR $J$-optimal
 state-feedback pairs
 instead of merely WR $J$-optimal state-feedback operators.

\begin{proof}[Proof of Lemma~\ref{lAREIRE}:]\label{pageproof-lAREIRE}
{\em Remark:} 
In this proof we also show that %
 all solutions 
 of the extended ARE are exactly
 all ``WR'' (meaning that $H_B\sub\Dom(\Kw)$) solutions of the IRE
 except that for the solutions of this extended ARE
 we have to add the requirement $H_B\sub\Dom(K_\w)$
 (this requirement is redundant if $S\in\cG\BL(U)$,
  because in $2.2^\circ$ we show that 
 $H_B\sub\Dom((\Bw\Ric)_\w)$,
  which implies that  $H_B\sub\Dom(SK_\w)$ %
 (because $H_B\sub\Dom(C_\w)$ because $\qD$ is WR)),
 and that we do not know whether $\qK$ and $\qF$ are well-posed
 (this is not a problem, since it is implicitly required
  by saying that $(\Ric,S,\qKF)$ is a solution of the IRE
  or that $(\Ric,S,K)$ is a WR solution of the ARE). %
This fact was originally shown in Proposition 9.8.10 of [M02],
 with an alternative, time-domain proof.

$1^\circ$ {\em \hIRE$\THEN$ARE:}
Multiply (\ref{ehXSK=}) by $(z-A)$ to the right
 and then let $s\to+\infty$ to obtain (\ref{eARE3}) on $\Dom(A)$
 (note that $\hqX(+\infty)=I-\hqF(+\infty)=0$,
   and that $(s-A)^{-1}x_0\to0$ in $\Dom(A)$ for all $x_0\in H$,
    by Lemma A.4.4(d3) of [M02]).
Let first $s\to+\infty$ and then $z\to+\infty$
 in (\ref{ehXSX=}) to obtain (\ref{eARE2}). 

{\em Remark:} The limits of the $B^*$-terms below exist since so do the others.

$2^\circ$ {\em ARE$\THEN$\hIRE:}
Now $\hqX(s)=X-\Kw V_s$, $\hqX(s)^*=X^*-V_s^* K$
 $(s\in\C_{\omega_A}^+)$, by Lemma~\ref{lCharFct}(c) and regularity,
 where $X:=I$, $V_s:=(s-A)^{-1}B$, $V_s^*:=\Bw^*(s-A)^{-*}$.
(This explicit $X$ makes it easier to follow the computations and
 allows us to prove the more general result (``extended ARE'' or eCARE)
 given in [M02] and mentioned below Lemma~\ref{lAREIRE}.) %
Naturally, $\hqD(s)=D+\Cw V_s$.

$2.1^\circ$ {\em (\ref{ehXSK=}):}
Multiply (\ref{ehXSK=}) by $z-A$ to the right to obtain
\begin{align} %
  X^*SK-V_s^*K^*SK = -D^*JC-V_s^*(C^*JC+s^*\Ric+\Ric A).
\end{align}
Use (\ref{eARE1}) to obtain
 $X^*SK+D^*JC=-V_s^*(s^*\Ric-A^*\Ric)=-\Bw^*\Ric$,
 which is true, by (\ref{eARE3}).

$2.2^\circ$ {\em (\ref{ehXSX=}):}
We have $[I-r(r-A)^{-1}](z-A)^{-1}=A(r-A)^{-1}(z-A)^{-1}
  = (r-A)^{-1}[I-z(z-A)^{-1}]$,
 hence $[\Bw^*\Ric-(\Bw^*\Ric)_\w](z-A)^{-1}B
 =\wlim_{r\to+\infty}\Bw^*\Ric (r-A)^{-1}[I-z(z-A)^{-1}]B
 =\wlim_{r\to+\infty}\Bw^*\Ric (r-A)^{-1}B$,
 which exists, by (\ref{eARE2}),
 hence so does the weak limit $(\Bw^*\Ric)_\w(z-A)^{-1}B$.
By (\ref{eARE2}) and the above,
\begin{equation}
  \label{eBw-Bww3}
  \Bw^*\Ric V_z -(\Bw^*\Ric)_\w V_z =X^*SX-D^*JD,
\end{equation}
 where $V_z:=(z-A)^{-1}B$ %
 (an alternative proof is given in Lemma 9.11.5(a) of [M02]).
Apply (\ref{eARE3}) to $r(r-A)^{-1}x_0$ and let $r\to+\infty$ to
 obtain
 \begin{equation}
   \label{eXSKwHB2}
 D^*JC_\w x_0+X^*SK_\w x_0=-(\Bw^*\Ric)_\w x_0
 \end{equation}
 (in particular, $(\Bw^*\Ric)_\w x_0$ exists)
 for all $x_0\in\Dom(\Cw)\cap\Dom(\Kw)$.
Subtract the left side of (\ref{ehXSX=}) from the right
 and use (\ref{eXSKwHB2}) and its dual
 to obtain
 \begin{align}\label{eregre-tod22d1}
 D^*JD-X^*SX - (\Bw^*\Ric)_\w V_z - [(\Bw^*\Ric)_\w V_s]^*  
   + T %
 \end{align}  %
 where $T:=V_s^*[(z+\bar s)\Ric + C^*J\Cw -K^*S\Kw]V_z$,
 $V_s^*:=\Bw^*(s-A)^{-*}$. %
But $[(z+\bar s)\Ric + C^*J\Cw -K^*S\Kw]
  := \wlim_{r\to+\infty}
 [(z+\bar s)\Ric + C^*JC -K^*SK]r(r-A)^{-1}$,
 and $[\cdots]=[(z+\bar s)\Ric - A^*\Ric - \Ric A^* ]
  = [(s-A)^*\Ric+\Ric(z-A)]$, by (\ref{eARE1}),
 hence 
 \begin{equation}
   T=\wlim_{r\to+\infty}[\Bw^*\Ric r(r-A)^{-1}(z-A)^{-1}B
   + \Bw^*(s-A)^{-*}\Ric r(r-A)^{-1}B]
 = (\Bw^*\Ric)_\w V_z + (\Bw^*\Ric V_s)^*.
 \end{equation}
Thus, (\ref{eregre-tod22d1}) becomes 
  $D^*JD-X^*SX- [(\Bw^*\Ric)_\w V_s]^*  
    + (\Bw^*\Ric V_s)^*=0^*=0$,
 by (\ref{eBw-Bww3}).
\end{proof}

Above we also showed the following:
\begin{cor}[ARE $\IFF$ \hIRE]\label{cAREIRE} %
Any solution $(\Ric,S,K)$ of the ARE
 having $S\in\cG\BL(U)$ %
 satisfies the \hIRE\ (\ref{ehIRE}).
 with $\hqX(s):=I-\Kw(s-A)^{-1}B$.
If, in addition, $S\gg0$,
 then $\qABKF$ is a WR WPLS.\noproof %
\end{cor}

(Use Lemma~\ref{lhIRESgg0} for the last claim;
 weak regularity ($H_B\sub\Dom(\Kw)$) was shown above.)

\begin{proof}[Proof of Lemma~\ref{lBboundedARE}:]\label{pageproof-lBboundedARE}
Since $B$ is bounded, now $\qD$ (and $\qF$ if any) is ULR
 and a control in WPLS form is necessarily
 given by a state-feedback pair, %
 by Lemmata 6.3.16(b) and 8.3.18 of [M02].
In particular, \SIRE\ or \hSIRE\ implies the IRE and the \hIRE,
 by Lemma~\ref{lSIREqKF=IRE},
 hence the ARE, by  Lemma~\ref{lAREIRE}.
Conversely, if $(\Ric,S,K)$ is a WR solution of the ARE, 
 then it is admissible (because $\qX$ is ULR and $\hqX(+\infty)=0$),
 hence we obtain the \hIRE\ (hence \SIRE\ and \hSIRE) from Lemma~\ref{lAREIRE}.
If $D^*JD\in\cG\BL$, then any solution of the ARE is WR,
 by Lemma 6.3.17 of [M02].
\end{proof}

\begin{proof}[Proof of Theorem~\ref{MTICinftyKF}:]\label{pageproof-MTICinftyKF}
$0^\circ$ We shall use the following assumptions,
 all of which will be established in the proof of Remark~\ref{rcA}:

$0.1^\circ$ $\BL\sub\cA\sub\TIC$,
 i.e., 
 $\BL(H_1,H_2)\sub\cA(H_1,H_2)\sub\TIC(H_1,H_2)$
 for all Hilbert spaces $H_1,H_2$,

$0.2^\circ$ $\cA$ is closed w.r.t.\ addition, composition, inversion,
 scalar multiplication and added stability ($\C\cA^{-1}+\cA\cA\sub\cA$,
 and
 $\efn^{\omega \cdot}\qD\efn^{-\omega\cdot}\in\cA$ for all
  $\omega<0,\ \qD\in\cA$).

($0.1^\circ$ and $0.2^\circ$ imply that $\BL\sub\cA_\infty\sub\TIC_\infty$
 and that $\C\cA_\infty^{-1}+\cA_\infty\cA_\infty\sub\cA_\infty$,
 because the ``$\omega$-shift'' commutes with these operations.)

$0.3^\circ$ $\cA$ is closed w.r.t.\ spectral factorization,
 by Theorem~\ref{SpF1}.

$0.4^\circ$ The maps in $\cA$ are {\em UR} (uniformly regular),
 i.e., we have $\|\hqE(s)-\hqE(+\infty)\|\to0$, 
  as $s\to+\infty$, %
 for all $\qE\in\cA$ (hence for all $\qE\in\cA_\infty$).

$0.5^\circ$ In (a2) we also use the following:
1.  $\cA=\cA^\rmd$.
2.  %
If $g_1,g_2\in\L^1(\R;\BL)$ and $E\in\BL$,
 then $\qE:=f*=g_1*(E+g_2*)$, where $f:=g_1E+g_1*g_2\in\L^1(\R;\BL)$,
 and $\qE_+=f_+*\in\cA$, 
 where $(\qE_+ u)(t) := (\qE \piit u)(t)$,
 $f_+:=\khi_{\R_+}f$. %

$0.6^\circ$ We have $\|\qE(ir)-\qE(+\infty)\|\to0$, as $|r|\to\infty$,
 for all $\qE\in\cA$, by the Riemann--Lebesgue Lemma [M02]. %

(a1) By taking $\alpha$ big enough, we have $\qDp\in\cA$, hence $\qXp,\qXp^{-1}\in\cA$
 in the proof of Theorem~\ref{PosJcKF}
 (by $0.3^\circ$).
It follows that $\qX,\qM\in\cA_\infty$. %
Therefore, $\qF=I-\qX$, $\qFClL=\qM-I$,
 $\qN=\qDClL=\qD\qM$, $\qBClL\tau=\qB\tau\qM$ (see (\ref{eALSb}))
 are in $\cA_\infty$, by $0.2^\circ$.

By Theorem~\ref{GenSpF}(iv), we have $\qXp^*S\qXp=\qDp^*J_+\qDp$
 with $\qDp,\qXp\in\cA$, 
 hence $\hqXp^*S\hqXp=\hqDp^*J_+\hqDp$ on $i\R$,
 hence $S=D_+^*J_+D_+=D^*JD$, by $0.6^\circ$.

(a2) %
By (\ref{eqK+}), we have 
\begin{equation}
  \label{e}
  \qK_+^\rmd \tau^t u
 = -\qC_+^\rmd J_+^* \Refl\qN_+S^{-1}\Refl\pi_+\tau^t u
 = -\qC_+^\rmd \tau^t J_+ \Refl\qN_+S^{-1}\Refl\piit u,
\end{equation}
 where $(\Refl\label{pageRefl} u)(t):=u(-t)$. 
Set $\qE_1:=\qC_+^\rmd\tau$. %
The top row of $\qE_1$ is in $\cA$, by the assumption in (a2)
 (see (\ref{eALSp})).
One easily verifies that the bottom row
 of $\qE_1$ equals $f\mapsto \efn^{-\alpha\cdot}\qA^* *f$,
 hence $\qE_1\in\cA$ (increase $\alpha$ if necessary)).
Set $\qE_2:=-J_+\qN_+ S^{-1}\in\cA$
 to observe that
 $(\qK_+^\rmd \tau u)(t)= (\qE_1\Refl\qE_2\Refl\piit u)(t)\ \all t\in\R$,
 so that $\qK_+^\rmd\tau\in\cA$,
 by $0.5^\circ$
 (because $\qE_2=E+h* \TTHEN \Refl\qE_2\Refl = E+h(-\cdot)*$),
We conclude that $\qK^\rmd\tau\in\cA_\infty$
 (since $\qK=\efn^{\alpha\cdot}\qK_+$, as noted below (\ref{eqK+})).
Since $\qKClL=\qM\qK$, $\qCClL=\qC+\qN\qKClL$,
 the remaining claims follow from this and (a1).

(a3) 
Apply Lemma~\ref{lcAsystem} to $\ALSClL$ to get ``$\in\cA_\omega$''.
The latter claim follows from Theorem~\ref{ALSopt}.

(b) %
Since $\qX$ is UR, we have $X\in\cG\BL(U)$
 (Proposition 6.3.1(b1) of [M02]).
Therefore, we can choose $\qKF$ so that $F=0$,
 i.e., so that $K$ is a UR %
 $J$-optimal state-feedback operator, by (\ref{eAllqKF}).
This leads to the ARE (\ref{eARE}), by Theorem~\ref{ARE}.
Obviously, Theorem~\ref{ARE}(i) implies Theorem~\ref{GenSpF}(i),
 hence the equivalence holds.
From $1^\circ$ of the proof of Theorem~\ref{ARE}
 we observe that the limit converges in norm to $S-D^*JD$.

(c) $\|\hf(r+i\cdot)\|_\infty\le \|\efn^{-r\cdot}f\|_1\to0$, as $r\to+\infty$.
\end{proof}

Before proving the main result, Theorem~\ref{PosJcKF},
 we explain how it was obtained.
As mentioned above, Theorem~\ref{ALSopt} has already been known 
 in the positive case. 
Our contribution was 1. to find the necessary and sufficient conditions
 in Theorem~\ref{GenSpF}, particularly the 
 ``spectral factorization condition'' (iv);
 2. to show (Lemma~\ref{lqScoercive}(b))
 that if the 
 Popov Toeplitz operator is uniformly positive
 ($\PTO\ge\eps I$),
 then so is the
 ``shifted Popov function'' ($\hqS(\alpha+i\cdot,\alpha+i\cdot)\ge\eps I$),
 so that the condition (iv) is satisfied
 by the standard positive spectral factorization result 
 (Theorem~\ref{SpF1}(a)). See $2^\circ$ below for details.
\begin{proof}[Proof of Theorem~\ref{PosJcKF}:]\label{pageproof-PosJcKF}
$1^\circ$ {\em ``If'':}
If $\qKF$ is $J$-optimal, then $\qKClL x_0\in\gUstar(x_0)\ \all x_0$. %

$2^\circ$ {\em ``Only if'':}
Assume the FCC, so that the assumptions of Theorem~\ref{GenSpF}
 are satisfied, by Theorem~\ref{J-coercive}.
By Lemma~\ref{lqScoercive}(b), we have $\hqS(s,s)\ge\eps I$
 on $\C^+_{\omega_0}$.
Fix some $\alpha>\omega_0$ to conclude that
  $\hqDp(ir)^*J_+\hqDp(ir)=\hqS(\alpha+ir,\alpha+ir)\ge\eps I\ \all r\in\R$,
 i.e., that $\qDp^*J_+\qDp\ge\eps I$.
Consequently, 
 there is a spectral factorization $\qD_+^*J_+\qD_+ =\qX_+^*S\qX_+$
 (i.e., $S\in\cG\BL(U),\ \qX_+\in\cG\TIC(U)$), by Theorem~\ref{SpF1}.
Thus, Theorem~\ref{GenSpF}(iv)\&(i) imply that there is a $J$-optimal
 state-feedback pair $\qKF$ for $\ALS$ over $\gUstar$,
  with $\qF=I-\qX$. %
\end{proof}

\begin{proof}[Proof of Theorem~\ref{QRCF-Uout}:]\label{pageproof-QRCF-Uout}
$1^\circ$ {\em (iii)$\THEN$(ii)$\THEN$(i):} This is trivial.
$2^\circ$ {\em (i)$\THEN$(iii):}
Assume (i). 
The map $\tqD:=\sbm{\qD\cr I}$ is $I$-coercive over $\dUout^{\tALS}$
 (because $\p{\tqD u,I\tqD u}=\|\tqD u\|_2^2=\|\qD u\|_2^2+\|u\|_2^2$;
  here $\tALS=\ssysbm{\qA\|\qB\crh\tqC\|\tqD}$, $\tqC:=\sbm{\qC\cr 0}$).
Therefore, we can apply Theorem~\ref{PosJcKF}
 to obtain
 an $I$-optimal (over $\dUout^{\tALS}$)
 state-feedback pair $\qKF$ for $\tALS:=\sbm{\ALS\cr \ssysbm{0\|I}}$;
 let $\Ric$ be the corresponding solution of the IRE
 (see Theorem~\ref{IRE}; then $\Ric=\tqCClL^*I\tqCClL\ge0$, $S\gg0$);
 let $\tALSClL$ be the corresponding closed-loop system of $\tALS$
 and $\ALSClL$ that of $\ALS$ (Definition~\ref{dAdmKF0}).
By Lemma~\ref{lC2ClLstable}, 
 the maps $\tqCClL$ and $\tqDClL$ %
 are stable.
Since $\tqDClL=\sbm{\qDClL\cr \qM}$,
 where $\qN:=\qDClL:=\qD\qM$, $\qM:=(I-\qF)^{-1}$,
 the maps $\qN,\qM$ are stable.
Similarly, $\qCClL$ and $\qKClL$ are stable, hence $\ALSClL\in\SOS$
 (here $\ALSClL$ refers to $\ALS$ under $\qKF$).
By, e.g., Lemma~\ref{lOptIRE0}(c2),
 we have $\tqDClL^*I\tqDClL = S$. %
Let $E:=S^{1/2}$ and apply (\ref{eAllqKF})
 to normalize $S=\tqDClL^*\tqDClL=\qN^*\qN+\qM^*\qM$ to identity.
By Lemma~\ref{lSgg0}, $\qN,\qM$ are q.r.c.
\end{proof}

\begin{proof}[Proof of Corollary~\ref{cUout-B1}:]\label{pageproof-cUout-B1}
$1^\circ$ We add a copy of $u$ to the output, i.e., 
 we define a WPLS $\ALS^e$ on $(U\times W,H,U\times Y)$ by
 setting
 $\qA^e:=\qA,\ \qB^e:=\bsysbm{\qB&\qH}$,
 $\qC^e:=\sbm{0\cr \qC},\ \qD^e:=\sbm{I&0\cr \qD&\qG}$.
Set $\qQR:=0,\ \qRR=\bbm{0&I},\ \sZ:=\{0\},\ \uZ:=\L^2,\ \vartheta=0$
 to have 
 \begin{equation}
   \dU_*^{\ALS^e}(x_0)=\{\sbm{u\cr 0}\in\L^2(\R_+;U\times W)\I
         y:=\qC x_0+\qD u\in\L^2\}
  =\dUout(x_0)\times\{0\}.
 \end{equation}

With $J^e:=I\in\BL(U\times Y)$
 we get the cost function $\cJ^e(x_0,\sbm{u\cr 0})=\|y\|_2^2+\|u\|_2^2$,
 where $u^e=\sbm{u\cr w}$ is the input and
 $y^e=\sbm{u\cr y}$ the output of $\ALS^e$.
Since $\|\sbm{u\cr 0}\|_{\dU_*^{\ALS^e}}
 =\max\{\|u\|_2,\|\sbm{u\cr y}\|_2,\|0\|_{\sZ}\}$, 
 we have  $\PTO^e\gg0$,
 and $\Ric^e$ and $\qK_{\ClL1}^e:=\qK_{\ClL}^e\sbm{I\cr 0}$
 are the same as $\Ric$ and $\qKClL$
 in the proof of Theorem~\ref{QRCF-Uout}, respectively,
 (by the uniqueness of the optimal control $\qK_{\ClL1}^e$),
 and $\qK_{\ClL2}^e=0$
 (since $\qKClL^e x_0\in\dU_*^{\ALS^e}(x_0)\ \all x_0$).

As in the last paragraph of the proof of Corollary~\ref{cExpStabB1},
 we see that $\bbm{0&I}\qM^e\in\BL(U\times W,W)$
 and that we can have $\shat{\qM^e}(\alpha)=\sbm{\hqM(\alpha)&0\cr 0&I}$,
 so that $\qM^e_{11}=\qM$ (being unique modulo constant, by (\ref{eCharFct}),
   because $\qB_\ClL^e\sbm{I\cr0}=\qBClL$ and $\qK_{\ClL1}^e=\qKClL$),
 hence  $\bbm{0&I}\qF^e=\bbm{0&I}(\qM^e)^{-1}=\bbm{0&0}$,
 $\qF^e_{11}=I-(\qM^e_{11})^{-1}=\qF$, 
 $\qK^e_1 =(\qM^e_{11})^{-1}\qK_{\ClL1}^e=\qM^{-1}\qKClL=\qK$
 and $\qK_2^e=0$, as required.
By Lemma~\ref{lC2ClLstable}, $\ALSClL^e$ is SOS-stable,
 hence so is $\tALSClL$ (being a subset of $\ALSClL^e$).

$2^\circ$ {\em Q.r.c.:}
The two first columns of the resulting closed-loop system $\tALSClL$ 
 equal $\ALSClL$ extended by $\bsysbm{0\|0}$; %
 in particular, $\qN:=\qDClL=\tqDClL\sbm{I\cr 0}$ and $\qM=\tqM_{11}$ are q.r.c.\
 (by the choice of $\qKF$).
But $\TIC\owns\tqM=(I-\tqF)^{-1}=\sbm{\qM&\qM\qE\cr 0&I}$ and
 $\TIC\owns\tqN=\bsysbm{\qD&\qG}\tqM = \bbm{\qN&\qGClL}$,
 where $\qGClL=\qG+\qN\qE$.
If $\tqM \sbm{u\cr w},\tqN\sbm{u\cr w}\in\L^2$, then $w,\qM(u+\qE w)\in\L^2$,
 hence $w,\qM u\in\L^2$ (since $\qM\qE\in\TIC$),
 and $\qN u+\qGClL w\in\L^2$, hence $\qN u\in\L^2$ (since $\qGClL\in\TIC$),
 hence $u\in\L^2$ (since $\qM,\qN$ are q.r.c.),
 hence $\tqM,\tqN$ are q.r.c.
\end{proof}

\begin{proof}[Proof of Theorem~\ref{JointlyUout}:]\label{pageproof-JointlyUout}
(By Theorem~\ref{QRCF-Uout}, these pairs $\qKF$ or $\qHG$
 (or any output-stabilizing pairs)
 exist iff $\ALS,\ALS^\rmd$ satisfy the output-FCC.)

$1^\circ$ Choose $\qE$ as in Corollary~\ref{cUout-B1}
 to make $(\ALStotal)_L$ SOS-stable 
 (since it is contained in $\tALSClL$).

$2^\circ$ {\em $\tqN,\tqM$ are l.c.}
Since the maps in (\ref{eS4A:4.4}) are the inverses of each
 other, we observe that $\bbm{\tqN&\tqM}\sbm{-\qY_1\cr \qX_1}=I$.

$3^\circ$ {\em ``Moreover I/O-'':}
Actually, we have shown above that if $\qHGd$ is any (by $1^\circ$)
 {\em I/O-stabilizing} (i.e., one that makes the I/O map
 $(\qD^\rmd)_\ClL$ of $(\ALSd)_\ClL$ stable)
 state-feedback pair for $\ALSd$,
 then $\tqN,\tqM$ are l.c.
By duality, the ``moreover'' claim holds.

$4^\circ$ {\em D.c.f.;  $(\ALStotal)_\tL^\rmd$ is SOS-stable:}
By $3^\circ$, we have the r.c.f.\ and l.c.f.\ $\qD=\qN\qM^{-1}=\tqM^{-1}\tqN$.
By Lemma 4.3(iii) of [S98a],
 we can find $\tqX,\tqY\in\TIC$ that complete them (and $\qX,\qY$)
 to a d.c.f.
(Given any $\tqX_0,\tqY_0\in\TIC$ for which $\tqX_0\qM-\tqY_0\qN=I$, set
 $\tqY:=\tqY_0+(\tqX_0\qY_1-\tqY_0\qX_1)\tqM,\
  \tqX:=\tqX_0+(\tqX_0\qY_1-\tqY_0\qX_1)\tqN$.) %

But the inverse of $\sbm{\qM&\qY\cr \qN&\qX}$
 in $\TIC_\infty(Y\times U)$ is given in (\ref{eS4A:4.4})
 and it is unique, hence %
  also the maps $\qF_\tL,\qE_\tL$
 must be stable ($\in\TIC$).
We conclude that also $(\ALStotal)_\tL^\rmd$ is SOS-stable
 (its output map equals that of $(\ALSd)_\ClL$).

$5^\circ$ {\em Externally stabilizing:} %
We complete the proof by showing that
 any jointly admissible pairs $\qKF,\qHG$ that make 
 $(\ALStotal)_L$  and $(\ALStotal)_\tL^\rmd$ SOS-stable
 actually make them externally stable
 (i.e., that also $\qB_L,\qH_L,\qC_\tL,\qK_\tL$ are stable;
  it obviously also follows that $\ALSClL$ is externally stable).
Now (cf.\ (6.170) and (6.171) of [M02])
\begin{equation}
  \qB_L=\qB\qM=\qB_\tL\qM-\qH\tqM\qN
 =\qB_\tL\qM-\qH_\tL\qN
\end{equation}
 is stable.
Therefore, $\qH_L\tqM=\qH_\tL+\qB_L\qE_\tL$ is stable,
 and so is $\qH_L\tqN=\qB_\tL + \qB(\qM\qE\tqN-I)$,
 because $\qB(-\qY_1\tqN-I) = \qB(-\qM\tqX)=\qB_L\tqX$;
 consequently, $\qH_L$ is stable
 (since $\tqM,\tqN$ are l.c., by the d.c.f.\ (\ref{eS4A:4.4})).
By duality, also $(\ALStotal)_\tL^\rmd$ is externally stable.

$6^\circ$ {\em ``Moreover, SOS-'':} %
By the above, also any other  (by $1^\circ$)
 SOS-stabilizing $\qHGd$ for $\ALSd$
 is externally stabilizing.
By duality, any SOS-stabilizing $\qKF$ for $\ALS$ is externally
 stabilizing.

$7^\circ$ {\em Final equivalence:} %
We have shown above ``(ii)$\THEN$(i)''
 (the converse is obvious).
The proof of (i)$\THEN$(v)$\THEN$(ii) is obtained as
 in Theorem 7.2.4(c1) of [M02],
 using in place of Theorem 6.7.10(d)(viii)
  the fact that a WPLS is externally stable iff its I/O map
 is stable and the WPLS is input-detectable and output-stabilizable
 (since $\qC=\qCClL-\qD\qKClL$ and similarly for $\qB$).
\end{proof}

\begin{proof}[Proof of Lemma~\ref{lBnotmax}:]\label{pageproof-lBnotmax}
$1^\circ$ The first UR claim is from Lemma 2.5 of [C03] %
 (due to G. Weiss [WC99]). 

$2^\circ$ {\em (iv)$\IFF$(ii)$\THEN$(i)$\IFF$(iii):}
By Theorems \ref{IRE} and~\ref{ARE},
 conditions (i) and (iii), and (ii) and (iv) are equivalent. 
By definition, (ii) implies (i).

$3^\circ$ {\em (i)$\THEN$(ii):}
Assume (i).
By $1^\circ$, the I/O map $\sbm{\qD\cr \qF}$ of $\ALSext$ is UR.
It follows that $X:=\hqX(+\infty)=I-F$ is invertible,
 by Lemma 6.3.1(b1) of [M02],
 hence $F$ can be normalized to zero, by (\ref{eAllqKF}),
 hence (ii) holds.

$4^\circ$ {\em $\wlim=0$:}
Let $s=z\to+\infty$ in (\ref{ehXSX=})
 to obtain that $S=D^*JD$
 (since $2s\|\Ric\|Ms^{-1-2\eps}\to0$).

$5^\circ$ {\em Positively $J$-coercive case:}
This follows from Theorem~\ref{PosJcKF}.

$6^\circ$ {\em Bounded $B$:}
Naturally, (v) is necessary. Conversely, (v) leads to (\ref{eALSopt}),
 and the generator $\Kopt$ of $\qKopt$ is %
 a uniformly line-regular state-feedback operator
 for $\ALS$ (see Lemma 8.3.18 of [M02] for details). 
Since obviously $\ALSClL\sbm{I\cr 0}=\ALSopt$,
 the operator $\Kopt$ is  $J$-optimal.
By continuity, $\wlim=0$.
\end{proof}

We obtain the \hIRE\ once the \hSIRE\ holds at a single point
 (when we use characteristic functions in place of transfer functions
  and do not consider well-posedness):
\begin{lemma}[\hSIRE$\THEN$\hIRE]\label{lhSIREthenhIRE} %
Assume that the \hSIRE\ (\ref{ehSIRE}) holds
 (with $\chqD$ in place of $\hqD$)
 for some $s,z\in\rho(A)$,
 $\Ric=\Ric^*\in\BL(H)$,
 $\hqKopt(z)\in\BL(H,U)$. %

Fix this $z$.
Define $\chqX(s):=I-(z-s)K(s-A)^{-1}(z-A)^{-1}B\in\BL(U)\ \all s\in\rho(A)$,
 $S:=\hqS(z,z)$, $K:=\hqKopt(z)(z-A)\in\BL(\Dom(A),U)$
 to obtain the \hIRE\ (\ref{ehIRE}) for $s=z$
 (replace $\hqD$ by $\chqD$ and $\hqX$ by $\chqX$).
By Lemma~\ref{lIREiffhIREtechnique}(e) and the proof of Lemma~\ref{lhIRE},
 it follows that (\ref{ehIRE}) holds for all $s,z\in\rho(A)$.\noproof
\end{lemma}

In suitably positive problems,
 such as the LQR problem or most other problems of Section~\ref{sStabOpt}, %
 we typically have $S\gg0,\ \Ric\ge0$. %
In this case the maps in Lemma~\ref{lhSIREthenhIRE} are well-posed:
\begin{lemma}[\hIRE$\lland S\gg0 \TTHEN$SOS]\label{lhIRESgg0} %
Assume that $\Ric\ge0$, $S\gg0$ and $K\in\BL(\Dom(A),U)$
 are s.t.\ (\ref{ehIRE}) (the \hIRE) holds
 (use $\chqD$ in place of $\hqD$) %
 for some $s=z\in\rho(A)$
 and some $\hqX(z)\in\BL(U)$.

{\bf (a)}
Then $(\Ric,S,\qKF)$ is a solution of the IRE and the \hIRE\
 and $\qABKF$ is a WPLS,
 where
 $\hqK(s):=K(s-A)^{-1}$
 and $\hqX(s):=\hqX(z)+(s-z)K(z-A)^{-1}(s-A)^{-1}B$
 (for the fixed $z$ and $\hqX(z)$ of the previous paragraph),
 $\qF:=I-\qX$.

{\bf (b)}
Assume, in addition, that $\qC=\sbm{\qC_1\cr 0},\ \qD=\sbm{\qD_1\cr I},\ J=\sbm{*&0\cr 0&*}\gg0$
 for some operators $\qC_1,\qD_1$
 (then the above assumption $S\gg0$ becomes redundant).
If %
  1. $\hqX(s)\in\cG\BL(U)$ for some $s\in\rho_\infty(A)$,
 2. $\dim U<\infty$,
 3. $B$ is not maximally unbounded (or $\qX$ is UR) and $X\in\cG\BL(U)$, 
 or 4. $\qKF$ is admissible,
 then $\qKF$ is SOS-stabilizing. 
\end{lemma}

See Proposition 2.2.5 of [M02] for further sufficient conditions for the
 last claim.
Note from Lemma~\ref{lCharFct} that if $\qABKF$ is a WPLS, then 
 the formulas for $\chqK$ and $\chqX:=I-\chqF$ are as in Lemma~\ref{lhIRESgg0}.

The last paragraph of Lemma~\ref{lhIRESgg0} together with
 Lemma~\ref{lhSIREthenhIRE}
 shows that if the LQR-\hSIRE\ 
 has a nonnegative solution at a single point $z=s\in\rho_\infty(A)$,
 then this solution is SOS-stabilizing;
 in particular, then the output-FCC holds
 and there is a smallest nonnegative solution
 (see Corollary \ref{cPosIRE}(c)).
\begin{proof}[Proof of Lemma~\ref{lhIRESgg0}:]
(See Sections 9.12 and~10.7 of [M02] for similar results.)

$1^\circ$ {\em $\qK$:}
By (the proof of) Lemma~\ref{lAPA+CJC},
 we observe that  (\ref{eKSK=}) holds ``on $\Dom(A)\times \Dom(A)$''.
Since $S\gg0$, it follows from (\ref{eKSK=})
 that $K\qA^\cdot:\Dom(A)\to\L^2([0,t);U)$
 extends continuously to $\qKt:H\to\L^2([0,t);U)$.
Obviously, $\qAK$ is (the left column of) a WPLS, hence (\ref{eKSK=}) holds.

$2^\circ$ {\em $\qX$:}
As in Lemma~\ref{lhSIREthenhIRE}, 
 we observe that the \hIRE\ (\ref{ehIRE}) holds for all $s,z\in\rho(A)$.
Fix some $\omega>\max\{\omega_A,0\}$. %
Then the right-hand-side of (\ref{ehXSX=}) is bounded on $\C_\omega^+$
 (since $2\re s\|(s-A)^{-1}B\|^2\le\|\qB\|^2_{\BL(\L^2_\omega,H)}$
  for $s\in\C_\omega^+$, by, e.g., (b3) on p.~176 of [M02]),
 hence so is $\hqX$.
By Lemma 6.3.15 of [M02],
 it follows that $\ssysbm{A\|B\crh -K\| }$
 are the generators of a WPLS $\ssysbm{\qA\|\qB\crh -\qK\|\qX}$
 (where the value of $\hqX$ could be fixed arbitrarily at a single point
  had we not already done it).
The IRE follows from Lemma~\ref{lhIRE}.

(b) 
$1^\circ$
We first show that any of 2., 3. and 4. implies 1.:
4. If $\qKF$ is admissible (i.e., $\qX\in\cG\TIC_\infty(U)$),
 then $\hqX(s)\in\cG\BL(U)$ for each $s$ in some right half-plane.
2. From (\ref{ehXSX=}) we observe that $\hqX(s)^*S\hqX(s)\ge J_{22}\gg0\ \all s\in\rho(A)$;
 this shows the invertibility of $\hqX(s)$ for all $s$
 if $\dim U<\infty$.
3. If $B$ is not maximally unbounded,
 then $\hqX$ is UR %
 and hence
 $X:=\hqX(+\infty)\in\cG\BL(U)$
 implies that $\hqX(s)\in\cG\BL(U)$ for
 real $s$ big enough. %

Thus, we may assume that $\hqX(s_0)$ is invertible for some
 $s_0\in\rho_\infty(A)$; %
 but this leads to $S\ge \hqX(s_0)^{-*}J_{22}\hqX(s_0)^{-1}\gg0$,
 so the assumption $S\gg0$ is now redundant.

$2^\circ$ {\em $\hqX\in\cG\Hoooo$, i.e., $\qKF$ is admissible:}
From $\hqX(s)^*S\hqX(s)\ge J_{22}\gg0$ we deduce
 that $\hqX(s)^{-1}$ is uniformly bounded (wherever it exists).
Since $\rho_\infty(A)$ is connected, it follows that
  $\hqX(s)^{-1}$ exists for all $s\in\rho_\infty(A)$.

$3^\circ$ {\em $\ALSClL$ is SOS-stable:}
From (\ref{e0APA+CJC}) we observe that 
 $\int_0^t \|(\qCClL x_0)(t)\|^2\,dt\le \p{x_0,\Ric x_0}\ \all x_0\in H$,
 hence $\|\qCClL\|^2_{\BL(H,\L^2)}\le\|\Ric\|<\infty$.
But $\qCClL=\qC+\qD\qKClL=\sbm{\qC_1+\qD_1\qKClL\cr \qKClL}$,
 hence also $\qKClL$ is stable.
From (\ref{e0S=NJN+BPB}) we observe that $\qNt$ is uniformly bounded,
 hence $\qN\in\TIC$.
But $\qN:=\qD\qM=\sbm{\qD_1\qM\cr \qM}$, hence $\qM\in\TIC$ too.
\end{proof}

\begin{proof}[Proof of Corollary~\ref{cPosIRE}:]\label{pageproof-cPosIRE} %
Claims (a) and (b) were established on p.~\pageref{pageproof-cPosIREab}.
Most of claim (c) follows from Sections 10.7 and 10.1 of [M02], but
 we give here a self-contained proof.

Since a $\gUstar$-stabilizing solution is admissible
 (and $\Ric\ge0$ since the cost function $\cJ$ is nonnegative),
 the necessity follows from (a) or (b). Below we establish 
 the sufficiency and further claims.

Let $\tALS,\tJ$ denote the system and cost operator whose IRE
 is used in the result under study
 (so $\tALS:=\sbm{\ALS\cr 0&I}$ and $\tJ=I$
  in the proof of Theorem~\ref{QRCF-Uout}, p.~\pageref{pageproof-QRCF-Uout};
  we need do not study its special case, Corollary~\ref{cUout-Meromorphic}).

$1^\circ$ {\em Assume that $(\Ric,S,\qKF)$ is a solution
 of the IRE for $\tALS,I$ with $\Ric\ge0$
 and $\qKF$ admissible for $\tALS$
 (equivalently, to $\ALS$ or to any other extension of $\qAB$)
 or $\dim U<\infty$:}
By Lemma~\ref{lhIRESgg0},
 $S\gg0$ and $\qKF$ is SOS-stabilizing.
By $\tALSClL$ and $\ALSClL$ we denote the closed-loop systems
  corresponding to $\tALS$ and $\ALS$ (under $\qKF$).

$1.1^\circ$ {\em For the output IRE
 (i.e., the one in for Theorem~\ref{QRCF-Uout})}
 we conclude that also $\ALSClL$ is then SOS-stable
 (being contained in $\tALSClL$),
 so the output-FCC holds for $\ALS$.
Thus, we have the sufficiency. 

If $\Ric$ is the $\dUout^{\tALS}$-stabilizing one and
 $\Ric'$ (with some $S',\bsysbm{\qK'&\qF'}$)
 is any other admissible nonnegative solution,
 then $\qKClL'x_0\in\dUout^{\tALS}(x_0)\ \all x_0$
 (being output stabilizing, as noted above),
 and $\Ric'\ge {\hbox{$\tqCClL$}'}^*J{\tqCClL}'$,
 by (\ref{e0APA+CJC}),
 hence, for $u':=\qKClL' x_0,\ \ty':=\tqC x_0+\tqD u'=\hbox{$\tqCClL$}' x_0,
 \ u:=\qKClL x_0$,
 we have
 $\p{x_0,\Ric' x_0}\ge\p{y',Jy'}
   = \cJ(x_0,u')\ge\cJ(x_0,u)=\p{x_0,\Ric x_0}$.
Thus, $\Ric$ is the smallest admissible nonnegative solution.

$1.2^\circ$ {\em For the state IRE (Corollary~\ref{cOptExpstab1}),}
 we have $\tqC=\sbm{\qA\cr 0},\ \tqD=\sbm{\qB\tau\cr I}$,
 hence 
 \begin{equation}
\tqCClL
   :=\tqC+\tqD\qKClL=\bbm{\qA+\qB\tau\qKClL\cr 0}=\bbm{\qAClL\cr 0},
 \end{equation}
 hence $\qAClL x_0\in \L^2\ \all x_0\in H$,
 hence $\ALSClL$ is exponentially stable, by Lemma~\ref{lExpStable},
 hence $(\Ric,S,\qKF)$ is $\dUexp$-stabilizing for $\ALS$
 (hence unique)
 and the state-FCC holds for $\ALS$.

$2^\circ$ {\em The ARE:}
By Corollary~\ref{cAREIRE}, any nonnegative solution of the ARE
 solves the \hIRE\ (\ref{ehIRE}) and
 makes $\qABKF$ a WR WPLS with $X=I$ (i.e., $\hqF(+\infty)=I-X=0$).
Thus, if $\dim U<\infty$ or $B$ is not maximally unbounded or $\qX$ is UR,
 then $\qKF$ is SOS-stabilizing (hence admissible),
 by Lemma~\ref{lhIRESgg0},
 so the sufficiency follows from $1^\circ$.
\end{proof}

\begin{proof}[Proof of Corollary~\ref{cBARE}:]\label{pageproof-cBARE} 
Set $\tqC:=\sbm{\qC\cr \qA\cr0},\ \tqD:=\sbm{\qD\cr \qB\tau\cr I}$
 to make the output of $\tALS:=\ssysbm{\qA\|\qB\crh\tqC\|\tqD}$
 equal to $(y;x;u)$, where $y$ is the output of $\ALS$,
 both under the initial state $x_0$ and input $u$.
Obviously, the equation (\ref{eBARE}) equals (\ref{eARE}) for $\tALS$
 and $J:=\diag(Q,T,R)$;
 $\tALS$ is positively $J$-coercive;
 and $\tALS$ is exponentially detectable
 (by Lemma 6.6.25 of [M02]), %
 hence estimatable.
Moreover, the FCC condition obviously equals the FCC for $\tALS$
 and $\dUout$.
By Lemma~\ref{lBnotmax}, conditions (i)--(v) are equivalent,
 hence the FCC holds iff (\ref{eBARE}) has a $\dUout$-stabilizing solution.
The rest follows from Corollary~\ref{cPosIRE}(c)
 (for Theorem~\ref{QRCF-Uout})
 applied to $\sbm{Q^{1/2}\qC \cr T^{1/2}\qA }$,
   $\sbm{Q^{1/2}\qD E\cr T^{1/2}\qB\tau E}$ and $R^{1/2} u$
  in place of $\qC$, $\qD$ and $u$, respectively,
 where $E:=R^{-1/2}u$
 (alternatively,  slightly modify its proof (for our different $J$)).
By Theorem~\ref{QRCF-Uout}, $K$ is (q.r.c.-)SOS-stabilizing.
In (a), also $\qAClL x_0\in \L^2\ \all x_0$, hence then
 $K$ is exponentially stabilizing.
\end{proof}

\begin{proof}[Proof of Theorem~\ref{GreatestSol}:]\label{pageproof-GreatestSol} 
We actually show the claim for \SIRE's
 (equivalently, \optIRE's), to obtain a more general claim.
Let $\qKClL$ and $\tqKClL$ be controls in WPLS form for $\ALS$
 and let $(\Ric,\qKClL)$, $(\Ric,\tqKClL)$ satisfy the \SIRE.
Assume that $\qAClLt x_0\to0$ as $t\to+\infty$.

$1^\circ$ Compute $(\tqKClLt)^*$(\ref{e3DJC+BPA=0})$+$(\ref{e3APA+CJCmix})
 to obtain
 $\Ric=(\tqAClLt)^*\Ric\qAClLt+(\tqCClLt)^*J\qCClL$.
Exchange $\Ric,\tqKClL$ and $\tRic,\qKClL$ to conclude that
 $\Ric-\tRic^*=(\tqAClLt)^*(\Ric-\tRic)\qAClLt$,
 which converges to zero weakly,
 as $t\to+\infty$,
 if also $\tqAClLt\to0$ strongly.
Thus, $\Ric$ is the unique strongly internally stabilizing
 solution of the \SIRE.

$2^\circ$ 
Assume that $\tqS\ge0$.
Given any $x_0\in H$, set $u:=\qKClL x_0$,
 $\tu:=(\qKClL-\tqKClL) x_0\in\Lloc^2(\R_+;U)$,
 $y:=\qC x_0+\qD u=\tqCClL x_0 + \qD\tu$.
Now $x_T:=x(T)=\qA^T x_0+\qB^T u =\tqAClL^T x_0+\qB^T \tu$,
 hence
 \begin{equation}
   \label{eGsolp-5}
   \p{y,\piOT Jy}=\p{x_0,\tRic x_0}-\p{x_T,\tRic x_T}
 + \p{\tu,\qSt \tu},  %
 \end{equation}
 by (\ref{e3APA+CJC}),
 (\ref{e3DJC+BPA=0}) and (\ref{eS2XSX=}) (with tildes).
But $\p{y,\piOT Jy}\to\p{x_0,\Ric x_0}$
 (by (\ref{e3APA+CJC}))
 and $x_T\to0$, as $T\to+\infty$,
 hence $\qSt\ge0$ implies that $\p{x_0,\Ric x_0}\ge \p{x_0,\tRic x_0}$.
\end{proof}

\begin{proof}[Proof of Theorem~\ref{Assumptions}:]\label{pageproof-Assumptions}
(a) This follows from  Proposition 10.3.2(e2) 
 (ULR from Theorem 9.2.3) of [M02].

(b) This follows from  10.3.2(e1)\&(e2) and 9.2.2(1.)\&(3.)\&(4.) of [M02].

(c) $1^\circ$ {\em Case $\qA B\in\Lloc^1$:} 10.3.2(e2).
$2^\circ$ {\em Case $B$ not maximally unbounded:}
By Corollary~\ref{cPosIRE}(b) and Theorem~\ref{PosJcKF},
 (i) implies that the ARE has a $\dUexp$-stabilizing solution
 with $S=D^*JD\gg0$
 (in fact, also this is an equivalent condition,
 by Proposition 9.9.12 of [M02]), hence $D^*D\gg0$, hence (ii) holds.
The rest follows from 10.3.2(e1)\&(e2) of [M02].
\end{proof}

\begin{proof}[Proof of Corollary~\ref{cMTICinftyKF+}:]\label{pageproof-cMTICinftyKF+} %
(a) Combine (the proof of) Corollary~\ref{cPosIRE}(a)
 with Theorem~\ref{MTICinftyKF}(a1).

(b) This follows from (a)
 (because $\C\cA_\infty^{-1}+\cA_\infty\cA_\infty\sub\cA_\infty$,
  as noted in the proof of Theorem~\ref{MTICinftyKF}),
 as one observes from the proofs of the corollaries and the remark
 (see formulas (6.170) and (6.171) on p.~241 of [M02]
  for the claims on $L$ and $\tL$).

(c) The first claim follows from Corollary~\ref{cPosIRE}(a)
 and Theorem~\ref{MTICinftyKF}(c). %
The other two claims can be proved as the original ones
 (the only exception is that if $B$ is maximally unbounded,
  we obtain the uniform regularity of $\qF$ 
  in the same way as the weak regularity was 
  obtained in the proof of Corollary~\ref{cPosIRE}
  (now $Ks(s-A)^{-1}x_0$ converges uniformly $\all x_0\in H_B$;
   a more detailed proof is given in Lemma 9.11.5(e) of [M02]),
 because we have here required
 the ``$\wlim$'' in the ARE[s] to converge uniformly;

(d) This follows from Theorem~\ref{MTICinftyKF}(a3). %
\end{proof}

\begin{proof}[Proof of Remark~\ref{rcA}:]\label{pageproof-rcA}
Practically the same proofs still hold. 
We give below some details.

$1^\circ$ {\em Case $\cA=\WTICLC$:}
Observe first that if $T$ is a compact operator, then so are $T^*$, $ST$, $TS$
 and $(X+T)^{-1}-X^{-1}$ (whenever they exist).
By  Theorem~\ref{MTICinftyKF},
 it obviously suffices to prove that $f(t)$ is compact for a.e.\ $t$,
 equivalently, that $\hf(z)$ is compact for all $z$ on some right half-plane,
 where $\qX=X+f*$
 (since then the same holds for $\qM=\qX^{-1}$ etc.),
 and that $(z-A)^{-*}K^*$ is compact for all $z$ on some right half-plane
 if $(z-A)^{-*}C^*$ is.
Multiply the \hIRE\ (\ref{ehIRE}) by $S^{-1}\hqX(s)^{-*}$ to the left
 to observe this.
For ``finite-dimensional'', the same proof applies, mutatis mutandis.

$2^\circ$ {\em Cases $\cA=\cA_{\H^2}$ 
 and $\cA=\cA_2$:}
The first claims follow easily from Theorem 8.4.9 of [M02].
The equivalence with Theorem \ref{BwARE}(3.) is from Lemma 6.8.1(a)\&(d1)
 of [M02].
\end{proof}

\NotesFor{Section~\ref{sStabOptproofs}} %
Lemma~\ref{lSgg0} seems to be new,
 whereas Lemma~\ref{lC2ClLstable} (from [M02])
 is a simple generalization of a classical result.
\section{Conclusions}\label{sConclusions} %

We summarize here the Riccati equation and optimization theory
 developed in this article, 
 thus explaining how and to which extent the finite-dimensional results
 can be extended to WPLSs.
The general setting being thus resolved, 
 it seems that 
 in the future the WPLS RE research 
 should focus
 on the special cases where these results can be strengthened
 to give better applicability and on nonstandard REs
 (cf.\ [C03] and [M03b]).

For finite-dimensional $U,H,Y$,
 the equivalence of the 
 following conditions is fairly %
 well known:
 \begin{itemlist}
   \item[(i)] {\bf (\pmb{$\ex!\uopt$})} For each initial state $x_0$ there is 
    a unique optimal control. \label{pagei-iv}
   \item[(ii)] {\bf (\pmb{$u(t)=Kx(t)$})} 
 There is a unique optimal state-feedback operator.
   \item[(iii)] {\bf (FCC \& coercive)} The FCC holds
 and the cost function is $J$-coercive.\label{pageiii}
   \item[(iv)] {\bf (RE)} The Riccati equation 
 (ARE) has a stabilizing solution.
 \end{itemlist}

The cost functions (\ref{ecJx2u2}) 
 and (\ref{ecJy2u2}) are {\em $J$-coercive},
 and so is any other cost function that dominates the natural
 square norm of the input (p.~\pageref{J-coercive}) %
 (otherwise the ``infimal cost'' %
 would be achieved
  by no input or by many inputs).
The FCC means that there are some admissible inputs
 for each initial state $x_0$.

Also in the infinite-dimensional case it has been known
 that roughly the same four
 conditions are equivalent
 even when $A$ and $C$ are unbounded operators (if $B$ is bounded). 

In this article, we have generalized this equivalence to the class
 of  {\em WPLSs},
 thus allowing for rather unbounded $A$, $B$, $C$.
Our main results consist of the results ``1.--3c.'' below on
 the equivalence of (i)--(iv),
 and on the corollaries of them
 (particularly of ``3b.'', including rather indirect ones,
 such as the results of Section~\ref{sStabOpt}):
 \begin{itemlist}
   \item[1.] \label{page1.}
If the system is {\bf sufficiently regular},
 then (i)--(iv) are equivalent. %
\footnote{%
To be exact, claims 1.--3c.\ are true when $\dim U<\infty$ 
 and $\gUstar=\dUexp$ 
 or
 when we assume the cost to be coercive (as usual);
 otherwise (iii) is not implied by the other conditions.
Moreover, in (iv) we have required the indicator
 (the ``$S$'' or $\qSt$ on pp.~\pageref{eAREminB}, \pageref{eS2XSX=}, 
 \pageref{eXSX=})
 to be one-to-one,
 although (ii) and (iv) are equivalent even without that
 assumption if the word ``unique'' is deleted (pp.\ \pageref{IRE}\&\pageref{ARE};
 of course, nonuniqueness can only happen
 when the cost function is
  noncoercive (singular)).}%
   \item[2.] If the system is {\bf weakly regular},
 then (ii) and (iv) are equivalent (Theorem~\ref{ARE}).
   \item[3a.] \label{page3a.}
If we replace the (infinitesimal algebraic) RE
 by the integral RE ({\em IRE}), then (ii) and (iv) are equivalent
 for {\bf any WPLS} (Theorems \ref{IRE} and \ref{GenSpF}).
   \item[3b.] \label{page3b.}
In fact, for the IRE, (i)--(iv) are equivalent
 if the cost is nonnegative  (Theorem~\ref{PosJcKF} with 3a.\&3c.).
   \item[3c.] \label{page3c.}
For its variant, the \SIRE, 
  (i)--(iv) are equivalent in general
if we allow for possibly ill-posed state feedback in (ii) (Theorem~\ref{SIRE}).
 \end{itemlist}

Examples of ``1.'' are (roughly) %
 Theorems \ref{INTROcvKLQR0}, \ref{Assumptions}(vii) and
  Corollary~\ref{cBARE}; Remark 9.9.14 of [M02];
 for $J$-coercive systems also Theorem~\ref{MTICinftyKF}(b)
 and Remark~\ref{rcA};
 for positively $J$-coercive systems also
 Theorem~\ref{lBnotmax},
 Corollaries \ref{cPosIRE}(b)\&(c) and \ref{cMTICinftyKF+}(a)\&(c).
Except for Theorem~\ref{INTROcvKLQR0}, these results are new
 (except that \ref{Assumptions} and \ref{rcA} are from [M02],
 whose Chapters 9--10 also contain further results).
In ``3a.'' and ``3b.'', (ii) refers to a unique pair (modulo (\ref{eAllqKF})),
 not necessarily to an operator.

We have already presented 2. in [M97] in the stable case
 and 1.--3a.\ in [M02] in the general case. 
In this article we have repeated most of them and established 3b. and 3c.
Most earlier results were special cases of ``1.'';
 e.g., in [vK93] at least most implications can be found,
 for smooth Pritchard--Salamon systems.

In the WPLS setting, 
 the stable case of the implication (ii)$\THEN$(iv) 
 (and (iii)$\THEN$(ii) for the Wiener class)
 was originally solved in [S97] and [WW97]. 
The implication (iii)$\THEN$(i)
 was established in [FLT88]
 (for WPLSs having bounded $C$; see [Z96] for general WPLSs)
 in the case of the standard LQR cost function $\|y\|_2^2+\|u\|_2^2$.

As is well-known, in some cases a fifth equivalent condition is the existence
 of a (coprime) $J$-inner factorization of the I/O map
 (a spectral factorization in the stable case);
 see [M02] for details
 (see Theorem~\ref{QRCF-Uout}(iii) %
  and Lemma~\ref{lGenSpFproof} 
  for a special case).

It has been known for the standard LQR cost function that
 the optimal state is generated by a $C_0$-semigroup %
 and that the optimal control (and output) is generated by an admissible
 output operator for this (closed-loop) semigroup, as in [FLT88] and [Z96]. 
It has not been known that the output is admissible also for the original
 semigroup, or that the state-feedback loop is well-posed
 w.r.t.\ external perturbation (i.e., that $\ALSext$ and $\ALSClL$
 are WPLSs; see pp.\ \pageref{PosJcKF}\&\pageref{eALSExt}).
These facts are contained in ``3b.'' 
 and they do not hold for indefinite ($\cJ(0,\cdot)\not\ge0$) cost functions,
 by Example 8.4.13 of [M02].

Nevertheless, even in the indefinite case (3c.), we have the 
 ``ARE on $\Dom(A+BK)$'' (\ref{eGenAREAopt}) 
 whenever $\qD$ is uniformly regular and a unique $J$-optimal control
 exists for each initial state.
The solution of this ARE leads to the optimal control
 $u(t)=-(D^*JD)^{-1}(\Bw^*\Ric+D^*J\Cw)x(t)$
 for a.e.\ $t>0$.
To get AREs given on $\Dom(A)$
 (such as (\ref{eARE})), one has to restrict to ``2.'',
 and usually one wants to
 use further assumptions to simplify the ARE (as in ``1.'').

There is some ongoing research on the computational aspects of the ARE,
 but further results are needed for sufficient applicability.
Thus, the main contribution of this article consists of the abstract
 Riccati equation and optimization theory
 and of the stabilization and factorization results of Section~\ref{sStabOpt}.

\appendix
\section{Symbols $\hqA,\hqB,\hqC,\hqD,\hqD_\ALS,...$}\label{sSymbols}

In this appendix we present the frequency-domain symbols of WPLSs,
 and recall that
 $\sbm{\hqA&\shat{\qB\tau}\cr \hqC&\hqD}:\sbm{x_0\cr\hu}\to\sbm{\hx\cr \hy}$
 holds on $\C_{\omega}^+$,
 when $\ALS$ is $\omega$-stable, $u\in\L^2_\omega(\R_+;U)$ and $x_0\in H$.
Here $\hqA(s)=(s-A)^{-1}$, $\shat{\qB\tau}=(s-A)^{-1}B$,
 $\hqC=C(s-A)^{-1}$
 (and $\hqD=D+\Cw(s-A)^{-1}B$ if $\hqD$ is weakly regular),
 and $\hu,\hx,\hy$ are the Laplace transforms of $u,x,y$.

We also record some corollaries on ``compatible pairs'',
 to be referred in this article
 in the regular case only and in [M03b] in the general case.

Not all WPLSs are weakly regular
 (if they are, the values $\Cc=\Cw,\ \Dc=D$ will do below),
 but yet (\ref{exy=ABCDintro}), (\ref{eIntroTrFct})--(\ref{ehxhuhyClL})
 and other classical equations can be recovered for all WPLSs:
\begin{lemma}[Compatible pair $\pmbold{(\Cc,\Dc)}$]\label{lComp} %
Let $\ALS$ be a WPLS on $(U,H,Y)$.
Then there are a Banach space $W$ and $\Cc\in\BL(W,Y)$,
 $\Dc\in\BL(U,Y)$ such that $\Dom(A)\sub W\sub H$
 continuously, 
 and $\hqD(s)=\Dc+\Cc(s-A)^{-1}B$ for $s\in\C_{\omega_A}^+$.

Assume that $x_0\in H$ and $u,x\in\L^2_\omega(\R_+;*)$,
 where $x:=\qA x_0+\qB\tau u$.
Then $y:=\qC x_0+\qD u\in\L^2_\omega$,
 and equations $\hx=(\cdot-A)^{-1}(x_0+B\hu)$,
 $\hy=\hqD_{\ALS}\hu+C(\cdot-A)^{-1}x_0$ hold on $\C_\omega^+\pois\sigma(A)$
 and a.e.\ on $(\omega+i\R)\pois\sigma(A)$.
\end{lemma} %

(We can always take $W:=H_B:=(s-A)^{-1}BU+\Dom(A)$ for any $s\in\rho(A)$
 (this is independent of $s$), but even so $\Cc,\Dc$ need not be unique.
Necessarily always $H_B\sub W$. 
See Lemma \ref{lCharFct}(b3) for $\omega_A$
 and p.~\pageref{pageCharFct} for $\hqD_\ALS$.)

See, e.g., Section 6.3 of [M02] for more on $\Cc,\Dc$
 (e.g., Lemma 6.3.10(c) for equations (\ref{exy=ABCDintro})).

\begin{proof}
Combine theorems and lemmas 6.3.9, 6.3.10(a), %
 6.7.8, 6.3.20 and 6.2.11(c1) of [M02] 
 to get all this with $\omega':=\max(\omega,\omega_A)$
 in the last equation.
By holomorphicity (in $H_{-1}$),
 we can extend equations $(s-A)\hx(s)=x_0+B\hu(s)$
 and $\hy = \Cc \hx+\Dc \hu$ to $\C_\omega^+$ and to (a.e.) $\omega+i\R$
 (the proof of Lemma 6.3.20 of [M02];
  e.g., $x$ is continuous $\C_\omega^+\to W$). %
But $\Dc+\Cc(s-A)^{-1}B - \hqD(z) = (z-s)\Cc(s-A)^{-1}(z-A)^{-1}B
 = (z-s)C(s-A)^{-1}(z-A)^{-1}B$ for $z\in\C_{\omega_A}^+$,
 $s\in\rho(A)$, hence also the last claim holds
 (see Lemma~\ref{lCharFct}(c)).
\end{proof}

By $\rconn(V)$\label{pagerconn} we denote the ``rightmost maximal connected component'' of $V$,
 i.e., the maximal connected component that contains some right half-plane,
 provided that such exists. %
If $F$ is holomorphic on $V$, $G$ is holomorphic on $W$,
 and $F=G$ on some $(r,+\infty)$, then $F=G$ on $\rconn(V\cap W)$,
 by holomorphicity (when $V,W$ are open and contain some right half-plane).
This will be applied below:
\begin{lemma}[\pmbold{$\hqA,\hqB,\hqC,\hqD$}]\label{lCharFct} %
Let $\ALS=\qABCD$ be a WPLS on $(U,H,Y)$ %
 and $\omega\in\R$.
  \begin{itemlist} %

    \item[(a)]
If $\hqD$ is holomorphic on $\C_\omega^+$,
 then for all $s,z\in\rconn(\rho(A)\cap\C_\omega^+)\supset \C_{\omega_A}^+$,
 we have $\hqD_\ALS(s)=\hqD(s)$ and
 \begin{equation}
   \label{eCharFct}
   \hqD(s)-\hqD(z)=(z-s)C(s-A)^{-1}(z-A)^{-1}B.
 \end{equation}

    \item[(b1)] If $\hqC$ is holomorphic on $\C^+_\omega$,
 then so is $\hqD$,
 and, for all $s,z\in\C_\omega^+\pois\sigma(A)$ and $s'\in\C_\omega^+$,
 equations  (\ref{eCharFct}),
 $\hqC(s)=C(s-A)^{-1}$, $\hqD_\ALS(s)=\hqD(s)$
 and $\hqD(s')-\hqD(z)=(z-s')\hqC(s')(z-A)^{-1}B$
 hold.

    \item[(b2)] If $\qC$ is $\omega$-stable
 (or $\omega'$-stable for all $\omega'>\omega$),
 then $\hqC,\hqD$ are holomorphic on $\C_\omega^+$.

    \item[(b3)] $\ALS$ is $\alpha$-stable for any  %
 $\alpha>\omega_A:=\inf_{t>0}[t^{-1}\log\|\qA^t\|]$,%
 and $\C_{\omega_A}^+\sub\rho(A)$.

   \item[(c)]
We have $\hqD_\ALS=\Dc+\Cc(\cdot-A)^{-1}B$ on $\rho(A)$,
 and $\hqD_\ALS\in\H(\rho(A);\BL(U,Y))$,
 when $\Cc,\Dc$ are as in Lemma~\ref{lComp}.

    \item[(d)] $\shat{\qA x_0}(s)=(s-A)^{-1}x_0$ for $s\in\C_{\omega_A}^+$,
 and $\shat{\qB\tau u}(s)=(s-A)^{-1}B\hu$
 for $s\in\C_{\max\{\omega_A,\omega\}}^+$, 
 when $x_0\in H$, $u\in\L^2_\omega(\R_+;U)$. %

   \item[(e)] If $u,y\in\L^2_\omega(\R_+;*)$, $x_0\in H$, %
 where $y:=\qC x_0+\qD u$,
 then $\hy=C(\cdot-A)^{-1}x_0+\hqD_\ALS\hu$ on $\rconn(\rho(A)\cap\C_\omega^+)$.
   \item[(f)] ``$\C_\omega^+$'' may be replaced by %
 ``$\C_\omega^+\pois E$'' in (a) and (b1)
 if $E$ has no limit points in $\C_\omega^+$.
In particular, this applies with $\omega=0$ (resp.\ with some $\omega<0$)
 if $\ALS$ is output-stabilizable  (resp.\ exponentially stabilizable)
 and $\dim U<\infty$.
  \end{itemlist}
\end{lemma}

Here $\rho(A):=\rho(A)$ is the resolvent set of $A$, 
 and $\rho_\infty(A)\label{pagerhoinfty} :=\rconn(\rho(A))$ 
 is its maximal connected component containing $\C_{\omega_A}^+$.
By $\hqC$ we mean the map $\C_\omega^+\to\BL(H,Y)$
 that satisfies $\hqC x_0=\shat{\qC x_0}\ \all x_0\in H$
 ($\hqC$ exists and is holomorphic for any $\omega\ge\omega_A$,
 by Theorem 3.10.1 of [HP57]),
 or its holomorphic extension to a right half-plane.
The {\em characteristic function $\check\qD:=\hqD_\ALS$}\label{pageCharFct} 
 of $\ALS$ is
 defined by extending the equation (\ref{eCharFct}) from $s,z\in\C_{\omega_A}^+$
 to all $s,z\in\rho(A)$ %
 (cf.\ [SW03], Section~2). %
By taking adjoints, we observe that $\hqD_{\ALS^d}(s)=\hqD_\ALS(s^*)^*$ %
 for all $s\in\rho(A^*)=\rho(A)^*$. %

By Example~\ref{exaHans}, we may have $\hqD_\ALS=-1$ on the unit disc
 even if $\hqD\equiv0$ on $\C$ (and hence $\qD=0$),
 despite of bounded generators.
In general, $\hqD_\ALS$ depends on the whole realization $\ALS$ of $\qD$,
 not merely on $\qD$ and $A$, %
 hence we write $\check\qD$ only when the realization is obvious from the 
 context.
Typically, $\chqX=I-\chqF$ refers to $\tsysbm{\qA\|\qB\crh-\qK\|\qX}$
 (or $\ALSext$ for $\chqF$) and
 $\chqM=I+\chqFClL$ to $\tsysbm{\qAClL\|\qBClL\crh \qKClL\|\qM}$
 (or $\ALSClL$ for $\chqFClL$).

\begin{proof}[Proof of Lemma~\ref{lCharFct}:]
(a) See, e.g., [M02], Lemma 6.2.11(d2) for $s,z\in\C_{\omega_A}^+$.
Since both sides are holomorphic on the connected set
 $\rconn(\rho_\infty(A)\cap\C_{\omega}^+)$, (a) holds.

(b3)\&(d) See, e.g., Lemma 6.1.10(a2) and Theorem 6.2.11 of [M02].

(b2) By Lemma F.3.2(d) of [M02],
 there is $\hqC\in\Hstrong^2(\C^+;\BL(H,Y))$
 s.t.\ $\shat{\qC x}=\hqC x$ on $\C^+_\omega$ for all $x\in H$.
By Lemmas 6.1.11 and 6.2.1, $\hqD\in\H(\C_\omega^+;\BL(H,Y))$. %
(Similarly, $\qC\in\BL(H,\L^2_{\omega'})$ implies that
 $\hqC,\hqD\in\H(\C_{\omega'}^+,\BL)$; if this holds for all $\omega'>\omega$,
 then $\hqC,\hqD\in\H(\C_\omega^+,\BL)$.)

(c) $\Dc+\Cc(s-A)^{-1}B - \hqD(z) = (z-s)\Cc(s-A)^{-1}(z-A)^{-1}B
 = (z-s)C(s-A)^{-1}(z-A)^{-1}B$ for $z\in\C_{\omega_A}^+$,
 $s\in\rho(A)$. %
(Here $(\Cc,\Dc)$ may be any compatible pair for $\ALS$.)
From the definition (see (\ref{eCharFct})) we observe that $\hqD$
 is holomorphic (use Lemma A.4.4(a) of [M02]). %

(b1)
$1^\circ$ ``$\qC$'':
We have $\hqC(s)(s-A)=C$ on $\C_{\omega_A}^+\cap\C_\omega^+$, 
 hence on $\C_\omega^+$, by holomorphicity
 (see Lemma A.4.4(b) of [M02]), hence $\hqC(s)=C(s-A)^{-1}$
 on $\C_\omega^+\pois\sigma(A)$.

$2^\circ$ ``$s'$ and (\ref{eCharFct})'': Set $\alpha:=\max\{\omega,\omega_A\}$.
By (a) and $1^\circ$, we have
   $\hqD(s)-\hqD(z)=(z-s)\hqC(s)(z-A)^{-1}B$
 for all $s,z\in\C_\alpha^+$,
 hence for all $s\in\C_\omega^+$, $z\in\C_{\alpha}^+$
 (since this equation specifies a (unique) holomorphic extension
  of $\hqD(s)$ to $s\in\C_\omega^+$;
  recall that we identify a function with extensions to
 any right half-planes). %
By $1^\circ$, this leads to (\ref{eCharFct}) 
 for all $s\in\C_\omega^+\pois\sigma(A)$, $z\in\C_{\alpha}^+$.
Substitute it to $[\hqD(s')-\hqD(z)]-[\hqD(s)-\hqD(z)]$ %
 to obtain (\ref{eCharFct}) for  $s,z\in\C_\omega^+\pois\sigma(A)$ %
 (use the resolvent equation $(z-s)(s-A)^{-1}(z-A)^{-1}
  =(s-A)^{-1}-(z-A)^{-1}$).

$3^\circ$ ``$\qD(s)=\hqD_\ALS(s)$'':
Fix $z\in\C_{\omega_A}^+$.
By the definition of $\hqD_\ALS$,
 we have $\hqD-\hqD_\ALS=0$ on $\C_{\omega_A}^+$;
 by definition and $2^\circ$,
 we have $\hqD_\ALS(s)-\hqD(z)= (z-s)C(s-A)^{-1}(z-A)^{-1} =\hqD(s)-\hqD(z)$
 for $s\in\C_\omega^+\pois\sigma(A)$.

(e) Now $\hy\in\H^2(\C_\omega^+;Y)$ and $C(\cdot-A)^{-1}x_0+\hqD_\ALS\hu
 \in\H(\rho(A);Y)$,
 and $\hy=\hqC x_0+\hqD\hu=C(\cdot-A)^{-1}x_0+\hqD_\ALS\hu$
 on $\C_\omega^+\cap\C_{\omega_A}^+\sub\rconn(\rho(A)\cap\C_\omega^+)$. %

(f) The same proofs still apply for (a) and (b1).
As noted below Corollary~\ref{cUout-Meromorphic},
 $\hqC$ and $\hqD$ are meromorphic on $\C^+\pois E$
 (resp.\ $\C_\omega^+\pois E$ for some $\omega<0$)
 when $\dUout(x_0)\ne\tyhja$ (resp.\ $\dUexp(x_0)\ne\tyhja$)
 for all $x_0\in H$.
\end{proof}

\begin{example}{(\pmbold{$\qD=0$} but \pmbold{$\hqD_\ALS=-1$} on the unit disc).}\label{exaHans} %
In the example on p.~843 of [W94a] we have
 $H=\ell^2(\Z)$, $U=Y=\C$, %
 $A$ is the right shift $Ae_k=e_{k+1}$, $B=e_1$,
 $C=\p{\cdot,e_0}$, $D=0$, where $\{e_k\}_{k\in\Z}$ is the natural base of $H$.

Thus, $A,B,C,D$ are bounded, $\sigma(A)=\sigma(A^{-1})$ is the unit circle,
 $\hqD\equiv0$ (hence $\qD=0$) but $\hqD_\ALS=-1$ on the unit disc.
\end{example}

A transfer function ($\hqD$) is uniquely defined by the I/O map ($\qD$)
 and vice versa.
We can allow for a holomorphic extension of $\hqD$ and still
 keep uniqueness if we require the domain to be a half-plane
 with a discrete set of singularities:
\begin{remark}[Transfer function ($\pmbold{\hqD}$)]\label{rTrFct} %
We define a {\em transfer function}  (see Theorem~\ref{TransferFct1})
 on the
 ``maximal half-plane of discrete meromorphicity\label{pagemeromorphic}'',
 i.e., on $\C_\omega^+\pois E$, if $\omega\in\R$, 
 $E$ is {\em discrete}\label{pagediscrete0} on $\C_\omega^+$ (i.e., has no limit points on $\C_\omega^+$)
 and $\hqD$ has a holomorphic extension onto $\C_\omega^+\pois E$
 (as in (f) above; similarly for $\hqA,\shat{\qB\tau},\hqC$).

Obviously, this defines $\hqD$ uniquely. 
It has been thought that $\C_\omega^+\pois E$ being connected would suffice,
 but that is not true,
 as shown in Example~\ref{exaqD-Different-a},
 where $E=\sigma(A)$ is a half-line, so that $\rho(A)$ 
 is connected but yet the value of the characteristic function $\hqD_\ALS$
 on $\C^-$ depends on the realization of $\hqD=1/\sqrt{s}$.
\end{remark}

As explained above,
 by taking two different branches of
 (the transfer function $\hqD(s):=$) $1/\sqrt{s}$,
 we can have different values of $\hqD_\ALS(-1)$,
 even for characteristic functions 
 of realizations of $\hqD$ 
 (hence holomorphic extensions of $\hqD\raj{\C^+}$)
 whose resolvent sets are connected:
\begin{example}{(Transfer function cannot be uniquely extended around nondiscrete sets, not even for connected $\rho(A)$).}\label{exaqD-Different-a} %
Let $\ALS_1,\ALS_2$ be realizations of $s\mapsto 1/\sqrt{s}$
 (the primary branch on $\C^+$)
 with $\sigma(A_1)=i\R_+$, $\sigma(A_2)=-i\R_+$.
Then $\hqD_1=\hqD_2$ on $\C^+$ 
 but $\hqD_{\ALS_1}(-1)=-i$, $\hqD_{\ALS_2}(-1)=i$.
Moreover, $-1\in\rho_\infty(A_k)=\rho(A_k)\ (k=1,2)$.
\end{example}

\begin{proof}
A realization of $\hqD(s)=1/\sqrt{s}$
 with $\sigma(A)=\R_-$ is given in [O96]. 
Since $\hqD(+\infty)=0$ exists, $\hqD$ is regular.
Set $\tA:=-iA,\ \tB:=-\sqrt{i}B$ to obtain 
 $\htqD(s)=\Cw(s+iA)^{-1}(-\sqrt{i})B = \sqrt{i}\Cw(is-A)^{-1}B
  = \sqrt{i}\hqD(is)$,
 which is obviously holomorphic outside $i\R_+$
 and is a branch of $1/\sqrt{s}$ (since they obviously coincide on $\R_+$).
Thus, we have obtained the realization $\ALS_1:=\ssysbm{\tA\|\tB\crh C\|0}$
 of $\tqD:=\sqrt{i}\qD$.
Similarly, one obtains $\ALS_2$ and the rest is straight-forward.
\end{proof}

The standard formulas of (\ref{elALSClLgen})
 also apply to controlled WPLS forms:
\begin{theorem}[$\shat{\oALS}$]\label{oALSgen} %
Let $\oqK$ be a control for $\ALS$ in WPLS form,
 and choose $(\Cc,\Dc)$ as in Lemma~\ref{lComp}.
Then $A_0=A+B K_0$ and $C_0=\Cc+\Dc K_0$ on $\Dom(\oA)$,
 $(s-A_0)^{-1}-(s-A)^{-1}=(s-A)^{-1}B K_0 (s-A_0)^{-1}$,
 and $C_0(s-A_0)^{-1}=C(s-A)^{-1}+\hqD_\ALS(s)K_0(s-A_0)^{-1}$
 for all $s\in\rho(A)\cap\rho(A_0)$.

If $\oqK$ is $\omega$-stable, then %
 $\sigma(A_0)\cap\overline{\C_\omega^+}\sub\sigma(A)$ and
 $\sigma_p(A_0)\cap\overline{\C_\omega^+}\sub\sigma_p(A)$.
\end{theorem}

In particular, output-stabilizing feedback does not
 add unstable spectrum ($\sigma(\oA)\cap\cRHP$).
Actually, we prove the stronger claim
  $\rho(A)\cap$``$\rho(\hoqK)$''$\sub\rho(\oA)$
 in $2^\circ$ below.
Moreover, only the eigenvalues
 ($\sigma_p(A):=\{s\in\C\I (s-A)x_0=0$ for some $x_0\in H\pois\{0\}\}$)
 of $A$ may be those of $A_0$ on $\overline{\C_\omega^+}$.
See also Lemma~\ref{lWPLSform}. %

By duality (see Lemma~\ref{lWPLSformDual}),
  $\sigma(A^*)\cap\overline{\C_\omega^+}\sub\sigma(A_0^*)$ and
 $\sigma_p(A^*)\cap\overline{\C_\omega^+}\sub\sigma_p(A_0^*)$
 if $\qB$ is $\omega$-stable.
Recall that $\sigma(A^*)=\sigma(A)^*$. %

\begin{proof}
$1^\circ$ See Lemma 8.3.17(a) of [M02] for $A_0$ and $C_0$;
 here $(\Cc,\Dc)$ is any compatible pair for $\ALS$.
From (8.61) and (8.63)--(8.64) of [M02]
 (in (8.63) ``$\Cc$'' should be ``$+\Cc$''),
 we get the other two equations. %

$2^\circ$ {\em Assume that $\Omega\sub\C$ is open and connected
 and contains some right half-plane.
Assume that $\hoqK$ has a holomorphic extension %
  $\Omega\to\BL(H,U)$.
Then $\Omega':=\rho(A)\cap\Omega\sub\rho(\oA)$
 and $(s-\oA)^{-1}=(s-A)^{-1}[I+B\hoqK(s)]=:f(s)$ on 
 $\Omega'$:}
Fix $z>\max\{\omega_A,\omega_{\oA}\}$.
Since $(z-\oA)^{-1}=f(z)\in\BL(H)$ is one-to-one,
 our claim follows from Lemma~\ref{lPseudoRes} once we
 have (\ref{eRes}).
Set $R_s:=(s-A)^{-1}$ to have, for any $s\in\Omega'$, that
\[ f(s)-f(z) = (R_s-R_z)[I+B\hoqK(z)] + R_s B[\hoqK(s)-\hoqK(z)], \]
 and 
 $f(s)f(z) %
 = R_sR_z [I+B\hoqK(z)] + R_s B\hoqK(s)R_z[I+B\hoqK(z)]$.
By these and the Resolvent equation,
 $f(s)-f(z)-(z-s)f(s)f(z)  = R_s B W$, %
 where
\[ W:=\hoqK(s)-\hoqK(z)-(z-s)\hoqK(s)R_z[I+B\hoqK(z)], \]
 hence $W(z-\oA)=\hoqK(s)(z-\oA) -K_0 -(z-s)\hoqK(s) 
  = \hoqK(s)(s-\oA)-\oK=0$ on $\Dom(\oA)$.
(We used here the fact that 
 $\hoqK(s)(s-\oA)=\oK$ for $s>z$,
 hence for any $s\in\Omega'$.)
Since $\Ran(z-\oA)=H$, equation (\ref{eRes}) holds and we are done.

$3^\circ$ {\em Case $\oqK$ $\omega$-stable, $\sigma(\oA)$:}
Let $z\in\rho(A)\cap\overline{\C_\omega^+}$;
 we should show that $z\in\rho(\oA)$.
W.l.o.g., $\omega\le0$ and $z=0$ (replace $A$ by $A-z$
 as in Lemma 6.1.9 of [M02]). %
We have $\C_\omega^+\cap\rho(A)\sub\rho(\oA)$, by $2^\circ$,
 hence $s\in\rho(\oA)$ for small $s>0$.
Since $\hoqK\in\Hstrong^2$, we have $s\hoqK(s)\to0$ (uniformly), %
 as $s\to0+$, by Lemma~\ref{lsf(s)to0H2}. 
But $(s-A)^{-1}\to A^{-1}$ and $(s-A)^{-1}B\to A^{-1}B$,
 hence $s(s-\oA)^{-1}\to0$, %
 hence $0\not\in\sigma(\oA)$, by Lemma~\ref{lSGexplodes}.

$4^\circ$ {\em Case $\oqK$ $\omega$-stable, $\sigma_p(A)$:}
As in $3^\circ$,
 assume now that $0\in\overline{\C_\omega^+}\pois\sigma_p(A)$, $\omega\le0$.
If $x_0\in\Dom(\oA)$ and $\oA x_0=0$, then $\oqA x_0\equiv x_0$, 
 hence $\oqK x_0\equiv \oK x_0$,
 hence $\oK x_0=0$ (since $\oqK\in\L^2$),
 hence $A x_0=\oA x_0+B\oK x_0=0$,
 hence ($x_0\in\Dom(A)$ and) %
 $x_0=0$.
\end{proof}

Next we give similar but stronger
 results for (well-posed) state feedback,
 i.e., equation (\ref{eALSb}) in the frequency domain.
By Lemma~\ref{lCharFct} (and Definition~\ref{dAdmKF0}),
 both sides of (\ref{eqXinv=qM}) equal
 the Laplace transforms of the components of $\ALSClL$ on some right half-plane
 (see Proposition 6.6.18 of [M02] for details and further results).
By Example~\ref{exaHans}, the Laplace transforms
 need not equal (\ref{eqXinv=qM}) outside $\rho_\infty(A)$.
Nevertheless, (\ref{eqXinv=qM}) itself holds
 wherever both sides are defined:
\begin{theorem}[\pmbold{$\rho(A)\pois\sigma(\hqM) \sub \rho(\AClL)
 \ \ \&\ \ \hqM_{\ALSClL}=\hqX_\ALS^{-1}$}]\label{qXinv=qM} %
Let $\qKF$, $\qM$ and $\ALS=\qABCD$ be as in Definition~\ref{dAdmKF0},
 and set $\qX:=I-\qF\ (=\qM^{-1})$.
Let $s\in\rho(A)$.
Then $s\in\rho(\AClL) \IFF \hqXc(s)\in\cG\BL(U)$
 ($\IF \re s\ge\omega$ if $\qKClL$ or $\qBClL$ is $\omega$-stable).
Moreover, for all $s\in\rho(A)\cap\rho(\AClL)$, we have
\begin{equation}
  \label{eqXinv=qM}
\begin{aligned} %
\bsysbm{(s-\AClL)^{-1} \| (s-\AClL)^{-1}\BClL  \crh
 \CClL(s-\AClL)^{-1} \|  (\hqDClL)_{\ALSClL}(s)  \cr
 \KClL(s-\AClL)^{-1} \| \hqM_{\ALSClL}(s)} \hspace{16.5em}\\ %
 \ \ \  =\bsysbm{(s-A)^{-1}[I+B\KClL(s-\AClL)^{-1}] \| %
      (s-A)^{-1}B\hqM_{\ALSClL}(s)  \crh
  C(s-A)^{-1}+\hqD_\ALS(s)\hqM_{\ALSClL}(s)K(s-A)^{-1} \|
       \hqD_\ALS(s)\hqM_{\ALSClL}(s)  \cr
  \hqM_{\ALSClL}(s)K (s-A)^{-1} \|  \hqX_{\ALSext}(s)^{-1}}.
\end{aligned}
\end{equation}
\end{theorem}
\toistoa{
(Here $\qM-I$ is the I/O map of the corresponding closed-loop system $\ALSClL$,
 as in [S04] or Definition 6.6.10 of [M02].
Note that $\qKF$ is admissible for $\qAB$ iff
 $\ssysbm{\qA\|\qB\crh -\qK\|\qX}$ has a flow inverse,
 in which case $\ALSinv = \ALSClL+\ssysbm{0\|0\crh 0\|I}$.) 
} %

Observe that above we refer to the characteristic functions of 
 these specific realizations
 (and $\qM_{\ALSClL}:=I+(\qFClL)_{\ALSClL}$). 
E.g., if $\qABKF$ the system of Example~\ref{exaHans} %
 but with $F:=1/2$,
 then $\hqF\equiv 1/2\equiv \hqX$ on $\C$,
 hence $\hqM\equiv 2$,
 but $\hqXc=3/2=\hqM_{\ALSClL}^{-1}$ on the unit disc, %
 whereas $\hqXc\equiv 1/2\equiv\hqM_{\ALSClL}^{-1}$ %
 for the alternative realization $\qABKF:=\ssyspm{0\|0\crh 0\|1/2}$.

Exchange the roles of $\ALS$ and $\qABCDClL$,
 to obtain further formulas and implications
 (with the roles of $\pm\bsysbm{\qK\|\qF}$ and $\mp\bsysbm{\qKClL\|\qFClL}$
  exchanged, as in Lemma 6.6.14 of [M02]).
Also Theorem~\ref{oALSgen} applies both ways.

\begin{proof}[Proof of Theorem~\ref{qXinv=qM}:]
Set $R_s:=(s-A)^{-1}\ (s\in\rho(A))$,
$T_s:=(s-\AClL)^{-1}\ (s\in\rho(\AClL))$.

I Let $s\in\rho(A)$. %
Let $(\Kc,\Fc)$ be a compatible pair for $\qABKF$,
 $\Xc:=I-\Fc$.
Set $X:=\hqXc(s)=\Xc-\Kc(s-A)^{-1}B$.

$1^\circ$ Assume that $M:=X^{-1}\in\cG\BL(U)$ exists.
We shall show that $T=(s-\AClL)^{-1}$,
 where $T:=(I+R_sBMK)R_s\in\BL(H)$ (actually, $T\in\BL(H,H_B)$).

By Theorem~\ref{oALSgen}
 and, e.g., the proof of (b1) of Proposition 6.6.18 of [M02],
 we have $\AClL=A+B\KClL$ and $\Xc\KClL=\Kc$
 on $\Dom(\AClL)$. %
Consequently, on $\Dom(\AClL)$ we have %
 $T(s-\AClL)
 = (I+R_sBM\Kc)R_s(s-A-B\KClL)
 = (I+R_sBM\Kc)(I-R_s B\KClL)
 = I+R_sBM[\Kc-X\KClL-\Kc R_s B\KClL]
 = I+R_sBM[\Kc-\Xc\KClL]=I$.

By the dual result,
 also $(\bar s-\AClL^*)=(s-\AClL)^{*}$ has a bounded left-inverse,
 hence $s-\AClL$ has a bounded right-inverse,
 hence $s\in\rho(\AClL)$
 (because $s-\AClL$ is also closed, densely defined and one-to-one). 

$2^\circ$ Assume that $s\in\rho(A)\cap\rho(\AClL)$.
Then $\hqXc(s)=\hqM_{\ALSClL}(s)^{-1}$,
 by part II of this proof, hence then $\hqXc(s)\in\cG\BL(U)$.

$3^\circ$ {\em ``$\IF\ \re s\ge\omega$'':}
This follows from Theorem~\ref{oALSgen}
 (use duality for $\qBClL$).

II
Choose some $z>\max\{\omega_A,\omega_{\AClL}\}$.
Fix $s\in\rho(A)\cap\rho(\AClL)$.

$1^\circ$ ``$\hqM_{\ALSClL}$'':
 $\hqXc(s):=\hqX(z)+(z-s)KR_sR_zB$
 (see p.~\pageref{pageCharFct}), and 
 $\hqM_{\ALSClL}(s):=\hqM(z)+(z-s)\KClL T_sT_z\BClL$.

Therefore, $S:=\hqXc(s)\hqM_{\ALSClL}(s)
 =I+(z-s)\hqX(z)\KClL T_sT_z\BClL
  + (z-s)KR_s R_zB\linebreak %
 \hqM(z)  +
 (z-s)^2KR_sR_zB \KClL T_sT_z\BClL$.
But $\KClL(z-\AClL)^{-1}=\qKClL(z)$,
 and $\hqX(z)\hqKClL(z)=\hqK(z)$ (by (6.133) of [M02]);
 similarly, $R_zB\hqM(z)=\shat{\qBClL}(z)=(z-\AClL)^{-1}\BClL$,
 and, by (6.137), $R_zB \KClL T_z=T_z-R_z$ %
 hence
 $S-I= (z-s) K V  \BClL$,
 where $V=-R_zT_s+R_sT_z+(z-s)R_s(T_z-R_z)T_s$. %
By the resolvent equation, $V=0$, hence $S=I$.
Similarly, $\hqM_{\ALSClL}(s)\hqXc(s)=...=I$. 
$2^\circ$ ``$\hqBClL,\hqKClL,\hqAClL$'':
Apply  Theorem~\ref{oALSgen} to $\bsysbm{\qK\|-\qX}$
 in place of $\qCD$ to obtain that
 $0(s-\oA)^{-1} = KR_s - \hqX_\ALS(s)\oK T_s$,
 i.e., $KR_s = \hqX_\ALS(s)\oK T_s$.
By $1^\circ$, this equals $\hqM_{\ALSClL}(s) KR_s = \oK T_s$,
 as claimed.
The formula for $(s-\AClL)^{-1}\BClL$
 follows by duality.

$3^\circ$ ``$\hqCClL$, $\hqDClL$'':
Apply the above for $\ssysbm{\qC\|0&\qD\cr \qK\|0&\qF}$
 in place of $\qKF$ to obtain 
 that $\sbm{I&\hqD_\ALS\hqXc^{-1}\cr 0&\hqXc^{-1}}
  = \sbm{I&(\hqDClL)_{\ALSClL}\hqM_{\ALSClL}\cr 0&\hqM_{\ALSClL}}$
 and $\sbm{\CClL\cr \KClL}(s-\AClL)^{-1}
  = \sbm{C+(\hqDClL)_{\ALSClL}\hqM_{\ALSClL}K\cr \hqM_{\ALSClL}K} (s-A)^{-1}$.
\end{proof}

If $\sigma(A)$ is nice and $\dim U<\infty$, then any
 I/O-stabilizing state-feedback pair allows for a holomorphic extension
 of $\hqN,\hqM$ over the imaginary axis:
\begin{lemma}\label{lNMholomUfinite} %
Use notation of  Definition~\ref{dAdmKF0}, $\qKF$ being admissible.
If $\qN,\qM$ are stable, $\dim U<\infty$,
 $\sigma(A)\cap i\R$ consists of isolated poles,
 and $\inf\rho_\infty(A)\le0$,
 then there is an open $\Omega\sub\C$
 s.t.\ $\cRHP\cup(\rho_\infty(A)\pois Z_g)\sub\Omega$
 and $\hqN,\hqM$ have holomorphic extensions to $\Omega$.
Moreover, then $\rho_\infty(A)\pois Z_g\sub\rho_\infty(\AClL)$,
 and the only nonremovable singularities of (\ref{eqXinv=qM})
 on $\rho_\infty(A)$ are isolated poles.
\end{lemma}

Here $Z_g:=\{s\in\rho_\infty(A)\I \det\hqX_{\ALSext}=0\}$ (it is discrete).
The requirement $\inf\rho_\infty(A)\le0$
 means that you can reach the imaginary axis from $+\infty$ through $\rho(A)$.
By the proof ($6^\circ$),
 the extensions equal the characteristic functions
 on $\rho_\infty(A)\pois Z_g$, which contains almost every point of $i\R$.

For this kind of systems, ``q.r.c.'' and ``r.c.'' are equivalent
 in the sense of Lemma~\ref{lNMrcUfinite}.

\begin{proof}
$1^\circ$ %
For each $t\in\R$, choose a maximal disc (punctured if necessary)
 with center $it$ and contained in $\rho(A)$.
Let $G$ be the union of these discs. %
Obviously, $G$ and $G\cap\C^+$ are connected open subsets of $\rho(A)$,
 and $i\R\pois\sigma(A)\sub G$. 

$2^\circ$ {\em We have $G\sub\rho_\infty(A)$,
 hence $i\R\pois\sigma(A)\sub\rho_\infty(A)$}
 (because $\inf\rho_\infty(A)\le0$ implies that $G\cap\rho_\infty(A)\ne\tyhja$).

$3^\circ$ {\em $\Omega_1:=\rho_\infty(A)\cap\C^+$ is connected:}
If $\Omega_1\sub V\cup W,\ V\cap W =\tyhja$,
 and, w.l.o.g., $G\cap\C^+\sub V$,
 then (the boundary) $\d W\sub\C^+$, %
 hence $\d W\cap\rho_\infty(A)=\tyhja$,
 hence $W$ is a component of $\rho_\infty(A)$ or $W=\tyhja$, QED.

$4^\circ$ {\em $g^{-1}=\hqM$ on $\Omega_1$:}
By Theorem~\ref{TransferFct1}, $\hqN,\hqM\in\H^\infty(\C^+;\BL)$.
Let $f:=\hqD_\ALS$, $g:=\hqX_{\ALSext}$, so that $f,g\in\H(\rho(A);\BL)$.
Since $\det g\not\equiv0$ on $\rho_\infty(A)$,
 the set $Z_g:=\{s\in\rho_\infty(A)\I \det g(s)=0\}$ %
 is discrete in $\rho_\infty(A)$.
By continuity, $g^{-1}=\hqM$ and $fg^{-1}=\hqN$
 on the whole $\Omega_1$.

$5^\circ$ {\em $\hqM,\hqN$ have unique holomorphic extensions to $\Omega$:}
Because $\Omega_2:=\rho_\infty(A)\cup\C^+\pois Z_g$ is open and connected,
 and $fg^{-1},g^{-1}$ are holomorphic on $\rho_\infty(A)\pois Z_g$,
 $\hqN,\hqM$ have unique holomorphic extensions
 $\hqN_e,\hqM_e:\Omega_2\to\BL$.
Also $\Omega:=\Omega_2\cup i\R$ is open and connected.
By the assumption, $f,g$ do not have essential singularities on $i\R$,
 hence so do not $\hqN_e,\hqM_e$; but they cannot have poles either,
 because they are bounded on $\C^+$,
 hence their singularities on $i\R$ are removable, QED.

$6^\circ$ {\em We have $\rho_\infty(A)\pois Z_g\sub\rho(\AClL)$:}
By Theorem~\ref{qXinv=qM}, 
 $V:=\rho_\infty(A)\pois Z_g\sub\rho(\AClL)$
 (hence $V\sub\rho_\infty(\AClL)$, hence $G\sub\rho_\infty(\AClL)$),
 and $\hqN_e,\hqM_e$ %
 coincide with the characteristic functions  of $\hqN,\hqM$ on $V$.

$7^\circ$ {\em Nonremovable singularities:}
The functions $(s-A)^{-1},C(s-A)^{-1},K(s-A)^{-1},(s-A)^{-1}B,\hqD_\ALS,\hqX_{\ALSext}$
 are holomorphic on $\rho(A)$
 and have at most isolated poles at those of $(s-A)^{-1}$
 (by the resolvent equation: e.g., $(s-A)^{-1}B=(s_0-s)(s-A)^{-1}(s_0-A)^{-1}B$
  and $(s_0-A)^{-1}B\in\BL(U,H)$). %
But $g^{-1}$ has at most isolated poles on $\rho_\infty(A)$,
 hence so do also the other elements of both sides of (\ref{eqXinv=qM}),
 since they are obtained from the functions mentioned above
 through sums and products.
\end{proof}

\NotesFor{Appendix~\ref{sSymbols}} %
In 1997, the author defined compatible pairs (under a different name)
 and developed state-feedback and Riccati equation theories for them.
The concept was further developed by O. Staffans and G. Weiss,
 including the existence for all WPLSs (see p.~202 of [M02]);
 the second paragraph of 
 Lemma~\ref{lComp} seems to be new (although partially in [M02]).
Also part of Lemma~\ref{lCharFct} is known, 
 due to Weiss, Staffans and others;
 see, e.g., [S04] or Chapter~6 of [M02] for further notes and details.

Example~\ref{exaHans} is due to Hans Zwart (p.~843 of [W94a]).
Example~\ref{exaqD-Different-a} was constructed in correspondence
 with Olof Staffans.
The rest of this appendix seems to be new.

The importance of characteristic functions and exact domains
 of transfer functions requires some explanation:
Recently, the so called reciprocal system theory
 has been used as a powerful tool
 to simplify WPLSs theory and AREs, due to Ruth Curtain and others
 (see, e.g., [C03]).
Late 2002, we pointed out that one must use the characteristic
 function instead of the transfer function in the reciprocal systems theory
 (and suggested a separate name and symbol,
 i.e., $\hqD_\ALS$ instead of $\hqD$).
Moreover, the reciprocal equations are justified
 on $\rconn(\rho(A)\cap\C^+)$ only,
 not on $\rho(A)$ or $\rho_\infty(A)$ in general,
 which restricts the applicability of the theory,
 although one can often circumvent these problems,
 as shown in [M03b] (and [M03]).
These findings, thereafter widely spread, were our motivations
 behind Lemma~\ref{lCharFct}, Example~\ref{exaqD-Different-a}
 and Theorem~\ref{qXinv=qM}.
Later we observed these results useful also in the standard RE theory.

\section{Laplace transforms}\label{sLaplace} %
In this section we list %
 some results on the Laplace transform
 (see (\ref{eLaplace})). Here $X$ stands for an arbitrary Banach space.

\begin{lemma}[\pmbold{$s\hf(s)\to f(0)$}]\label{lshf2f(0)} %
If $f\in\L^1_\omega(\R_+;X)$ is continuous at $0$, and $\omega\in\R$,
 then $s\hf(s)\to f(0)$ as $s\to+\infty$.
\end{lemma}

\begin{proof}
$\|s\hf(s)-sf(0)/s\|=\|s\int_0^\infty[f(t)-f(0)]\efn^{-st}\,dt\|
 =\|s\int_0^\del+s\int_\del^\infty\|
 \le \eps/2 + |s\efn^{-(s-\omega)\del}|\|f\|_{\L^1_\omega} 
 + |\efn^{-s\del}|\|f(0)\|<\eps$ for $\del>0$ small
 and $s$ big enough. %
\end{proof}

\begin{lemma}\label{lhf'} %
If $f\in\L^2_\omega(\R_+;X)$,
 then $\shat{f'\,}(s)=s\hf(s)-f(0)\ \all s\in\C_\omega^+$.
\end{lemma}

(This follows by, e.g., partial integration for $f\in\W^{1,2}_\omega$.
This is the obvious extension (sometimes even the definition)
 of the (Sobolev) distribution derivative.)

Knowing $\p{\hf(s),\hg(s)}$ on a half-plane $\C_\omega^+$
 characterizes $\p{f(t),g(t)}$ on $\R_+$: 
\begin{lemma}[\pmbold{$\p{\hf,\hg}_H=0\ \THEN\ \p{f,g}_H=0$} a.e.]\label{lhfhg=0fg=0} %
Let $f,g\in\L^2_\omega(\R_+;H)$,
 $F,G\in\L^2_\omega(\R_+;Y)$, $\omega\le\alpha<\beta$.
If $\p{\hf(s),\hg(s)}_H=\p{\hF(s),\hG(s)}_Y$ for $s\in\C_\alpha^+\pois\C_\beta^+$,
 then $\p{f(t),g(t)}_H=\p{F(t),G(t)}_Y$ for a.e.\ $t\ge0$.
\end{lemma}

\begin{proof} %
Take $F=0=G$ and $\omega=\alpha=0$ w.l.o.g.\
 (use $f\mapsto \efn^{-\alpha\cdot}\sbm{f\cr F},\
 g\mapsto\efn^{-\alpha\cdot}\sbm{g\cr -G}$,
 $\omega\mapsto \alpha$).
Set $h(t):=\p{f(t),g(t)}_H$, so that $h\in\L^1(\R_+;H)$.
Then $h(t)=0$ a.e.\ $\IFF\ \hh(s)=0\ (s\in\C^+)\ %
 \IFF\ \hh(s)=0\ (s\in(0,\beta))$.
But $\hh(s)=\int_0^\infty \efn^{-st}h(t)\,dt
  = \int_0^\infty \efn^{-st}\p{f(t),g(t)}\,dt
  = \p{f(t),g(t)}_{\L^2_{s/2}}
  = (2\pi)^{-1} \p{\hf,\hg}_{\L^2(s/2+i\R;H)}=0$.
\end{proof}

Next we allow for two different values for $s$ ($s$ and $z$)
 to obtain a stronger result
 (that leads to the \hoptIRE, \hIRE\ and \hSIRE): 
\begin{lemma}[\pmbold{$\p{\hf,\hg}_Y=(s+\bar z)\p{\hF,\hG}_H \IFF\ \p{\tau f,\piOt g}_{\L^2}=\big|_0^t \p{\tau F,G}_H$}]\label{lhFG=hfg} %
Assume that $f,g\in\L^2_\omega(\R_+;Y)$,
 $F,G\in\L^2_\omega\cap\cC(\R_+;H)$, $\omega\in\R$.
Then the following are equivalent:
\begin{itemlist}
  \item[(i)] $\p{\hf(s),\hg(z)}_Y=
  (\bar z+s)\p{\hF(s),\hG(z)}_H - \p{\hF(s),G(0)}_H$
 for a.e.\ $s,z\in\C_\omega^+$.
  \item[(v)] $\p{\piOt\tau^r f,\piOt g}_{\L^2}=\p{F(r+t),G(t)}_H-\p{F(r),G(0)}_H$
 for all $t\ge0$, $r\in\R$.
\end{itemlist} %
\end{lemma}

By holomorphicity, we may replace $\C_\omega^+$
 by any of its subsets having a limit point in $\C_\omega^+$.
Warning: (v)  must hold for all $r\in\R$, not merely for $r\ge0$.
In some applications this can be achieved by interchanging $f$ with $g$.

\begin{proof} %
(We do not require $F(0),G(0)$ be zero
 (i.e., continuous from the left).)

By $\Lap_s^r h:=\int_\R \efn^{-rs}h(r,t)\,dr$ we
 denote the value of the Laplace transform of $h$ w.r.t.\ $r$ at $s$.
Recall that we use the two-sided Laplace transform ($\int_{-\infty}^\infty$, %
 not $\int_0^\infty$), although there is no difference for $f,g,F,G$
 (since they are zero on $\R_+$ unlike possibly $\tau^r f,\tau^r F$).

We transform  the left- and right-hand sides of (v)
 to obtain (i)
 ($\Lap_\barz^t\Lap_s^r \khi_{\R_+}(t)$, i.e., first w.r.t.\ $r$,
 then one-sidedly w.r.t.\ $t$;
 the multiplication by $\khi_{\R_+}(t)$ 
 makes the original expressions equal on the whole $\R\times \R$).
(By going the equations backwards with norm signs,
 one observes that the integrals converge absolutely.
Therefore, the Fubini Theorem is admissible;
 obviously, all functions are product measurable.) %
On the left we have
 (when $\re s>\omega,\ \re z>\max\{0,\omega+\re s\}$)
 \begin{align}
 &\Lap_\barz^t\khi_{\R_+}(t) \Lap_s^r \int_0^t \p{f(r+p),g(p)}_Y\,dp
  &=& \Lap_\barz^t \khi_{\R_+}(t) \int_0^t \p{\efn^{sp}\hf(s),g(p)}_Y\,dp\\
  &= \int_0^\infty     \int_p^\infty \efn^{-\bar z t}
         \p{\efn^{sp}\hf(s),g(p)}_Y\,dt\,dp
  &=& \int_0^\infty \frac{\efn^{-\bar z p}}{\bar z}
          \p{\efn^{sp}\hf(s),g(p)}_Y\,dp\\
  &= \frac1{\bar z}
          \p{\hf(s),\hg(z-\bar s)}_Y,
 \end{align}
 and on the right (recall that $F(t)=0=G(t)$ for $t<0$): 
 \begin{align}
 \Lap_\barz^t \khi_{\R_+}(t) \efn^{st} \p{\hF(s),G(t)}_H
    - \Lap_\barz^t \khi_{\R_+}(t) \p{\hF(s),G(0)}_H
  =  \p{\hF(s),\hG(z-\bar s)}_H    - \frac1\barz \p{\hF(s),G(0)}_H
 \end{align}
Since the integrands are in $\L^2_{\alpha}$ for all $\alpha\ge\omega$,
 the transforms must be equal on $\C_\omega^+$
 iff ($\khi_{\R_+}(t)$ times) the original expressions are equal.
Multiply the above results by $\barz$ to obtain
\begin{align}
  \p{\hf(s),\hg(z-\bar s)}_Y
  = \barz \p{\hF(s),\hG(z-\bar s)}_H    -  \p{\hF(s),G(0)}_H
\end{align}
Replace $z-\bar s$ by $z$ to obtain (v)
 for $s\in\C_\omega^+,\ z\in\C^+_{\max\{\omega,-\re s\}}$;
 use holomorphicity to allow for any $z\in\C_\omega^+$.
\end{proof}

If the Laplace transform of a function is a resolvent,
 then the function is a semigroup:
\begin{lemma}[$\pmbold{\hqA(s)=(s-A)^{-1}\ \THEN\ \qA^t}$ is a semigroup]\label{lResSG} %
If $\qA:\R_+\to\BL(X)$,
 $\qA x_0\in\cC(\R_+;X)\ (\all x_0\in X)$,
 $\qA^0=I$, $\|\qA^t\|\le M\efn^{\omega t}\ (t\ge0)$,
 $\hqA(s)=(s-A)^{-1}\ (s\in\C_\omega^+)$
 for some $\omega,M\in\R$, some linear operator $A:X\supset\Dom(A)\to X$
 and some Banach space $X$,
 then $\qA$ is a $C_0$-semigroup with generator $A$. %
\end{lemma}

\begin{proof} %
$1^\circ$ {\em $\Dom(A)$ is dense in $X$:}
By Lemma~\ref{lshf2f(0)},
 $x_0=\qA^0 x_0=\lim_{s\to+\infty} s(s-A)^{-1}x_0\in\overline{\Dom(A)}$
 for all $x_0\in X$, hence $\overline{\Dom(A)}=X$. %

$2^\circ$ {\em The rest:}
Let $x_0\in\Dom(A)$. Then the Sobolev derivative $\qA x_0'$ 
 is actually a continuous function:
 $\shat{(\qA x_0)'}:=s\shat{\qA x_0}(s)-x_0
 =s(s-A)^{-1}x_0-(s-A)(s-A)^{-1}x_0=A(s-A)^{-1}x_0=(s-A)^{-1}Ax_0$,
 i.e., $(\qA x_0)'=A\qA x_0=\qA A x_0\in\cC(\R_+;X)$
 and $\qA[\Dom(A)]\sub\Dom[A]$
 (since the inverse Laplace transform of $(s-A)^{-1}x_0$
  converges also in the topology of $\Dom(A)$,
 because, obviously, $\|\hqA x_0\|_{\H^2(\C_{\omega+1}^+;\Dom(A))}^2
  =2\pi\|\qA x_0\|_{\L^2(\C_{\omega+1}^+;\Dom(A))}^2
  \le M\| x_0\|_{\Dom(A)}^2:=M(\|x_0\|_H^2+\|Ax_0\|_H^2$). %

However, if $x\in\W^{1,2}_\alpha(R_+;X)$ for some $\alpha<\infty$,
 $\omega\ge\omega$,
 and $x$ solves the problem $x'=Ax,\ x(0)=x_0\in\Dom(A)$,
 then $x\in\L^2_\omega(\R_+;\Dom(A))$ (as above) and
 $s\hx(s)-x(0)=:\shat{x'}(s)=\shat{Ax}(s)=A\hx(s)$,
 hence $\hx(s)=(s-A)^{-1}x_0$ on $\C_\alpha^+$,
 hence $x = \qA x_0$, by the uniqueness of Laplace transforms.

But, for any $T\ge0$,
  $\qA^{T+\cdot}x_0$ is the solution of $x'=Ax,\ x(0)=\qA^T x_0$,
 hence $\qA^{T+\cdot}x_0=\qA^\cdot \qA^T x_0$.
Since $x_0\in\Dom(A)$ was arbitrary, we have
 $\qA^{T+t}=\qA^t\qA^T$ on $\Dom(A)$.
By $1^\circ$ and continuity, it follows that 
 $\qA^{T+t}=\qA^t\qA^T$, hence $\qA$ is a $C_0$-semigroup.
Since $\hqA(s)=(s-A)^{-1}$, $A$ is the generator of $\qA$.
\end{proof}

A one-to-one function satisfying the Resolvent equation (\ref{eRes})
 is a resolvent:
\begin{lemma}[Pseudoresolvent is a resolvent]\label{lPseudoRes} %
Let $X$ be a Banach space, $\tyhja\ne E\sub\C$, and let $f:E\to\BL(X)$.
Then $f(s)=(s-A)^{-1}$ ($s\in E$) for some linear operator $A$ on $X$
 iff $f(s_0)$ is one-to-one for some $s_0\in E$ and 
 \begin{equation}
   \label{eRes}
   f(s)-f(s_0)=(s_0-s)f(s)f(s_0)\ \ \ \all s\in E.
 \end{equation}\itemlistnoproof
\end{lemma}

(The proof of Theorem I.9.3 of [Pazy] applies here too.  %
Note that $E$ need not be open nor connected
 and that $f(s)=(s-A)^{-1}$ means that $f(s)(s-A)=I_{\Dom(A)}$
 and $(s-A)f(s)=I_X$. %
Naturally, $A$ is unique and closed and $E\sub\rho(A)$.)

The norm of a resolvent is unbounded at the boundary of the resolvent set:
\begin{lemma}[\pmbold{$\|(s-A)^{-1}\|_{\BL(X)}\ge 1/d(s,\sigma(A))$}]\label{lSGexplodes} %
Let $A$ be a linear operator $\Dom(A)\to X$.
If $s\in\rho(A)$, then $\|(s-A)^{-1}\|_{\BL(X)}\ge 1/d(s,\sigma(A))$
 (the inverse of the distance from $s$ to the spectrum of $A$).\noproof
\end{lemma}

(This is well known; see, e.g., Lemma 3.2.8(iii) of [S04].)

\begin{lemma}\label{lsf(s)to0H2} %
If $f\in\Hstrong^2(\C^+;\BL(X))$, %
 then $sf(s)\to0$ in $\BL(X)$ as $s\to0+$.   
\end{lemma}

Thus, $sf(s+ir)\to0$ for any $r\in\R$.

\begin{proof}
Obviously, %
 $g\in\Hstrong^2$, where $g(z):=z^{-1}f(z^{-1})$.
By Lemma F.3.2(b) (p.~1017) of [M02], %
 $g(z)\to0$ as $z\to+\infty$.
\end{proof}

\NotesFor{Appendix~\ref{sLaplace}} %
Lemmas \ref{lshf2f(0)}, \ref{lhf'},
 \ref{lPseudoRes}
 and~\ref{lSGexplodes} %
 are well known. 
{\bf Acknowledgments.} %
The author wants to thank Olof Staffans, Ruth Curtain, Mark Opmeer, 
 and others on useful comments on the manuscript.
Part of this work was written with the support of the
 Academy of Sciences and part at Institut Mittag--Leffler.

\vspace{-2.8ex}
\section*{References} %

\begin{itemlistbibitem} %

\KMBibitem{[AN96]}{Damir Z. Arov and Michael A. Nudelman.
Passive linear stationary dynamical scattering systems with continuous time.
{\em Integr.\ Equ.\ Oper.\ Theory},
 24:1--45, 1996.}

\KMBibitem{{[C03]}}{Ruth F. Curtain.
Riccati equations for well-posed linear systems: the generic case.
To appear in
{\em Systems and Control Letters}, 2003.}

\KMBibitem{[CWW01]}{Ruth F. Curtain, George Weiss and Martin Weiss. 
Stabilization of irrational transfer functions by controller with 
  internal loop. 
{\em Systems, Approximation, Singular Integral Operators, and
Related Topics, Proceedings of IWOTA 2000}, %
 Alexander A. Borichev and Nikolai K. Nikolski (eds.), Birkhäuser, 2001.}

\KMBibitem{{[FLT88]}}{Franco Flandoli, Irena Lasiecka and Roberto Triggiani.
Algebraic {Riccati} equations with non-smoothing
                  observation arising in hyperbolic and
                  {Euler}--{Bernoulli} boundary control problems.
{\em Annali Mat. Pura Appl.}, 153:307--382, 1988.}

\KMBibitem{[G92]}{Michael Green.
{$H^\infty$} controller synthesis by {$J$}-lossless
coprime factorization.
{\em SIAM J. Control Optim.},
 30:522--547, 1992.}

\KMBibitem{{[GL73]}}{Israel C. Gohberg and Yuri Laiterer. %
The factorization of operator-functions relative to a contour.
III. Factorization in algebras. (Russian)
{\em Math. Nachr.} 55:33--61, 1973.}

\KMBibitem{{[HP57]}}{Einar Hille and Ralph S. Phillips.
{\em Functional Analysis and Semi-Groups}.
AMS, Providence, revised edition, 1957.}

\KMBibitem{[IOW99]}{Vlad Ionescu, Cristian Oar\u a and Martin Weiss.
{\em Generalized Riccati Theory and Robust Control. 
A Popov Function Approach.}
Wiley, 1999.}

\KMBibitem{{[vK93]}}{Bert van Keulen.
{\em {$H_\infty$}-Control for Distributed Parameter
                  Systems: A State Space Approach}.
{Birkh\"auser}, 1993.}

\KMBibitem{[LR95]}{Peter Lancaster and Leiba Rodman.
{\em Algebraic Riccati Equations.}
Clarendon Press, Oxford, 1995.} %

\KMBibitem{{[LT00]}}{Irena Lasiecka and Roberto Triggiani. %
{\em Control Theory for Partial Differential Equations:
Continuous and Approximation Theories.
I: Abstract Parabolic Systems.}
Cambridge University Press, 2000.}

\KMBibitem{{[M97]}}{Kalle M. Mikkola.
On the Stable $\H^2$ and $\H^\infty$
Infinite-Dimensional Regular Problems and Their Algebraic Riccati Equations.
Technical Report A383, Institute of Mathematics, 
Helsinki University of Technology, Espoo, Finland, 1997.}

\KMBibitem{{[M02]}}{Kalle M. Mikkola.
{\em Infinite-Dimensional Linear Systems, Optimal Control
        and Algebraic Riccati Equations}.\\  
{\tt http://www.math.hut.fi/\~{}kmikkola/research/thesis/}\\ %
Doctoral dissertation, Institute of Mathematics, 
Helsinki University of Technology, Espoo, Finland, 2002.}

\KMBibitem{{[M03]}}{Kalle M. Mikkola. %
Riccati equations and optimal control for regular linear
systems and beyond. %
Mittag-Leffler seminar lecture.
Institut Mittag-Leffler, Sweden, February 27, 2003.}

\KMBibitem{{[M03b]}}{Kalle M.\ Mikkola.
Reciprocal and resolvent Riccati equations
 for well-posed linear systems.
In preparation, 2003.} %

\KMBibitem{{[MSW03]}}{Jarmo Malinen, Olof Staffans and George Weiss.
When is a linear system conservative?
Report No. 46, 02/03.
Institut Mittag-Leffler, Sweden, 2003.}

\KMBibitem{{[Q03]}}{Alban Quadrat.
On a general structure of the stabilizing controllers based on stable range.
To appear in SIAM J. Control Optim., 2003.}

\KMBibitem{{[O96]}}{
Raimund J. Ober.
System-theoretic aspects of completely symmetric 
systems.  Recent developments in operator theory and its 
applications (Winnipeg, MB, 1994).
{\em Oper. Theory Adv.},
  233--262, 
Appl., 87, Birkhäuser, Basel, 1996.}

\KMBibitem{{[OC04]}}{Ruth F. Curtain and Mark R. Opmeer. %
New Riccati equations for well-posed linear systems.
To appear, 2004.} %

\KMBibitem{{[RR85]}}{Marvin Rosenblum and James Rovnyak.
{\em Hardy Classes and Operator Theory}.
Oxford University Press,
 New York, 1985.}

\KMBibitem{[S89]}{Malcolm C. Smith.
On stabilization and the existence of coprime factorizations.
{\em IEEE Trans. Autom. Control},
 34:1005--1007, 1989.}

\KMBibitem{{[S97]}}{Olof J.\ Staffans. 
Quadratic optimal control of stable well-posed linear systems.
{\em Trans.\ Am.\ Math.\ Soc.},
349:3679--3715, 1997.}

\KMBibitem{{[S98a]}}{Olof J.\ Staffans.
Coprime factorizations and well-posed linear systems.
{\em SIAM J. Control Optim.},
36:1268--1292, 1998.}

\KMBibitem{{[S98b]}}{Olof J.\ Staffans.
Quadratic optimal control of well-posed linear systems.
{\em SIAM J. Control Optim.},
37:131--164, 1998.} %

\KMBibitem{{[S98c]}}{Olof J.\ Staffans.
Feedback representations of critical controls for
                  well-posed linear systems.
{\em Internat.\ J. Robust Nonlinear Control},
 8:1189--1217, 1998.}

\KMBibitem{{[S04]}}{Olof J. Staffans. %
Well-Posed Linear Systems.
Cambridge university press.
To appear in 2004. %
{\tt %
http://www.abo.fi/\~ {}staffans/psfiles/wellpos.ps}, 2003/08/01.}

\KMBibitem{{[SW03]}}{Olof J. Staffans and George Weiss. %
Transfer functions of regular linear systems.
{Part III}: inversions and duality.
To appear in Integral Equations and Operator 
Theory 2003, 41 pp.} %

\KMBibitem{{[W94a]}}{George Weiss. %
Transfer functions of regular linear systems.  
{P}art {I}: Characterizations of regularity.
{\em Trans.\ Am.\ Math.\ Soc.},
 342:827--854, 1994.}

\KMBibitem{{[W94b]}}{George Weiss. %
Regular linear systems with feedback.
{\em Math.\ Control Signals Systems},
7:23--57, 1994.}

\KMBibitem{{[WC99]}}{George Weiss and Ruth F. Curtain.
Exponential stabilization of vibrating systems by collocated feedback.
{\em Proceedings of the 7th Mediterranean
 Conference on Control and Automation (MED99)},
1705--1722, Haifa, Israel, 1999.}

\KMBibitem{{[WR00]}}{George Weiss and Richard Rebarber.
Optimizability and estimatability for infinite-dimensional linear systems.
{\em SIAM J. Control Optim.},
 39:1204--1232, 2000.}

\KMBibitem{{[WST01]}}{George Weiss, Olof J.\ Staffans and Marius Tucsnak.
Well-posed linear systems---a survey with emphasis on conservative systems.
{\em Int. J. Appl. Math. Comput. Sci.},
11:7--33, 2001.}

\KMBibitem{{[WW97]}}{Martin Weiss and George Weiss. 
Optimal control of stable weakly regular linear systems.
{\em Math.\ Control Signals Systems},
10:287--330, 1997.}

\KMBibitem{{[Z96]}}{Hans Zwart.
Linear Quadratic Optimal Control for Abstract Linear Systems.
In {\em Modelling and Optimization of Distributed Parameter
Systems with Applications to Engineering},
 Chapman {\&} Hall, New York, 175--182, 1996.}

\end{itemlistbibitem}

\end{document} %

